\documentclass[10pt]{amsart}
\title[Quantum Periods for Fano Manifolds]{Quantum Periods for
  3-Dimensional Fano Manifolds}

\author[Coates]{Tom Coates}
\address{Department of Mathematics\\
Imperial College London\\
180 Queen's Gate\\
London SW7 2AZ
\\UK}
\email{t.coates@imperial.ac.uk}

\author[Corti]{Alessio Corti}
\address{Department of Mathematics\\
Imperial College London\\
180 Queen's Gate\\
London SW7 2AZ\\
UK}
\email{a.corti@imperial.ac.uk}

\author[Galkin]{Sergey Galkin}
\address{IPMU, MIPT, IUM, Laboratory of Algebraic Geometry, and Universit\"at Wien}
\email{Sergey.Galkin@phystech.edu}

\author[Kasprzyk]{Alexander Kasprzyk}
\address{Department of Mathematics\\
Imperial College London\\
180 Queen's Gate\\
London SW7 2AZ\\
UK}
\email{a.m.kasprzyk@imperial.ac.uk}

\usepackage{amsrefs}
\usepackage{graphicx}
\usepackage{color}
\usepackage{amsfonts}
\usepackage{amssymb}
\usepackage{amscd}
\usepackage{array}
\usepackage{subfigure}
\usepackage{xypic}
\usepackage[margin=2.5cm]{geometry}
\usepackage{booktabs}
\usepackage{hyperref}
\usepackage{longtable}
\usepackage{pdflscape}
\usepackage{colortbl}

\allowdisplaybreaks[3]
\DeclareMathOperator{\Amp}{Nef}
\newcommand{\be}{{\boldsymbol{e}}}

\newcommand{\CC}{\mathbb{C}}
\newcommand{\cC}{\mathcal{C}}
\newcommand{\Cstar}{\CC^\times}
\DeclareMathOperator{\ev}{ev}
\DeclareMathOperator{\Ext}{Ext}
\newcommand{\FF}{\mathbb{F}}
\DeclareMathOperator{\Fl}{Fl}
\newcommand{\hG}{\widehat{G}}
\DeclareMathOperator{\GL}{GL}
\DeclareMathOperator{\Gr}{Gr}
\DeclareMathOperator{\Hom}{Hom}
\DeclareMathOperator{\id}{id}
\DeclareMathOperator{\Ker}{Ker}
\newcommand{\LL}{\mathbb{L}}
\newcommand{\cO}{\mathcal{O}}
\DeclareMathOperator{\OGr}{OGr}
\newcommand{\PP}{\mathbb{P}}
\DeclareMathOperator{\Pic}{Pic}
\newcommand{\QQ}{\mathbb{Q}}
\newcommand{\RR}{\mathbb{R}}
\DeclareMathOperator{\rk}{rk}
\DeclareMathOperator{\Spec}{Spec}
\DeclareMathOperator{\SpGr}{SpGr}
\newcommand{\TT}{\mathbb{T}}
\newcommand{\tw}{\text{\rm tw}}
\newcommand{\cV}{\mathcal{V}}
\newcommand{\Vstd}{V_{\text{\rm std}}}
\newcommand{\Vtriv}{V_{\text{\rm triv}}}
\newcommand{\vir}{\text{\rm vir}}
\newcommand{\vol}{\text{\rm vol}}
\newcommand{\ZZ}{\mathbb{Z}}

\newcommand{\all}{{1 \cdots r}}
\newcommand{\GIT}{/\!\!/}

\newcommand{\MM}[2]{\mathrm{MM}_{#1\text{--}#2}} 

\theoremstyle{plain}
\newtheorem{theorem}{Theorem}[section]
\newtheorem{proposition}[theorem]{Proposition}
\newtheorem{lemma}[theorem]{Lemma}
\newtheorem*{lem*}{Lemma}

\newtheorem{corollary}[theorem]{Corollary}

\theoremstyle{definition}

\newtheorem{remark}[theorem]{Remark}
\newtheorem*{remark*}{Remark}
\newtheorem{example}[theorem]{Example}
\newtheorem{assumptions}[theorem]{Assumptions}

\begin{document}

\begin{abstract}
  The quantum period of a variety $X$ is a generating function for
  certain Gromov--Witten invariants of $X$ which plays an important
  role in mirror symmetry. In this paper we compute the quantum
  periods of all 3-dimensional Fano manifolds.  In particular we show
  that 3-dimensional Fano manifolds with very ample anticanonical
  bundle have mirrors given by a collection of Laurent polynomials
  called Minkowski polynomials.  This was conjectured in joint work
  with Golyshev.  It suggests a new approach to the classification of
  Fano manifolds: by proving an appropriate mirror theorem and then
  classifying Fano mirrors.

  Our methods are likely to be of independent interest.  We rework the
  Mori--Mukai classification of \mbox{3-dimensional} Fano manifolds,
  showing that each of them can be expressed as the zero locus of a
  section of a homogeneous vector bundle over a GIT quotient $V/G$,
  where $G$ is a product of groups of the form $\GL_n(\CC)$ and $V$ is
  a representation of $G$.  When $G=\GL_1(\CC)^r$, this expresses the
  Fano 3-fold as a toric complete intersection; in the remaining
  cases, it expresses the Fano 3-fold as a tautological subvariety of
  a Grassmannian, partial flag manifold, or projective bundle thereon.
  We then compute the quantum periods using the Quantum Lefschetz
  Hyperplane Theorem of Coates--Givental and the Abelian/non-Abelian
  correspondence of Bertram--Ciocan-Fontanine--Kim--Sabbah.
\end{abstract}

\maketitle

\renewcommand\thesection{\Alph{section}}

\section{Introduction}
The quantum period of a Fano manifold $X$ is a generating function for
Gromov--Witten invariants. It is a deformation invariant of $X$ that 
carries detailed information about quantum cohomology.  In this paper
we give closed formulas for the quantum periods for all 3-dimensional
Fano manifolds.  As a consequence we prove a conjecture, made jointly
with Golyshev, that identifies Laurent polynomials which correspond
under mirror symmetry to each of the 98~deformation families of
3-dimensional Fano manifolds with very ample anticanonical bundle.  We
also exhibit Laurent polynomial mirrors for the remaining
7~deformation families.  Our arguments rely on the classification of
\mbox{3-dimensional} Fano manifolds, due to Iskovskikh and
Mori--Mukai: this is a difficult theorem whose proof, even today,
requires delicate arguments in explicit birational geometry.  On the
other hand our mirror Laurent polynomials have a simple combinatorial
definition and classification.  Given a suitable mirror theorem this
classification would give a straightforward, combinatorial, and
uniform alternative proof of the classification of 3-dimensional Fano
manifolds.  The general outlines of such a mirror theorem are beginning to emerge~\citelist{\cite{Kontsevich:ICM}\cite{SYZ}\cite{Auroux:complement}\cite{Auroux:special}\cite{Katzarkov--Kontsevich--Pantev}}, as are some promising approaches to proving it~\citelist{\cite{Kontsevich--Soibelman:1}\cite{Kontsevich--Soibelman:2}\cite{Gross--Siebert:affine}\cite{Gross}\cite{Gross--Siebert:1}\cite{Gross--Siebert:2}}.

Let $X$ be a Fano manifold, that is, a smooth projective variety such
that the anticanonical bundle ${-K_X}$ is ample.  The quantum period
$G_X(t)$ of $X$, defined in \S\ref{sec:quantum_period}~below, is a
generating function for certain genus-zero Gromov--Witten invariants
of $X$.  It satisfies a differential equation:
\begin{equation}
  \label{eq:QDE}
  \left(\sum_{k=0}^r t^k p_k(D) \right)G_X = 0 
\end{equation}
where $D = t {d \over dt}$ and the $p_k$ are polynomials, called the
quantum differential equation for $X$.  The quantum differential
equation carries information about the quantum cohomology of $X$: the
local system of solutions to the quantum differential equation is an
irreducible piece of the restriction of the Dubrovin connection (in
the Frobenius manifold given by the quantum cohomology of $X$) to the
line in $H^\bullet(X)$ spanned by $c_1(X)$.  In \S\S1--105 below we
give closed formulas for the quantum periods of the 105 deformation
families of 3-dimensional Fano manifolds.

In joint work with Golyshev \cite{CCGGK} we introduced \emph{Minkowski
  polynomials}: these are a collection of Laurent polynomials $f$ in
three variables such that the Newton polytope $\Delta$ of $f$ is a
reflexive polytope, defined\footnote{Some of these Laurent polynomials
  correspond under mirror symmetry to 3-dimensional Fano manifolds
  which admit a small toric degeneration \cite{Batyrev}.  These
  Laurent polynomials were considered earlier by Galkin
  \citelist{\cite{Galkin:preprint}\cite{Galkin:thesis}}.}  in terms of
Minkowski decompositions of the facets of $\Delta$.  Given a Laurent
polynomial $f$, one can define the period of $f$:
\[
\pi_f(t) = \Bigl(\frac{1}{2\pi \tt{i}}\Bigr)^n \int_{|x_1|=\cdots =|x_n|=1}
\frac1{1-tf(x_1,\ldots,x_n)}
\frac{dx_1}{x_1} \cdots   \frac{dx_n}{x_n}
\]
and this satisfies a differential equation called the Picard--Fuchs
equation:
\begin{equation}
  \label{eq:PF}
  \left(\sum_{k=0}^r t^k P_k(D) \right)\pi_f = 0 
\end{equation}
where the $P_k$ are polynomials.  There are 3747 Minkowski polynomials
(up to monomial change of variables) but
Akhtar--Coates--Galkin--Kasprzyk showed that these Laurent polynomials
together generate only $165$ periods \cite{ACGK}.  That is, Minkowski
polynomials fall into 165 equivalence classes where $f$ and $g$ are
equivalent if and only if they have the same period.  The quantum
differential equation \eqref{eq:QDE} of a K\"ahler manifold has the
property that every complex root of the polynomial $p_0$ is an
integer---this reflects the fact that the quantum cohomology algebra
of $X$ carries an integer grading---and we say that a Laurent
polynomial $f$ is of \emph{manifold type} if the Picard--Fuchs
operator \eqref{eq:PF} has the property that every complex root of
$P_0$ is an integer.  Coates--Galkin--Kasprzyk have computed the
Picard--Fuchs operators for the Minkowski polynomials numerically
\cite{CGK}.  Their results, which are computer-assisted rigorous and
which pass a number of stringent checks, show that exactly 98 of the
165 Minkowski periods are of manifold type.

We conjectured, jointly with Golyshev, that the 98 Minkowski periods
of manifold type\footnote{We expect that the remaining Minkowski
  periods correspond to smooth $3$-dimensional Fano orbifolds.}
correspond under mirror symmetry to the 98 deformation families of
3-dimensional Fano manifolds with very ample anticanonical bundle
\cite{CCGGK}.  That is, there is a one-to-one correspondence between
deformation families of 3-dimensional Fano manifolds $X$ with very
ample anticanonical bundle and equivalence classes of Minkowski
polynomials $f$, such that\footnote{This is a very weak notion of
  mirror symmetry.  It is natural to conjecture much more: that the
  Minkowski polynomials $f$ give mirrors to the Fano manifolds $X$ in
  the sense of Kontsevich's Homological Mirror Symmetry program.}  the
Fourier--Laplace transform $\hG_X$ of the quantum period of $X$
coincides with the period $\pi_f$ of $f$.  Assuming the numerical
calculations of Minkowski periods in \cite{CGK}, our results here
prove this conjecture.

\subsection*{The Classification of Fano $3$-Folds}

There are exactly $105$ deformation families of Fano
\mbox{3-folds}. Of these, $17$ parameterise \mbox{3-folds} $X$ with
Picard rank $\rho (X) = b_2 (X)=1$. All but one of these $17$ families
were known to Fano himself. The first modern rank-$1$ classification,
in the style of Fano's double projection from a line, is due to
Iskovskikh \citelist{\cite{Isk:1}\cite{Isk:2}\cite{Isk:3}}. More
recently Mukai, in a program announced in \cite{Mukai:natacad} and
still ongoing, re-proved the rank-$1$ classification from the study of
exceptional vector bundles
\citelist{\cite{Mukai:Trieste}\cite{Mukai:CS}\cite{Mukai-CG}\cite{Mukai-CSI}\cite{Mukai:new_developments}\cite{Mukai-CSII}\cite{Mukai-G2}}
. In
particular, Mukai gave new model constructions for some of the
rank-$1$ Fano \mbox{3-folds} as linear sections of homogeneous spaces;
we make use of these models below. Mori and Mukai
\citelist{\cite{MM:Manuscripta}\cite{MM81}\cite{MM84}\cite{Mori--Mukai:erratum}\cite{MM:fanoconf}}
proved that there are precisely $88$ families of nonsingular Fano
\mbox{3-folds} of rank $\geq 2$; their proof was a spectacular display
of the power of Mori's then-new theory of extremal rays.

The model constructions given by Mori and Mukai are, however, not
well-suited for the calculation of quantum periods. Indeed, these
model constructions are in terms of extremal rays: typically $X$ is
constructed by giving an extremal contraction $f\colon X\to Y$, for
instance the blow up of some curve in $Y$. For example, consider
family number~13 in the table of 3-dimensional Fano manifolds of
Picard rank~$3$ in \cite{MM:fanoconf}:

\subsubsection*{Rank~$3$, number~13: Mori--Mukai construction}

$X$ is the blow-up of a hypersurface $W\subset \PP^2\times \PP^2$ with
centre a curve $C$ of bi-degree $(2,2)$ on it such that
$C\hookrightarrow W \to \PP^2\times
\PP^2\overset{p_i}{\longrightarrow} \PP^2$ is an embedding for both
$i=1, 2$, where $p_i$ is the projection to the $i$th factor of the product
$\PP^2 \times \PP^2$.

\smallskip

This construction, elegant though it is, and natural from the point of
view of extremal rays, is not well-adapted for doing calculations in
Gromov--Witten theory. There are procedures for computing
Gromov--Witten invariants of blow-ups
\citelist{\cite{Gathmann}\cite{Hu:curves}\cite{Hu:surfaces}\cite{Lai}\cite{Manolache}}
but, because they are not based on a satisfactory structural
understanding of blow-ups on the Gromov--Witten side, they are very
difficult to use.  Instead, our preferred tools are those for which we
have a good structural understanding on the Gromov--Witten side:
Givental's mirror theorem \cite{Givental:toric}, the Quantum Lefschetz
theorem of Coates--Givental \cite{Coates--Givental}, and the
Abelian/non-Abelian correspondence of
Bertram--Ciocan-Fontanine--Kim--Sabbah \cite{CFKS}. These tools
require that $X$ be constructed inside the GIT quotient $F=V\GIT G$ of
a vector space $V$ by the action of a complex Lie group $G$ as the
zero-locus of a general section of a homogeneous vector bundle $E \to
V \GIT G$.  Thus we rework the Mori--Mukai classification of
3-dimensional Fano manifolds, proving:

\pagebreak[2]
\begin{theorem}\label{thm:models}
  Let $X$ be a 3-dimensional Fano manifold. Then there exist:
  \begin{itemize}
  \item a vector space $V=\CC^n$;
  \item a representation of $G=\prod_{i=1}^r \GL_{k_i}(\CC)$ on $V$; and
  \item a representation $\rho$ of $G$;
  \end{itemize}
  such that $X$ is the vanishing locus, inside a GIT quotient $F=V\GIT
  G$ with respect to a suitably chosen stability condition, of a
  section of the vector bundle $E \to F$ determined by $\rho$.
\end{theorem}

We think of $F$ as what Miles Reid would call a \emph{key variety}: by
construction, $F$ is endowed with a universal property characterising
the embedding $X\hookrightarrow F$. Both the algebraic geometry and
the Gromov--Witten theory of $X$ are inherited from $F$ through the
universal property.

The proof of Theorem~\ref{thm:models} occupies a substantial portion
of this paper. For many of the 105 families the proof is
straightforward; for a few families it is rather
tricky. In the majority of cases, $G = \GL_1(\CC)^r$ and so $X$ is a
complete intersection in a toric variety (and in practice a complete
intersection of codimension at most~$3$). Here is a typical example:

\subsubsection*{Rank~$3$, number~13: our construction}

$X$ is the codimension-$3$ complete intersection in $\PP^2\times \PP^2\times \PP^2$ of general sections of the line bundles $\cO(1,1,0)$, $\cO(1,0,1)$, and $\cO(0,1,1)$.

\smallskip

An immediate consequence of Theorem~\ref{thm:models} is that the
moduli space of 3-dimensional Fano manifolds is unirational: the
obvious map from $\PP\big(H^0(F,E)\big)$ to the moduli space of $X$ is
dominant.

\subsection*{Highlights} With our model constructions in hand, we then
compute the quantum periods.  Most of these calculations are routine,
but a number are more interesting.  The varieties 2--2
(\S\ref{sec:2-2}), 3--2 (Example~\ref{ex:ql_with_mirror_map}), 3--5
(\S\ref{sec:3-5}), and 4--2 (\S\ref{sec:4-2}) require sophisticated
applications of the Quantum Lefschetz theorem.  Challenging (and new)
applications of the Abelian/non-Abelian correspondence include
Theorem~\ref{thm:rank_1_A_nA}, which gives a uniform treatment of
seven of the seventeen 3-dimensional Fano manifolds of rank-$1$, and
the varieties 2--17 (\S\ref{sec:2-17}), 2--20 (\S\ref{sec:2-20}),
2--21 (\S\ref{sec:2-21}), 2--22 (\S\ref{sec:2-22}), and 2--26
(\S\ref{sec:2-26}).  

We draw the reader's attention, too, to
\S\ref{sec:moduli_not_unirational}, where we exhibit an example of a
high-dimensional Fano manifold with non-unirational moduli space.  In
essence, this means that there is no explicit\footnote{Our work here
  relies on the explicit construction of 3-dimensional Fano manifolds
  given in Theorem~\ref{thm:models}.  But we hope that, in the future,
  a more conceptual approach will be possible.  Such an approach is
  likely to construct Fano manifolds via deformation methods, as in
  the Gross--Siebert mirror symmetry program
  \citelist{\cite{Gross--Siebert:affine}\cite{Gross--Siebert:1}\cite{Gross--Siebert:2}},
  as opposed to explicit descriptions in the style of Theorem~\ref{thm:models}.
} way to write down a general Fano $n$-manifold for $n$ large.

\subsection*{Perspectives and Future Directions}

As discussed above, Minkowski polynomials have a combinatorial
definition and are classified directly from this definition.  Given an
appropriate mirror theorem, therefore, we could reverse the
perspective of this paper and recover the classification of
3-dimensional Fano manifolds from the classification of their mirror
Laurent polynomials. Even once such a mirror theorem has been proved,
the calculations in this paper are likely to remain a very  efficient way
\emph{in practice} to determine the mirror partner to a 3-dimensional
Fano manifold.  Our results suggest, too, that one should search for
higher-dimensional Fano manifolds systematically by searching for
their Laurent polynomial mirrors. This is discussed in our joint work
with Golyshev, where we outline a program to classify 4-dimensional
Fano manifolds using these ideas \cite{CCGGK}.

We know of no \emph{a~priori} reason why every 3-dimensional Fano
manifold admits a construction as in Theorem~\ref{thm:models}. At
present this can be proven only post-classification, by a case-by-case
analysis. The obvious generalization of Theorem~\ref{thm:models} fails
in high dimensions (see \S\ref{sec:moduli_not_unirational} for an
example in dimension~66) but it may still hold in low dimensions.  In
particular, does the generalization of Theorem~\ref{thm:models} hold
in dimension~4? For now perhaps the following remarks are not out of
place. Since the beginning of the subject people have asked what can
birational geometry do for Gromov--Witten theory. For instance a
natural question that was asked early on was how do Gromov--Witten
invariants transform under birational maps, for instance crepant
birational morphisms or blow-ups of nonsingular centres. By now we
have learned that these questions are often very subtle; in the case
of blow-ups of a smooth centre we have a procedure but not a good
structural understanding of the problem. On the other hand, in some
areas, we have made good progress in Gromov--Witten calculus: the
Abelian/non-Abelian correspondence being the most general and best
example. Perhaps now is the right time to ask what can Gromov--Witten
theory do for birational geometry: what view of birational
classification do we get if we take seriously\footnote{For instance,
  our model constructions of the 3-dimensional Fano manifolds of
  Picard rank $\geq 2$ can be used to better organise the calculations
  in \cite{Matsuki}---which we found very helpful on many occasions.
  We do not pursue this line here, apart from a few scattered comments
  in the text below.} the perspective of the Abelian/non-Abelian
correspondence? Does something like Theorem~\ref{thm:models} hold, and
if so why?

\subsection*{Remarks on the Rank-$1$ Case} As far as we know, most of
our constructions of 3-dimensional Fano manifolds of Picard rank $\geq
2$ are new.  In Picard rank $1$ this is not the case: all of the
models that we give are either already in the literature or were known
to Mukai.  As we have said, Mukai gave model constructions for some of
the rank-$1$ 3-dimensional Fano manifolds $X$ as linear sections of
homogeneous manifolds $G/P$ in their canonical projective
embedding. In other words, $X$ is the complete intersection of
$G/P\subset \PP^N$ with a linear subspace of the appropriate
codimension in $\PP^N$. Mukai's models are not always the best for our
purposes.  The Abelian/non-Abelian correspondence is currently known
to hold only for Lie groups of type $A$, and so we prefer to exhibit
$X$ as a subvariety of $F=A\GIT G$ where $G$ is a product of groups of
the form $\GL_k(\CC)$.  Our rank-$1$ models are thus in some sense
simpler than Mukai's; in each case they either occur as an
intermediate step in Mukai's published construction or were known to
Mukai.

\subsection*{Remarks on Quantum Periods of Fano Manifolds}

Golyshev, based on a heuristic involving mirror symmetry and modular
forms, gave a conjectural form of the matrices of small quantum
multiplication by the anticanonical class for each of the rank-$1$
Fano \mbox{3-folds} \cite{Golyshev}, and verified it by explicit
calculation of Gromov--Witten numbers (unpublished).  This work is the
fundamental source of the perspective taken in this paper; it is also
an important antecedent to the more precise conjecture (joint with
Golyshev) that we prove here.  The regularized quantum period of
rank-$1$ Fano \mbox{3-folds} was computed by Beauville
\cite{Beauville}, Kuznetsov (unpublished), and Przyjalkowski
\citelist{\cite{Przyjalkowski:68}\cite{Przyjalkowski:LG}\cite{Przyjalkowski:QC}}.
Ciolli has computed the small quantum cohomology rings of 13
higher-rank Fano \mbox{3-folds} \cite{Ciolli}.

\subsection*{Acknowledgements} This research was supported by a Royal
Society University Research Fellowship (TC); ERC Starting Investigator
Grant number~240123; the Leverhulme Trust; AG Laboratory NRU-HSE, RF
government grant, ag.~11.G34.31.0023; EPSRC grant EP/I008128/1; JSPS
Grant-in-Aid for Scientific Research 10554503; the WPI Initiative of
MEXT, Japan; NSF Grant DMS-0600800; NSF FRG Grant DMS-0652633; FWF
Grant P20778; and the ERC Advanced Grant GEMIS. SG thanks Bumsig Kim
and Alexander Kuznetsov for useful conversations.

\subsection*{Plan of the Paper}
Sections~\ref{sec:quantum_period}--\ref{sec:dim12} are devoted to some
preliminaries and examples, mostly to fix our notation. In particular
we summarise all the results from Gromov--Witten theory that we need.
The subsequent sections 1--105 are self-contained essays, one for each
of the deformation families of 3-dimensional Fano manifolds, giving:
the standard known model construction; our model construction; a proof
that the two constructions coincide; the calculation of the
regularized quantum period; and---where appropriate---a match with a
Minkowski period of manifold type. In more detail:
\S\ref{sec:quantum_period} gives the definition of and basic facts
about quantum periods; \S\ref{sec:toric} treats toric Fano manifolds
and Givental's mirror theorem; \S\ref{sec:toric_ci} introduces
notation for Fano complete intersections in toric varieties and
discusses the Quantum Lefschetz theorem; \S\ref{sec:blowing-things-up}
provides some geometric constructions and notation that are used in
our model constructions; \S\ref{sec:A_nA} summarizes the
Abelian/non-Abelian correspondence; and in \S\ref{sec:dim12} we
compute the quantum periods for Fano manifolds of dimensions~$1$
and~$2$.  The table in Appendix~\ref{appendix} exhibits Laurent polynomial mirrors
for each of the~105 deformation families of 3-dimensional Fano manifolds.

\section{The J-Function and the Quantum Period}
\label{sec:quantum_period}

Let $X$ be a smooth projective variety over $\CC$.  For
$\beta \in H_2(X;\ZZ)$, let $X_{0,1,\beta}$ denote the moduli space of
degree-$\beta$ stable maps to $X$ from genus-zero curves with one
marked point \citelist{\cite{Kontsevich--Manin}\cite{Kontsevich:enumeration}}; let
$[X_{0,1,\beta}]^\vir \in H_\bullet(X_{0,1,\beta};\QQ)$ denote the
virtual fundamental class of $X_{0,1,\beta}$
\citelist{\cite{Li--Tian}\cite{Behrend--Fantechi}}; let
$\ev \colon X_{0,1,\beta} \to X$ denote the evaluation map at the marked
point; and let $\psi \in H^2(X_{0,1,\beta};\QQ)$ denote the first
Chern class of the universal cotangent line at the marked point.  The
\emph{$J$-function} of $X$ is a generating function for genus-zero
Gromov--Witten invariants of $X$:
\begin{equation}
  \label{eq:JX}
  J_X(\sigma + \tau) = e^{\sigma/z} e^{\tau/z}
  \Bigg( 1 +
  \sum_{\substack{\beta \in H_2(X;\ZZ): \\ \beta \ne 0}}
  Q^\beta e^{\langle \beta,\tau \rangle}
  \ev_\star \bigg([X_{0,1,\beta}]^\vir \cap
  {1 \over z(z-\psi)}
  \bigg)
  \Bigg)
\end{equation}
Here $\sigma \in H^0(X;\QQ)$, $\tau \in H^2(X;\QQ)$, $Q^\beta$ is the
representative of $\beta$ in the group ring $\QQ[H_2(X;\ZZ)]$, and we expand ${1 \over z(z-\psi)}$ as the series
$\sum_{k \geq 0} z^{-k-2} \psi^k$.  Let $\Lambda_X$ denote the completion of
$\QQ[H_2(X;\ZZ)]$ with respect to the valuation $v$ defined by
$v(Q^\beta) = \langle \beta, \omega \rangle$, where $\omega$ is the K\"ahler class of $X$. The $J$-function is a function on
$H^0(X;\QQ) \oplus H^2(X;\QQ)$ that takes values in
$H^\bullet\big(X;\Lambda_X\big)[\![z^{-1}]\!]$.  It plays a key role
in mirror symmetry \citelist{\cite{Givental:equivariant}
  \cite{Givental:toric} \cite{Cox--Katz}}.  We have:
\begin{equation}
  \label{eq:JXasymptotics}
  J_X(\sigma + \tau) = 1 + (\sigma + \tau) z^{-1} + O(z^{-2})
\end{equation}
where $1$ is the unit element in $H^\bullet(X)$.

Suppose now that $X$ is a Fano manifold, i.e.~a smooth projective variety over $\CC$ such that the anticanonical line bundle ${-K}_X$ is ample. 
Consider the component of the $J$-function
$J_X(\sigma + \tau)$ along the unit class $1 \in H^\bullet(X;\QQ)$.
Set $\sigma = \tau = 0$, $z = 1$, and replace $Q^\beta \in \Lambda_X$ by
$t^{\langle \beta,{-K_X} \rangle}$.  The resulting formal power series
in the variable $t$ is called the \emph{quantum period} of $X$:
\[
G_X(t) = 1 + \sum_{d \geq 2}
\sum_{\substack{\beta \in H_2(X;\ZZ) : \\ \langle \beta, {-K_X} \rangle
  = d}}
t^d \left \langle \phi_\vol \cdot \psi^{d-2} \right \rangle^X_{0,1,\beta}
\]
where $\phi_\vol$ is a top-degree cohomology class on $X$ such that
$\int_X \phi_\vol = 1$, and the correlator denotes a Gromov--Witten
invariant:
\[
\left \langle \phi_\vol \cdot \psi^{d-2} \right \rangle^X_{0,1,\beta} = \int_{[X_{0,1,\beta}]^\vir} \ev^\star (\phi_\vol) \cup \psi^{d-2}
\]
Write the quantum period as:
\begin{align*}
  & G_X(t) = 1 + \sum_{d \geq 2} c_d t^d
\intertext{The \emph{regularized quantum period} of $X$ is:}
  & \hG_X(t) = 1 + \sum_{d \geq 2} d!  c_d t^d
\end{align*}


\subsection{The Big $J$-function and the Small $J$-function}
\label{sec:bigJsmallJ}
Our $J$-function $J_X(t)$ is sometimes called the ``small
$J$-function''; it coincides with the $J$-function defined by Givental
in \cite{Givental:toric}. For the small $J$-function $J_X(t)$, the
parameter $t$ is taken to lie in $H^0(X) \oplus H^2(X)$.  Other
authors consider a ``big J-function'' $J(t)$ where the parameter $t$
ranges over all of $H^\bullet(X)$.  The big $J$-functions $J(t)$
considered by Coates--Givental and Ciocan-Fontanine--Kim--Sabbah
coincide with our $J_X(t)$, except for an overall factor of $z$, when
$t$ is restricted to lie in $H^0(X) \oplus H^2(X)$: to see this, apply
the String Equation and the Divisor Equation
\cite{Pandharipande}*{\S1.2} to the definition of the big $J$-function
\citelist{\cite{Coates--Givental}*{equation~11}\cite{CFKS}*{equation~5.2.1}}.
The overall factor of $z$ here comes from an unfortunate mismatch of
conventions.

\section{Fano and Nef Toric Manifolds}
\label{sec:toric}

Let $T = (\Cstar)^r$.  Write $\LL = \Hom(\Cstar,T)$ for the lattice of
subgroups of $T$, and write $\LL^\vee$ for the dual lattice
$\Hom(T,\Cstar)$.  Elements of $\LL^\vee$ are characters of $T$.
Consider an $r \times N$ integer matrix $M$ of rank $r$ such that the
columns of $M$ span a strictly convex cone $\cC$ in $\RR^r$.  The
columns of $M$ define characters of $T$, via the canonical isomorphism
$\LL^\vee \cong \ZZ^r$, and hence determine an action of $T$ on
$\CC^N$.  Given a stability condition $\omega \in \LL^\vee \otimes
\RR$ we can form the GIT quotient:
\[
X_\omega := \CC^N \GIT_\omega \, T
\]
Any smooth projective toric variety $X$ arises via this construction
for some choice of $M$ and $\omega$; we refer to the matrix $M$ as
\emph{weight data} for $X$ and to $\omega$ as a \emph{stability
  condition} for $X$.

There is a wall-and-chamber decomposition of $\cC \subset \LL^\vee
\otimes \RR$, called the \emph{secondary fan}, and if stability
conditions $\omega_1$ and $\omega_2$ lie in the same chamber then the
GIT quotients $X_{\omega_1}$ and $X_{\omega_2}$ coincide.  Write $c_i
\in \LL^\vee$ for the $i$-th column of $M$, and $\langle
c_{i_1},\dots,c_{i_k}\rangle$ for the $\RR_{\geq 0}$-span of the
columns $c_{i_1},\dots,c_{i_k}$. The walls of the secondary fan are
given by all cones of the form $\langle c_{i_1},\dots,c_{i_k}\rangle$
that have dimension $r-1$.  The chambers of the secondary fan are the
connected components of the complement of the walls; these are
$r$-dimensional open cones in $\cC \subset \LL^\vee \otimes \RR$.  We
always take our stability condition $\omega$ to lie in a
chamber. Given such an $\omega$, the \emph{irrelevant ideal} $I_\omega
\subset \CC[x_1,\ldots,x_N]$ is:
\[
I_\omega = \big( x_{i_1} x_{i_2} \cdots x_{i_r} \mid \omega \in \langle c_{i_1},\ldots,c_{i_r} \rangle \big)
\]
and the \emph{unstable locus} is $V(I_\omega) \subset \CC^N$.  The GIT
quotient $X_\omega$ is:
\begin{equation}
  \label{eq:GIT_quotient}
  X_\omega = \Big( \CC^N \setminus V(I_\omega)\Big) \Big/ T
\end{equation}
The variety $X_\omega$ is nonsingular if and only if
$c_{i_1},\ldots,c_{i_r}$ is an integer basis for $\LL^\vee$ for each
$\{i_1,\ldots,i_r\}$ such that $\omega \in \langle
c_{i_1},\ldots,c_{i_r} \rangle$.

Suppose now that $M$, $\omega$ are respectively weight data and a
stability condition for $X$.  A character $\xi \in \LL^\vee$ defines a
line bundle $L_\xi$ on $X$ and hence a cohomology class $c_1(L_\xi)
\in H^2(X;\QQ)$.  Thus the columns of $M$ define cohomology classes
$D_1,\ldots,D_N \in H^2(X;\QQ)$.  Define the \emph{$I$-function} of $X$:
\[
I_X(\tau) = e^{\tau/z}
\sum_{\substack{\beta \in H_2(X;\ZZ)}}
Q^\beta e^{\langle \beta, \tau \rangle}
{
\prod_{i=1}^{i=N} \prod_{m \leq 0} D_i + m z
\over
\prod_{i=1}^{i=N} \prod_{m \leq \langle \beta, D_i \rangle} D_i + m z
}
\]
Here $\tau \in H^2(X;\QQ)$ and $Q^\beta$ is, as before, the
representative of $\beta$ in the group ring $\QQ[H_2(X;\ZZ)]$. The
$I$-function $I_X$ is a function on $H^2(X;\QQ)$ that takes values in
$H^\bullet\big(X;\Lambda_X\big)[\![z^{-1}]\!]$.  Note that all but
finitely many terms in the infinite products cancel, and that 
\[
\frac{1}{D_i + m z} = \frac{1}{mz} - \frac{D_i}{(mz)^2} + \frac{D_i^2}{(mz)^3} + \cdots
\] 
is well-defined as an element of $H^\bullet(X)[\![z^{-1}]\!]$.

\begin{theorem}[Givental]
  \label{thm:toric_mirror}
  Let $X$ be a toric manifold such that ${-K_X}$ is nef.  Then:
  \[
  J_X\big(\theta(\tau)\big) = I_X(\tau)
  \]
  for some function $\theta \colon H^2(X;\QQ) \to H^0\big(X;\Lambda_X\big) \oplus
  H^2\big(X;\Lambda_X\big)$.  Furthermore, the function $\theta$ is uniquely
  determined by the expansion:
  \[
  I_X(\tau) = 1 + \theta(\tau) z^{-1} + O(z^{-2})
  \]
  If $X$ is Fano then $\theta(\tau) = \tau$.
\end{theorem}

\begin{proof}
  This follows immediately from Givental's mirror theorem for toric
  varieties \cite{Givental:toric}.
\end{proof}

\begin{corollary}
  \label{cor:toric_mirror}
  Let $X$ be a Fano toric manifold and let $D_1,\ldots,D_N \in
  H^2(X;\QQ)$ be the cohomology classes of the
  torus-invariant divisors on $X$.  The quantum period of $X$ is:
  \[
  G_X(t) = \sum_{\substack{
      \beta \in H_2(X;\ZZ) : \\
      \text{$\langle \beta, D_i \rangle \geq 0$ $\forall i$}
    }}
  \frac{
    t^{\langle \beta,{-K_X}\rangle}
  }{
    \prod_{i=1}^N \langle \beta, D_i \rangle !
  }
  \]
\end{corollary}

\begin{proof}
  The quantum period $G_X$ is obtained from the component of the
  $J$-function $J_X(\tau)$ along the unit class $1 \in
  H^\bullet(X;\QQ)$ by setting $\tau = 0$, $z = 1$, and $Q^\beta =
  t^{\langle \beta, {-K_X} \rangle}$.  Now apply
  Theorem~\ref{thm:toric_mirror}.
\end{proof}

\begin{example}[number 36 on the Mori--Mukai list of 3-dimensional
  Fano manifolds of rank 2]
  \label{ex:toric}
  Here $X$ is the projective bundle $\PP(\cO\oplus\cO(2))$ over
  $\PP^2$.  This is a toric variety with weight data:
  \[
  \begin{array}{rrrrrl}
    1 & 1 & 1 & -2 & 0 & \hspace{1.5ex} L\\
    0 & 0 & 0 & 1 & 1 & \hspace{1.5ex} M
  \end{array}
  \]
  and nef cone $\Amp(X)$ spanned by $L$ and $M$.  The $L$ and $M$ next to the weight data here denote the line bundles associated to the standard basis of $\LL^\vee = \ZZ^2$; we use this notation, and its natural extension to the case where $\LL^\vee = \ZZ^r$ with $r \ne 2$, freely throughout the paper.  
  Corollary~\ref{cor:toric_mirror} yields:
  \[
  G_X(t) = \sum_{d_1=0}^\infty \sum_{d_2 = 2 d_1}^\infty
  { t^{d_1 + 2 d_2} \over (d_1!)^3 \, (d_2 - 2 d_1)! \, d_2!}
  \]
  and regularizing gives:
  \[
  \hG(t) = 1 + 2t^{2} + 6t^{4} + 60t^{5} + 20t^{6} + 840t^{7} +
  70t^{8} + 7560t^{9} + \cdots
  \]
\end{example}

\section{Fano Toric Complete Intersections}
\label{sec:toric_ci}

\begin{assumptions}
  \label{assumptions:toric_ci} Throughout \S\ref{sec:toric_ci}, take
  $Y$ to be a smooth projective toric variety such that ${-K_Y}$ is nef, and take
  $X$ to be a smooth Fano complete intersection in $Y$ defined by a
  section of $E = L_1 \oplus \cdots \oplus L_s$ where each $L_i$ is a
  nef line bundle.  Let $\rho_i = c_1(L_i)$, and let $\Lambda = \rho_1
  + \cdots + \rho_s$.  By the Adjunction Formula:
  \[
  {-K_{X}} = \big({-K_Y} - \Lambda\big)\big|_X
  \]
  We assume that the line bundle ${-K_Y}- \Lambda$ on $Y$ is nef on
  $Y$, that is, we assume that $\langle \beta, {-K_Y}- \Lambda \rangle
  \geq 0$ for all $\beta$ in the Mori cone of $Y$.
\end{assumptions}

\subsection{The Quantum Lefschetz Theorem}
\label{sec:quantum_lefschetz}

We will compute the quantum period of $X$ by computing certain twisted
Gromov--Witten invariants of the ambient space $Y$ using the Quantum
Lefschetz theorem \cite{Coates--Givental}.  Consider the
$\Cstar$-action on the total space of $E$ given by rescaling fibers
(with the trivial action on the base).  Let $\lambda$ denote the first
Chern class of the line bundle $\cO(1)$ over $\CC P^\infty \cong B
\Cstar$, so that the $\Cstar$-equivariant cohomology of a point is
$\QQ[\lambda]$, and let $\be(\cdot)$ denote the $\Cstar$-equivariant
Euler class.  Coates--Givental (ibid.) define a complex of
$\Cstar$-equivariant sheaves $E_{0,1,\beta}$ on $Y_{0,1,\beta}$.  In
this case $E_{0,1,\beta}$ is a $\Cstar$-equivariant vector bundle over
$Y_{0,1,\beta}$, and there is a $\Cstar$-equivariant evaluation map
$E_{0,1,\beta} \to \ev^\star E$.  Let $E_{0,1,\beta}'$ be the kernel of this
evaluation map.  The \emph{twisted $J$-function} is:
\begin{equation}
  \label{eq:twisted_J}
  J_{\be,E}(\sigma + \tau) =
  e^{\sigma/z} e^{\tau/z}
  \Bigg( 1 +
  \sum_{\substack{\beta \in H_2(Y;\ZZ): \\ \beta \ne 0}}
  Q^\beta e^{\langle \beta,\tau \rangle}
  \ev_\star \bigg([Y_{0,1,\beta}]^\vir \cap
  \be(E_{0,1,\beta}') \cap
  {1 \over z(z-\psi)}
  \bigg)
  \Bigg)
\end{equation}
Here $\sigma \in H^0(Y;\QQ)$, $\tau \in H^2(Y;\QQ)$, $Q^\beta$ is the
representative of $\beta$ in the group ring $\QQ[H_2(Y;\ZZ)]$, and we
expand ${1 \over z(z-\psi)}$ as the series $\sum_{k \geq 0} z^{-k-2}
\psi^k$.  The twisted $J$-function is\footnote{Coates--Givental
  consider a ``big twisted $J$-function'' $J_{\be,E}(t)$ where the
  parameter $t$ ranges over all of $H^\bullet(X)$.  Exactly as in
  \S\ref{sec:bigJsmallJ} this coincides with our twisted $J$-function,
  up to an overall factor of $z$, when $t$ is restricted to lie in
  $H^0(X) \oplus H^2(X)$.}  a function on $H^0(Y;\QQ) \oplus
H^2(Y;\QQ)$ that takes values in
$H^\bullet\big(Y;\Lambda_Y[\lambda]\big)[\![z^{-1}]\!]$.  It
satisfies:
\begin{equation}
  \label{eq:twistedJasymptotics}
  J_{\be,E}(\sigma + \tau) = 1 + (\sigma + \tau) z^{-1} + O(z^{-2})
\end{equation}
where $1$ is the unit element in $H^\bullet(Y)$.  The twisted
$J$-function admits a non-equivariant limit $J_{Y,X}$ which satisfies:
\begin{equation}
  \label{eq:KKP}
  j_\star J_X\big(j^\star(\sigma+\tau)\big) =
  J_{Y,X}(\sigma+\tau) \cup \prod_{i=1}^{i=s} \rho_i
\end{equation}
Here $j:X \to Y$ is the inclusion, and the equality holds after
applying the homomorphism between $\Lambda_X$ and $\Lambda_Y$ induced
by $j$.  Since we can determine the quantum period $G_X$ from the
component of $J_X$ along the unit class $1 \in H^\bullet(X)$, we can
determine $G_X$ from the component of $J_{Y,X}$ along the unit class
$1 \in H^\bullet(Y)$.

The Quantum Lefschetz theorem determines the twisted $J$-function
$J_{\be,E}$ from the twisted $I$-function:
\[
I_{\be,E}(\tau) = \sum_{\beta \in H_2(Y;\ZZ)} Q^\beta J_\beta(\tau)
\prod_{i=1}^{s} \prod_{m=1}^{\langle \beta,\rho_i\rangle}
(\lambda + \rho_i + m z)
\]
where:
\[
J_Y(\tau) = \sum_{\beta \in H_2(Y;\ZZ)} Q^\beta J_\beta(\tau)
\]
and so, in particular, $J_0(\tau) = e^{\tau/z}$.
\begin{proposition}
  \label{pro:ABC}
  Under Assumptions~\ref{assumptions:toric_ci}, we have:
  \[
  I_{\be,E}(\tau) = A(\tau) + B(\tau) z^{-1} + O(z^{-2})
  \]
  for some functions:
  \begin{align*}
    & A \colon H^2(Y;\QQ) \to H^0\big(Y;\Lambda_Y\big) \\
    & B \colon H^2(Y;\QQ) \to H^0\big(Y;\Lambda_Y[\lambda]\big) \oplus H^2\big(Y;\Lambda_Y[\lambda]\big)
  \end{align*}
  If ${-K_X}$ is the restriction of an ample line bundle on $Y$,
  i.e.~if $\langle \beta, {-K_Y}- \Lambda \rangle > 0 $ for all $\beta$ in
  the Mori cone of $Y$, then $A$ is the constant function with value
  the unit class $1 \in H^0(Y;\QQ)$ and $B(\tau) = \tau + C(\tau) 1$ with:
  \[
  C(\tau) = \sum_{\substack{\beta \in H_2(Y;\ZZ): \\ \langle
      \beta, {-K_Y}- \Lambda\rangle = 1}} n_\beta Q^\beta e^{\langle
    \beta,\tau \rangle}
  \]
  for some rational numbers $n_\beta$.  In general we have:
  \[
  A \equiv 1 \mod \{Q^\beta : \text{$\beta \ne 0$, $\beta$ in the Mori
    cone of $Y$}\}
  \]
\end{proposition}

\begin{proof}
  Combine the fact that $J_0(\tau) = e^{\tau/z} = 1 + \tau z^{-1} +
  O(z^{-2})$ with the fact that $I_{\be,E}$ is homogeneous of degree zero with
  respect to the grading:
  \begin{align*}
    \deg Q^\beta = \langle\beta,{-K_Y}- \Lambda \rangle &&
    \deg z = 1 && \deg \lambda = 1 &&
    \text{$\deg \phi = k$ if $\phi \in H^{2k}(Y;\QQ)$}
  \end{align*}
  With respect to this grading, $A(\tau)$ is homogeneous of degree
  zero and $B(\tau)$ is homogeneous of degree one.
\end{proof}

\begin{theorem}
  \label{thm:toric_ci_ql}
  Under Assumptions~\ref{assumptions:toric_ci}, with $A, B, C$ as in
  Proposition~\ref{pro:ABC}, we have:
  \begin{align*}
    & J_{\be,E}(\theta(\tau)) = \frac{I_{\be,E}(\tau)}{A(\tau)}
    && \text{where } \theta(\tau) = \frac{B(\tau)}{A(\tau)}
  \end{align*}
  If ${-K_X}$ is the restriction of an ample class on $Y$ then $J_{\be,E}(\tau) = e^{-C(\tau)/z} I_{\be,E}(\tau)$.
\end{theorem}

\begin{proof}
  The first statement is a slight generalization of Corollary~7 in
  Coates--Givental \cite{Coates--Givental}, and is proved in exactly
  the same way.  When ${-K_X}$ is the restriction of an ample class on
  $Y$, combining the first statement with Proposition~\ref{pro:ABC}
  gives:
  \[
  J_{\be,E}(\tau + C(\tau) 1) = I_{\be,E}(\tau)
  \]
  The String Equation~\cite{Pandharipande}*{\S1.2} now implies that:
  \[
   J_{\be,E}(\tau + C(\tau) 1) = e^{C(\tau)/z}  J_{\be,E}(\tau)
   \]
   completing the proof.
\end{proof}

The twisted $I$-function admits a non-equivariant limit:
\[
I_{Y,X}(\tau) = \sum_{\beta \in H_2(Y;\ZZ)} Q^\beta J_\beta(\tau)
\prod_{i=1}^{s} \prod_{m=1}^{\langle \beta,\rho_i\rangle}
(\rho_i + m z)
\]

\begin{corollary}
  \label{cor:toric_ci_JYX}
   Under Assumptions~\ref{assumptions:toric_ci}, with $A, B, C$ as in
  Proposition~\ref{pro:ABC}, we have:
  \[
  I_{Y,X}(\tau) = A(\tau) + B'(\tau) z^{-1} + O(z^{-2})
  \]
  where $B'(\tau) = B(\tau)\big|_{\lambda=0}$, and:
  \begin{align*}
    & J_{Y,X}(\theta(\tau)) = \frac{I_{Y,X}(\tau)}{A(\tau)}
    && \text{where } \theta(\tau) = \frac{B'(\tau)}{A(\tau)}
  \end{align*}
  If ${-K_X}$ is the restriction of an ample class on $Y$ then
  $J_{Y,X}(\tau) = e^{-C(\tau)/z} I_{Y,X}(\tau)$.
\end{corollary}

\begin{proof}
  Take the non-equivariant limit $\lambda \to 0$ of
  Proposition~\ref{pro:ABC} and Theorem~\ref{thm:toric_ci_ql}.
\end{proof}

\begin{corollary}
  \label{cor:QL}
  Let the toric complete intersection $X$ and the toric variety $Y$ be such that Assumption~\ref{assumptions:toric_ci} holds.  Let $D_1,\ldots,D_N \in H^2(Y;\QQ)$ be the cohomology classes
  of the torus-invariant divisors on $Y$, and let the classes $\rho_i$ and $\Lambda = \rho_1 + \cdots + \rho_s$ be as in Assumption~\ref{assumptions:toric_ci}.  Suppose that the line bundles ${-K}_Y$ and ${-K_Y}- \Lambda$ on $Y$ are ample.  Then the quantum
  period of $X$ is:
  \[
  G_X(t) = e^{{-c} t} \sum_{\substack{
      \beta \in H_2(Y;\ZZ) : \\
      \text{$\langle \beta, D_i \rangle \geq 0$ $\forall i$}
      }}
    t^{\langle \beta,{-K_Y}- \Lambda\rangle}
    \frac{
      \prod_{j=1}^s \langle \beta, \rho_j \rangle !
    }{
      \prod_{i=1}^N \langle \beta, D_i \rangle !
    }
    \]
    where $c$ is the unique rational number such that the right-hand
    side has the form $1 + O(t^2)$.
\end{corollary}

\begin{proof}
  Recall that $G_X$ is obtained from the component of the $J$-function
  $J_X(\sigma + \tau)$ along the unit class $1 \in H^\bullet(X;\QQ)$
  by setting $\sigma = \tau = 0$, $z = 1$, and $Q^\beta = t^{\langle
    \beta, {-K_X} \rangle}$.  In view of equation \eqref{eq:KKP}, we need
  the component of $J_{Y,X}(0)$ along $1 \in H^\bullet(Y;\QQ)$.
  Computing $J_Y(\tau)$ using Theorem~\ref{thm:toric_mirror}, we see
  that:
  \[
  I_{Y,X}(\tau) = e^{\tau/z}
  \sum_{\substack{\beta \in H_2(X;\ZZ)}}
  Q^\beta e^{\langle \beta, \tau \rangle}
  {
    \prod_{i=1}^{i=N} \prod_{m \leq 0} D_i + m z
    \over
    \prod_{i=1}^{i=N} \prod_{m \leq \langle \beta, D_i \rangle} D_i + m z
  }
  \prod_{j=1}^{s} \prod_{m=1}^{\langle \beta,\rho_j\rangle}
  (\rho_j + m z)
  \]
  Applying Corollary~\ref{cor:toric_ci_JYX}, we see that the component
  of $J_{Y,X}(\tau)$ along $1 \in H^\bullet(Y;\QQ)$ is:
  \begin{align*}
    & e^{-C(\tau)/z}     \sum_{\substack{
        \beta \in H_2(Y;\ZZ) : \\
        \text{$\langle \beta, D_i \rangle \geq 0$ $\forall i$}
      }}
    Q^\beta e^{\langle \beta, \tau \rangle}
    \frac{
      \prod_{j=1}^{s} \prod_{m=1}^{\langle \beta,\rho_j\rangle}
      (m z)
    }{
      \prod_{i=1}^{N} \prod_{1 \leq m \leq \langle \beta, D_i \rangle} (m z)
    }
    \intertext{where:}
    & C(\tau) = \sum_{\substack{\beta \in H_2(Y;\ZZ): \\ \langle
        \beta, {-K_Y}- \Lambda\rangle = 1}} n_\beta Q^\beta e^{\langle
      \beta,\tau \rangle}
  \end{align*}
  for rational numbers $n_\beta$ as in
  Proposition~\ref{pro:ABC}. Setting $\tau = 0$, $z = 1$, and $Q^\beta = t^{\langle
    \beta, {-K_Y}- \Lambda \rangle}$ yields:
  \[
  G_X(t) = e^{{-c} t} \sum_{\substack{
      \beta \in H_2(Y;\ZZ) : \\
      \text{$\langle \beta, D_i \rangle \geq 0$ $\forall i$}
    }}
  t^{\langle \beta,{-K_Y}- \Lambda\rangle}
  \frac{
    \prod_{j=1}^s \langle \beta, \rho_j \rangle !
  }{
    \prod_{i=1}^N \langle \beta, D_i \rangle !
  }
  \]
  for some rational number $c$.  We saw in \S\ref{sec:quantum_period}
  that the right-hand side has no linear term in $t$; this determines
  $c$.
\end{proof}

\begin{remark} \label{rem:twisting_by_multiples_of_K} Comparing
  Corollary~\ref{cor:QL} with Corollary~\ref{cor:toric_mirror}, we see
  that if each of the line bundles $L_1,\ldots,L_s$ in
  Corollary~\ref{cor:QL} is a tensor power or fractional tensor power
  of ${-K_Y}$ then we can compute the quantum period $G_X$ from the
  quantum period $G_Y$ and the line bundles $L_i$ alone, without
  needing to know the full $J$-function $J_Y$.
\end{remark}

\begin{example}
  \label{ex:W22}
  Let $X$ be the divisor on $Y = \PP^2 \times \PP^2$ of bigree
  $(2,2)$.  The toric variety $Y$ has weight data:
  \[
  \begin{array}{rrrrrrl}
    1 & 1 & 1 & 0 & 0 & 0 & \hspace{1.5ex} L\\
    0 & 0 & 0 & 1 & 1 & 1 & \hspace{1.5ex} M \\
  \end{array}
  \]
  and the nef cone $\Amp(Y)$ is spanned by $L$ and $M$.  The variety
  $X$ is a member of the ample linear system $|2L+2M|$, and ${-(}K_Y + X)
  \sim L+M$ is ample.  Corollary~\ref{cor:QL} yields:
  \begin{align*}
    & G_X(t) = e^{-4t} \sum_{l=0}^\infty \sum_{m=0}^\infty
    t^{l+m}
    \frac{(2l+2m)!}{(l!)^3 (m!)^3}
    \intertext{and regularizing gives:}
    & \hG_X(t) = 1+44 t^2+528 t^3+11292 t^4+228000 t^5+4999040
    t^6\\ & \qquad \qquad \qquad +112654080 t^7+2613620380 t^8+61885803840 t^9+ \cdots
  \end{align*}
\end{example}

\begin{example}
  \label{ex:ql_with_mirror_map}
  Let $F$ be the toric variety with weight data:
  \[
  \begin{array}{rrrrrrrl}
    1 & 1 & 0 & 0 & -1 & 0 & 0 &\hspace{1.5ex} L \\
    0 & 0 & 1 & 1 & -1 & 0 & 0 & \hspace{1.5ex} M \\
    0 & 0 & 0 & 0 &  1  & 1 & 1 & \hspace{1.5ex} N
  \end{array}
  \]
  and nef cone $\Amp F$ spanned by $L$, $M$, and $N$.  Let $X$ be a
  member of the nef linear system $|M+2N|$.  We have that
  $-K_F=L+M+3N$ is ample, so $F$ is a Fano variety, and that
  $-K_F - \Lambda\sim L+N$ is nef but not ample on $F$.  As is discussed in detail in \S\ref{sec:3-2}, even though ${-K_F} - \Lambda$ is not ample on $F$, it
  becomes ample  when restricted to $X$; thus the variety $X$ is Fano.

  Write $p_1$, $p_2$, $p_3 \in H^\bullet(F;\ZZ)$ for the first Chern
  classes of $L$, $M$, $N$ respectively; these classes form a basis
  for $H^2(F;\ZZ)$.  Identify the group ring $\QQ[H_2(F;\ZZ)]$ with
  the polynomial ring $\QQ[Q_1,Q_2,Q_3]$ via the $\QQ$-linear map that
  sends the element $Q^\beta \in \QQ[H_2(F;\ZZ)]$ to $Q_1^{\langle
    \beta,p_1\rangle} Q_2^{\langle \beta,p_2\rangle} Q_3^{\langle
    \beta,p_3\rangle}$.  Theorem~\ref{thm:toric_mirror} gives:
  \begin{multline*}
    J_F(\tau) = e^{\tau/z}
    \sum_{(l, m, n) \in \ZZ^3}
    Q_1^l Q_2^m Q_3^n e^{\langle \beta,\tau \rangle}
    \frac{\prod_{k = -\infty}^0 (p_1 + k z)^2}
    {\prod_{k = -\infty}^{l} (p_1 + k z)^2}
    \frac{\prod_{k = -\infty}^0 (p_2 + k z)^2}
    {\prod_{k = -\infty}^{m} (p_2 + k z)^2} \\
    \times
    \frac{\prod_{k = -\infty}^0 (p_3 + k z)^2}
    {\prod_{k = -\infty}^{n} (p_3 + k z)^2}
    \frac{\prod_{k = -\infty}^0 (p_3-p_2-p_1 + k z)}
    {\prod_{k = -\infty}^{n-l-m} (p_3-p_2-p_1 + k z)}
  \end{multline*}
  and, since $p_1^2 = p_2^2 = p_3^2(p_3-p_2-p_1) = 0$ in the cohomology of $F$, this reduces to:
  \[
  J_F(\tau) = e^{\tau/z}
  \sum_{l, m, n \geq 0}
  \frac{Q_1^l Q_2^m Q_3^n e^{\langle \beta,\tau \rangle}}
  {
    \prod_{k=1}^l (p_1 + k z)^2
    \prod_{k=1}^m (p_2 + k z)^2
    \prod_{k=1}^n (p_3 + k z)^2
  }
  \frac{\prod_{k = -\infty}^0 (p_3-p_2-p_1 + k z)}
  {\prod_{k = -\infty}^{n-l-m} (p_3-p_2-p_1 + k z)}
  \]
  Thus:
  \[
   I_{\be,E}(\tau) = e^{\tau/z}
  \sum_{l, m, n \geq 0}
  \frac
  {
    Q_1^l Q_2^m Q_3^n e^{\langle \beta,\tau \rangle}
    \prod_{k=1}^{m+2n} (\lambda + p_2+2p_3 + k z)
  }
  {
    \prod_{k=1}^l (p_1 + k z)^2
    \prod_{k=1}^m (p_2 + k z)^2
    \prod_{k=1}^n (p_3 + k z)^2
  }
  \frac
  {
    \prod_{k = -\infty}^0 (p_3-p_2-p_1 + k z)
  }
  {
    \prod_{k = -\infty}^{n-l-m} (p_3-p_2-p_1 + k z)
  }
  \]
  We now apply Theorem~\ref{thm:toric_ci_ql}.  Setting $\tau = 0$, we
  find that:
  \[
  I_{\be,E}(0) = A+ B  z^{-1} + O(z^{-2})
  \]
  where:
  \begin{align*}
    A & = 1\\
    B & = (2Q_3 + 6 Q_2 Q_3) 1 + (p_3-p_2-p_1)\sum_{m > 0}
    {(-1)^{m-1} Q_2^m \over m} \\
    &= (2Q_3 + 6 Q_2 Q_3) 1 + (p_3-p_2-p_1)\log(1+Q_2)
  \end{align*}
  Thus:
  \[
  J_{\be,E}(B) = I_{\be,E}(0)
  \]
  The String Equation gives:
  \[
  J_{\be,E}(c 1 + \tau) = e^{c/z} J_{\be,E}(\tau)
  \]
  so:
  \[
  J_{\be,E}\big((p_3-p_2-p_1)\log(1+Q_2)\big) =
  e^{-(2Q_3 + 6 Q_2 Q_3) /z} I_{\be,E}(0)
  \]

  The twisted $J$-function satisfies:
  \[
  J_{\be,E}(t_1 p_1 + t_2 p_2 + t_3 p_3) =
  e^{(t_1 p_1 + t_2 p_2 + t_3 p_3)/z}
  \Big( 1 +
  \sum_{l, m, n \geq 0}
  Q_1^l Q_2^m Q_3^n e^{l t_1} e^{m t_2}e^{n t_3}
  c_{l,m,n}\Big)
  \]
  for classes $c_{l,m,n} \in H^\bullet(F;\QQ[\lambda])[[z^{-1}]]$ that
  do not depend on $t_1$, $t_2$, $t_3$.  So, substituting $t_1 = t_2 =
  {-\log(1+Q_2)}$, $t_3 = \log(1+Q_2)$, we see that:
  \[
  J_{\be,E}\big((p_3-p_2-p_1)\log(1+Q_2)\big)
  =
  e^{(p_3-p_2-p_1)\log(1+Q_2)/z}
  \Big[ J_{\be,E}(0)\Big]_{Q_1 = {Q_1 \over 1+Q_2}, Q_2 = {Q_2 \over
      1+Q_2}, Q_3 = Q_3(1+Q_2)}
  \]
  The change of variables:
  \begin{align}
    Q_1 = {Q_1 \over 1+Q_2} &&
    Q_2 = {Q_2 \over 1+Q_2} &&
    Q_3 = Q_3(1+Q_2) \notag
    \intertext{is called the mirror map; the inverse change of
      variables is:}
    Q_1 = {Q_1 \over 1-Q_2} &&
    Q_2 = {Q_2 \over 1-Q_2} &&
    Q_3 = Q_3(1-Q_2)
    \label{eq:inverse_mirror_map}
  \end{align}
  and so:
  \begin{align*}
    J_{\be,E}(0) &=
    \Big[e^{-(p_3-p_2-p_1)\log(1+Q_2)/z}
    J_{\be,E}\big((p_3-p_2-p_1)\log(1+Q_2)\big)\Big]_{Q_1 = {Q_1 \over 1-Q_2}, Q_2 = {Q_2 \over
      1-Q_2}, Q_3 = Q_3(1-Q_2)}  \\
  &=
    e^{(p_3-p_2-p_1)\log(1-Q_2)/z}
    \Big[e^{-(2Q_3 + 6 Q_2 Q_3) /z} I_{\be,E}(0)\Big]_{Q_1 = {Q_1 \over 1-Q_2}, Q_2 = {Q_2 \over
      1-Q_2}, Q_3 = Q_3(1-Q_2)}
  \end{align*}
  Taking the non-equivariant limit yields:
  \begin{multline*}
    J_{F,X}(0) = e^{(p_3-p_2-p_1)\log(1-Q_2)/z} e^{{-2} Q_3 - 4 Q_2
      Q_3} \times \\
    \sum_{l, m, n \geq 0}
    \frac
    {
      Q_1^l Q_2^m Q_3^n (1-Q_2)^{n-l-m}
      \prod_{k=1}^{m+2n} (p_2+2p_3 + k z)
    }
    {
      \prod_{k=1}^l (p_1 + k z)^2
      \prod_{k=1}^m (p_2 + k z)^2
      \prod_{k=1}^n (p_3 + k z)^2
    }
    \frac
    {
      \prod_{k = -\infty}^0 (p_3-p_2-p_1 + k z)
    }
    {
      \prod_{k = -\infty}^{n-l-m} (p_3-p_2-p_1 + k z)
    }
  \end{multline*}

  Recall that the quantum period $G_X$ is obtained from the component
  of $J_X(0)$ along the unit class $1 \in H^\bullet(X;\QQ)$ by setting
  $z = 1$ and $Q^\beta = t^{\langle \beta, {-K_X} \rangle}$.  Consider
  equation \eqref{eq:KKP}.  To obtain $G_X$, therefore, we need to
  extract the component of $J_{F,X}(0)$ along the unit class $1 \in
  H^\bullet(F;\QQ)$, set $z=1$, set $Q_1 = t$, set $Q_2 = 1$, and set $Q_3
  =t$.  This gives:
  \begin{align*}
    & G_X(t) = e^{-6t} \sum_{l=0}^\infty \sum_{m = 0}^\infty
    t^{2l+m} \frac{(2l+3m)!}{(l!)^2 (m!)^2 ((l+m)!)^2}
    \intertext{Regularizing gives:}
    & \hG_X(t) = 1 + 58t^2 + 600t^3 + 13182t^4 + 247440t^5 +
    5212300t^6 + \\
    & \qquad \qquad \qquad \qquad 111835920t^7 +
    2480747710t^8 + 56184565920t^9 + \cdots
  \end{align*}
\end{example}

\subsection{Weighted Projective Complete Intersections}

We will need also an analog of Corollary~\ref{cor:QL} where the
ambient space is weighted projective space, regarded as a smooth toric
Deligne--Mumford stack rather than as a singular variety.

\begin{proposition}
  \label{pro:wps}
  Let $Y$ be the weighted projective space $\PP(w_0,\ldots,w_n)$, let
  $X$ be a smooth Fano variety given as a complete intersection in $Y$
  defined by a section of $E = \cO(d_1) \oplus \cdots \oplus
  \cO(d_m)$, and let ${-k} = w_0 + \cdots + w_n - d_1 - \cdots - d_m$.
  Suppose that each $d_i$ is a positive integer, that ${-k}>0$, and
  that:
  \begin{equation}
    \label{eq:divisibility}
    \text{$w_i$ divides $d_j$ for all $i$, $j$ such that $0 \leq i \leq n$ and
      $1 \leq j \leq m$}
  \end{equation}
  Then the quantum period of $X$ is:
  \[
  G_X(t) = e^{{-c} t} \sum_{d \geq 0}
  t^{{-k} d}
  \frac{
    \prod_{j=1}^m (d d_j) !
  }{
    \prod_{i=1}^n (d w_i) !
  }
  \]
  where $c$ is the unique rational number such that the right-hand
  side has the form $1 + O(t^2)$.
\end{proposition}

\begin{proof}
  This follows immediately from Corollary~1.9 in \cite{CCLT}.
  Corollary~1.9 as stated there is false, however, because it omits
  the divisibility hypothesis \eqref{eq:divisibility}.  This
  hypothesis ensures that the bundle $E$ is convex, and hence ensures
  both (a) that the twisted $J$-function denoted by $J^{\text{\rm
      tw}}$ in \cite{CCIT}*{Corollary~5.1} admits a non-equivariant
  limit $J_{Y,X}$, and (b) that this non-equivariant limit satisfies
  \eqref{eq:KKP}: see \cite{CGIJJM}*{\S5}.  Both (a) and (b) are used
  implicitly in the proof of \cite{CCLT}*{Corollary~1.9}.  Under the
  additional divisibility assumption \eqref{eq:divisibility}, however,
  the proof of \cite{CCLT}*{Corollary~1.9} goes through.  This proves
  the Proposition.
\end{proof}

\section{Geometric constructions}
\label{sec:blowing-things-up}

\begin{lemma}
  \label{lem:blowups}
  Let $G$ be a nonsingular algebraic variety, let $V^{n+1}$ and $W^n$
  be locally free sheaves on $G$ of ranks $n+1$ and $n$ respectively,
  and let $f\colon V\to W$ be a homomorphism of sheaves.  Denote by
  $\pi\colon \PP(V)\to G$ the projective space bundle of lines in $V$,
  so that there is a tautological exact sequence:
  \[
  0\to S \to \pi^\star V \to Q\to 0
  \]
  with $S^\star := \cO(1)$.  Recall that elements of $V^\star$, being linear functions on $V$, define canonical sections of the line bundle $\cO(1)$ on $\PP(V)$, and that the corresponding homomorphism $\pi^\star V^\star \to \cO(1)$ induces an isomorphism $V^\star \cong \pi_\star
  \cO(1)$.  The section $f\in \Hom_G (V, W)$ determines a section
  $\tilde{f}\in H^0\bigl(\PP(V), \pi^\star W \otimes \cO(1) \bigr)$ by means of
  the following canonical identifications:
  \[
  \Hom_G (V, W)=H^0(G, W\otimes V^\star)=H^0(G,W\otimes \pi_\star
  \cO(1))=H^0\bigl(\PP(V), \pi^\star W \otimes \cO(1)\bigr)
  \]
  Let $F=Z(\tilde{f})\subset \PP(V)$ be the subscheme of $\PP(V)$
  where $\tilde{f}$ vanishes. Denote by $Z\subset G$ the subscheme where $f$ drops
  rank; that is, the ideal of $Z$ is the ideal defined by the $n+1$
  minors of size $n$ of $f$. Assume (a) that $f$ has generically maximal
  rank; (b) that it drops rank in codimension 2 (this is the expected
  codimension); and (c) that $Z$ is nonsingular\footnote{The last
    assumption (c) is probably not necessary.}. Then $F$ is the blow up of $G$ along $Z$.
\end{lemma}

\begin{proof}
  The statement is local on $G$ so fix a point $P\in Z \subset G$, and
  a Zariski open neighbourhood $P\in U=\Spec A$ with trivializations
  $V|_{U}=A^{n+1}$, $W|_{U}=A^n$. The morphism $f|_{U}$ is given by a
  $n\times (n+1)$ matrix $M$ with entries in $A$. Because $Z$ is
  nonsingular, at least one of the $(n-1)\times (n-1)$ minors of $A$
  is non-zero at $P$, and then, after changing trivializations and
  shrinking $U$ if necessary, we may assume that
  \[
  M=
  \begin{pmatrix}
    1 & 0 & \dots & 0 & 0 & 0 \\
    0 & 1 & \dots & 0 & 0 & 0 \\
    \vdots & \vdots & \vdots & \vdots & \vdots & \vdots \\
    0 & 0  & \dots  & 1 & 0 & 0 \\
    0 & 0 & \dots & 0 & x & y
  \end{pmatrix}
  \]
 It is clear that the ideal generated by the $n\times n$ minors of $M$
 is the ideal generated by the two rightmost minors $x$, $y$ (and, since $Z$
 is nonsingular, $x, y$ form part of a regular system of parameters at
 $P$). Denoting by $x_0,\dots, x_n$ the dual basis of $V^\star$,
 $F|_{U}=F\cap \pi^{-1}(U)\subset \PP(V|_{U})\cong U\times \PP^{n}$ is given
 by the $n$ equations in $n+1$ variables:
\[
M\cdot
\begin{pmatrix}x_0\\\vdots \\x_n \end{pmatrix} =0
\]
The first $n-1$ equations just say $x_0=\cdots = x_{n-2}=0$, while the
last equation states that $F|_{U}$ is the variety
$(xx_{n-1}+yx_n=0)\subset U\times \PP^1_{x_{n-1},x_n}$, that is, $F$
is the blow-up of $Z\subset G$.
\end{proof}

We will need the following well-known construction.

\begin{lemma}
  \label{lem:P(E)} Let $G$ be a complex Lie group acting on a space
  $A$, $X=A\GIT G$ a geometric quotient, and $\rho \colon G \to
  GL_r(\CC)$ a complex representation.
  \begin{itemize}
  \item[(1)] $\rho$ naturally induces a vector bundle $E=E(\rho)$ on
    $X$. Explicitly, $E(\rho)=(A\times \CC^r)\GIT G$ where $G$ acts
    as
\[
g\colon (a,v) \mapsto (ga, \rho(g) v)
\]
\item[(2)] Let $F=\PP(E)$ be the bundle of $1$-dimensional subspaces
  of the vector bundle in (1). Then $F=(A \times \CC^r) \GIT (G \times
  \CC^\times)$ where $G$ acts as in (1), and $\CC^\times$ acts
  trivially on the first factor and by rescaling on the second factor.
\item[(3)] Let $F=\PP(E)$ be as in (2). The tautological line bundle
  $\cO(-1)$ on $F$ is induced as in (1) by the $1$-dimensional
  representation of $G\times \CC^\times$ that is trivial on the first
  factor and standard on the second.
  \end{itemize}
\end{lemma}

We will also need to know how to compute the quantum period of a
product in terms of the quantum periods of the factors.

\begin{proposition}[The small $J$-function of a product]
  Let $X$ and $Y$ be smooth projective varieties over $\CC$.  Recall
  that there is a canonical isomorphism $H^\bullet(X \times Y;\QQ)
  \cong H^\bullet(X;\QQ) \otimes H^\bullet(Y;\QQ)$, and that $\Lambda_{X \times Y}$ is a completion of $\Lambda_X \otimes \Lambda_Y$.  Let $\tau_X \in
  H^2(X)$ and $\tau_Y \in H^2(Y)$.  Then:
  \[
  J_{X \times Y}\big(\tau_X \otimes 1 + 1 \otimes \tau_Y\big) =
  J_X(\tau_X) \otimes J_Y(\tau_Y)
  \]
\end{proposition}
\begin{proof}
  Combine:
  \begin{itemize}
  \item the differential equations
    \cite{Coates--Givental}*{equation~16} that characterize the
    $J$-function;
  \item the fact that the small quantum product $\ast_\tau$, $\tau \in
    H^2$, is uniquely determined by three-point Gromov--Witten
    invariants and the Divisor Equation;
  \item the product formula for Gromov--Witten invariants
    \citelist{\cite{Kontsevich--Manin}\cite{Behrend:products}}
    relating three-point Gromov--Witten invariants of $X \times Y$ to
    those of $X$ and of $Y$.
  \end{itemize}
\end{proof}

\begin{corollary}[The quantum period of a product]
  \label{cor:products}
  Let $X$ and $Y$ be smooth projective varieties over $\CC$.  Then:
  \[
  G_{X \times Y}(t) = G_X(t)\, G_Y(t)
  \]
 \qed
\end{corollary}

\subsection*{Notation for Grassmannians}
\label{sec:notation}

We denote by $\Gr=\Gr(r,n)$, the manifold of $r$-dimensional vector
subspaces of $\CC^n$. Notation for the universal sequence:
\[
0\to S \to \CC^n \to Q \to 0
\]
where $S$ is the rank $r$ universal bundle of subspaces and $Q$ is the
rank $n-r$ universal bundle of quotients.  If $\lambda=(\lambda_1\geq
\lambda_2\geq \dots)$ is a partition or Young diagram, we denote by
$Z_\lambda \subset \Gr$ the Schubert variety corresponding to
$\lambda$ and by $\sigma_\lambda \in H^\bullet \bigl(\Gr;\ZZ\bigr)$
its class in cohomology.  It is well-known that
$c_i(S^\star)=\sigma_{1^i}$ and $c_i(Q)=\sigma_i$ for $i=1,2,3,\dots$.

We will need:
\begin{itemize}
\item the \emph{Pieri formula}: if $\lambda$ is a partition and $k\geq 0$
  an integer then
  \[
  \sigma_\lambda \cdot \sigma_k = \sum_{\stackrel{\mu \geq
      \lambda}{\text{adds $k$ boxes no two in a column}}} \sigma_{\mu}
  \]
\item the following elementary facts for $\Gr(2,5)$:
  \begin{itemize}
  \item The Pl\"ucker embedding sends the Schubert variety $Z_2=\big\{W\mid W\cap \langle
    e_0,e_1\rangle \ne \{0\}\big\}$ to the subset of $\PP^9$ defined
    by the equations $ z_{23}=z_{24}=z_{34}=0$ and:
    \[
    \rk
    \begin{pmatrix}
      z_{02} & z_{03} & z_{04}\\
      z_{12} & z_{13} & z_{14}
    \end{pmatrix}
    <2
    \]
  \item The Pl\"ucker embedding sends the Schubert variety $Z_{1,1}=\{W
    \mid W\subset \langle e_0,e_1,e_2,e_3\rangle\}\cong \Gr(2,4)$ to
    a nonsingular quadric.
  \end{itemize}
\end{itemize}

\section{The Abelian/non-Abelian Correspondence}
\label{sec:A_nA}

Our other main tool for computing quantum periods is the
Abelian/non-Abelian correspondence of Ciocan-Fontanine--Kim--Sabbah
\cite{CFKS}.  This expresses genus-zero Gromov--Witten invariants (or
twisted Gromov--Witten invariants) of $X \GIT G$, where $G$ is a
complex reductive Lie group and $X$ is a smooth projective variety, in
terms of genus-zero Gromov--Witten invariants (or twisted
Gromov--Witten invariants) of $X\GIT T$ where $T$ is a maximal torus
in $G$.  The computations for $X \GIT T$ are typically much easier ---
the methods of \S\S\ref{sec:toric}--\ref{sec:toric_ci} often apply,
for example --- so the Abelian/non-Abelian correspondence is a
powerful tool for calculations.  Ten of the seventeen smooth Fano
3-folds of Picard rank one are toric varieties or toric complete
intersections, and thus can be treated using the methods of
\S\S\ref{sec:toric}--\ref{sec:toric_ci}; the following Theorem allows
a uniform treatment of the remaining seven cases.

\begin{theorem}
  \label{thm:rank_1_A_nA}
  Let $\Gr$ denote the Grassmannian $\Gr(r,n)$ of $r$-dimensional
  subspaces of $\CC^n$; let $S \to \Gr$ denote the universal bundle of
  subspaces; and let $E \to \Gr$ denote the vector bundle:
  \[
  E =
  \Big( \det S^\star \Big)^{\oplus a} \oplus
  \Big( \det S^\star \otimes \det S^\star \Big)^{\oplus b} \oplus
  \Big( S^\star \otimes \det S^\star \Big)^{\oplus c} \oplus
  \Big( S \otimes \det S^\star \Big)^{\oplus d} \oplus
  \Big( \textstyle \bigwedge^2 S^\star \Big)^{\oplus e}
  \]
  Let $X$ be the subvariety of $\Gr$ cut out by a generic section
  of $E$, and suppose that:
  \[
  k = a + 2 b + (r+1)c + (r-1)d + (r-1) e - n
  \]
  is strictly negative.  Consider the cohomology algebra
  $H^\bullet\big((\PP^{n-1})^r;\QQ\big)$.  Let $p_i \in
  H^2\big((\PP^{n-1})^r\big)$, $1 \leq i \leq r$, denote the first
  Chern class of $\pi_i^\star \cO(1)$ where $\pi_i \colon
  (\PP^{n-1})^r \to \PP^{n-1}$ is projection to the $i$th factor of
  the product.  Let $p_\all = p_1 + \cdots + p_r$ and, for
  $(l_1,\ldots,l_r) \in \ZZ^r$, let $|l| = l_1 + \ldots + l_r$.
  Let:
  \begin{align*}
    \Gamma_{l_1,\ldots,l_r} & =
    \Bigg(\prod_{k=1}^{|l|} (p_\all+ k)\Bigg)^a
    \Bigg(\prod_{k=1}^{2|l|} (2p_\all+ k)\Bigg)^b
    \Bigg(\prod_{j=1}^r \prod_{k=1}^{|l| + l_j} (p_\all + p_j+ k)\Bigg)^c
    \\ & \qquad \qquad \qquad
    \Bigg(\prod_{j=1}^r \prod_{k=1}^{|l| - l_j} (p_\all - p_j+ k)\Bigg)^d
    \Bigg(\prod_{i=1}^{r-1} \prod_{j=i+1}^r \prod_{k=1}^{l_i+l_j} (p_i+p_j+k)\Bigg)^e
    \\
    \intertext{and let:}
    \Omega & = \prod_{i=1}^{r-1} \prod_{j=i+1}^{r} (p_j - p_i)
  \end{align*}
  The element:
  \begin{equation}
    \label{eq:A_nA_twisted_J}
    \sum_{l_1 = 0}^\infty
    \cdots
    \sum_{l_r = 0}^\infty
    {
      (-1)^{|l|(r-1)} t^{{-k}|l|}
      \Gamma_{l_1,\ldots,l_r}
      \over
      \prod_{j=1}^{r}
      \prod_{k=1}^{k=l_r} (p_j + k)^n
    }
    \prod_{i=1}^{r-1}
    \prod_{j=i+1}^r
    \big(p_j-p_i + (l_j-l_i) \big)
  \end{equation}
  of $H^\bullet\big((\PP^{n-1})^r;\QQ\big) \otimes \QQ[\![t]\!]$ is
  divisible by $\Omega$.  Let $I_\tw(t)$ be the scalar-valued function
  obtained by dividing \eqref{eq:A_nA_twisted_J} by $\Omega$ and
  taking the component along $H^0\big((\PP^{n-1})^r;\QQ\big)$.  Then
  the quantum period $G_X$ of $X$ satisfies:
  \[
  G_X(t) = e^{\alpha t} I_\tw(t)
  \]
  where $\alpha$ is the unique rational number such that the
  right-hand side has the form $1 + O(t^2)$.
\end{theorem}

\begin{proof}
  The expression \eqref{eq:A_nA_twisted_J} is divisible by $\Omega$
  because it is totally antisymmetric in $p_1,\ldots,p_r$.  We know
  \emph{a priori} that $G_X(t) = 1 + O(t^2)$, so if there exists
  $\alpha \in \QQ$ such that $G_X(t) = e^{\alpha t} I_\tw(t)$ then
  this $\alpha$ is uniquely determined by the condition $e^{\alpha t}
  I_\tw(t) = 1 + O(t^2)$.  For the rest we use the Abelian/non-Abelian
  correspondence.  Consider the situation as in \S3.1 of \cite{CFKS}
  with:
  \begin{itemize}
  \item the space that is denoted by $X$ in \cite{CFKS} set equal to $A
    = \CC^{rn}$, regarded as the space of $r \times n$ matrices;
  \item $G = \GL_r(\CC)$, acting on $A$ by left-multiplication;
  \item $T = (\Cstar)^r$, the diagonal torus in $G$;
  \item the group that is denoted by $S$ in \cite{CFKS} set equal to the trivial group;
  \item $\cV$ equal to the representation:
    \[
    \Big( \det \Vstd \Big)^{\oplus a} \oplus
    \Big( \det \Vstd \otimes \det \Vstd \Big)^{\oplus b} \oplus
    \Big( \Vstd \otimes \det \Vstd \Big)^{\oplus c} \oplus
    \Big( \Vstd^\star \otimes \det \Vstd \Big)^{\oplus d} \oplus
    \Big( \textstyle \bigwedge^2 \Vstd \Big)^{\oplus e}
    \]
    where $\Vstd$ is the standard representation of $G$.
  \end{itemize}
  Then $A \GIT G$ is the Grassmannian $\Gr = \Gr(r,n)$ and $A \GIT T$
  is $(\PP^{n-1})^r$. The Weyl group $W=S_r$ permutes the $r$ factors
  of the product $(\PP^{n-1})^r$. The representation $\cV$ induces the
  vector bundle $\cV_G = E$ over $A \GIT G = \Gr$, and the
  representation $\cV$ induces the vector bundle:
  \begin{multline*}
    \cV_T =
    \Big(\cO(1,1,\ldots,1)\Big)^{\oplus a} \oplus
    \Big(\cO(2,2,\ldots,2)\Big)^{\oplus b} \oplus
    \Big( \oplus_{j=1}^r \cO(1,1,\ldots,1) \otimes \pi_j^\star \cO(1) \Big)^{\oplus c}
    \\ \oplus
    \Big( \oplus_{j=1}^r \cO(1,1,\ldots,1) \otimes \pi_j^\star \cO(-1) \Big)^{\oplus d} \oplus
    \Big( \oplus_{i=1}^{r-1} \oplus_{j=i+1}^r \pi_i^\star \cO(1)
    \otimes \pi_j^\star \cO(1) \Big)^{\oplus e}
  \end{multline*}
  over $A \GIT T = (\PP^{n-1})^r$.

  We fix a lift of $H^\bullet(A \GIT G;\QQ)$ to $H^\bullet(A \GIT
  T,\QQ)^W$ in the sense of \cite{CFKS}*{\S3}; there are many possible
  choices for such a lift, and the precise choice made will be
  unimportant in what follows.  The lift allows us to regard
  $H^\bullet(A \GIT G;\QQ)$ as a subspace of $H^\bullet(A \GIT
  T,\QQ)^W$, which maps isomorphically to the Weyl-anti-invariant part
  $H^\bullet(A \GIT T,\QQ)^a$ of $H^\bullet(A \GIT T,\QQ)$ via:
  \[
  \xymatrix{
    H^\bullet(A \GIT T,\QQ)^W \ar[rr]^{\cup \Omega} &&
    H^\bullet(A \GIT T,\QQ)^a}
  \]
  We compute the quantum period of $X$ by computing the $J$-function
  of $\Gr = A \GIT G$ twisted \cite{Coates--Givental} by the
  Euler class and the bundle $\cV_G$, using the Abelian/non-Abelian
  correspondence \cite{CFKS}.

  We begin by computing the $J$-function of $A \GIT T$ twisted by the
  Euler class and the bundle $\cV_T$.  As in
  \S\ref{sec:quantum_lefschetz}, and as in \cite{CFKS}, consider the
  bundles $\cV_T$ and $\cV_G$ equipped with the canonical
  $\Cstar$-action that rotates fibers and acts trivially on the base.
  We will compute the twisted $J$-function $J_{\be,\cV_T}$ of $A \GIT
  T$ using the Quantum Lefschetz theorem; $J_{\be,\cV_T}$ was defined
  in equation \eqref{eq:twisted_J} above, and is the restriction to
  the locus $\tau \in H^0(A \GIT T) \oplus H^2(A \GIT T)$ of what was
  denoted by $J^{S \times \Cstar}_{\cV_T}(\tau)$ in \cite{CFKS}.  The
  toric variety $A \GIT T$ is Fano, and Theorem~\ref{thm:toric_mirror}
  gives:
  \[
  J_{A \GIT T}(\tau) =
  e^{\tau/z}
  \sum_{l_1 = 0}^\infty
  \cdots
  \sum_{l_r = 0}^\infty
  {
    Q_1^{l_1} \cdots Q_r^{l_r}
    e^{l_1 \tau_1} \cdots e^{l_r \tau_r}
    \over
    \prod_{j=1}^{r}
    \prod_{k=1}^{k=l_j} (p_j + k z)^n
  }
  \]
  where $\tau = \tau_1 p_1 + \cdots + \tau_r p_r$ and we have
  identified the group ring $\QQ[H_2(A \GIT T;\ZZ)]$ with
  $\QQ[Q_1,\ldots,Q_r]$ via the $\QQ$-linear map that sends $Q^\beta$
  to $Q_1^{\langle \beta, p_1 \rangle} \cdots Q_r^{\langle \beta,
    p_r\rangle}$.  Each line bundle summand in $\cV_T$ is nef, and the
  condition $k<0$ ensures that $c_1(A \GIT T) - c_1(\cV_T)$ is ample,
  so Theorem~\ref{thm:toric_ci_ql} gives:
  \begin{equation}
    \label{eq:JeVT_A_nA_rank_1}
    J_{\be,\cV_T}(\tau) =
    e^{c(Q_1 e^{\tau_1} + \cdots + Q_r e^{\tau_r})/z}
    e^{\tau/z}
    \sum_{l_1 = 0}^\infty
    \cdots
    \sum_{l_r = 0}^\infty
    {
      Q_1^{l_1} \cdots Q_r^{l_r}
      e^{l_1 \tau_1} \cdots e^{l_r \tau_r}
      \Gamma_{l_1,\ldots,l_r}(\lambda,z)
      \over
      \prod_{j=1}^{r}
      \prod_{k=1}^{k=l_j} (p_j + k z)^n
    }
  \end{equation}
  for some rational number $c$, where:
  \begin{multline*}
    \Gamma_{l_1,\ldots,l_r}(\lambda,z) =
    \Bigg(\prod_{k=1}^{|l|} (\lambda + p_\all+ k z)\Bigg)^a
    \Bigg(\prod_{k=1}^{2|l|} (\lambda + 2p_\all+ k z)\Bigg)^b
    \Bigg(\prod_{j=1}^r \prod_{k=1}^{|l| + l_j} (\lambda + p_\all + p_j+ k z)\Bigg)^c
    \\
    \Bigg(\prod_{j=1}^r \prod_{k=1}^{|l| - l_j} (\lambda + p_\all - p_j+ k z)\Bigg)^d
    \Bigg(\prod_{i=1}^{r-1} \prod_{j=i+1}^r \prod_{k=1}^{l_i+l_j} (\lambda + p_i+p_j+k z)\Bigg)^e
  \end{multline*}
  The prefactor $e^{c (Q_1 e^{\tau_1} + \cdots + Q_r e^{\tau_r})/z}$ in
  \eqref{eq:JeVT_A_nA_rank_1} comes from the prefactor
  $e^{-C(\tau)/z}$ in Theorem~\ref{thm:toric_ci_ql}.

  Consider now $A \GIT G$ and a point $t \in H^\bullet(A \GIT G)$.  In
  \cite{CFKS}*{\S6.1} the authors consider the lift $\tilde{J}^{S
    \times \Cstar}_{\cV_G}(t)$ of their twisted $J$-function $J^{S
    \times \Cstar}_{\cV_G}(t)$ determined by a choice of lift
  $H^\bullet(A \GIT G;\QQ) \to H^\bullet(A \GIT T,\QQ)^W$.  We
  restrict to the locus $t \in H^0(A \GIT G;\QQ) \oplus H^2(A \GIT
  G;\QQ)$, considering the lift:
  \begin{align*}
    \tilde{J}_{\be,\cV_G}(t) := \tilde{J}^{S \times \Cstar}_{\cV_G}(t)
    && t \in H^0(A \GIT G;\QQ) \oplus H^2(A \GIT G;\QQ)
  \end{align*}
  of our twisted $J$-function $J_{\be,\cV_G}$ determined by our choice
  of lift $H^\bullet(A \GIT G;\QQ) \to H^\bullet(A \GIT T,\QQ)^W$.
  Let $p$ be the ample generator for $H^2(A \GIT G; \ZZ) \cong \ZZ$
  and identify the group ring $\QQ[H_2(A \GIT G;\ZZ)]$ with $\QQ[q]$
  via the $\QQ$-linear map which sends $Q^\beta$ to $q^{\langle
    \beta,p \rangle}$.  Theorems~4.1.1 and~6.1.2 in \cite{CFKS} imply
  that:
  \[
  \tilde{J}_{\be,\cV_G}\big(\theta(t)\big) \cup \Omega = \Big[
  \textstyle
  \Big( \prod_{i=1}^{r-1} \prod_{j=i+1}^r
  \big(z {\partial \over \partial \tau_j}
  -
  z {\partial \over \partial \tau_i}
  \big) \Big) J_{\be,\cV_T}(\tau) \Big]_{\tau = t, Q_1  = \cdots = Q_r = (-1)^{r-1} q}
  \]
  for some\footnote{The map $\theta$ here is the inverse to the map denoted by $\varphi$ in~\cite{CFKS}; it is grading-preserving where
    cohomology classes have their usual degree and $q$ has degree
    ${-2}k$.  Furthermore $\theta$ is the identity map modulo $q$.
    It follows that $\theta(0) = c' q \in H^0(A \GIT G;\QQ) \otimes
    \Lambda_{A \GIT G}$ for some $c' \in \QQ$, and that $c'= 0$
    whenever $k<-1$.}  function $\theta \colon H^2(A \GIT G;\QQ) \to
  H^\bullet(A \GIT G; \Lambda_{A \GIT G})$ such that $\theta(0) = c'
  q \in H^0(A \GIT G;\QQ) \otimes \Lambda_{A \GIT G}$. Setting $t = 0$
  gives:
  \[
  \tilde{J}_{\be,\cV_G}(c'q) \cup \Omega =
  e^{\pm c r q/z}
  \sum_{l_1 = 0}^\infty
  \cdots
  \sum_{l_r = 0}^\infty
  {
    (-1)^{|l|(r-1)} q^{|l|}
    \Gamma_{l_1,\ldots,l_r}(\lambda,z)
    \over
    \prod_{j=1}^{r}
    \prod_{k=1}^{k=l_j} (p_j + k z)^n
  }
  \prod_{i=1}^{r-1} \prod_{j=i+1}^{r} \big(p_j-p_i + (l_j-l_i) z \big)
  \]
  The String Equation gives:
  \[
  \tilde{J}_{\be,\cV_G}(c'q)  = e^{c'q/z} \tilde{J}_{\be,\cV_G}(0)
  \]
  and therefore:
  \begin{equation}
    \label{eq:A_nA_almost_there}
    \tilde{J}_{\be,\cV_G}(0) \cup \Omega =
    e^{\alpha q/z}
    \sum_{l_1 = 0}^\infty
    \cdots
    \sum_{l_r = 0}^\infty
    {
      (-1)^{|l|(r-1)} q^{|l|}
      \Gamma_{l_1,\ldots,l_r}(\lambda,z)
      \over
      \prod_{j=1}^{r}
      \prod_{k=1}^{k=l_j} (p_j + k z)^n
    }
    \prod_{i=1}^{r-1} \prod_{j=i+1}^{r} \big(p_j-p_i + (l_j-l_i) z \big)
  \end{equation}
  where $\alpha = {-c}' \pm c r$.  Note that if $k<-1$ then $\alpha=
  0$, for in that case both $c$ and $c'$ are zero.  Note also that $
  \Gamma_{l_1,\ldots,l_r}(0,1)$ coincides with what was denoted $
  \Gamma_{l_1,\ldots,l_r}$ in the statement of the Theorem.

  We saw in Example~\ref{ex:ql_with_mirror_map} how to extract the
  quantum period $G_X$ from the twisted $J$-function
  $J_{\be,\cV_G}(0)$: we take the non-equivariant limit $\lambda \to
  0$, extract the component along the unit class $1 \in H^\bullet(A
  \GIT G;\QQ)$, set $z=1$, and set $Q^\beta = t^{\langle \beta, {-K_X}
    \rangle}$.  Thus we consider the right-hand side of
  \eqref{eq:A_nA_almost_there}, take the non-equivariant limit,
  extract the coefficient of $\Omega$, set $z=1$, and set $q=t^{-k}$.
  The Theorem follows.
\end{proof}

\section{Fano Manifolds of Dimension $1$ and $2$}
\label{sec:dim12}
As a warm-up exercise, and because we will need some of these results
in the three-dimensional calculation, we now compute the quantum
periods for all Fano manifolds of dimension $1$ and $2$.

\begin{example}
  \label{ex:P1}
  There is a unique Fano manifold of dimension~$1$: the projective
  line $\PP^1$.  This is the toric variety with weight data:
  \[
  \begin{array}{rr}
    1 & 1
  \end{array}
  \]
  and nef cone given by the non-negative half-line in $\RR$.
  Corollary~\ref{cor:toric_mirror} gives:
  \[
  G_{\PP^1}(t) = \sum_{d=0}^\infty
  \frac{t^{2d}}{(d!)^2}
  \]
\end{example}

\subsection*{del~Pezzo Surfaces} There are 10 deformation families of
Fano manifolds of dimension~$2$: these are the del~Pezzo surfaces.  It is
well-known that, up to deformation:
\begin{itemize}
\item there is a unique smooth Fano surface of degree~$9$, being the projective plane $\PP^2$;
\item there are two smooth Fano surfaces of degree~$8$, being the
  Hirzebruch surface $\FF_1$ and the product of projective lines
  $\PP^1 \times \PP^1$;
\item there is a unique deformation class of smooth Fano surfaces
  $S_d$ of degree~$d$, $1 \leq d \leq 7$.
\end{itemize}
Given this, it is easy to see that the del~Pezzo surfaces can be
constructed, and their quantum periods calculated, as follows.

\begin{example}
  \label{ex:P2}
  The del~Pezzo surface $\PP^2$ is the toric variety with weight data:
  \[
  \begin{array}{rrr}
    1 & 1 & 1
  \end{array}
  \]
  and nef cone equal to the non-negative half-line.
  Corollary~\ref{cor:toric_mirror} gives:
  \[
  G_{\PP^2}(t) = \sum_{d=0}^\infty
  \frac{t^{3d}}{(d!)^3}
  \]
\end{example}

\begin{example}
  \label{ex:P1xP1}
  The del~Pezzo surface $\PP^1 \times \PP^1$ is the toric variety with
  weight data:
  \[
  \begin{array}{rrrrl}
    1 & 1 & 0 & 0 & \hspace{1.5ex} L \\
    0 & 0 & 1 & 1 & \hspace{1.5ex} M \\
  \end{array}
  \]
  and nef cone equal to $\langle L, M \rangle$.  (Here and henceforth, $\langle L_1,\ldots,L_k\rangle$ denotes the cone spanned by $L_1,\ldots,L_k$.)
  Corollary~\ref{cor:toric_mirror} gives:
  \[
  G_{\PP^1 \times \PP^1}(t) = \sum_{l=0}^\infty \sum_{m=0}^\infty
  \frac{t^{2l+2m}}{(l!)^2(m!)^2}
  \]
\end{example}

\begin{example}
  \label{ex:F1}
  The del~Pezzo surface $\FF_1$ is the toric variety with weight data:
  \[
  \begin{array}{rrrrl}
    1 & 1 & -1 & 0 & \hspace{1.5ex} L \\
    0 & 0 & 1 & 1 & \hspace{1.5ex} M \\
  \end{array}
  \]
  and nef cone equal to $\langle L, M \rangle$.
  Corollary~\ref{cor:toric_mirror} gives:
  \[
  G_{\FF_1}(t) = \sum_{l=0}^\infty \sum_{m=l}^\infty
  \frac{t^{l+2m}}{(l!)^2(m-l)!m!}
  \]
\end{example}

\begin{example}
  \label{ex:S7}
  The del~Pezzo surface $S_7$ is the toric variety with weight data:
  \[
  \begin{array}{rrrrrl}
    1 & 0 &  1  & -1  & 0 &    \hspace{1.5ex} L \\
    0 & 1 &  1  &  0   &-1 &  \hspace{1.5ex} M \\
    0 & 0 & -1  &  1  & 1 &    \hspace{1.5ex} N \\
  \end{array}
  \]
  and nef cone equal to $\langle L, M, N \rangle$.
  Corollary~\ref{cor:toric_mirror} gives:
  \[
  G_{S_7}(t) = \sum_{l=0}^\infty \sum_{m=0}^\infty \sum_{n=\max(l,m)}^{l+m}
  \frac{t^{l+m+n}}{l!m!(l+m-n)!(n-l)!(n-m)!}
  \]
\end{example}

\begin{example}
  \label{ex:S6}
  The del~Pezzo surface $S_6$ is the toric variety with weight data:
  \[
  \begin{array}{rrrrrrl}
    1 & 0 & 0 &  0  & 1 & -1  &  \hspace{1.5ex} A\\
    0 & 1 & 0 &  0  & 1 & 0  &  \hspace{1.5ex} B\\
    0 & 0 & 1 &  0  & 0 &  1 &  \hspace{1.5ex} C\\
    0 & 0 & 0 &  1  & -1 &  1 &  \hspace{1.5ex} D
  \end{array}
  \]
  and nef cone equal to $\langle A+B, B+C, C+D, A+B+C,
  B+C+D\rangle$.  Corollary~\ref{cor:toric_mirror} gives:
  \[
  G_{S_6}(t) = \sum_{a=0}^\infty \sum_{b=0}^\infty \sum_{c=0}^\infty \sum_{d=\max(a-c,0)}^{a+b}
  \frac{t^{a+2b+2c+d}}{a!b!c!d!(a+b-d)!(c+d-a)!}
  \]
\end{example}

\begin{example}
  \label{ex:S5}
  The del~Pezzo surface $S_5$ is a hypersurface of bidegree $(1,2)$ in
  $\PP^1 \times \PP^2$.  The ambient space $\PP^1 \times \PP^2$ is the
  toric variety with weight data:
  \[
  \begin{array}{rrrrrl}
    1 & 1 & 0 & 0 & 0 & \hspace{1.5ex} L \\
    0 & 0 & 1 & 1 & 1 & \hspace{1.5ex} M \\
  \end{array}
  \]
  and nef cone equal to $\langle L, M \rangle$, and $S_5$ is a member
  of $|L+2M|$ on $\PP^1 \times \PP^2$.  Corollary~\ref{cor:QL} gives:
  \[
  G_{S_5}(t) = e^{-3t}
  \sum_{l=0}^\infty \sum_{m=0}^\infty
  t^{l+m} \frac{(l+2m)!}{(l!)^2(m!)^3}
  \]
\end{example}

\begin{example}
  \label{ex:S4}
  A complete intersection of type $(2,2)$ in $\PP^4$ is a del~Pezzo
  surface~$S_4$.  Proposition~\ref{pro:wps} gives:
  \[
  G_{S_4}(t) = e^{-4t}
  \sum_{d=0}^\infty
  t^d \frac{(2d)!(2d)!}{(d!)^5}
  \]
\end{example}

\begin{example}
  \label{ex:S3}
  A cubic surface in $\PP^3$ is a del~Pezzo surface~$S_3$.
  Proposition~\ref{pro:wps} gives:
  \[
  G_{S_3}(t) = e^{-6t}
  \sum_{d=0}^\infty
  t^d \frac{(3d)!}{(d!)^4}
  \]
\end{example}

\begin{example}
  \label{ex:S2}
  A quartic surface in $\PP(1,1,1,2)$ is a del~Pezzo surface~$S_2$.
  Proposition~\ref{pro:wps} gives:
  \[
  G_{S_2}(t) = e^{-12t}
  \sum_{d=0}^\infty
  t^d \frac{(4d)!}{(d!)^3(2d)!}
  \]
\end{example}

\begin{example}
  \label{ex:S1}
  A sextic surface in $\PP(1,1,2,3)$ is a del~Pezzo surface~$S_1$.
  Proposition~\ref{pro:wps} gives:
  \[
  G_{S_1}(t) = e^{-60t}
  \sum_{d=0}^\infty
  t^d \frac{(6d)!}{(d!)^2(2d)!(3d)!}
  \]
\end{example}

\section{Notation for 3-Dimensional Fano Manifolds}

We fix notation for 3-dimensional Fano manifolds as follows.
\begin{itemize}
\item $\PP^3$ denotes $3$-dimensional complex projective space;
\item $Q^3$ denotes a quadric hypersurface in $\PP^4$;
\item $V_k$ denotes the 3-dimensional Fano manifold of Picard rank~$1$, Fano index~$1$, and degree~$k$;
\item $B_k$ denotes the 3-dimensional Fano manifold of Picard rank~$1$, Fano index~$2$, and degree~$8k$;
\item $\MM{\rho}{k}$ denotes the $k$th entry in the Mori--Mukai list~\cite{MM:fanoconf} of 3-dimensional Fano manifolds of Picard rank~$\rho$, with the exception of the case $\rho=4$ where we reorder the manifolds, placing the 13th entry in Mori--Mukai's rank-4 list~\cite{MM:fanoconf}*{pages~48--49} in between the first and second elements of that list.  This reordering ensures that, for each $\rho$, the sequence $\MM{\rho}{1}$,~$\MM{\rho}{2}$,~$\MM{\rho}{3}$,\ldots~is in order of increasing degree.
\end{itemize}

\newcounter{CustomSectionCounter}
\renewcommand\thesection{\arabic{CustomSectionCounter}}

\addtocounter{CustomSectionCounter}{1}
\section{The Fano Manifold $\PP^3$}
\label{anchor:P3}

\subsection*{Name:} $\PP^3$

\subsection*{Iskovskikh Classification:} This is case 1 in \cite{Isk:2}*{Table~6.5}.

\subsection*{Construction:}
The Fano toric variety $X$ with weight data:
\[
\begin{array}{rrrrl} 
  1 & 1 & 1 & 1 & \hspace{1.5ex} L\\ 
\end{array}
\]
and $\Amp(X)$ spanned by $L$.

\subsection*{The quantum period:}  Corollary~\ref{cor:toric_mirror} yields:
\[
G_X(t) = \sum_{d =0}^\infty {t^{4d} \over (d!)^4}  
\]
and regularizing gives:
\[
\hG_X(t) = 1 + 24 t^4 + 2520 t^8 + 369600 t^{12} + \cdots
\]
\subsection*{Minkowski period sequence:} \href{http://www.grdb.co.uk/search/period3?id=1&printlevel=2}{1}


\addtocounter{CustomSectionCounter}{1}
\section{The Fano Manifold $Q^3$}
\label{anchor:Q3}

\subsection*{Name:} $Q^3$

\subsection*{Iskovskikh Classification:} This is case 2 in \cite{Isk:2}*{Table~6.5}.

\subsection*{Construction:} A divisor $X$ of degree 2 on $F=\PP^4$.

\subsection*{The quantum period:} The toric variety $F$ has weight
data:
\[
\begin{array}{rrrrrl} 
  1 & 1 & 1 & 1 & 1 & \hspace{1.5ex} L\\ 
\end{array}
\]
and $\Amp(F) = \langle L \rangle$.  We have:
\begin{itemize}
\item $F$ is a Fano variety;
\item $X\sim 2L$ is ample;
\item $-(K_F+X)\sim 3L$ is ample.
\end{itemize}
Corollary~\ref{cor:QL} yields:
\[
G_X(t) = \sum_{d = 0}^\infty t^{3d} \frac{(2d)!}{(d!)^5}
\]
and regularizing gives:
\[
\hG_X(t) = 1+12 t^3+540 t^6+33600 t^9+2425500 t^{12} + \cdots
\]

\subsection*{Minkowski period sequence:} \href{http://www.grdb.co.uk/search/period3?id=3&printlevel=2}{3}


\addtocounter{CustomSectionCounter}{1}
\section{The Fano Manifold $B_1$}
\label{anchor:B1}
\label{sec:B1}

\subsection*{Name:} $B_1$

\subsection*{Iskovskikh Classification:} This is case 3 in \cite{Isk:2}*{Table~6.5}.

\subsection*{Construction:} A sextic hypersurface $X$ in $\PP(1,1,1,2,3)$.

\subsection*{The quantum period:} Proposition~\ref{pro:wps} yields:
\[
G_X(t) = \sum_{d =0}^\infty t^{2d} \frac{(6d)!}{(d!)^3 (2d)! (3d)!}
\]
and regularizing gives:
\[
\hG_X(t) = 1+120 t^2+83160 t^4+81681600 t^6+93699005400 t^8+117386113965120 t^{10} + \cdots
\]

\subsection*{Minkowski period sequence:} None.  Note that the
anticanonical line bundle of $B_1$ is not very ample.


\addtocounter{CustomSectionCounter}{1}
\section{The Fano Manifold $B_2$}
\label{anchor:B2}
\label{sec:B2}

\subsection*{Name:} $B_2$

\subsection*{Iskovskikh Classification:} This is case 4 in \cite{Isk:2}*{Table~6.5}.

\subsection*{Construction:} A quartic hypersurface $X$ in $\PP(1,1,1,1,2)$.

\subsection*{The quantum period:} Proposition~\ref{pro:wps} yields:
\[
G_X(t) = \sum_{d = 0}^\infty t^{2d} \frac{(4d)!}{(d!)^4 (2d)!}
\]
and regularizing gives:
\[
\hG_X(t) = 1+24 t^2+2520 t^4+369600 t^6+63063000 t^8+11732745024 t^{10}+ \cdots
\]

\subsection*{Minkowski period sequence:} \href{http://www.grdb.co.uk/search/period3?id=140&printlevel=2}{140}


\addtocounter{CustomSectionCounter}{1}
\section{The Fano Manifold $B_3$}
\label{anchor:B3}

\subsection*{Name:} $B_3$

\subsection*{Iskovskikh Classification:} This is case 5 in \cite{Isk:2}*{Table~6.5}.

\subsection*{Construction:} A divisor $X$ of degree 3 on $F=\PP^4$.

\subsection*{The quantum period:} The toric variety $F$ has weight
data:
\[
\begin{array}{rrrrrl} 
  1 & 1 & 1 & 1 & 1 & \hspace{1.5ex} L\\ 
\end{array}
\]
and $\Amp(F) = \langle L \rangle$.  We have:
\begin{itemize}
\item $F$ is a Fano variety;
\item $X\sim 3L$ is ample;
\item $-(K_F+X)\sim 2L$ is ample.
\end{itemize}
Corollary~\ref{cor:QL} yields:
\[
G_X(t) = \sum_{d=0}^\infty t^{2d} \frac{(3d)!}{(d!)^5}
\]
and regularizing gives:
\[
\hG_X(t) = 1+12 t^2+540 t^4+33600 t^6+2425500 t^8+190702512 t^{10} + \cdots
\]

\subsection*{Minkowski period sequence:} \href{http://www.grdb.co.uk/search/period3?id=106&printlevel=2}{106}


\addtocounter{CustomSectionCounter}{1}
\section{The Fano Manifold $B_4$}
\label{anchor:B4}

\subsection*{Name:} $B_4$

\subsection*{Iskovskikh Classification:} This is case 6 in \cite{Isk:2}*{Table~6.5}.

\subsection*{Construction:} A codimension-2 complete intersection $X$
of type $(2L) \cap (2L)$ in the toric variety $F=\PP^5$.

\subsection*{The quantum period:} The toric variety $F$ has weight
data:
\[
\begin{array}{rrrrrrl} 
  1 & 1 & 1 & 1 & 1 & 1 & \hspace{1.5ex} L\\ 
\end{array}
\]
and $\Amp(F) = \langle L \rangle$.  We have:
\begin{itemize}
\item $F$ is a Fano variety;
\item $X$ is the complete intersection of two ample divisors on $F$
\item $-(K_F+\Lambda)\sim 2L$ is ample.
\end{itemize}
Corollary~\ref{cor:QL} yields:
\[
G_X(t) = \sum_{d = 0}^\infty t^{2d} \frac{(2d)!(2d)!}{(d!)^6}
\]
and regularizing gives:
\[
\hG_X(t) = 1+8 t^2+216 t^4+8000 t^6+343000 t^8+16003008 t^{10} + \cdots
\]

\subsection*{Minkowski period sequence:} \href{http://www.grdb.co.uk/search/period3?id=75&printlevel=2}{75}


\addtocounter{CustomSectionCounter}{1}
\section{The Fano Manifold $B_5$}
\label{anchor:B5}

\label{sec:B5}
\subsection*{Name:} $B_5$

\subsection*{Iskovskikh Classification:} This is case 7 in \cite{Isk:2}*{Table~6.5}.

\subsection*{Construction:}  A complete intersection $X$ in
$\Gr(2,5)$ cut out by a section of $\cO(1)^{\oplus 3}$, where $\cO(1)$
is the pullback of $\cO(1)$ on projective space under the Pl\"ucker
embedding.

\subsection*{The quantum period:}  The line bundle $\cO(1)$ is the ample
generator of $\Pic(\Gr(2,5))$, hence $\cO(1)$ coincides with
$\det(S^\star)$ where $S$ is the universal bundle of subspaces on
$\Gr(2,5)$.  We apply Theorem~\ref{thm:rank_1_A_nA} with $a=3$ and
$b=c=d=e=0$, obtaining:
\[
G_X(t) = \sum_{l=0}^\infty \sum_{m=0}^\infty
(-1)^{l+m}
t^{2l+2m}
\frac
{
  \big((l+m)!\big)^3
}
{
  (l!)^5 (m!)^5 
}
\big(1-5 (m-l)H_m \big)
\]
where $H_m$ is the $m$th harmonic number.  Regularizing yields:
\[
\hG_X(t) = 1+6 t^2+114 t^4+2940 t^6+87570 t^8+2835756 t^{10}+\cdots
\]

\subsection*{Minkowski period sequence:} \href{http://www.grdb.co.uk/search/period3?id=46&printlevel=2}{46}


\addtocounter{CustomSectionCounter}{1}
\section{The Fano Manifold $V_2$}
\label{anchor:V2}

\subsection*{Name:} $V_2$

\subsection*{Iskovskikh Classification:} This is case 8 in \cite{Isk:2}*{Table~6.5}.

\subsection*{Construction:} A sextic hypersurface $X$ in $\PP(1,1,1,1,3)$.

\subsection*{The quantum period:} Proposition~\ref{pro:wps} yields:
\[
G_X(t) = e^{-120t} \sum_{d = 0}^\infty t^{d} \frac{(6d)!}{(d!)^4 (3d)!}
\]
and regularizing gives:
\begin{multline*}
  \hG_X(t) = 1+68760 t^2+55200000 t^3+61054781400 t^4+71591389125120
  t^5+88808827978814400 t^6\\
  +114426010259814758400 t^7+151686694219076253783000 t^8\\+205548259807393951744128000 t^9 + \cdots
\end{multline*}

\subsection*{Minkowski period sequence:} None.  Note that the
anticanonical line bundle of $V_2$ is not very ample.


\addtocounter{CustomSectionCounter}{1}
\section{The Fano Manifold $V_4$}
\label{anchor:V4}

\subsection*{Name:} $V_4$

\subsection*{Iskovskikh Classification:} This is cases~9 and~10 in \cite{Isk:2}*{Table~6.5}.  These cases are deformation equivalent: they can both be described as 
complete intersections of type $(2,4)$ in $\PP(1,1,1,1,1,2)$.

\subsection*{Construction:} A divisor $X$ of degree 4 on $F=\PP^4$.

\subsection*{The quantum period:} The toric variety $F$ has weight
data:
\[
\begin{array}{rrrrrl} 
  1 & 1 & 1 & 1 & 1 & \hspace{1.5ex} L\\ 
\end{array}
\]
and $\Amp(F) = \langle L \rangle$.  We have:
\begin{itemize}
\item $F$ is a Fano variety;
\item $X\sim 4L$ is ample;
\item $-(K_F+X)\sim L$ is ample.
\end{itemize}
Corollary~\ref{cor:QL} yields:
\[
G_X(t) = e^{-24t}\sum_{d=0}^\infty t^{d} \frac{(4d)!}{(d!)^5}
\]
and regularizing gives:
\begin{multline*}
  \hG_X(t) = 1+1944 t^2+215808 t^3+35295192 t^4+5977566720
  t^5+1073491139520 t^6+199954313717760 t^7\\
  +38302652395770840 t^8+7497487505353251840 t^9 + \cdots
\end{multline*}

\subsection*{Minkowski period sequence:} \href{http://www.grdb.co.uk/search/period3?id=165&printlevel=2}{165}


\addtocounter{CustomSectionCounter}{1}
\section{The Fano Manifold $V_6$}
\label{anchor:V6}

\subsection*{Name:} $V_6$

\subsection*{Iskovskikh Classification:} This is case 11 in \cite{Isk:2}*{Table~6.5}.

\subsection*{Construction:} A codimension-2 complete intersection $X$
of type $(2L) \cap (3L)$ in the toric variety $F=\PP^5$.

\subsection*{The quantum period:} The toric variety $F$ has weight
data:
\[
\begin{array}{rrrrrrl} 
  1 & 1 & 1 & 1 & 1 & 1 & \hspace{1.5ex} L\\ 
\end{array}
\]
and $\Amp(F) = \langle L \rangle$.  We have:
\begin{itemize}
\item $F$ is a Fano variety;
\item $X$ is the complete intersection of two ample divisors;
\item $-(K_F+\Lambda) \sim L$ is ample.
\end{itemize}
Corollary~\ref{cor:QL} yields:
\[
G_X(t) = e^{-12t} \sum_{d=0}^\infty t^{d} \frac{(2d)!(3d)!}{(d!)^6}
\]
and regularizing gives:
\begin{multline*}
  \hG_X(t) = 1+396 t^2+17616 t^3+1217052 t^4+85220640 t^5+6349812480
  t^6+490029523200 t^7\\+38883641777820 t^8+3152020367254080 t^9 + \cdots
\end{multline*}

\subsection*{Minkowski period sequence:} \href{http://www.grdb.co.uk/search/period3?id=164&printlevel=2}{164}


\addtocounter{CustomSectionCounter}{1}
\section{The Fano Manifold $V_8$}
\label{anchor:V8}

\subsection*{Name:} $V_8$

\subsection*{Iskovskikh Classification:} This is case 12 in \cite{Isk:2}*{Table~6.5}.

\subsection*{Construction:} A codimension-3 complete intersection $X$
of type $(2L) \cap (2L) \cap (2L)$ in the toric variety $F=\PP^6$.

\subsection*{The quantum period:} The toric variety $F$ has weight
data:
\[
\begin{array}{rrrrrrrl} 
  1 & 1 & 1 & 1 & 1 & 1 & 1 & \hspace{1.5ex} L\\ 
\end{array}
\]
and $\Amp(F) = \langle L \rangle$.  We have:
\begin{itemize}
\item $F$ is a Fano variety;
\item $X$ is the complete intersection of three ample divisors;
\item $-(K_F+\Lambda)\sim L$ is ample.
\end{itemize}
Corollary~\ref{cor:QL} yields:
\[
G_X(t) = e^{-8t} \sum_{d=0}^\infty t^{d} \frac{(2d)!(2d)!(2d)!}{(d!)^7}
\]
and regularizing gives:
\begin{multline*}
  \hG_X(t) = 1+152 t^2+3840 t^3+157656 t^4+6428160 t^5+280064960
  t^6+12618762240 t^7\\+584579486680 t^8
  +27660007173120 t^9 + \cdots
\end{multline*}

\subsection*{Minkowski period sequence:} \href{http://www.grdb.co.uk/search/period3?id=163&printlevel=2}{163}


\addtocounter{CustomSectionCounter}{1}
\section{The Fano Manifold $V_{10}$}
\label{anchor:V10}

\subsection*{Name:} $V_{10}$

\subsection*{Iskovskikh Classification:} This is case 13 in \cite{Isk:2}*{Table~6.5}.

\subsection*{Construction:}  A complete intersection $X$ in $\Gr(2,5)$,
cut out by a section of $\cO(1)^{\oplus 2} \oplus \cO(2)$ where
$\cO(1)$ is the pullback of $\cO(1)$ on projective space under the
Pl\"ucker embedding.  

\subsection*{The quantum period:} We apply
Theorem~\ref{thm:rank_1_A_nA} with $a=2$, $b=1$, and $c=d=e=0$.  This
yields:
\[
G_X(t) = e^{-6t} \sum_{l=0}^\infty \sum_{m=0}^\infty
(-1)^{l+m}
t^{l+m}
\frac
{
  \big((l+m)!\big)^2 (2l+2m)!
}
{
  (l!)^5 (m!)^5 
}
\big(1-5 (m-l)H_m \big)
\]
where $H_m$ is the $m$th harmonic number.  Regularizing gives:
\begin{multline*}
  \hG_X(t) = 1+78 t^2+1320 t^3+37746 t^4+1051920 t^5+31464780
  t^6+971757360 t^7 \\
  +30859805970 t^8+1000739433120 t^9 +\cdots
\end{multline*}

\subsection*{Minkowski period sequence:} \href{http://www.grdb.co.uk/search/period3?id=160&printlevel=2}{160}


\addtocounter{CustomSectionCounter}{1}
\section{The Fano Manifold $V_{12}$}
\label{anchor:V12}

\subsection*{Name:} $V_{12}$

\subsection*{Iskovskikh Classification:} This is case 14 in \cite{Isk:2}*{Table~6.5}.

\subsection*{Construction:} A subvariety $X$ of $\Gr(2,5)$ cut out by
a generic section of $\big(S^\star \otimes \det S^\star\big) \oplus
\det S^\star$, where $S$ is the universal bundle of subspaces on
$\Gr(2,5)$.

\subsection*{A remark on the construction:} The paper of Mukai
\cite{Mukai-CSI} is devoted to this case and it is shown there that
$X$ is a complete intersection of $7$ hyperplane sections of the
($10$-dimensional) orthogonal Grassmannian $\OGr (5,10)$ in its spinor
embedding in $\PP^{15}$. This model contains $X$ as a \emph{linear}
section and, perhaps more important, is the largest hyperplane
``un-section'' of $X$. Our construction, on the other hand, is
better-suited for the fast calculation of the quantum period of $X$.

Write $V=\CC^5$; in what follows, for ease of notation, we denote by
$\cO(1)$ the line bundle $\det S^\star$ on $\Gr(2,V)=\Gr(2,5)$.  Let
$\Sigma\subset \Gr(2,V)$ be the vanishing locus of a general section
$s$ of the vector bundle $S^\star \otimes \cO(1)$. Below we sketch a general
construction of a natural linear embedding $\Sigma \subset \OGr
(5,10)$; this shows that our construction and Mukai's construction
coincide.  To compute the quantum period of $X$, however, we need
rather less.  Gromov--Witten invariants are deformation-invariant so,
since there is a unique deformation family of $V_{12}$s
\citelist{\cite{Isk:1}\cite{Isk:2}}, it suffices to show that our
construction gives a smooth member of this family.  In other words, it
suffices to prove that $\Sigma$ is a rank-$1$ Fano \mbox{4-fold} of
Fano index $2$---hence coindex $3$ in Mukai's terminology---and degree
$12$.

The Picard rank of $\Sigma$ is $1$ by Sommese's Theorem
\cite{Lazarsfeld}*{Theorem~7.1.1} and, from the exact sequence:
\[
0\to T_\Sigma\to T_{\Gr(2,5)}\big|_\Sigma \to S^\star \otimes \cO(1) \big|_\Sigma \to 0 
\]
we get that:
\[
-K_\Sigma=\Bigl(-K_{Gr(2,5)} \otimes \wedge^2 \bigl(S \otimes \cO(-1)\bigr)\Bigr)\big|_\Sigma=\cO_\Sigma (2)
\]
That is, $\Sigma$ is a Fano \mbox{4-fold} of Fano index $2$. It
remains to show that $\Sigma$ has degree $12$; this is a small
calculation in Schubert calculus:
\[
[\Sigma]= c_2(S^\star \otimes \det S^\star)=\sigma_{1,1}+2\sigma_1^2 =3\sigma_{1,1}+2\sigma_2
\]
and therefore:
\[
\deg \Sigma=[\Sigma]\sigma_1^4=\sigma_{1,1}\sigma_1^4+2\sigma_1^5=2+10=12
\]

We next sketch the promised construction of a linear embedding $\Sigma
\subset \OGr(5,10)$. The first task is to construct a rank-$5$
vector bundle on $\Sigma$---the bundle that will be the pull-back of
the tautological sub-bundle of $\OGr(5,10)$.

We claim that $\Ext^1_\Sigma (S^\star, Q)=\CC$ and take $E$ the unique
nontrivial extension. To calculate this $\Ext$ group consider the
Koszul resolution of $\cO_\Sigma$:
\[
0\to \cO(-3)\to S \otimes \cO(-1)\to\cO_{\Gr(2,V)}\to \cO_\Sigma \to 0 
\]  
Tensoring by $S\otimes Q$ and using $H^1\bigl(\Gr(2,V); S\otimes
Q\bigr)=H^2\bigl(\Gr(2,V); S\otimes Q\bigr)=\{0\}$ (Borel--Weil--Bott)
and $H^2\bigl(\Gr(2,V); S\otimes Q \otimes \cO(-3)\bigr)=H^3\bigl(\Gr(2,V);
S\otimes Q \otimes \cO(-3)\bigr)=\{0\}$ (Borel--Weil--Bott) we get:
\[
\Ext^1_\Sigma(S^\star,
Q)=H^1(\Sigma; S\otimes Q)=H^2\bigl(\Gr(2,V); S\otimes
Q\otimes S \otimes \cO(-1)\bigr)=\CC 
\]
again by Borel--Weil--Bott.

As anticipated, denote now by $E$ the unique nontrivial rank-$5$ extension:
\[
0\to Q \to E \to S^\star \to 0 
\]
The bundle $E$ fits into a natural self-dual ``diagram of $9$'':
\[
\xymatrix{
          & 0\ar[d]               & 0\ar[d]              & 0\ar[d] &  \\
0\ar[r]&S\ar[d]\ar[r]&E^\star\ar[d]\ar[r]&Q^\star\ar[d]\ar[r]&0\\  
0\ar[r]&V\ar[d]\ar[r]&W\ar[d]\ar[r]&V^\star\ar[d]\ar[r]&0\\
0\ar[r]&Q\ar[d]\ar[r]&E\ar[d]\ar[r]&S^\star\ar[r]\ar[d]&0\\
          &0 & 0 & 0 &}
\]
where $W=V\oplus V^\star$. The diagram makes it clear that $E\subset
V\oplus V^\star$ is isotropic when $V\oplus V^\star$ is equipped with
the canonical nondegenerate symmetric bilinear form. Thus $E$ induces
a map $\Sigma \to \OGr(5, V\oplus V^\star)$.
 
\subsection*{The quantum period:} We apply
Theorem~\ref{thm:rank_1_A_nA} with $a = c = 1$ and $b = d = e = 0$.
This yields:
\[
G_X(t) = e^{-5t} \sum_{l, m \geq 0}
({-t})^{l+m}\sum_{l=0}^\infty \sum_{m=0}^\infty
\frac
{
  \big(l+m)! (2l+m)! (l+2m)!
}
{
  (l!)^5 (m!)^5 
}
\big(1+(m-l)(H_{2l+m} + 2 H_{l + 2m}-5H_m) \big)
\]
where $H_k$ denotes the $k$th harmonic number.  Regularizing gives:
\begin{multline*}
  \hG_X(t) = 1+48 t^2+600 t^3+13176 t^4+276480 t^5+6259800
  t^6+146064240 t^7 \\
  +3505282200 t^8+85882130880 t^9 +\cdots
\end{multline*}

\subsection*{Minkowski period sequence:} \href{http://www.grdb.co.uk/search/period3?id=150&printlevel=2}{150}


\addtocounter{CustomSectionCounter}{1}
\section{The Fano Manifold $V_{14}$}
\label{anchor:V14}

\subsection*{Name:} $V_{14}$

\subsection*{Iskovskikh Classification:} This is case 15 in \cite{Isk:2}*{Table~6.5}.

\subsection*{Construction:} A complete intersection $X$ in $\Gr(2,6)$, cut
out by a section of $\cO(1)^{\oplus 5}$ where $\cO(1)$ is the pullback
of $\cO(1)$ on projective space under the Pl\"ucker embedding
\citelist{\cite{Mukai-CG}\cite{Gushel:83}\cite{Gushel:92}}.

\subsection*{The quantum period:} We apply
Theorem~\ref{thm:rank_1_A_nA} with $a=5$ and $b=c=d=e=0$.  This
yields:
\[
G_X(t) = e^{-4t} \sum_{l, m \geq 0}
(-1)^{l+m}
t^{l+m}
\frac
{
  \big((l+m)!\big)^5
}
{
  (l!)^6 (m!)^6
}
\big(1-6 (m-l)H_m \big)
\]
where $H_m$ is the $m$th harmonic number.  Regularizing gives:
\begin{multline*}
  \hG_X(t) = 1+32 t^2+312 t^3+5520 t^4+91680 t^5+1651640 t^6+30604560
  t^7 \\ +583436560 t^8+11352768000 t^9+\cdots
\end{multline*}

\subsection*{Minkowski period sequence:} \href{http://www.grdb.co.uk/search/period3?id=147&printlevel=2}{147}


\addtocounter{CustomSectionCounter}{1}
\section{The Fano Manifold $V_{16}$}
\label{anchor:V16}

\subsection*{Name:} $V_{16}$

\subsection*{Iskovskikh Classification:} This is case 16 in \cite{Isk:2}*{Table~6.5}.

\subsection*{Construction:} The vanishing locus $X$ of a general section
of the vector bundle:
\[
\wedge^2 S^\star \oplus (\det S^\star)^{\oplus 3}
\]
on $\Gr(3,6)$.

\subsection*{A remark on the construction:} The paper
\cite{Mukai-CSII} of Mukai is devoted to this case and it is shown
there that $X$ is a complete intersection of $3$ hyperplane sections
of the ($6$-dimensional) symplectic Grassmannian $\SpGr (3, 6)$ of
complex Lagrangian $3$-dimensional subspaces $W\subset \CC^6$ where
$\CC^6$ is equipped with the standard symplectic form $\omega\in
\wedge^2 \CC^{6\, \star}$, in the Pl\"ucker embedding inherited from
$\Gr(3, 6)$. Indeed, the natural surjection $\wedge^2 \CC^{6\, \star}
\to \wedge^2 S^\star$ induces an isomorphism:
\[
H^0\bigl(\Gr(3,6); \wedge^2 \CC^{6\, \star}\bigr) \cong
H^0\bigl(\Gr(3,6); \wedge^2 S^\star\bigr)
\]
that allows us to view $\omega$ as an element of $H^0\bigl(\Gr(3,6);
\wedge^2 S^\star\bigr)$ with zero locus $\SpGr(3,6)$.  Thus the
construction given above coincides with that given by Mukai (ibid.).

\subsection*{The quantum period:} We apply
Theorem~\ref{thm:rank_1_A_nA} with $a=3$, $b=c=d=0$, and $e=1$.  This
yields:
\[
G_X(t) = 1 + 12 t^2 + 32 t^3 + 121 t^4 + 336 t^5 + \textstyle\frac{2548}{3} t^6 +
1888 t^7 + \frac{60481}{16} t^8 + \frac{185350}{27} t^9 + \cdots
\]
Regularizing gives:
\begin{multline*}
  \hG_X(t) = 1 + 24 t^2 + 192 t^3 + 2904 t^4 + 40320 t^5 + 611520 t^6 + 9515520 t^7 + \
  152412120 t^8 + 2491104000 t^9 + \cdots
\end{multline*}

\subsection*{Minkowski period sequence:} \href{http://www.grdb.co.uk/search/period3?id=143&printlevel=2}{143}


\addtocounter{CustomSectionCounter}{1}
\section{The Fano Manifold $V_{18}$}
\label{anchor:V18}

\subsection*{Name:} $V_{18}$

\subsection*{Iskovskikh Classification:} This is case 17 in \cite{Isk:2}*{Table~6.5}.

\subsection*{Construction:} The vanishing locus $X$ of a general
section of the vector bundle:
\[
\bigl( S \otimes \det S^\star\big) \oplus \det S^{\star \,\oplus 2} 
\]
on $\Gr(5,7)$.

\subsection*{A remark on the construction:} The paper \cite{Mukai-G2}
is devoted to this case and it is shown there that $X$ is a complete
intersection of $2$ hyperplane sections of a ($5$-dimensional)
homogeneous space $\Sigma=G_2/P$ for the exceptional Lie group
$G_2$. It is not hard to argue from first principles that $\Sigma$ is
the vanishing locus of a general section of $S^\star \otimes \det
S^\star$. We sketch this here, assuming that the reader is acquainted
with basic facts about the geometry of the Lie group $G_2$. Fix a
$7$-dimensional complex vector space $V=\CC^7$ and a $3$-form
$\varphi\in \wedge^3 V^\star$ in the generic $GL_7(\CC)$-orbit; we may
take:
\[
\varphi =
dx^{124}+dx^{235}+dx^{346}+dx^{457}+dx^{561}+dx^{672}+dx^{713} 
\]
where $dx^{ijk}=dx^i\wedge dx^j\wedge dx^k$. Then:
\[
\Sigma = \bigl\{W\in \Gr(2,V)\mid \varphi (w_1,w_2, \cdot) \equiv 0 \;
\text{for all}\; w_1,w_2\in W \bigr\} 
\]
As usual denote by $0\to S \to V\to Q\to 0$ the tautological sequence
on $\Gr(2,V)$. Note that $\rk S^\star =2$, hence $\wedge^3S^\star =0$,
and therefore there is a natural homomorphism $\wedge^3 V^\star \to
Q^\star \otimes (\wedge^2 S^\star)$.  This homomorphism allows us to:
\begin{itemize}
\item view $\varphi$ as an element $s_\varphi \in
  H^0\bigl(\Gr(2,7);Q^\star\otimes (\det S^\star)  \bigr)$; and
\item identify $\Sigma$ with $Z(s_\varphi)$.
\end{itemize}
Finally, we get our construction upon identifying $\Gr(2,V)$ with $\Gr(5,V^\star)$.
 
\subsection*{The quantum period:} We apply
Theorem~\ref{thm:rank_1_A_nA} with $a=2$, $d=1$ and $b=c=e=0$.  This
yields:
\[
G_X(t) = \textstyle
1 + 9t^2 + 20t^3 + \frac{261}{4}t^4 + 153t^5 + \frac{1317}{4}t^6 + 621t^7 +
\frac{67581}{64}t^8 + \frac{351641}{216}t^9 + \cdots
\]
Regularizing gives:
\begin{multline*}
  \hG_X(t) = 1 + 18t^2 + 120t^3 + 1566t^4 + 18360t^5 + 237060t^6 +
  3129840t^7 \\ +  42576030t^8 + 590756880t^9 + \cdots
\end{multline*}

\subsection*{Minkowski period sequence:} \href{http://www.grdb.co.uk/search/period3?id=124&printlevel=2}{124}


\addtocounter{CustomSectionCounter}{1}
\section{The Fano Manifold $V_{22}$}
\label{anchor:V22}

\subsection*{Name:} $V_{22}$

\subsection*{Iskovskikh Classification:} This is case 18 in \cite{Isk:2}*{Table~6.5}.

\subsection*{Construction:} The vanishing locus $X$ of a general
section of the vector bundle:
\[
\bigl( S \otimes \det S^\star\big)^{\oplus \, 3}
\]
on $\Gr(3,7)$ (cf.~\citelist{\cite{Mukai:Trieste}\cite{Mukai:new_developments}}).

\subsection*{The quantum period:} 
We apply Theorem~\ref{thm:rank_1_A_nA} with $d=3$ and $a=b=c=e=0$.
This yields:
\[
G_X(t) = \textstyle 1 + 6t^2 + 10t^3 + \frac{53}{2}t^4 + 48t^5 +
\frac{977}{12}t^6 + 120t^7 + \frac{5117}{32}t^8 +
\frac{5210}{27}t^9 + \cdots
\]
Regularizing gives:
\[
\hG_X(t) = 1 + 12t^2 + 60t^3 + 636t^4 + 5760t^5 + 58620t^6 +
604800t^7 + 6447420t^8 + 70022400t^9 + \cdots
\]

\subsection*{Minkowski period sequence:} \href{http://www.grdb.co.uk/search/period3?id=113&printlevel=2}{113}

\addtocounter{CustomSectionCounter}{1}

\section{The Fano Manifold $\MM{2}{1}$}
\label{sec:2-1}
\label{anchor:2--1}

\subsection*{Mori--Mukai name:} 2--1 

\subsection*{Mori--Mukai construction:} The blow-up of $B_1$ (see \S\ref{sec:B1}) with
centre an elliptic curve which is the intersection of two members of $
|{-\frac{1}{2}}K_{B_1}|$.

\subsection*{Our construction:} A divisor $X$ of bidegree $(1,1)$ in
the product $\PP^1 \times B_1$.

\subsection*{The two constructions coincide:} Apply
Lemma~\ref{lem:blowups} with $V = \cO_{B_1} \oplus \cO_{B_1}$,
$W={-\frac{1}{2}}K_{B_1}$, and $f\colon V \to W$ the map given by the two
sections of ${-\frac{1}{2}}K_{B_1}$ that define the elliptic curve.

\subsection*{The quantum period:} Combining Example~\ref{ex:P1}, the
calculation in Section~\ref{sec:B1}, and Corollary~\ref{cor:products},
we have:
\[
G_{\PP^1 \times B_1}(t) = \sum_{l=0}^\infty \sum_{m =0}^\infty
t^{2l+2m} \frac{(6m)!}{(l!)^2 (m!)^3 (2m)! (3m)!}
\]
Applying Remark~\ref{rem:twisting_by_multiples_of_K} yields:
\[
G_X(t) = e^{-61t} \sum_{l=0}^\infty \sum_{m =0}^\infty
t^{l+m} \frac{(6m)!(l+m)!}{(l!)^2 (m!)^3 (2m)! (3m)!}
\]
and regularizing gives:
\begin{multline*}
  \hG_X(t) = 1+10380 t^2+2082840 t^3+650599740 t^4+199351017360
  t^5+64604751907800 t^6\\ +21521865311226000 t^7+7344504146141322300
  t^8+2554251417295177437600 t^9 + \cdots
\end{multline*}

\subsection*{Minkowski period sequence:} None.  Note that the
anticanonical line bundle of $X$ is not very ample.


\addtocounter{CustomSectionCounter}{1}
\section{The Fano Manifold $\MM{2}{2}$}
\label{sec:2-2}
\label{anchor:2--2}

\subsection*{Mori--Mukai name:} 2--2

\subsection*{Mori--Mukai construction:} A double cover of $ \PP^1
\times \PP^2$ branched along a divisor of bidegree $ (2,4)$.

\subsection*{Our construction:} A member $X$ of $|2L+4M|$ in the toric
variety $F$ with weight data:
\[
\begin{array}{rrrrrrl} 
  \multicolumn{1}{c}{x_0} & 
  \multicolumn{1}{c}{x_1} & 
  \multicolumn{1}{c}{y_0} & 
  \multicolumn{1}{c}{y_1} & 
  \multicolumn{1}{c}{y_2} & 
  \multicolumn{1}{c}{w} & \\ 
  \cmidrule{1-6}
  1 & 1 & 0 & 0 & 0 & 1  & \hspace{1.5ex} L\\ 
  0 & 0 & 1 & 1 & 1 & 2 & \hspace{1.5ex} M \\
\end{array}
\]
and $\Amp F = \langle L, L+2M \rangle$.  We have:
\begin{itemize}
\item $-K_F=3L+5M$ is ample, that is $F$ is a smooth Fano orbifold\footnote{By `smooth orbifold', we mean `smooth Deligne--Mumford stack over $\CC$'.  Excellent introductions to Deligne--Mumford stacks can be found in~\cite{Fantechi} and~\cite{Vistoli}*{Appendix}; note that in the latter reference Deligne--Mumford stacks are called `algebraic stacks'. By `smooth Fano orbifold', we mean `smooth orbifold such that the coarse moduli space is a Fano variety'.};
\item $X\sim 2L+4M$ is nef;
\item $-(K_F+X)\sim L+M$ is ample.
\end{itemize}

\subsection*{The two constructions coincide:} Consider the defining
equation of $ X$ to be $ w^2=f_{2,4}$, where $f_{2,4}$ is a
bihomogeneous polynomial of degrees $2$ in $x_0$,~$x_1$ and $4$ in
$y_0$,~$y_1$,~$y_2$.  Let $p \colon F \dashrightarrow \PP^1 \times
\PP^2$ be the rational map which sends (contravariantly) the
homogeneous co-ordinate functions $[x_0,x_1,y_0,y_1,y_2] $ on
$\PP^1_{x_0,x_1} \times \PP^2_{y_0,y_1,y_2}$ to
$[x_0,x_1,y_0,y_1,y_2]$.  The restriction of $p$ to $X$ is a morphism,
which exhibits $X$ as a double cover of $\PP^1 \times \PP^2$ branched
over the locus $(f_{2,4}=0) \subset \PP^1_{x_0,x_1} \times
\PP^2_{y_0,y_1,y_2}$.

\subsection*{Remarks on our construction:} Next we make some comments
on the geometry of $X$ and the embedding $X\subset F$ that are not
logically necessary for the computation of the quantum period: this
subsection can safely be skipped by the impatient reader.  In
particular we explain why in this case $\Amp F$ is a proper subset of
$\Amp X$.  The toric variety $F$ is defined by the requirement that
$\Amp F = \langle L, L+2M\rangle $; the unstable locus is
$(x_0=x_1=0)\cup (y_0=y_1=y_2=w=0)$ and:
\[
F=\bigl[(\CC^\times)^2\times(\CC^\times)^4/\TT^2\bigr] 
\]
Note that $F$ is itself a Fano variety---or, more precisely, a smooth
Fano orbifold---and $X$ is a nef divisor on $F$ such that ${-(K_F+X)}$
is ample, so the given model is well-adapted for computing the quantum
cohomology of $X$ via Quantum Lefschetz.  The linear system
$|L|=|x_0,x_1|$ defines a morphism $f\colon F \to \PP^1_{x_0,x_1}$
with fibre the weighted projective space $\PP(1,1,1,2)$; the
restriction $f|_{X}\colon X\to \PP^1$ is one of the two extremal
contractions of $X$. On the other hand, the linear system
$|M|=|y_0,y_1,y_2|$ is not base point free on $F$: the base locus is a
section $C$ of the morphism $f$. When restricted to $X$, however, this
linear system is free and it defines the ``other'' extremal
contraction $X \to \PP^2$. In particular, we see that $\langle L,L+2M
\rangle = \Amp F\subsetneq \Amp X = \langle L, M \rangle$. How can we
see the rest of $\Amp X$?

Let us denote by $F^\prime$ the toric variety corresponding to the
``other'' chamber, so that $\Amp F^\prime=\langle L+2M, M\rangle$ and
the unstable locus is now $(y_0=y_1=y_2=0)\cup (x_0=x_1=w=0)$. Note
that $F^\prime$ is the flip of $F$ along the curve
$C=(y_0=y_1=y_2=0)\subset F$. $X$ is a member of $|2L+4M|$, a nef
linear system on $F^\prime$, but ${-(K_{F^\prime}+X)}$ is not nef on
$F'$ and so this construction of $X$ is not well-adapted for computing
the quantum cohomology of $X$ via Quantum Lefschetz. Nevertheless,
$\Amp X = \Amp F + \Amp F^\prime$, so we need $F^\prime$ to see all of
$\Amp X$. The linear system $|y_0,y_1,y_2|$ is free on $F^\prime$ and
it defines an extremal contraction $g\colon F^\prime \to \PP^2$ with
fibre $\PP^2$; this also gives the missing extremal contraction of
$X$.

\subsection*{The quantum period:} If we assume a mirror theorem for
toric orbifolds in the form \cite{Iritani:integral}*{Conjecture~4.3}
then we can apply the Quantum Lefschetz theorem for orbifolds
\cite{CCIT}, exactly as in Proposition~\ref{pro:wps}, to obtain:
\[
G_X(t) = e^{-14t} \sum_{l=0}^\infty \sum_{m=0}^\infty
t^{l+m} \frac{(2l+4m)!}{(l!)^2 (m!)^3 (l+2m)!}
\]
Regularizing gives:
\begin{multline}
  \label{eq:regularized_quantum_period_2_2}
  \hG_X(t) = 1+470 t^2+21216 t^3+1562778 t^4+114717120 t^5+9003183140
  t^6+731280419520 t^7\\
  +61092935052730 t^8+5214279501137280 t^9+ \cdots
\end{multline}

\subsection*{Minkowski period sequence:} None.  Note that the
anticanonical line bundle of $X$ is not very ample.

\subsection*{The quantum period, alternative construction:} There is
as yet no proof of \cite{Iritani:integral}*{Conjecture~4.3} in the
literature so we give an alternative calculation of the quantum period
for $X$.  This uses a different model of $X$, as a member of $|2N|$ in
the toric variety $F$ with weight data:
\[
\begin{array}{rrrrrrrl} 
  \multicolumn{1}{c}{x_0} & 
  \multicolumn{1}{c}{x_1} & 
  \multicolumn{1}{c}{y_0} & 
  \multicolumn{1}{c}{y_1} & 
  \multicolumn{1}{c}{y_2} & 
  \multicolumn{1}{c}{w} & 
  \multicolumn{1}{c}{z} & 
  \\ 
  \cmidrule{1-7}
  1 & 1 & 0 & 0 & 0 & -1 & 0 & \hspace{1.5ex} L\\ 
  0 & 0 & 1 & 1 & 1 & -2 & 0 & \hspace{1.5ex} M \\
  0 & 0 & 0 & 0 & 0 &   1 & 1 & \hspace{1.5ex} N
\end{array}
\]
and $\Amp F = \langle L, M, N \rangle$.  The variety $F$ is the
projective bundle $\PP(\cO_{\PP^1 \times \PP^2}(-1,-2) \oplus
\cO_{\PP^1 \times \PP^2})$ over $\PP^2$.  Let $p\colon F \to \PP^1
\times \PP^2$ be the projection map, and consider the defining
equation of $X$ to be:
\[
z^2 - w^2 f_{2,4} = 0
\]
where $f_{2,4}$ is a bihomogeneous polynomial of degrees $2$ in
$x_0$,~$x_1$ and $4$ in $y_0$,~$y_1$,~$y_2$.  The restriction $p|_X
\colon X \to \PP^1 \times \PP^2$ exhibits $X$ as a double cover of
$\PP^1 \times \PP^2$ branched over the locus $(f_{2,4}=0) \subset
\PP^1_{x_0,x_1} \times \PP^2_{y_0,y_1,y_2}$. 

We now compute the quantum period of $X$.  Let $p_1$, $p_2, p_3 \in
H^\bullet(F;\ZZ)$ denote the first Chern classes of $L$, $M$, and $N$
respectively; these classes form a basis for $H^2(F;\ZZ)$.  Write
$\tau \in H^2(F;\QQ)$ as $\tau = \tau_1 p_1 + \tau_2 p_2+ \tau_3 p_3$
and identify the group ring $\QQ[H_2(F;\ZZ)]$ with the polynomial ring
$\QQ[Q_1,Q_2,Q_3]$ via the $\QQ$-linear map that sends the element
$Q^\beta \in \QQ[H_2(F;\ZZ)]$ to $Q_1^{\langle \beta,p_1\rangle}
Q_2^{\langle \beta,p_2\rangle} Q_3^{\langle \beta,p_3\rangle}$.  The
toric variety $F$ is Fano; Theorem~\ref{thm:toric_mirror} gives:
\[
J_F(\tau) = e^{\tau/z} 
\sum_{l, m, n \geq 0} 
\frac{Q_1^l Q_2^m Q_3^n e^{l \tau_1} e^{m \tau_2} e^{n \tau_3}}
{
  \prod_{k=1}^l (p_1 + k z)^2 
  \prod_{k=1}^m (p_2 + k z)^3
  \prod_{k=1}^n (p_3 + k z) 
}
\frac{\prod_{k = -\infty}^0 (p_3-p_1-2p_2 + k z)}
{\prod_{k = -\infty}^{n-l-2m} (p_3-p_1-2p_2 + k z)}
\]
and hence:
\[
I_{\be,E}(\tau) = e^{\tau/z} 
\sum_{l, m, n \geq 0} 
\frac{
  Q_1^l Q_2^m Q_3^n e^{l \tau_1} e^{m \tau_2} e^{n \tau_3}
  \prod_{k=1}^{2n} (\lambda + 2p_3 + k z)
}
{
  \prod_{k=1}^l (p_1 + k z)^2 
  \prod_{k=1}^m (p_2 + k z)^3
  \prod_{k=1}^n (p_3 + k z) 
}
\frac{\prod_{k = -\infty}^0 (p_3-p_1-2p_2 + k z)}
{\prod_{k = -\infty}^{n-l-2m} (p_3-p_1-2p_2 + k z)}
\]
We have:
\[
I_{\be,E}(\tau) = A(\tau)+ B(\tau)  z^{-1} + O(z^{-2})
\]
where:
\begin{align*}
  A(\tau) & = \sum_{n=0}^\infty Q_3^n e^{n \tau_3} \frac{(2n)!}{(n!)^2}\\
  & = (1 - 4 Q_3 e^{\tau_3})^{-1/2} \\
  B(\tau) & = \sum_{n=1}^\infty Q_1 e^{\tau_1} Q_3^n e^{n \tau_3} 
  \frac{(2n)!}{n!(n-1)!} 
  + \sum_{n=2}^\infty Q_2 e^{\tau_2} Q_3^n e^{n \tau_3} 
  \frac{(2n)!}{n!(n-2)!} \\
  & \qquad \qquad
  + \sum_{n = 0}^\infty Q_3^n e^{n \tau_3} \frac{(2n)!}{(n!)^2}
  \Big( (\lambda + 2p_3) H_{2n} - p_3 H_n - (p_3 - p_1 - 2p_2) H_n \Big)
  \\
  \intertext{and $H_m$ is the $m$th harmonic number.  In the notation of Corollary~\ref{cor:toric_ci_JYX}, we have:}
  A(\tau) &= (1 - 4 Q_3 e^{\tau_3})^{-1/2} \\
  B'(\tau) &= \sum_{n=1}^\infty Q_1 e^{\tau_1} Q_3^n e^{n \tau_3} 
  \frac{(2n)!}{n!(n-1)!} 
  + \sum_{n=2}^\infty Q_2 e^{\tau_2} Q_3^n e^{n \tau_3} 
  \frac{(2n)!}{n!(n-2)!} \\
  & \qquad \qquad
  + \sum_{n = 0}^\infty Q_3^n e^{n \tau_3} \frac{(2n)!}{(n!)^2}
  \Big(p_3 (2H_{2n} - H_n) - (p_3 - p_1 - 2p_2) H_n \Big) \\
  & = 2 Q_1 e^{\tau_1} Q_3 e^{\tau_3} (1-4 Q_3 e^{\tau_3})^{-3/2} + 
  12 Q_2 e^{\tau_2} Q_3^2 e^{2 \tau_3} (1-4 Q_3 e^{\tau_3})^{-5/2} \\
  & \qquad \qquad 
  - p_3 (1 - 4 Q_3 e^{\tau_3})^{-1/2} \log (1 - 4 Q_3 e^{\tau_3}) 
  - (p_3 - p_1 - 2p_2) 
  \sum_{n = 0}^\infty Q_3^n e^{n \tau_3} \frac{(2n)!}{(n!)^2}
  H_n
\end{align*}
Corollary~\ref{cor:toric_ci_JYX} gives:
\[
J_{Y,X}(\theta(\tau)) = (1 - 4 Q_3 e^{\tau_3})^{1/2} I_{Y,X}(\tau)
\]
where:
\begin{align*}
  & \theta(\tau) = \tau + \frac{2 Q_1 e^{\tau_1} Q_3 e^{\tau_3}}{1-4 Q_3
    e^{\tau_3}} + \frac{12 Q_2 e^{\tau_2} Q_3^2 e^{2 \tau_3}}{(1-4
    Q_3 e^{\tau_3})^2} - p_3 \log (1 - 4 Q_3 e^{\tau_3}) - (p_3 -
  p_1 - 2p_2) F \\
  & F = (1-4 Q_3 e^{\tau_3})^{1/2}
  \sum_{n = 0}^\infty Q_3^n e^{n \tau_3} \frac{(2n)!}{(n!)^2}
  H_n
\end{align*}
and:
\[
I_{Y,X}(\tau) = 
e^{\tau/z} 
\sum_{l, m, n \geq 0} 
\frac{
  Q_1^l Q_2^m Q_3^n e^{l \tau_1} e^{m \tau_2} e^{n \tau_3}
  \prod_{k=1}^{2n} (2p_3 + k z)
}
{
  \prod_{k=1}^l (p_1 + k z)^2 
  \prod_{k=1}^m (p_2 + k z)^3
  \prod_{k=1}^n (p_3 + k z) 
}
\frac{\prod_{k = -\infty}^0 (p_3-p_1-2p_2 + k z)}
{\prod_{k = -\infty}^{n-l-2m} (p_3-p_1-2p_2 + k z)}
\]

From equation~\ref{eq:KKP}, we have that:
\[
j_\star J_X\big(j^\star \theta(\tau)\big) = 2p_3 (1 - 4 Q_3 e^{\tau_3})^{1/2} I_{Y,X}(\tau)
\]
where $j \colon X \to F$ is the inclusion map and equality holds
after applying the map of coefficient rings $\Lambda_X \to
\Lambda_F$ induced by $j$.  Note that $j^\star (p_3 - p_1 -2 p_2)
=0$; this reflects the fact that $X$ is disjoint from the divisor
$w=0$.  Consider the classes $p_1' = j^\star p_1$ and $p_2' =
j^\star p_2$.  These form a basis for $H^2(X)$, and we identify the
group ring $\QQ[H_2(X;\ZZ)]$ with the polynomial ring $\QQ[q_1,q_2]$
via the $\QQ$-linear map that sends the element $Q^\beta \in
\QQ[H_2(F;\ZZ)]$ to $q_1^{\langle \beta,p'_1\rangle} q_2^{\langle
  \beta,p'_2\rangle}$.  The map $\Lambda_X \to \Lambda_F$ induced by
$j$ sends $q_1$ to $Q_1 Q_3$ and $q_2$ to $Q_2 Q_3^2$.  We have:
\[
j^\star \theta(\tau) = (\tau_1 + \tau_3) p_1' + (\tau_2  + 2 \tau_3) p_2' +
\frac{2 Q_1 e^{\tau_1} Q_3 e^{\tau_3}}{1-4 Q_3
  e^{\tau_3}} + \frac{12 Q_2 e^{\tau_2} Q_3^2 e^{2 \tau_3}}{(1-4
  Q_3 e^{\tau_3})^2} - (p_1'+2p_2')\log (1 - 4 Q_3 e^{\tau_3})
\]
and thus, from equation~\ref{eq:JX}:
\begin{multline*}
  J_X\big(j^\star \theta(\tau)\big) = 
  \exp\left(\Big(\textstyle \frac{2 Q_1 e^{\tau_1} Q_3 e^{\tau_3}}{1-4 Q_3
      e^{\tau_3}} + \frac{12 Q_2 e^{\tau_2} Q_3^2 e^{2 \tau_3}}{(1-4
      Q_3 e^{\tau_3})^2} \Big)/z\right) \times \\
  J_X\big((\tau_1 + \tau_3) p_1' + (\tau_2  + 2 \tau_3) p_2' \big)
  \Big|_{Q_1 = \frac{Q_1}{1 - 4 Q_3 e^{\tau_3}}, Q_2 = \frac{Q_2}{(1 - 4 Q_3 e^{\tau_3})^2}}
\end{multline*}
Making the inverse change of variables $Q_1 = Q_1 (1-4 Q_3
e^{\tau_3})$, $Q_2 = Q_2 (1-4 Q_3 e^{\tau_3})^2$, we see
that\footnote{The right-hand side of \eqref{eq:2-2_almost_there}
  depends on $Q_1$, $Q_2$, $Q_3$ only through the products $Q_1 Q_3$
  and $Q_1 Q_3^2$, but this is not manifest from the formula.  We
  will see it explicitly for the coefficient of $2p_3$ in
  \eqref{eq:2-2_almost_there} below.}:
\begin{equation}
  \label{eq:2-2_almost_there}
  j_\star J_X(0) = 
  e^{{-(2 Q_1 Q_3+12 Q_2 Q_3^2)}/z}
  2p_3 (1 - 4 Q_3)^{1/2} I_{Y,X}(0)
  \Big|_{Q_1 = Q_1(1 - 4 Q_3), Q_2 = Q_2(1 - 4 Q_3)^2}
\end{equation}

Recall that the quantum period $G_X$ is obtained from the component of
$J_X(0)$ along the unit class $1 \in H^\bullet(X;\QQ)$ by setting $z =
1$ and $Q^\beta = t^{\langle \beta, {-K_X} \rangle}$.  To obtain
$G_X$, therefore, we need to extract the coefficient of $2p_3$ on the
right-hand side of \eqref{eq:2-2_almost_there}, set $z=1$, and set:
\begin{align*}
  Q_1 Q_2 = t && Q_1 Q_3^2 = t
\end{align*}
(this amounts to setting $q_1 = q_2 = t$ and then applying the map of
coefficient rings $\Lambda_X \to \Lambda_F$ induced by the inclusion
$j$).  Observe that $p_3(p_3-p_1 - 2p_2) = 0$ in $H^\bullet(F)$.
Taking the coefficient of $2p_3$ on the right-hand side of
\eqref{eq:2-2_almost_there} and setting $z=1$ thus gives:
\begin{align*}
  & 
  e^{-(2 Q_1 Q_3+12 Q_2 Q_3^2)}
  \sum_{l=0}^\infty \sum_{m=0}^\infty \sum_{n=l+2m}^\infty
  Q_1^l Q_2^m Q_3^n (1-4Q_3)^{l+2m+\frac{1}{2}}
  \frac{
    (2n)!
  }
  {
    (l!)^2 (m!)^3 n! (n-l-2m)!
  } \\
  & =
  e^{-(2 Q_1 Q_3+12 Q_2 Q_3^2)}
  \sum_{l=0}^\infty \sum_{m=0}^\infty
  \frac{Q_1^l Q_2^m Q_3^{l+2m} (1-4Q_3)^{l+2m+\frac{1}{2}}}
  {(l!)^2 (m!)^3}
  \sum_{n=l+2m}^\infty
  Q_3^{n-l-2m} 
  \frac{
    (2n)!
  }
  {
    n! (n-l-2m)!
  } \\
    & =
    e^{-(2 Q_1 Q_3+12 Q_2 Q_3^2)}
  \sum_{l=0}^\infty \sum_{m=0}^\infty
  \frac{Q_1^l Q_2^m Q_3^{l+2m} (1-4Q_3)^{l+2m+\frac{1}{2}}}
  {(l!)^2 (m!)^3}
  \Big(\textstyle \frac{d}{dQ_3}\Big)^{l+2m}
  (1-4Q_3)^{-\frac{1}{2}}
  \\
  & =
    e^{-(2 Q_1 Q_3+12 Q_2 Q_3^2)}
  \sum_{l=0}^\infty \sum_{m=0}^\infty
  Q_1^l Q_2^m Q_3^{l+2m} 
  \frac{(2l+4m)!}
  {(l!)^2 (m!)^3(l+2m)!}
\end{align*}
Setting $Q_1 Q_3 = t$, $Q_1 Q_3^2 = t$ yields: 
\[
G_X(t) = e^{-14t} \sum_{l=0}^\infty \sum_{m = 0}^\infty t^{l+m} \frac{(2l+4m)!}{(l!)^2 (m!)^3 (l+2m)!}
\]
and regularizing gives \eqref{eq:regularized_quantum_period_2_2}, as before.


\addtocounter{CustomSectionCounter}{1}

\section{The Fano Manifold $\MM{2}{3}$}
\label{anchor:2--3}

\subsection*{Mori--Mukai name:} 2--3

\subsection*{Mori--Mukai construction:} The blow-up of $B_2$ with
centre an elliptic curve that is the intersection of two members of $
|{-\frac{1}{2}K_{B_2}}|$.

\subsection*{Our construction:} A divisor $X$ of bidegree $(1,1)$ in
the product $\PP^1 \times B_2$.

\subsection*{The two constructions coincide:} Apply
Lemma~\ref{lem:blowups} with $V = \cO_{B_2} \oplus \cO_{B_2}$,
$W={-\frac{1}{2}}K_{B_2}$, and $f\colon V \to W$ the map given by the two
sections of ${-\frac{1}{2}}K_{B_2}$ that define the elliptic curve.

\subsection*{The quantum period:}
Combining Example~\ref{ex:P1}, the calculation in
Section~\ref{sec:B2}, and Corollary~\ref{cor:products}, we have:
\[
G_{\PP^1 \times B_2}(t) = \sum_{l=0}^\infty \sum_{m =0}^\infty
t^{2l+2m} \frac{(4m)!}{(l!)^2 (m!)^4 (2m)!}
\]
Applying Remark~\ref{rem:twisting_by_multiples_of_K} yields:
\[
G_X(t) = e^{-13t} \sum_{l=0}^\infty \sum_{m =0}^\infty
t^{l+m} \frac{(4m)!(l+m)!}{(l!)^2 (m!)^4 (2m)!}
\]
and regularizing gives:
\begin{multline*}
  \hG_X(t) = 1 + 300t^2 + 8472t^3 + 438588t^4 + 21183120t^5 +
  1115221080t^6 + 60512230800t^7 +
  \\ 3385779824700t^8 + 193681282922400t^9 + \cdots
\end{multline*}

\subsection*{Minkowski period sequence:} None.  Note that the
anticanonical line bundle of $X$ is not very ample.


\addtocounter{CustomSectionCounter}{1}

\section{The Fano Manifold $\MM{2}{4}$}
\label{anchor:2--4}

\subsection*{Mori--Mukai name:} 2--4

\subsection*{Mori--Mukai construction:} The blow-up of $ \PP^3$ with
centre an intersection of two cubics.

\subsection*{Our construction:} A member $X$ of $|L+3M|$ in the toric
variety $F=\PP^1\times \PP^3$.

\subsection*{The two constructions coincide:} Apply
Lemma~\ref{lem:blowups} with $V = \cO_{\PP^3} \oplus \cO_{\PP^3}$, $W
= \cO_{\PP^3}(3)$, and $f\colon V \to W$ given by the two cubics that define the
centre of the blow-up.

\subsection*{The quantum period:}  The toric variety $F$ has weight
data:
\[
\begin{array}{cccccc|c}
  1 & 1 & 0 & 0 & 0 & 0 & L \\
  0 & 0 & 1 & 1 & 1 & 1 & M
\end{array}
\]
and $\Amp F = \langle L,M\rangle$.  We have:
\begin{itemize}
\item $F$ is a Fano variety;
\item $X\sim L+3M$ is ample;
\item $-(K_F+X)\sim L+M$ is ample.
\end{itemize}
Corollary~\ref{cor:QL} yields:
\[
G_X(t) = e^{-7t} \sum_{l=0}^\infty \sum_{m=0}^\infty
t^{l+m} 
\frac{(l+3m)!}{(l!)^2 (m!)^4}
\]
and regularizing gives:
\begin{multline*}
\hG_X(t) = 1 + 90 t^2 + 1518 t^3 + 46086 t^4 + 1327320 t^5 + 41383350 t^6 + 1329442380 t^7 \\
   + 43944315030 t^8 + 1483208104560 t^9 + \cdots
 \end{multline*}

\subsection*{Minkowski period sequence:} \href{http://www.grdb.co.uk/search/period3?id=161&printlevel=2}{161}


\addtocounter{CustomSectionCounter}{1}

\section{The Fano Manifold $\MM{2}{5}$}
\label{anchor:2--5}

\subsection*{Mori--Mukai name:} 2--5

\subsection*{Mori--Mukai construction:} The blow-up of $ B_3$ with
centre a plane cubic on it.

\subsection*{Our construction:} A member $X$ of $|3M|$ in the toric
variety $F$ with weight data:
\[ 
\begin{array}{rrrrrrl} 
\multicolumn{1}{c}{s_0} & 
\multicolumn{1}{c}{s_1} & 
\multicolumn{1}{c}{x} & 
\multicolumn{1}{c}{x_2} & 
\multicolumn{1}{c}{x_3} & 
\multicolumn{1}{c}{x_4} & \\ 
\cmidrule{1-6}
1 & 1 & -1 & 0 & 0 & 0 & \hspace{1.5ex} L\\ 
0 & 0 & 1 & 1 & 1 & 1 & \hspace{1.5ex} M \\
\end{array}
\]
and $\Amp F = \langle L,M\rangle$.  We have:
\begin{itemize}
\item $-K_F=L+4M$ is ample, that is $F$ is a Fano variety;
\item $X\sim 3M$ is nef;
\item $-(K_F+X)\sim L+M$ is ample.
\end{itemize}

\subsection*{The two constructions coincide:} The notation makes it
clear that $ s_0,s_1$ are sections of $ L$; $ xs_0, xs_1,x_2, x_3,x_4$
are sections of $ M$; and $F$ is a scroll over $ \PP^1$ with fibre $
\PP^3$. The morphism $F \to \PP^4 $ that sends (contravariantly) the
homogeneous co-ordinate functions $[x_0,\dots,x_4]$ to
$[xs_0,xs_1,x_2,x_3,x_4]$ is the blow-up along $x_0=x_1=0$.

\subsection*{The quantum period:} Corollary~\ref{cor:QL} yields:
\[
G_X(t) = e^{-6t} \sum_{l=0}^\infty \sum_{m=l}^\infty
t^{l+m}
\frac{(3m)!}
{(l!)^2(m-l)! (m!)^3}
\]
and regularizing gives:
\begin{multline*}
  \hG_X(t) = 1+66 t^2+816 t^3+20214 t^4+449640 t^5+11050500 t^6+
  278336520 t^7+\\ 7229175030 t^8+191680807920 t^9 + \cdots
\end{multline*}

\subsection*{Minkowski period sequence:} \href{http://www.grdb.co.uk/search/period3?id=158&printlevel=2}{158}


\addtocounter{CustomSectionCounter}{1}

\section{The Fano Manifold $\MM{2}{6}$}
\label{anchor:2--6}

\subsection*{Mori--Mukai name:} 2--6

\subsection*{Mori--Mukai construction:}
\begin{itemize}
\item[(a)] A divisor of bidegree $(2,2)$ on $\PP^2\times \PP^2$;
\item[(b)] A double cover of $W\subset \PP^2\times \PP^2$ (the divisor
  of bidegree (1,1) on $\PP^2\times \PP^2)$ whose branch locus is a
  member of $|{-K_W}|$.
\end{itemize}

\subsection*{Our construction:} A member $X$ of $|2L+2M|$ in the toric
variety $F=\PP^2\times \PP^2$.

\subsection*{The two constructions coincide:} Obvious.

\subsection*{The quantum period:} This is Example~\ref{ex:W22}.  We have:
\begin{multline*}
  \hG_X(t) = 1+44 t^2+528 t^3+11292 t^4+228000 t^5+4999040
  t^6+112654080 t^7 \\+2613620380 t^8+61885803840 t^9+ \cdots
\end{multline*}

\subsection*{Minkowski period sequence:} \href{http://www.grdb.co.uk/search/period3?id=149&printlevel=2}{149}


\addtocounter{CustomSectionCounter}{1}

\section{The Fano Manifold $\MM{2}{7}$}
\label{anchor:2--7}

\subsection*{Mori--Mukai name:} 2--7

\subsection*{Mori--Mukai construction:} The blow-up of a quadric
\mbox{3-fold} $Q\subset \PP^4$ with centre the intersection of two
members of $ |\cO_Q(2)|$.

\subsection*{Our construction:} A codimension-2 complete intersection
$X$ of type $(2M) \cap (L+2M)$ on the toric variety $F=\PP^1\times
\PP^4$.

\subsection*{The two constructions coincide:} Apply
Lemma~\ref{lem:blowups} with $V = \cO_Q \oplus \cO_Q$, $W = \cO_Q(2)$,
and $f\colon V \to W$ given by the two sections of $\cO_Q(2)$ that define
the centre of the blow-up.  This shows that $X$ is a divisor of
bidegree $(1,2)$ on $\PP^1 \times Q$, or in other words a complete
intersection of type $(2M) \cap (L+2M)$ on $\PP^1\times \PP^4$.

\subsection*{The quantum period:} 
The toric variety $F$ has weight data:
\[ 
\begin{array}{rrrrrrrl} 
1 & 1 & 0 & 0 & 0 & 0 & 0 & \hspace{1.5ex} L\\ 
0 & 0 & 1 & 1 & 1 & 1 & 1 & \hspace{1.5ex} M \\
\end{array}
\]
and $\Amp F = \langle L,M\rangle$. We have:
\begin{itemize}
\item $F$ is a Fano variety;
\item $X$ is the complete intersection of two nef divisors on $F$;
\item $-(K_F+\Lambda)\sim L+M$ is ample.
\end{itemize}
Corollary~\ref{cor:QL} yields:
\[
G_X(t) = e^{-5t} \sum_{l=0}^\infty \sum_{m=0}^\infty
t^{l+m}
\frac{(2m)!(l+2m)!}
{(l!)^2 (m!)^5}
\]
and regularizing gives:
\begin{multline*}
  \hG_X(t) = 1+36 t^2+348 t^3+6516 t^4+110880 t^5+2069820 t^6+39606000
  t^7\\
  +780530100 t^8+15697106880 t^9+ \cdots
\end{multline*}

\subsection*{Minkowski period sequence:} \href{http://www.grdb.co.uk/search/period3?id=148&printlevel=2}{148}


\addtocounter{CustomSectionCounter}{1}

\section{The Fano Manifold $\MM{2}{8}$}
\label{anchor:2--8}

\subsection*{Mori--Mukai name:} 2--8

\subsection*{Mori--Mukai construction:} 
\begin{itemize}
\item[(a)] A double cover of $B_7$ (the blow-up of $\PP^3$ at a
  point) with branch locus a member $ B$ of $ |{-K_{B_7}}|$ such that $
  B\cap D$ is nonsingular, where $ D$ is the exceptional divisor of the
  blow-up $ B_7 \to \PP^3$;
\item[(b)] A specialization of (a) where $B\cap D$ is reduced but
  singular.
\end{itemize}

\subsection*{Our construction:} A member $X$ of $|2L+2M|$ in the toric
variety $F$ with weight data:
\[
\begin{array}{rrrrrrl} 
\multicolumn{1}{c}{s_0} & 
\multicolumn{1}{c}{s_1} & 
\multicolumn{1}{c}{s_2} & 
\multicolumn{1}{c}{x} & 
\multicolumn{1}{c}{x_3} & 
\multicolumn{1}{c}{w} & \\ 
\cmidrule{1-6}
1 & 1 & 1 & -1 & 0 & 1 & \hspace{1.5ex} L\\ 
0 & 0 & 0 & 1 & 1 & 1 & \hspace{1.5ex} M \\
\end{array}
\]
and $\Amp F = \langle L,L+M\rangle$. We have:
\begin{itemize}
\item $-K_F=3(L+M)$ is nef and big but not ample, so that $F$ is \emph{not} a Fano variety;
\item $X\sim 2(L+M)$ is nef;
\item $-(K_F+X)\sim L+M$ is nef and big but not ample.
\end{itemize}

\subsection*{The two constructions coincide:} 
Consider the equation of $X$ in the form:
\[
w^2=x_3^2 a_2+x_3xb_3 + x^2c_4
\]
where $a_2$,~$b_3$, and~$c_4$ are generic homogeneous polynomials in
$s_0$,~$s_1$,~$s_2$ of degrees $2$,~$3$, and~$4$ respectively.  The
locus $(w=0) \subset F$ is a copy of $B_7$ and the branch locus meets the
exceptional divisor $D=(x=w=0) \cong \PP^2$ in a nonsingular conic.

\subsection*{Remarks on the birational geometry of $X$:}
Next we make a few comments on the geometry of $X$ and the embedding
$X\subset F$ that are not logically necessary for the computation of
the quantum period. The discussion is similar to the discussion of
2--2 in \S\ref{sec:2-2} above; it in particular shows that $X$ is a
Fano variety, which is not immediately clear from our construction.

The secondary fan manifestly has three maximal cones.  By definition
$\Amp F=\langle L,L+M\rangle$. The irrelevant ideal is $(ws_i, x_3s_i,
xs_i)$ and the unstable locus is:
\[
(s_0=s_1=s_2=0)\cup (w=x=x_3=0)
\] 
The linear system $|L|=|s_0,s_1,s_2|$ defines a morphism $f\colon F
\to \PP^2$ with fibre $\PP^2$ and $f|_{X}$ is a conic bundle (in
particular, an extremal contraction in the Mori category). 
The linear system $|L+M|=|w,x_3s_i, xs_is_j|$ gives a flopping
contraction of $\Pi=(x=x_3=0)\cong \PP^2$ with normal bundle
$\cO(-1)\oplus \cO(-2)$. Note, however, that $X\cap \Pi=\varnothing$:
this contraction maps $X$ isomorphically onto its image.

Denote by $F^\prime$ the toric variety such that $\Amp F^\prime
=\langle L+M,M\rangle$. The irrelevant ideal is:
\[
(x_3w, xw, x_3s_i, xs_i)
\]
and the unstable locus is:
\[
(x=x_3=0)\cup (s_0=s_1=s_2=w=0) 
\]
The linear system $|L+M|$ defines the flop of $F$. On the other hand,
$|M|$ defines a contraction $g\colon F\to \PP(1,1,1,1,2)$ which sends
(contravariantly) the homogeneous co-ordinate functions
$[x_0,x_1,x_2,x_3,y]$ on $\PP(1,1,1,1,2)$ to $[s_0x, s_1x, s_2x, x_3,
wx]$. The restriction $g|_{X}$ maps $X$ to the variety $Y$ with
equation
\[
y^2=x_3^2a_2(x_0,x_1,x_2)+x_3b_3(x_0,x_1,x_2) + c_4(x_0,x_1,x_2) 
\]
so $Y$ is the double cover of $\PP^3$ branched along a general quartic
surface $B$ with an ordinary node at $(0,0,0,1)$, and $g|_{X}\colon X
\to Y$ is an extremal divisorial contraction contracting $X\cap
(x=0)=(w^2=x_3^2a_2(s_0,s_1,s_2))\cong \PP^1\times \PP^1$ to the node
just mentioned.

It follows from the preceding discussion that $\Amp X =\Amp F + \Amp
F^\prime=\langle L, M\rangle$; in particular, therefore, $X$ is Fano.

Finally the chamber $\langle M, M-L \rangle$ is ``hollow'', that is,
taking the GIT quotient with respect to a stability condition from the
interior of this chamber leads to a rank 1 toric variety.

\subsection*{The quantum period:}  
Let $p_1$, $p_2 \in H^\bullet(F;\ZZ)$ denote the first Chern classes
of $L$ and $L \otimes M$ respectively; these classes form a basis for
$H^2(F;\ZZ)$.  Write $\tau \in H^2(F;\QQ)$ as $\tau = \tau_1 p_1 +
\tau_2 p_2$ and identify the group ring $\QQ[H_2(F;\ZZ)]$ with the
polynomial ring $\QQ[Q_1,Q_2]$ via the $\QQ$-linear map that sends the
element $Q^\beta \in \QQ[H_2(F;\ZZ)]$ to $Q_1^{\langle
  \beta,p_1\rangle} Q_2^{\langle \beta,p_2\rangle}$. We have:
\begin{align*}
  I_F(\tau) & = e^{\tau/z} 
  \sum_{l, m\geq 0} 
  \frac{Q_1^l Q_2^m e^{l \tau_1} e^{m \tau_2}}
  {
    \prod_{k=1}^l (p_1 + k z)^3
    \prod_{k=1}^m (p_2 + k z) 
  }
  \frac{\prod_{k = -\infty}^0 (p_2-p_1 + k z)}
  {\prod_{k = -\infty}^{m-l} (p_2-p_1 + k z)}
  \frac{\prod_{k = -\infty}^0 (p_2-2p_1 + k z)}
  {\prod_{k = -\infty}^{m-2l} (p_2-2p_1 + k z)}
  \\
  &= 1 +\tau z^{-1} + O(z^{-2})
\end{align*}
Assumptions~\ref{assumptions:toric_ci} hold and, in the notation of
Proposition~\ref{pro:ABC}, we have $A(\tau) = 1$ and $B(\tau) =
\tau$.  We now proceed exactly as in the proof of
Corollary~\ref{cor:QL}, obtaining:
\begin{multline*}
  I_{F,X}(\tau) = \\ e^{\tau/z} 
  \sum_{l, m\geq 0} 
  \frac{Q_1^l Q_2^m e^{l \tau_1} e^{m \tau_2}
    \prod_{k=1}^{2l+2m}(2p_1+2p_2+k z)}
  {
    \prod_{k=1}^l (p_1 + k z)^3
    \prod_{k=1}^m (p_2 + k z) 
  }
  \frac{\prod_{k = -\infty}^0 (p_2-p_1 + k z)}
  {\prod_{k = -\infty}^{m-l} (p_2-p_1 + k z)}
  \frac{\prod_{k = -\infty}^0 (p_2-2p_1 + k z)}
  {\prod_{k = -\infty}^{m-2l} (p_2-2p_1 + k z)}
\end{multline*}
and:
\[
G_X(t) = e^{-2t} \sum_{l=0}^\infty \sum_{m=2l}^\infty 
t^m
\frac{(2m)!}
{(l!)^3 m! (m-l)! (m-2l)!}
\]
Regularizing gives:
\begin{multline*}
\hG_X(t) = 
1+26 t^2+216 t^3+3582 t^4+54480 t^5+874700 t^6+15000720 t^7\\+256965310 t^8+4576672800 t^9 + \cdots
\end{multline*}

\subsection*{Minkowski period sequence:} \href{http://www.grdb.co.uk/search/period3?id=144&printlevel=2}{144}


\addtocounter{CustomSectionCounter}{1}

\section{The Fano Manifold $\MM{2}{9}$}
\label{anchor:2--9}

\subsection*{Mori--Mukai name:} 2--9

\subsection*{Mori--Mukai construction:} The blow-up of $\PP^3$ with
centre a curve $\Gamma$ of degree 7 and genus 5 that is an
intersection of cubics.

\subsection*{Our construction:} A codimension-2 complete intersection
$X$ of type $(L+M)\cap (2L+M)$ in the toric variety $F=\PP^3\times
\PP^2$.

\subsection*{The two constructions coincide:} The curve $\Gamma$ is
cut out by the equations:
\[ 
\rk \left( \begin{array}{ccc} l_0 & l_1 & l_2 \\ q_0 & q_1 &
    q_2 \end{array} \right) < 2
\]
where the $ l_i$ are linear forms and the $ q_j$ are quadratic
forms. Lemma~\ref{lem:blowups} implies that $X$ is the complete
intersection given by the two equations
\[ 
\begin{cases} l_0 y_0 + l_1 y_1 + l_2 y_2 = 0 \\ q_0 y_0 + q_1
  y_1 + q_2 y_2 = 0 \end{cases}
\] 
in $\PP^3 \times \PP^2$, where the first factor has co-ordinates $
x_0, x_1, x_2, x_3$ and the second factor has co-ordinates $ y_0, y_1,
y_2$.  In other words, $ X$ is a complete intersection of type $ (L+M)
\cap (2L+M)$ in $\PP^3 \times \PP^2$.

\subsection*{The quantum period:}
The toric variety $F$ has weight data:
\[ 
\begin{array}{rrrrrrrl} 
1 & 1 & 1 & 1 & 0 & 0 & 0 & \hspace{1.5ex} L\\ 
0 & 0 & 0 & 0 & 1 & 1 & 1 & \hspace{1.5ex} M \\
\end{array}
\]
and $\Amp F = \langle L,M\rangle$.  We have:
\begin{itemize}
\item $F$ is a Fano variety;
\item $X$ is the complete intersection of two ample divisors on $F$;
\item $-(K_F+\Lambda)\sim L+M$ is ample.
\end{itemize}
Corollary~\ref{cor:QL} yields:
\[
G_X(t) = e^{-3t} \sum_{l=0}^\infty \sum_{m=0}^\infty
t^{l+m}
\frac{(l+m)!(2l+m)!}
{(l!)^4 (m!)^3}
\]
and regularizing gives:
\begin{multline*}
  \hG_X(t) = 1+22 t^2+174 t^3+2514 t^4+34200 t^5+501070 t^6+7586880
  t^7\\
  +117858370 t^8+1870811040 t^9+ \cdots
\end{multline*}

\subsection*{Minkowski period sequence:} \href{http://www.grdb.co.uk/search/period3?id=139&printlevel=2}{139}


\addtocounter{CustomSectionCounter}{1}

\section{The Fano Manifold $\MM{2}{10}$}
\label{anchor:2--10}

\subsection*{Mori--Mukai name:} 2--10

\subsection*{Mori--Mukai construction:} The blow-up of $ B_4\subset
\PP^5$ with centre an elliptic curve that is an intersection of two
hyperplane sections.

\subsection*{Our construction:} A codimension-2 complete intersection
$X$ of type $(2M)\cap(2M)$ in the toric variety $F$ with weight data:
\[
\begin{array}{rrrrrrrl} 
  \multicolumn{1}{c}{s_0} & 
  \multicolumn{1}{c}{s_1} & 
  \multicolumn{1}{c}{x} & 
  \multicolumn{1}{c}{x_2} & 
  \multicolumn{1}{c}{x_3} & 
  \multicolumn{1}{c}{x_4} & 
  \multicolumn{1}{c}{x_5} & \\ 
  \cmidrule{1-7}
  1 & 1 & -1 & 0 & 0 & 0  & 0& \hspace{1.5ex} L\\ 
  0 & 0 & 1 & 1 & 1 & 1 & 1 & \hspace{1.5ex} M \\
\end{array}
\]
and $\Amp F = \langle L, M\rangle$.  We have:
\begin{itemize}
\item $-K_F=L+5M$ is ample, that is $F$ is a Fano variety;
\item $X$ is the complete intersection of two nef divisors on $F$;
\item $-(K_F+\Lambda)\sim L+M$ is ample.
\end{itemize}

\subsection*{The two constructions coincide:} The notation makes it
clear that $ s_0,s_1$ are sections of $ L$; $ xs_0, xs_1,x_2, x_3,
x_4, x_5$ are sections of $ M$; and $F$ is a scroll over $ \PP^1$ with
fibre $ \PP^4$. The morphism $ F \to \PP^4$ that sends
(contravariantly) the homogeneous co-ordinate functions $
[x_0,x_1,x_2, x_3, x_4,x_5]$ to $ [xs_0,xs_1,x_2,x_3,x_4,x_5]$ is the
blow-up along $(x_0=x_1=0) \subset \PP^4$.

\subsection*{The quantum period:}  Corollary~\ref{cor:QL} yields:
\[
G_X(t) = e^{-4t} \sum_{l = 0}^\infty \sum_{m = l}^\infty 
t^{l+m}
\frac{(2m)!(2m)!}
{(l!)^2 (m-l)! (m!)^4}
\]
and regularizing gives:
\begin{multline*}
  \hG_X(t) = 1+28 t^2+216 t^3+3516 t^4+49680 t^5+783640 t^6+12594960
  t^7\\
+208898620 t^8+3533634720 t^9+ \cdots
\end{multline*}

\subsection*{Minkowski period sequence:} \href{http://www.grdb.co.uk/search/period3?id=145&printlevel=2}{145}


\addtocounter{CustomSectionCounter}{1}

\section{The Fano Manifold $\MM{2}{11}$}
\label{anchor:2--11}

\subsection*{Mori--Mukai name:} 2--11.

\subsection*{Mori--Mukai construction:} The blow-up of $B_3\subset
\PP^4$ with centre a line on it.

\subsection*{Our construction:} A member $X$ of $|L+2M|$ in the toric
variety $F$ with weight data:
\[
\begin{array}{rrrrrrl} 
  \multicolumn{1}{c}{s_0} & 
  \multicolumn{1}{c}{s_1} & 
  \multicolumn{1}{c}{s_2} & 
  \multicolumn{1}{c}{x} & 
  \multicolumn{1}{c}{x_3} & 
  \multicolumn{1}{c}{x_4} \\ 
  \cmidrule{1-6} 
  1 & 1 & 1 & -1 & 0 & 0  & \hspace{1.5ex} L\\ 
  0 & 0 & 0 & 1 & 1 & 1 &  \hspace{1.5ex} M \\
\end{array}
\]
and $\Amp F = \langle L, M\rangle$.  We have:
\begin{itemize}
\item $-K_F=2L+3M$ is ample, that is $F$ is a Fano variety;
\item $X\sim L+2M$ is ample;
\item $-(K_F+X)\sim L+M$ is ample.
\end{itemize}

\subsection*{The two constructions coincide:} The notation makes it
clear that $ s_0,s_1,s_2$ are sections of $ L$; $ xs_0, xs_1,xs_2,
x_3, x_4$ are sections of $ M$; and $F$ is a scroll over $ \PP^2$ with
fibre $\PP^2$. The morphism $ F \to \PP^4$ that sends
(contravariantly) the homogeneous co-ordinate functions $ [x_0,x_1,
x_2, x_3, x_4]$ to $ [xs_0,xs_1,xs_2,x_3,x_4]$ is the blow-up along
the line $\ell = (x_0=x_1=x_2=0) \subset \PP^4$. We construct $X$ as
the proper transform of a general cubic $B_3\subset \PP^4$ containing
the line $\ell$. This $B_3$ has an equation of the form:
\[
x_0A+x_1B+x_2C =0
\]
where $A$, $B$, and $C$ are homogeneous quadratic polynomials in the
variables $x_0, x_1,\dots, x_4$. Thus $X$ is given in $F$ by the
equation:
\[
s_0A(s_0x,s_1x,s_2x,x_3,x_4)+s_1B(s_0x,s_1x,s_2x,x_3,x_4)+s_2C(s_0x,s_1x,s_2x,x_3,x_4)=0
\]

\subsection*{The quantum period:} Corollary~\ref{cor:QL} yields:
\[
G_X(t) = e^{-2t} \sum_{l = 0}^\infty \sum_{m = l}^\infty
t^{l+m}
\frac{(l+2m)!}
{(l!)^3 (m-l)! (m!)^2}
\]
and regularizing gives:
\begin{multline*}
  \hG_X(t) = 1+14 t^2+108 t^3+1074 t^4+13440 t^5+154760 t^6+1951320 t^7\\+24999730 t^8+325321920 t^9+ \cdots
\end{multline*}

\subsection*{Minkowski period sequence:} \href{http://www.grdb.co.uk/search/period3?id=120&printlevel=2}{120}


\addtocounter{CustomSectionCounter}{1}

\section{The Fano Manifold $\MM{2}{12}$}
\label{anchor:2--12}

\subsection*{Mori--Mukai name:} 2--12.

\subsection*{Mori--Mukai construction:} The blow-up of $\PP^3$ with
centre a curve $\Gamma$ of degree 6 and genus 3 that is an intersection of cubics.

\subsection*{Our construction:} A codimension-3 complete intersection
$X$ of type $(L+M) \cap (L+M) \cap (L+M)$ in the toric variety
$F=\PP^3\times \PP^3$.

\subsection*{The two constructions coincide:} The curve $\Gamma
\subset \PP^3_{x_0,x_1,x_3}$ is given by the condition:
\[ 
\rk \begin{pmatrix} l_{00} & l_{01} & l_{02} & l_{03} \\ l_{10} & l_{11} &
  l_{12} & l_{13} \\ l_{20} & l_{21} & l_{22} & l_{23}  \end{pmatrix}
< 3
\]
where the $ l_{ij}$ are linear forms in $x_0,\dots,x_3$.
Lemma~\ref{lem:blowups} implies that $X$ is a codimension-$3$ complete
intersection in $\PP^3_{x_0,x_1,x_2,x_3} \times
\PP^3_{y_0,y_1,y_2,y_3}$ given by the three equations:
\[
\begin{cases} l_{00} y_0 + l_{01} y_1 + l_{02} y_2 + l_{03} y_3 = 0 \\
  l_{10} y_0 + l_{11} y_1 + l_{12} y_2 + l_{13} y_3 = 0 \\l_{20} y_0 +
  l_{21} y_1 + l_{22} y_2 + l_{23} y_3 = 0  \end{cases}
\]
In other words, $ X$ is a complete intersection in $ \PP^3 \times
\PP^3$ of type $ (L+M) \cap (L+M) \cap (L+M)$.  An equivalent
description of this variety was given by Qureshi \cite{Qureshi}*{Proposition~6.4.1}.

\subsection*{The quantum period:}  The toric variety $F$ has weight
data:
\[ 
\begin{array}{rrrrrrrrl} 
1 & 1 & 1 & 1 & 0 & 0 & 0 & 0 & \hspace{1.5ex} L\\ 
0 & 0 & 0 & 0 & 1 & 1 & 1 & 1 &\hspace{1.5ex} M \\
\end{array}
\]
and $\Amp F = \langle L,M\rangle$. We have that:
\begin{itemize}
\item $F$ is a Fano variety;
\item $X$ is the complete intersection of three ample divisors on $F$;
\item $-(K_F+\Lambda)\sim L+M$ is ample.
\end{itemize}
Corollary~\ref{cor:QL} yields:
\[
G_X(t) = e^{-2t} \sum_{l=0}^\infty \sum_{m=0}^\infty
t^{l+m}
\frac{\big((l+m)!\big)^3}
{(l!)^4 (m!)^4}
\]
and regularizing gives:
\begin{multline*}
  \hG_X(t) = 1+14 t^2+72 t^3+882 t^4+8400 t^5+95180 t^6+1060080 t^7
  +12389650 t^8+146472480 t^9+ \cdots
\end{multline*}

\subsection*{Minkowski period sequence:} \href{http://www.grdb.co.uk/search/period3?id=118&printlevel=2}{118}


\addtocounter{CustomSectionCounter}{1}

\section{The Fano Manifold $\MM{2}{13}$}
\label{anchor:2--13}

\subsection*{Mori--Mukai name:} 2--13

\subsection*{Mori--Mukai construction:} The blow-up of a quadric
\mbox{3-fold} $Q\subset \PP^4$ with centre a curve $\Gamma$ of degree
6 and genus 2.

\subsection*{Our construction:} A codimension-3 complete intersection
$X$ of type $(L+M) \cap (L+M) \cap (2M)$ in the toric variety
$F=\PP^2\times \PP^4$.

\subsection*{The two constructions coincide:} Let $[s_0,s_1,y]$ be
homogeneous co-ordinates on $\PP(1,1,3)$.  We have that $\Gamma =
\PP(1,1,3) \cap Q$, where the embedding $\PP(1,1,3) \hookrightarrow
\PP^4$ sends (contravariantly) the homogeneous co-ordinate functions $
[x_0,\dots,x_4]$ to $ [s_0^3,s_0^2 s_1,s_0 s_1^2,s_1^3,y]$.  Thus
$\PP(1,1,3)\subset \PP^4$ is given by the condition:
\[ 
\rk \begin{pmatrix} x_0 & x_1 & x_2 \\ x_1 & x_2 & x_3 \end{pmatrix} <
2 \
\]
By Lemma~\ref{lem:blowups}, the blow-up $G$ of $\PP^4$ along
$\PP(1,1,3)$ is the complete intersection in $\PP^2_{y_0,\ldots,y_2}
\times \PP^4_{x_0,\ldots,x_4}$ cut out by the equations:
\[ 
\begin{cases} x_0 y_0 - x_1 y_1 + x_2 y_2 = 0 \\ x_1 y_0 - x_2 y_1 +
  x_3 y_2 = 0 \end{cases}
\]
Our Fano variety $X$ is the complete intersection of $G$ with a
quadric $q(x_0,x_1,x_2,x_3,x_4)$.  Thus $ X$ is a complete
intersection of type $ (L+M)\cap (L+M)\cap (2M)$ in $ \PP^2 \times
\PP^4$.

\subsection*{The quantum period:}
The toric variety $F$ has weight data:
\[ 
\begin{array}{rrrrrrrrl} 
1 & 1 & 1 & 0 & 0 & 0 & 0 & 0 & \hspace{1.5ex} L\\ 
0 & 0 & 0 & 1 & 1 & 1 & 1 & 1 &\hspace{1.5ex} M \\
\end{array}
\]
and $\Amp F = \langle L,M\rangle$.  We have that:
\begin{itemize}
\item $F$ is a Fano variety;
\item $X$ is the complete intersection of three nef divisors on $F$;
\item $-(K_F+\Lambda)\sim L+M$ is ample.
\end{itemize}
Corollary~\ref{cor:QL} yields:
\[
G_X(t) = e^{-3t} \sum_{l=0}^\infty \sum_{m=0}^\infty 
t^{l+m}
\frac{(2m)!\big((l+m)!\big)^2}
{(l!)^3 (m!)^5}
\]
and regularizing gives:
\begin{multline*}
  \hG_X(t) = 1+14 t^2+84 t^3+930 t^4+9720 t^5+108680 t^6+1259160 t^7+14951650 t^8+181377840 t^9+ \cdots
\end{multline*}

\subsection*{Minkowski period sequence:} \href{http://www.grdb.co.uk/search/period3?id=119&printlevel=2}{119}


\addtocounter{CustomSectionCounter}{1}

\section{The Fano Manifold $\MM{2}{14}$}
\label{anchor:2--14}

\subsection*{Mori--Mukai name:} 2--14

\subsection*{Mori--Mukai construction:} The blow-up of $B_5\subset
\PP^6$ with centre an elliptic curve that is an intersection of two
hyperplane sections.

\subsection*{Our construction:} A divisor\footnote{This is one of six
  cases of families of rank 2 Fano 3-folds (2--14, 2--17, 2--20,
  2--21, 2--22, 2--26) where the generic member is not a complete
  intersection in a toric variety. Of these, four (2--14, 2--20,
  2--22, 2--26) are blow-ups of $B_5$ along a curve: a complete
  intersection, a twisted cubic, a conic, and a line. Fano 3-folds in
  families 2--17 and 2--21 are blow-ups of a quadric 3-fold.} $X$ of
bidegree $(1,1)$ on $B_5 \times \PP^1$.

\subsection*{The two constructions coincide:} Let $[x_0,\ldots,x_6]$
be homogeneous co-ordinates on $\PP^6$, and let $F\to \PP^6$ be the
blow-up in the complete intersection $(x_0=x_1=0)$.  Our Fano variety
$X$ is the proper transform of $B_5\subset \PP^6$ under the blow-up.
Applying Lemma~\ref{lem:blowups} with $V = \cO_{\PP^6} \oplus
\cO_{\PP^6}$, $W=\cO_{\PP^6}(1)$, and $f\colon V \to W$ the map given by
$(x_0,x_1)$ shows that $F$ is the subvariety of
$\PP^6_{x_0,\dots,x_6}\times \PP^1_{y_0,y_1}$ given by the equation
$x_0y_0+x_1y_1=0$.  

\subsection*{The quantum period:} 
Combining Example~\ref{ex:P1}, the calculation in
Section~\ref{sec:B5}, and Corollary~\ref{cor:products}, we have:
\[
G_{B_5 \times \PP^1}(t) = \sum_{l=0}^\infty \sum_{m=0}^\infty \sum_{n=0}^\infty
(-1)^{l+m}
t^{2l+2m+2n}
\frac
{
  \big((l+m)!\big)^3
}
{
  (l!)^5 (m!)^5 (n!)^2 
}
\big(1-5 (m-l)H_m \big)
\]
where $H_m$ is the $m$th harmonic number.  Applying
Remark~\ref{rem:twisting_by_multiples_of_K} yields:
\[
G_X(t) = e^{-4t} \sum_{l=0}^\infty \sum_{m=0}^\infty \sum_{n=0}^\infty
(-1)^{l+m}
t^{l+m+n}
\frac
{
  (l+m+n)! \big((l+m)!\big)^3
}
{
  (l!)^5 (m!)^5 (n!)^2 
}
\big(1-5 (m-l)H_m \big)
\]
and regularizing gives:
\begin{multline*}
  \hG_X(t) = 1 + 16 t^2 + 90 t^3 + 1104 t^4 + 11460 t^5 + 133990 t^6 + 1588860 t^7 + 19463920 t^8 + 242996040 t^9+ \cdots
\end{multline*}

\subsection*{Minkowski period sequence:} \href{http://www.grdb.co.uk/search/period3?id=122&printlevel=2}{122}


\addtocounter{CustomSectionCounter}{1}

\section{The Fano Manifold $\MM{2}{15}$}
\label{anchor:2--15}

\subsection*{Mori--Mukai name:} 2--15

\subsection*{Mori--Mukai construction:} The blow-up of $ \PP^3$ with
centre the intersection of a quadric $A$ and a cubic $B$.

\subsection*{Our construction:} A member $X$ of $|2L+M|$ in the toric
variety $F$ with weight data:
\[ 
\begin{array}{rrrrrrl} 
  \multicolumn{1}{c}{s_0} & 
  \multicolumn{1}{c}{s_1} & 
  \multicolumn{1}{c}{s_2} & 
  \multicolumn{1}{c}{s_3} & 
  \multicolumn{1}{c}{x} & 
  \multicolumn{1}{c}{x_4} & \\ 
  \cmidrule{1-6}
  1 & 1 & 1 & 1 & -1 & 0  & \hspace{1.5ex} L\\ 
  0 & 0 & 0 & 0 & 1 & 1 & \hspace{1.5ex} M \\
\end{array}
\]
and $\Amp F = \langle L, M\rangle$. We have:
\begin{itemize}
\item $-K_F=3L+2M$ is ample, that is $F$ is a Fano variety;
\item $X\sim 2L+M$ is ample;
\item $-(K_F+X)\sim L+M$ is ample.
\end{itemize}

\subsection*{The two constructions coincide:} Apply
Lemma~\ref{lem:blowups} with $V = \cO_{\PP^3}(-1) \oplus \cO_{\PP^3}$,
$W=\cO_{\PP^3}(2)$, and $f\colon V \to W$ the map given by the matrix $
\begin{pmatrix}
  B & A
\end{pmatrix}
$.

\subsection*{The quantum period:}  Corollary~\ref{cor:QL} yields:
\[
G_X(t) = e^{-t} \sum_{l = 0}^\infty \sum_{m = l}^\infty
t^{l+m}
\frac{(2l+m)!}
{(l!)^4 (m-l)!m!}
\]
and regularizing gives:
\begin{multline*}
  \hG_X(t) = 1+12 t^2+36 t^3+564 t^4+3600 t^5+41700 t^6+360360 t^7+3839220 t^8+37749600 t^9+ \cdots
\end{multline*}

\subsection*{Minkowski period sequence:} \href{http://www.grdb.co.uk/search/period3?id=109&printlevel=2}{109}


\addtocounter{CustomSectionCounter}{1}

\section{The Fano Manifold $\MM{2}{16}$}
\label{anchor:2--16}

\subsection*{Mori--Mukai name:} 2--16

\subsection*{Mori--Mukai construction:} The blow-up of $B_4 \subset \PP^5$ with centre a conic on it.

\subsection*{Our construction:} A codimension-2 complete intersection
$X$ of type $(L+M)\cap(2M)$ in the toric variety $F$ with weight data:
\[
\begin{array}{rrrrrrrl} 
  \multicolumn{1}{c}{s_0} & 
  \multicolumn{1}{c}{s_1} & 
  \multicolumn{1}{c}{s_2} & 
  \multicolumn{1}{c}{x} & 
  \multicolumn{1}{c}{x_3} & 
  \multicolumn{1}{c}{x_4} & 
  \multicolumn{1}{c}{x_5} & \\ 
  \cmidrule{1-7}
  1 & 1 & 1 & -1 & 0 & 0  & 0& \hspace{1.5ex} L\\ 
  0 & 0 & 0 & 1 & 1 & 1 & 1 & \hspace{1.5ex} M \\
\end{array}
\]
and $\Amp F = \langle L, M\rangle$.  We have:
\begin{itemize}
\item $-K_F=2L+4M$ is ample, that is $F$ is a Fano variety;
\item $X$ is the complete intersection of two nef divisors on $F$;
\item $-(K_F+\Lambda)\sim L+M$ is ample.
\end{itemize}

\subsection*{The two constructions coincide:} The morphism $F \to
\PP^5$ which sends (contravariantly) the homogeneous co-ordinate
functions $[x_0,x_1,\ldots,x_5]$ to $[s_0 x, s_1 x, s_2 x, x_3, x_4,
x_5]$ blows up the plane $\Pi = (x_0 = x_1 = x_2 = 0)$.  We realise
$X$ as the complete intersection of the proper transform of a quadric
containing $\Pi$ and a generic quadric.

\subsection*{The quantum period:}  Corollary~\ref{cor:QL} yields:
\[
G_X(t) = e^{-2t} \sum_{l = 0}^\infty \sum_{m = l}^\infty 
t^{l+m}
\frac{(l+m)!(2m)!}
{(l!)^3 (m-l)! (m!)^3}
\]
and regularizing gives:
\begin{multline*}
  \hG_X(t) = 1+10 t^2+60 t^3+510 t^4+4920 t^5+47080 t^6+473760 t^7+4908190 t^8+51641520 t^9+ \cdots
\end{multline*}

\subsection*{Minkowski period sequence:} \href{http://www.grdb.co.uk/search/period3?id=104&printlevel=2}{104}


\addtocounter{CustomSectionCounter}{1}

\section{The Fano Manifold $\MM{2}{17}$}
\label{sec:2-17}
\label{anchor:2--17}

\subsection*{Mori--Mukai name:} 2--17

\subsection*{Mori--Mukai construction:} The blow-up of a quadric \mbox{3-fold}
$Q\subset \PP^4$ with centre an elliptic curve $\Gamma$ of degree 5 on it. 

\subsection*{Our construction:} The vanishing locus $X$ of a general
section of the vector bundle:
\[
\bigl(S^\star \boxtimes
\cO_{\PP^3}(1)\bigr) \oplus \bigl(\det S^\star \boxtimes \cO_{\PP^3}(1)\bigr) \oplus 
\bigl(\det S^\star \boxtimes \cO_{\PP^3}\bigr)
\]
on the key variety $F=\Gr(2,4)\times \PP^3$, where $S$ is the
universal bundle of subspaces on $\Gr(2,4)$.

\subsection*{The two constructions coincide:} First consider $\Gr(2,4)$
with tautological rank-$2$ sub-bundle $S \subset \CC^4$: it is
well-known that the vanishing locus $Z=Z(s)$ of a general section $s\in
\Gamma \bigl( \Gr(2, 4); E\bigr)$ where:
\[
E = S^\star \otimes \det S^\star
\]
is a del Pezzo surface of degree $5$. Indeed this can be shown as
follows: the adjunction formula immediately implies that
$-K_Z=-(K_X\otimes \det E)|_{Z}=\det S^\star$ is ample, that is $Z$ is
a del~Pezzo surface, and a small exercise in Schubert calculus shows
that $K_Z^2=5$.

Next we blow-up $Z\subset \Gr(2,4)$. Consider the $\PP^1$-bundle
$p\colon \PP(E^\star)\to \Gr(2,4)$ of lines in $E^\star$: under
$p^\star E\to \cO(1)$ we can identify $s\in \Gamma
\bigl(\Gr(2,4);E\bigr)=\Gamma \bigl(\PP(E^\star), \cO(1)\bigr)$ with a
section $\tilde{s}$ of $\cO(1)$ on $\PP(E^\star)$ and, by
Lemma~\ref{lem:blowups}:
\[
p\colon Y=Z(\tilde{s}) \subset \PP(E^\star)\to \Gr(2,4)
\quad
\text{blows up}
\quad
Z=Z(s)\subset \Gr(2,4)
\]
Next, identify:
\begin{itemize}
\item $\PP(E^\star)=\PP(S \otimes \det S)$ with $\PP(S)$. Write
  $V=\CC^4$ with basis $e_0,\dots, e_3$ and note that the
  tautological sequence 
\[
0\to S \to V \to Q \to 0
\]
on $\Gr(2,4)$ identifies $V^\star$ with $\Gamma \bigl(\Gr(2,4);
S^\star \bigr)$. In this notation, we can now also identify:
\[
\PP(S)=Z(\sigma) \subset \Gr(2,V)\times
\PP(V^\star)
\]
where $\sigma=e_0x_0+\cdots e_3x_3\in \Gamma \bigl(\Gr,V)\times
\PP(V^\star); S^\star \boxtimes \cO(1)\bigr)$ is a general section.
\item The line bundle $\cO(1)$ on $\PP(E^\star)$ with the line
  bundle $\det S^\star (1)$ on $\PP(S)$ and $\tilde{s}$ with a
  section that, abusing notation, we still denote by $\tilde{s}$:
\[
\tilde s\in \Gamma \bigl( \PP(S); \det S^\star (1)\bigr)
\]
\end{itemize}
Combining all of the above we identify the blow-up $Y$ of a del Pezzo
surface of degree $5$, $Z\subset \Gr(2,4)$, with the vanishing locus
of a general section $(\sigma, \tilde{s})$ of the bundle
\[
\bigl(S^\star \boxtimes \cO_{\PP^3}(1)\bigr)\oplus \bigl(\det S^\star \boxtimes \cO_{\PP^3}(1)\bigr)
\]
on $\Gr (2,4)\times \PP^3$.  It follows easily from this that our
construction and the Mori--Mukai construction coincide.

\subsection*{Abelianization:} Consider $\Gr(2,4)$ as the geometric
quotient $\CC^{8}\GIT GL_2(\CC)$ where we regard $\CC^8$ as the space
$M(2,4)$ of $2\times 4$ complex matrices and $GL_2(\CC)$ acts by
multiplication on the left. The universal bundle $S$ of subspaces on
$\Gr(2,4)$ is the bundle on $\CC^{8}\GIT \GL_2(\CC)$ determined by
$\Vstd^\star$, where $\Vstd$ is the standard representation of
$\GL_2(\CC)$. Consider the situation as in \S3.1 of \cite{CFKS} with:
\begin{itemize}
\item the space that is denoted by $X$ in \cite{CFKS} set equal to $A =
  \CC^{12}$, regarded as the space of pairs:
  \[
  \{(M,w) : \text{$M$ is a $2 \times 4$ complex matrix, $w \in \CC^4$
    is a vector}\}
  \]
\item $G = \GL_2(\CC) \times \Cstar$, acting on $A$ as:
  \[
  (g,\lambda) \colon (M, w) \mapsto (g M, \lambda w)
  \]
\item $T = (\Cstar)^3$, the diagonal subtorus in $G$;
\item the group that is denoted by $S$ in \cite{CFKS} set equal to the trivial group;
\item $\cV$ equal to the representation of $G$ given by 
\[
(\Vstd \boxtimes \Vstd)\oplus (\det \Vstd \boxtimes \Vstd) \oplus \det
\Vstd \boxtimes \Vtriv
\]
where $\Vtriv$ is the trivial $1$-dimensional representation of
$\Cstar$.
\end{itemize}
It is clear that $A \GIT G = F$, whereas $A \GIT T=\PP^3\times \PP^3
\times \PP^3$. The non-trivial element in the Weyl group $W=\ZZ/2\ZZ$
permutes the first and second factors in the product $\PP^3\times
\PP^3 \times \PP^3$. The representation $\cV$ induces the vector
bundle $\cV_G = E$ over $F$, whereas the representation $\cV$ induces
the vector bundle:
\[\cV_T = \cO(1,0,1) \oplus \cO(0,1,1)\oplus \cO(1,1,1)\oplus \cO(1,1,0)\]
over $A \GIT T$.

\subsection*{The Abelian/non-Abelian correspondence:}

Let $p_i \in H^2(A\GIT T;\QQ)$, $1 \leq i \leq 3$, denote the first
Chern class of $\pi_i^\star \cO_{\PP^3}(1)$ where $\pi_i \colon A \GIT
T \to \PP^3$ is projection to the $i$th factor of the product $A \GIT
T = \PP^3 \times \PP^3 \times \PP^3$.  Set $\Omega = (p_2-p_1)$.  We
fix a lift of $H^\bullet(A \GIT G;\QQ)$ to $H^\bullet(A \GIT T,\QQ)^W$
in the sense of \cite[\S3]{CFKS}.  As in the proof of
Theorem~\ref{thm:rank_1_A_nA} there are many possible choices for such
a lift, and the precise choice made will be unimportant in what
follows.  The lift allows us to regard $H^\bullet(A \GIT G;\QQ)$ as a
subspace of $H^\bullet(A \GIT T,\QQ)^W$, which maps isomorphically to
the Weyl-anti-invariant part $H^\bullet(A \GIT T,\QQ)^a$ of
$H^\bullet(A \GIT T,\QQ)$ via:
\[
\xymatrix{
  H^\bullet(A \GIT T,\QQ)^W \ar[rr]^{\cup \Omega} &&
  H^\bullet(A \GIT T,\QQ)^a}
\]
We compute the quantum period of $X$ by computing the $J$-function of
$F = A \GIT G$ twisted \cite{Coates--Givental} by the Euler class and
the bundle $\cV_G$, using the Abelian/non-Abelian correspondence
\cite{CFKS}.  An alternative method of calculation has been given by
Andrew Strangeway \cite{Strangeway}.

We first compute the $J$-function of $A \GIT T$ twisted by the Euler
class and the bundle $\cV_T$.  As in the proof of
Theorem~\ref{thm:rank_1_A_nA}, consider the bundles $\cV_T$ and
$\cV_G$ equipped with the canonical $\Cstar$-action that rotates
fibers and acts trivially on the base, and consider the twisted
$J$-function $J_{\be,\cV_T}$ of $A \GIT T$. $J_{\be,\cV_T}$ was
defined in equation \eqref{eq:twisted_J} above, and is the restriction
to the locus $\tau \in H^0(A \GIT T) \oplus H^2(A \GIT T)$ of what was
denoted by $J^{S \times \Cstar}_{\cV_T}(\tau)$ in \cite{CFKS}.  The
toric variety $A \GIT T = \PP^3 \times \PP^3 \times \PP^3$ is Fano,
and Theorem~\ref{thm:toric_mirror} gives:
\begin{equation}
  \label{eq:2-17_J_A_mod_T}
  J_{A \GIT T}(\tau) = 
  e^{\tau/z}
  \sum_{l_1 = 0}^\infty
  \sum_{l_2 = 0}^\infty
  \sum_{l_3 = 0}^\infty
  {
    Q_1^{l_1}  Q_2^{l_2}  Q_3^{l_3} 
    e^{l_1 \tau_1} e^{l_2 \tau_2} e^{l_3 \tau_3}
    \over
    \prod_{j=1}^{j=3} 
    \prod_{k=1}^{k=l_j} (p_j + k z)^4
  }
\end{equation}
where $\tau = \tau_1 p_1 + \tau_2 p_2 + \tau_3 p_3$ and we have
identified the group ring $\QQ[H_2(A \GIT T;\ZZ)]$ with
$\QQ[Q_1,Q_2,Q_3]$ via the $\QQ$-linear map that sends $Q^\beta$ to
$Q_1^{\langle \beta, p_1 \rangle} Q_2^{\langle \beta, p_2\rangle}
Q_3^{\langle \beta, p_3\rangle}$.  Each line bundle summand in $\cV_T$
is nef and $c_1(A \GIT T) - c_1(\cV_T)$ is ample, so
Theorem~\ref{thm:toric_ci_ql} gives:
\begin{multline}
  \label{eq:2-17_J_e_T}
  J_{\be,\cV_T}(\tau) = e^{-(Q_1 e^{\tau_1} + Q_2 e^{\tau_2} + Q_3
    e^{\tau_3})/z}
  e^{\tau/z} \\
  \sum_{l_1 = 0}^\infty \sum_{l_2 = 0}^\infty \sum_{l_3 = 0}^\infty
  { 
    Q_1^{l_1} Q_2^{l_2} Q_3^{l_3} 
    e^{l_1 \tau_1} e^{l_2 \tau_2} e^{l_3 \tau_3} 
    \Big(\prod_{1 \leq i < j \leq 3} \prod_{k=1}^{l_i+l_j} 
    (\lambda + p_i + p_j + k z)\Big) 
    \over
    \prod_{j=1}^{j=3} \prod_{k=1}^{k=l_j} (p_j + k z)^4 
  } \times
  \\
  \prod_{k=1}^{l_1+l_2+l_3} (\lambda + p_1+p_2+p_3+k z) 
\end{multline}

Consider now $F = A \GIT G = \Gr(2,4) \times \PP^3$ and a point $t \in
H^\bullet(F)$.  Let $\epsilon_1 \in H^2(F;\QQ)$ be the pullback to $F$
(under projection to the first factor) of the ample generator of
$H^2(\Gr(2,4))$, and let $\epsilon_2 \in H^2(F;\QQ)$ be the pullback
to $F$ (under projection to the second factor) of the ample generator
of $H^2(\PP^3)$.  Identify the group ring $\QQ[H_2(F;\ZZ)]$ with
$\QQ[q_1,q_2]$ via the $\QQ$-linear map which sends $Q^\beta$ to
$q_1^{\langle \beta,\epsilon_1 \rangle} q_2^{\langle \beta,\epsilon_2
  \rangle}$.  In \cite{CFKS}*{\S6.1} the authors consider the lift
$\tilde{J}^{S \times \Cstar}_{\cV_G}(t)$ of their twisted
$J$-function $J^{S \times \Cstar}_{\cV_G}(t)$ determined by a
choice of lift $H^\bullet(A \GIT G;\QQ) \to H^\bullet(A \GIT
T,\QQ)^W$.  We restrict to the locus $t \in H^0(A \GIT G;\QQ)
\oplus H^2(A \GIT G;\QQ)$, considering the lift:
\begin{align*}
  \tilde{J}_{\be,\cV_G}(t) := \tilde{J}^{S \times
    \Cstar}_{\cV_G}(t) && t \in H^0(A \GIT G;\QQ) \oplus H^2(A
  \GIT G;\QQ)
\end{align*}
of our twisted $J$-function $J_{\be,\cV_G}$ determined by our choice
of lift $H^\bullet(A \GIT G;\QQ) \to H^\bullet(A \GIT T,\QQ)^W$.
Theorems~4.1.1 and~6.1.2 in \cite{CFKS} imply that:
\[
\tilde{J}_{\be,\cV_G}\big(\theta(t)\big) \cup \Omega = \Big[ 
\big(z \textstyle {\partial \over \partial \tau_2} 
- 
z {\partial \over \partial \tau_1} 
\big) J_{\be,\cV_T}(\tau) \Big]_{\tau=t, Q_1  = Q_2 = -q_1, Q_3 = q_2}
\]
for some\footnote{As in Theorem~\ref{thm:rank_1_A_nA}, the map
  $\theta$ is grading preserving and satisfies $\theta \equiv \id$
  modulo $q_1, q_2$.  We will need only that $\theta(0) \in H^0(A
  \GIT G;\QQ) \otimes \Lambda_{A \GIT G}$, however, and we will see
  this explicitly below. \label{footnote:2-17}}  function $\theta \colon H^2(A \GIT G;\QQ)
\to H^\bullet(A \GIT G; \Lambda_{A \GIT G})$ such that $\theta(0) \in
H^0(A \GIT G;\QQ) \otimes \Lambda_{A \GIT G}$.  Setting $t = 0$
gives:
\begin{multline}
  \label{eq:2-17_J_e_G}
  \tilde{J}_{\be,\cV_G}\big(\theta(0)\big) \cup \Omega = \\
  e^{-(2q_1 + q_2)/z}
  \sum_{l_1 = 0}^\infty \sum_{l_2 = 0}^\infty \sum_{l_3 = 0}^\infty
  { 
    (-1)^{l_1+l_2} q_1^{l_1+ l_2} q_2^{l_3} 
    \Big(\prod_{1 \leq i < j \leq 3} \prod_{k=1}^{l_i+l_j} 
    (\lambda + p_i + p_j + k z)\Big) 
    \over
    \prod_{j=1}^{j=3} \prod_{k=1}^{k=l_j} (p_j + k z)^4 
  } \times
  \\
  \textstyle
  \Big(\prod_{k=1}^{l_1+l_2+l_3} (\lambda + p_1+p_2+p_3+k z) \Big)
  \big(p_2-p_1 + (l_2-l_1)z\big)
\end{multline}
The left-hand side here takes the form:
\begin{align}
  & (p_2-p_1) \Big( 1 + \theta(0) z^{-1} + O(z^{-2})\Big) \notag \\
  \intertext{whereas the right-hand side is:}
  \label{eq:2-17_asymptotics}
  & (p_2-p_1) \Big( 1 -q_1 z^{-1} + O(z^{-2})\Big) 
\end{align}
We conclude that $\theta(0) = {-q_1}$ and hence, via the String
Equation, that:
\begin{multline}
  \label{eq:2-17_almost_there}
  \tilde{J}_{\be,\cV_G}(0) \cup \Omega = \\
  e^{-(q_1 + q_2)/z}
  \sum_{l_1 = 0}^\infty \sum_{l_2 = 0}^\infty \sum_{l_3 = 0}^\infty
  { 
    (-1)^{l_1+l_2} q_1^{l_1+ l_2} q_2^{l_3} 
    \Big(\prod_{1 \leq i < j \leq 3} \prod_{k=1}^{l_i+l_j} 
    (\lambda + p_i + p_j + k z)\Big) 
    \over
    \prod_{j=1}^{j=3} \prod_{k=1}^{k=l_j} (p_j + k z)^4 
  } \times
  \\
  \textstyle
  \Big(\prod_{k=1}^{l_1+l_2+l_3} (\lambda + p_1+p_2+p_3+k z) \Big)
  \big(p_2-p_1 + (l_2-l_1)z\big)
\end{multline}

We saw in Example~\ref{ex:ql_with_mirror_map} how to extract the
quantum period $G_X$ from the twisted $J$-function $J_{\be,\cV_G}(0)$:
we take the non-equivariant limit $\lambda \to 0$, extract the
component along the unit class $1 \in H^\bullet(A \GIT G;\QQ)$, set
$z=1$, and set $Q^\beta = t^{\langle \beta, {-K_X} \rangle}$.  Thus we
consider the right-hand side of \eqref{eq:2-17_almost_there}, take the
non-equivariant limit, extract the coefficient of $\Omega$, set $z=1$,
set $q_1=t$, and set $q_2=t$.  This yields:
\begin{multline*}
  G_X(t) = e^{-2t} 
  \sum_{l_1=0}^\infty \sum_{l_2=0}^\infty \sum_{l_3=0}^\infty
  (-1)^{l_1+l_2} t^{l_1+l_2+l_3}
  \frac
  {
    (l_1+l_2)! (l_1+l_3)!(l_2+l_3)! (l_1+l_2+l_3)!
  }
  {
    (l_1!)^4 (l_2!)^4 (l_3!)^4
  }
  \times \\
  \Big(1 + (l_2-l_1)(H_{l_2+l_3} - 4H_{l_2})\Big)
\end{multline*}
where $H_k$ is the $k$th harmonic number.  Regularizing gives:
\begin{multline*}
  \hG_X(t) = 1 + 10 t^2 + 42 t^3 + 414 t^4 + 3300 t^5 + 29890 t^6 + 275940 t^7 + 2608270 t^8 + 25305000 t^9+\cdots
\end{multline*}

\subsection*{Minkowski period sequence:} \href{http://www.grdb.co.uk/search/period3?id=101&printlevel=2}{101}


\addtocounter{CustomSectionCounter}{1}

\section{The Fano Manifold $\MM{2}{18}$}
\label{anchor:2--18}

\subsection*{Mori--Mukai name:} 2--18
\label{sec:2-18}

\subsection*{Mori--Mukai construction:} A double cover of $ \PP^1
\times \PP^2$ with branch locus a divisor of bidegree $(2,2)$.

\subsection*{Our construction:} A member $X$ of $|2L+2M|$ in the toric
variety $F$ with weight data:
\[ 
\begin{array}{rrrrrrl} 
  \multicolumn{1}{c}{x_0} & 
  \multicolumn{1}{c}{x_1} & 
  \multicolumn{1}{c}{x_2} & 
  \multicolumn{1}{c}{y_0} & 
  \multicolumn{1}{c}{y_1} & 
  \multicolumn{1}{c}{w} & \\ 
  \cmidrule{1-6}
  1 & 1 & 1 & 0 & 0 & 1  & \hspace{1.5ex} L\\ 
  0 & 0 & 0 & 1 & 1 & 1 & \hspace{1.5ex} M \\
\end{array}
\]
and $\Amp F = \langle L, L+M\rangle$.  We have:
\begin{itemize}
\item $-K_F=4L+3M$ is ample, that is $F$ is a Fano variety;
\item $X\sim 2L+2M$ is nef;
\item $-(K_F+X)\sim 2L+M$ is ample.
\end{itemize}

\subsection*{The two constructions coincide:} The defining equation of
$ X$ is $ w^2=f_{2,2}(x_0,x_1,x_2;y_0,y_1)$, and so the morphism $X
\to \PP^2 \times \PP^1$ which sends the point
$[x_0:x_1:x_2:y_0:y_1:w]$ of $X$ to the point $[x_0:x_1:x_2:y_0:y_1]$
of $\PP^2 \times \PP^1$ exhibits $X$ as a double cover of $\PP^2
\times \PP^1$ branched over a divisor of bidegree $(2,2)$.

\subsection*{The quantum period:} Corollary~\ref{cor:QL} yields:
\[
G_X(t) = e^{-2t} \sum_{l=0}^\infty \sum_{m=0}^\infty 
t^{2l+m}
\frac{(2l+2m)!}
{(l!)^3 (m!)^2(l+m)!}
\]
and regularizing gives:
\[
\hG_X(t) = 1+6 t^2+48 t^3+282 t^4+2400 t^5+22020 t^6+184800 t^7+1684410 t^8+15798720 t^9+ \cdots
\]

\subsection*{Minkowski period sequence:} \href{http://www.grdb.co.uk/search/period3?id=74&printlevel=2}{74}


\addtocounter{CustomSectionCounter}{1}

\section{The Fano Manifold $\MM{2}{19}$}
\label{anchor:2--19}

\subsection*{Mori--Mukai name:} 2--19

\subsection*{Mori--Mukai construction:} The blow-up of $B_4\subset
\PP^5$ with centre a line on it.

\subsection*{Our construction:} A codimension-2 complete intersection
$X$ of type $(L+M) \cap (L+M)$ in the toric variety $F$ with weight
data:
\[ 
\begin{array}{rrrrrrrl} 
  \multicolumn{1}{c}{s_0} & 
  \multicolumn{1}{c}{s_1} & 
  \multicolumn{1}{c}{s_2} & 
  \multicolumn{1}{c}{s_3} & 
  \multicolumn{1}{c}{x} & 
  \multicolumn{1}{c}{x_4} & 
  \multicolumn{1}{c}{x_5} & \\ 
  \cmidrule{1-7}
  1 & 1 & 1 & 1 & -1 & 0 & 0 & \hspace{1.5ex} L\\ 
  0 & 0 & 0 & 0 & 1 & 1 & 1 & \hspace{1.5ex} M \\
\end{array}
\]
and  $\Amp F = \langle L, M\rangle$.  We have:
\begin{itemize}
\item $-K_F=3L+3M$ is ample, that is $F$ is a Fano variety;
\item $X$ is the complete intersection of two ample divisors on $F$;
\item $-(K_F+\Lambda)\sim L+M$ is ample.
\end{itemize}

\subsection*{The two constructions coincide:}
The morphism $F\to \PP^5$ that sends (contravariantly) the homogeneous
co-ordinate functions $[x_0,\dots,x_5]$ to $[xs_0, \dots, xs_3, x_4,
x_5]$ blows up the line $(x_0=\dots =x_3=0)$ in $\PP^5$. Now take the
proper transform of a $B_4$ containing this line.

\subsection*{The quantum period:}
Corollary~\ref{cor:QL} yields:
\[
G_X(t) = e^{-t} \sum_{l = 0}^\infty \sum_{m = l}^\infty 
t^{l+m}
\frac{(l+m)!(l+m)!}
{(l!)^4 (m-l)! (m!)^2}
\]
and regularizing gives:
\[
\hG_X(t) = 1+8 t^2+30 t^3+240 t^4+1920 t^5+13490 t^6+121800 t^7+953680 t^8+8465520 t^9+ \cdots
\]

\subsection*{Minkowski period sequence:} \href{http://www.grdb.co.uk/search/period3?id=86&printlevel=2}{86}


\addtocounter{CustomSectionCounter}{1}
\section{The Fano Manifold $\MM{2}{20}$}
\label{sec:2-20}
\label{anchor:2--20}

\subsection*{Mori--Mukai name:} 2--20

\subsection*{Mori--Mukai construction:} The blow-up of $B_5\subset
\PP^6$ with centre a twisted cubic on it.

\subsection*{Our construction:} The vanishing locus $X$ of a general
section of the vector bundle:
\[
E=\bigl(S^\star \boxtimes
\cO_{\PP^2}(1)\bigr) \oplus  \bigl(\det S^\star \boxtimes \cO_{\PP^2}\bigr)^{\oplus 3}
\]
on the key variety $F=\Gr(2,5)\times \PP^2$, where $S$ is the
universal bundle of subspaces on $\Gr(2,5)$.

\subsection*{The two constructions coincide:} Consider $\CC^5$ with
basis $e_0,\dots, e_4$.  Let $M(2,5)^\times$ denote the space of $2
\times 5$ complex matrices of full rank.  As is customary we represent
a point $W$ in $\Gr(2, \CC^5)$ by a matrix:
\[
\begin{pmatrix}
  a_0 & a_1 & a_2 & a_3 & a_4\\
  b_0 & b_1 & b_2 & b_3 & b_4
\end{pmatrix}
\in M(2,5)^\times
\]
up to the action of $GL_2(\CC)$ from the left. A basis element $e_i$
of $\CC^5$ gives a section of the rank-$2$ vector bundle $S^\star$
that evaluates as:
\[
e_i(W)=
\begin{pmatrix}
  a_i\\b_i
\end{pmatrix}
\]
Consider now the section:
\[
s=e_0x_0+e_1x_1+e_2x_2\in \Gamma \bigl(\Gr(2,5)\times \PP^2; S^\star
\boxtimes \cO(1)\bigr) 
\]
Let $Y \subset \Gr(2,5)\times \PP^2$ be the vanishing locus of $s$,
and let $p:Y \to \Gr(2,5)$ be the projection. $Y$ consists of pairs
$(W,x)\in M(2,5)^\times\times \PP^2$ such that $x=(x_0,x_1,x_2)$ is a
solution of the system:
\[
W\cdot 
\begin{pmatrix}
  x_0\\x_1\\x_2\\0\\0
\end{pmatrix}
=0
\] 
that is, $p\colon Y \to \Gr(2,5)$ blows up the locus $Z\subset
\Gr(2,5)$ consisting of those $W$ such that:
\[
\rk
\begin{pmatrix}
  a_0 & a_1 & a_2 \\
  b_0 & b_1 & b_2
\end{pmatrix}
<2
\]
In Pl\"ucker coordinates $x_{ij}=\det \begin{pmatrix} a_i & a_j\\b_i &
  b_j\end{pmatrix} $ this is the locus where $x_{01}=x_{02}=x_{12}=0$.
Thus $Z$ is the \emph{cubic scroll} defined by:
\begin{align*}
  x_{01}=x_{02}=x_{12}=0
  &&
  \text{and}
  &&
  \rk
  \begin{pmatrix}
    x_{03} & x_{13} & x_{14} \\
    x_{04} & x_{14} & x_{24} 
  \end{pmatrix}
  <2
\end{align*}
Intersecting with $3$ more hyperplane sections in the Pl\"ucker
embedding, we get the blow-up of $B_5$ along a twisted cubic.

\subsection*{Abelianization:} Consider $\Gr(2,5)$ as the geometric
quotient $\CC^{10}\GIT GL_2(\CC)$ where we regard $\CC^{10}$ as the
space $M(2,5)$ of $2\times 5$ complex matrices and $GL_2(\CC)$
acts by multiplication on the left. The universal bundle $S$ of
subspaces on $\Gr(2,5)$ is the bundle on $\CC^{10}\GIT \GL_2(\CC)$
determined by $\Vstd^\star$, where $\Vstd$ is the standard
representation of $\GL_5(\CC)$. Consider the situation as in \S3.1 of
\cite{CFKS} with:
\begin{itemize}
\item the space that is denoted by $X$ in \cite{CFKS} set equal to $A =
  \CC^{13}$, regarded as the space of pairs:
  \[
  \{(M,w) : \text{$M$ is a $2 \times 5$ complex matrix, $w \in \CC^3$
    is a vector}\}
  \]
\item $G = \GL_2(\CC) \times \Cstar$, acting on $A$ as:
  \[
  (g,\lambda) \colon (M, w) \mapsto (g M, \lambda w)
  \]
\item $T = (\Cstar)^3$, the diagonal subtorus in $G$;
\item the group that is denoted by $S$ in \cite{CFKS} set equal to the trivial group;
\item $\cV$ equal to the representation of $G$ given by:
  \[
  (\Vstd \boxtimes \Vstd)\oplus (\det \Vstd \boxtimes \Vtriv)^{\oplus 3} 
  \]
  where $\Vtriv$ is the trivial $1$-dimensional representation of
  $\Cstar$.
\end{itemize}
It is clear that $A \GIT G = F$, whereas $A \GIT T=\PP^4\times \PP^4
\times \PP^2$. The Weyl group $W=\ZZ/2\ZZ$ permutes the first and
second factors of the product $\PP^4\times \PP^4 \times \PP^2$. The
representation $\cV$ induces the vector bundle $\cV_G=E$ over $F$,
whereas the representation $\cV$ induces the vector bundle
\[
\cV_T = \cO(1,0,1) \oplus \cO(0,1,1)\oplus \cO(1,1,0)^{\oplus 3}
\]
over $A \GIT T$.

\subsection*{The Abelian/non-Abelian correspondence:} We proceed
exactly as in \S\ref{sec:2-17}, replacing:
\begin{itemize}
\item $\PP^3 \times \PP^3 \times \PP^3$ by $\PP^4 \times \PP^4 \times
  \PP^2$, throughout;
\item equation \eqref{eq:2-17_J_A_mod_T} by:
  \[
  J_{A \GIT T}(\tau) = 
  e^{\tau/z}
  \sum_{l_1 = 0}^\infty
  \sum_{l_2 = 0}^\infty
  \sum_{l_3 = 0}^\infty
  {
    Q_1^{l_1}  Q_2^{l_2}  Q_3^{l_3} 
    e^{l_1 \tau_1} e^{l_2 \tau_2} e^{l_3 \tau_3}
    \over
    \prod_{k=1}^{k=l_1} (p_1 + k z)^5
    \prod_{k=1}^{k=l_2} (p_2 + k z)^5
    \prod_{k=1}^{k=l_3} (p_3 + k z)^3
  }
  \]
\item equation \eqref{eq:2-17_J_e_T} by:
  \begin{multline*}
    J_{\be,\cV_T}(\tau) = e^{-(Q_1 e^{\tau_1} + Q_2 e^{\tau_2} + Q_3
      e^{\tau_3})/z}
    e^{\tau/z} \\
    \sum_{l_1 = 0}^\infty
    \sum_{l_2 = 0}^\infty
    \sum_{l_3 = 0}^\infty
    {
      Q_1^{l_1}  Q_2^{l_2}  Q_3^{l_3} 
      e^{l_1 \tau_1} e^{l_2 \tau_2} e^{l_3 \tau_3}
      \prod_{k=1}^{l_1+l_2} (\lambda + p_1+p_2+kz)^3
      \over
      \prod_{k=1}^{k=l_1} (p_1 + k z)^5
      \prod_{k=1}^{k=l_2} (p_2 + k z)^5
      \prod_{k=1}^{k=l_3} (p_3 + k z)^3
    } \times
    \\
    \textstyle
    \prod_{k=1}^{l_1+l_3} (\lambda + p_1+p_3+k z) 
    \prod_{k=1}^{l_2+l_3} (\lambda + p_2+p_3+k z) 
  \end{multline*}
\item $\Gr(2,4) \times \PP^3$ by $\Gr(2,5) \times \PP^2$, throughout;
\item equation \eqref{eq:2-17_J_e_G} by:
  \begin{multline*}
    \tilde{J}_{\be,\cV_G}\big(\theta(0)\big) \cup \Omega = \\
    e^{-(2q_1 + q_2)/z}
    \sum_{l_1 = 0}^\infty
    \sum_{l_2 = 0}^\infty
    \sum_{l_3 = 0}^\infty
    {
      (-1)^{l_1+l_2} q_1^{l_1+ l_2} q_2^{l_3} 
      \prod_{k=1}^{l_1+l_2} (\lambda + p_1+p_2+kz)^3
      \over
      \prod_{k=1}^{k=l_1} (p_1 + k z)^5
      \prod_{k=1}^{k=l_2} (p_2 + k z)^5
      \prod_{k=1}^{k=l_3} (p_3 + k z)^3
    } \times
    \\
    \shoveright{\textstyle
    \prod_{k=1}^{l_1+l_3} (\lambda + p_1+p_3+k z) 
    \prod_{k=1}^{l_2+l_3} (\lambda + p_2+p_3+k z) \times}
    \\
    \big(p_2-p_1 + (l_2-l_1)z\big)
  \end{multline*}
\item equation \eqref{eq:2-17_asymptotics} by:
  \[
  (p_2-p_1) \Big( 1 + O(z^{-2})\Big) 
  \]
\item the conclusion $\theta(0) = {-q_1}$ by $\theta(0) = 0$, and
  equation \eqref{eq:2-17_almost_there} by:
  \begin{multline*}
    \tilde{J}_{\be,\cV_G}(0) \cup \Omega = \\
    e^{-(2q_1+q_2)/z}
    \sum_{l_1 = 0}^\infty
    \sum_{l_2 = 0}^\infty
    \sum_{l_3 = 0}^\infty
    {
      (-1)^{l_1+l_2} q_1^{l_1+ l_2} q_2^{l_3} 
      \prod_{k=1}^{l_1+l_2} (\lambda + p_1+p_2+kz)^3
      \over
      \prod_{k=1}^{k=l_1} (p_1 + k z)^5
      \prod_{k=1}^{k=l_2} (p_2 + k z)^5
      \prod_{k=1}^{k=l_3} (p_3 + k z)^3
    } \times
    \\
    \shoveright{\textstyle
      \prod_{k=1}^{l_1+l_3} (\lambda + p_1+p_3+k z) 
      \prod_{k=1}^{l_2+l_3} (\lambda + p_2+p_3+k z) \times}
    \\
    \big(p_2-p_1 + (l_2-l_1)z\big)
  \end{multline*}
\end{itemize}
This yields:
\begin{multline*}
  G_X(t) = e^{-3t} 
  \sum_{l_1=0}^\infty \sum_{l_2=0}^\infty \sum_{l_3=0}^\infty
  (-1)^{l_1+l_2} t^{l_1+l_2+l_3}
  \frac
  {
    \big((l_1+l_2)!\big)^3 (l_1+l_3)!(l_2+l_3)! 
  }
  {
    (l_1!)^5 (l_2!)^5 (l_3!)^3
  }
  \times \\
  \Big(1 + (l_2-l_1)(H_{l_2+l_3} - 5H_{l_2})\Big)
\end{multline*}
where $H_k$ is the $k$th harmonic number.  Regularizing gives:
\begin{multline*}
  \hG_X(t) = 1 + 8 t^2 + 36 t^3 + 288 t^4 + 2220 t^5 + 18260 t^6 + 154560 t^7 + 1348480 t^8 + 11977560 t^9+\cdots
\end{multline*}

\subsection*{Minkowski period sequence:} \href{http://www.grdb.co.uk/search/period3?id=87&printlevel=2}{87}


\addtocounter{CustomSectionCounter}{1}
\section{The Fano Manifold $\MM{2}{21}$}
\label{sec:2-21}
\label{anchor:2--21}

\subsection*{Mori--Mukai name:} 2--21

\subsection*{Mori--Mukai construction:} The blow-up of a quadric
\mbox{3-fold} $Q\subset \PP^4$ with centre a rational normal curve of
degree 4 on it.

\subsection*{Our construction:} The vanishing locus $X$ of a general
section of the vector bundle:
\[
E=\bigl(S^\star \boxtimes \cO_{\PP^4}(1)\bigr)^{\oplus 2} \oplus
\bigl(\det S^\star \boxtimes \cO_{\PP^4}\bigr)
\]
on the key variety $F=\Gr(2,4)\times \PP^4$, where $S$ is the
universal bundle of subspaces on $\Gr(2,4)$.

\subsection*{The two constructions coincide:} 

Consider $\CC^4$ with basis $e_0,\dots, e_3$.  Let $M(2,4)^\times$
denote the space of $2 \times 4$ complex matrices of full rank, and
represent a point $W$ in $\Gr(2, \CC^4)$ by:
\[
W=
\begin{pmatrix}
  a_0 & a_1 & a_2 & a_3 \\
  b_0 & b_1 & b_2 & b_3
\end{pmatrix}
\in M(2,4)^\times
\]
up to the action of $\GL_2(\CC)$ from the left.  A basis element $e_i$, $0 \leq i \leq 3$, of
$\CC^4$ gives a section of the rank-$2$ vector bundle $S^\star$ that
evaluates as:
\[
e_i(W)=
\begin{pmatrix}
  a_i\\b_i
\end{pmatrix}
\]
Let $x_0,\dots, x_4$ be homogeneous coordinates on $\PP^4$, and
consider the two sections:
\begin{align*}
  s_1=e_0x_0+e_1x_1+e_2x_2+e_3x_3,
  &&
  s_2= e_0x_1+e_1x_2+e_2x_3+e_3x_4
\end{align*}
in $\Gamma \bigl(\Gr(2,4)\times \PP^4; S^\star \boxtimes
\cO(1)\bigr)$.  Let $Y \subset \Gr(2,4)\times \PP^4$ denote the locus
on which $s_1$, $s_2$ both vanish, and let $p\colon Y \to
\Gr(2,4)$, $q\colon Y \to \PP^4$ denote the projections to the two
factors of $\Gr(2,4)\times \PP^4$. The locus $Y$ consists of pairs
$(W,x)\in M(2,4)^\times\times \PP^4$ such that:
\[
W\subset \Ker 
\begin{pmatrix}
  x_0 & x_1 \\
  x_1 & x_2\\
  x_2 & x_3\\
  x_3 & x_4
\end{pmatrix}
\]
It follows that $q\colon Y \to \PP^4$ blows up the locus $Z$ given by
the condition:
\[
\rk 
\begin{pmatrix}
  x_0 & x_1 & x_2 & x_3\\
  x_1 & x_2 & x_3 & x_4
\end{pmatrix}
<2
\]
that is, the rational normal curve. Intersecting with $p^\star (H)$,
where $H\in |\det S^\star|$, gives the proper transform of a quadric
\mbox{3-fold} containing $Z$.

\subsection*{Abelianization:} Consider $\Gr(2,4)$ as the geometric
quotient $\CC^{8}\GIT GL_2(\CC)$ where we regard $\CC^{8}$ as the
space $M(2,4)\times $ of $2\times 4$ complex matrices and $GL_2(\CC)$
acts by multiplication on the left. The universal bundle $S$ of
subspaces on $\Gr(2,4)$ is the bundle on $\CC^{8}\GIT \GL_2(\CC)$
determined by $\Vstd^\star$, where $\Vstd$ is the standard
representation of $\GL_2(\CC)$. Consider the situation as in \S3.1 of
\cite{CFKS} with:
\begin{itemize}
\item the space that is denoted by $X$ in \cite{CFKS} set equal to $A =
  \CC^{13}$, regarded as the space of pairs:
  \[
  \{(M,w) : \text{$M$ is a $2 \times 4$ complex matrix, $w \in \CC^5$
    is a vector}\}
  \]
\item $G = \GL_2(\CC) \times \Cstar$, acting on $A$ as:
  \[
  (g,\lambda) \colon (M, w) \mapsto (g M, \lambda w)
  \]
\item $T = (\Cstar)^3$, the diagonal subtorus in $G$;
\item the group that is denoted by $S$ in \cite{CFKS} set equal to the trivial group;
\item $\cV$ equal to the representation of $G=GL_2(\CC)\times \CC^\times$ given by 
  \[
  (\Vstd \boxtimes \Vstd)^{\oplus 2}\oplus (\det \Vstd \boxtimes \Vtriv) 
  \]
  where $\Vtriv$ is the trivial $1$-dimensional representation of
  $\Cstar$.
\end{itemize}
It is clear that $A \GIT G = F$, whereas $A \GIT T=\PP^3\times \PP^3
\times \PP^4$. The Weyl group $W=\ZZ/2\ZZ$ permutes the first and
second factors of the product $\PP^3\times \PP^3 \times \PP^4$. The
representation $\cV$ induces the vector bundle $\cV_G = E$ over $F$,
whereas the representation $\cV$ induces the vector bundle:
\[
\cV_T = \cO(1,0,1)^{\oplus 2} \oplus \cO(0,1,1)^{\oplus 2}\oplus
\cO(1,1,0)
\]
over $A \GIT T = \PP^3\times \PP^3 \times \PP^4$.

\subsection*{The Abelian/non-Abelian correspondence:} Again we proceed
as in \S\ref{sec:2-17}, replacing:
\begin{itemize}
\item $\PP^3 \times \PP^3 \times \PP^3$ by $\PP^3 \times \PP^3 \times
  \PP^4$, throughout;
\item equation \eqref{eq:2-17_J_A_mod_T} by:
  \[
  J_{A \GIT T}(\tau) = 
  e^{\tau/z}
  \sum_{l_1 = 0}^\infty
  \sum_{l_2 = 0}^\infty
  \sum_{l_3 = 0}^\infty
  {
    Q_1^{l_1}  Q_2^{l_2}  Q_3^{l_3} 
    e^{l_1 \tau_1} e^{l_2 \tau_2} e^{l_3 \tau_3}
    \over
    \prod_{k=1}^{k=l_1} (p_1 + k z)^4
    \prod_{k=1}^{k=l_2} (p_2 + k z)^4
    \prod_{k=1}^{k=l_3} (p_3 + k z)^5
  }
  \]
\item equation \eqref{eq:2-17_J_e_T} by:
  \begin{multline*}
    J_{\be,\cV_T}(\tau) = e^{-(Q_1 e^{\tau_1} + Q_2 e^{\tau_2} + Q_3
      e^{\tau_3})/z}
    e^{\tau/z} \\
    \sum_{l_1 = 0}^\infty
    \sum_{l_2 = 0}^\infty
    \sum_{l_3 = 0}^\infty
    {
      Q_1^{l_1}  Q_2^{l_2}  Q_3^{l_3} 
      e^{l_1 \tau_1} e^{l_2 \tau_2} e^{l_3 \tau_3}
      \prod_{k=1}^{l_1+l_2} (\lambda + p_1+p_2+kz)
      \over
      \prod_{k=1}^{k=l_1} (p_1 + k z)^4
      \prod_{k=1}^{k=l_2} (p_2 + k z)^4
      \prod_{k=1}^{k=l_3} (p_3 + k z)^5
    } \times
    \\
    \textstyle
    \prod_{k=1}^{l_1+l_3} (\lambda + p_1+p_3+k z)^2
    \prod_{k=1}^{l_2+l_3} (\lambda + p_2+p_3+k z)^2
  \end{multline*}
\item $\Gr(2,4) \times \PP^3$ by $\Gr(2,4) \times \PP^4$, throughout;
\item equation \eqref{eq:2-17_J_e_G} by:
  \begin{multline*}
    \tilde{J}_{\be,\cV_G}\big(\theta(0)\big) \cup \Omega = \\
    e^{-(2q_1 + q_2)/z}
    \sum_{l_1 = 0}^\infty
    \sum_{l_2 = 0}^\infty
    \sum_{l_3 = 0}^\infty
    {
      (-1)^{l_1+l_2} q_1^{l_1+ l_2} q_2^{l_3} 
      \prod_{k=1}^{l_1+l_2} (\lambda + p_1+p_2+kz)
      \over
      \prod_{k=1}^{k=l_1} (p_1 + k z)^4
      \prod_{k=1}^{k=l_2} (p_2 + k z)^4
      \prod_{k=1}^{k=l_3} (p_3 + k z)^5
    } \times
    \\
    \shoveright{\textstyle
    \prod_{k=1}^{l_1+l_3} (\lambda + p_1+p_3+k z)^2
    \prod_{k=1}^{l_2+l_3} (\lambda + p_2+p_3+k z)^2 \times}
    \\
    \big(p_2-p_1 + (l_2-l_1)z\big)
  \end{multline*}
\item equation \eqref{eq:2-17_asymptotics} by:
  \[
  (p_2-p_1) \Big( 1 -2q_1z^{-1} +  O(z^{-2})\Big) 
  \]
\item the conclusion $\theta(0) = {-q_1}$ by $\theta(0) = {-2q_1}$, and
  equation \eqref{eq:2-17_almost_there} by:
  \begin{multline*}
    \tilde{J}_{\be,\cV_G}(0) \cup \Omega = \\
    e^{-q_2/z}
    \sum_{l_1 = 0}^\infty
    \sum_{l_2 = 0}^\infty
    \sum_{l_3 = 0}^\infty
    {
      (-1)^{l_1+l_2} q_1^{l_1+ l_2} q_2^{l_3} 
      \prod_{k=1}^{l_1+l_2} (\lambda + p_1+p_2+kz)
      \over
      \prod_{k=1}^{k=l_1} (p_1 + k z)^4
      \prod_{k=1}^{k=l_2} (p_2 + k z)^4
      \prod_{k=1}^{k=l_3} (p_3 + k z)^5
    } \times
    \\
    \shoveright{\textstyle
      \prod_{k=1}^{l_1+l_3} (\lambda + p_1+p_3+k z)^2
      \prod_{k=1}^{l_2+l_3} (\lambda + p_2+p_3+k z)^2 \times}
    \\
    \big(p_2-p_1 + (l_2-l_1)z\big)
  \end{multline*}
\end{itemize}
This yields:
\begin{multline*}
  G_X(t) = e^{-t} 
  \sum_{l_1=0}^\infty \sum_{l_2=0}^\infty \sum_{l_3=0}^\infty
  (-1)^{l_1+l_2} t^{l_1+l_2+l_3}
  \frac
  {
    (l_1+l_2)! \big((l_1+l_3)!\big)^2 \big((l_2+l_3)!\big)^2 
  }
  {
    (l_1!)^4 (l_2!)^4 (l_3!)^5
  }
  \times \\
  \Big(1 + (l_2-l_1)(2H_{l_2+l_3} - 4H_{l_2})\Big)
\end{multline*}
where $H_k$ is the $k$th harmonic number.  Regularizing gives:
\begin{multline*}
  \hG_X(t) = 1 + 8 t^2 + 24 t^3 + 240 t^4 + 1440 t^5 + 11960 t^6 + 89040 t^7 + 731920 t^8 + 5913600 t^9+\cdots
\end{multline*}

\subsection*{Minkowski period sequence:} \href{http://www.grdb.co.uk/search/period3?id=84&printlevel=2}{84}


\addtocounter{CustomSectionCounter}{1}
\section{The Fano Manifold $\MM{2}{22}$}
\label{sec:2-22}
\label{anchor:2--22}

\subsection*{Mori--Mukai name:} 2--22

\subsection*{Mori--Mukai construction:} The blow-up of $B_5\subset
\PP^6$ with centre a conic on it.

\subsection*{Our construction:} A complete intersection $X$ of type
$L\cap M \cap M \cap M$ in the flag manifold $\Fl=\Fl(1,2;\CC^5)$,
where $p\colon \Fl\to \PP^4$ and $q\colon \Fl \to \Gr=\Gr(2,5)$ are
the natural projections, $L=p^\star \cO(1)$, $M=q^\star \det S^\star$,
and $S$ is the universal bundle of subspaces on $\Gr$.

\subsection*{The two constructions coincide:} Note that $\Fl=\PP(S)$
is the projectivization of the universal bundle $S$ of subspaces on
$\Gr$. On $\Fl$ we have a natural surjection of vector bundles:
\[
q^\star S^\star \to L
\quad
\text{inducing}
\quad
H^0(\Fl, q^\star S^\star)\cong H^0(\Fl, L)
\]
Let $s\in H^0(\Fl, L)$ be a general section and $Y$ be the locus $(s=0)\subset \Fl$. It
is clear that $q\colon Y \to \Gr$ blows up $Z=(\tilde s=0) \subset \Gr$
where $\tilde s$ ``is'' $s$, now thought of as an element of $H^0(\Gr,
S^\star)$. We are done as $Z=Z_{1,1}$ maps to a quadric under the
Pl\"ucker embedding. 

\subsection*{Abelianization:} 
Consider the situation as in \S3.1 of \cite{CFKS} with:
\begin{itemize}
\item the space that is denoted by $X$ in \cite{CFKS} set equal to $A =
  \CC^{12}$, regarded as the space of pairs:
  \[
  \{(v,w) : \text{$v \in \CC^2$ is a row vector, $w$ is a $2 \times 5$
    complex matrix}\}
  \]
\item $G = \Cstar \times \GL_2(\CC)$, acting on $A$ as:
  \[
  (\lambda, g) \colon (v, w) \mapsto (\lambda v g^{-1}, gw)
  \]
\item $T = (\Cstar)^3$, the diagonal subtorus in $G$;
\item the group that is denoted by $S$ in \cite{CFKS} set equal to the trivial group;
\item $\cV$ equal to the representation of $G$ given by the direct sum
  of one copy of the standard representation of the first factor
  $\Cstar$ and three copies of the determinant of the standard
  representation of the second factor $\GL_2(\CC)$.
\end{itemize}
Then $A \GIT G$ is the flag manifold $\Fl=\Fl(1,2;\CC^5)$, whereas $A \GIT T$
is the toric variety with weight data:
\[
\begin{array}{rrrrrrrrrrrrl} 
1 & 1 & 1 & 1 & 1 & 0 & 0 & 0 & 0 & 0 & -1 & 0  & \hspace{1.5ex}  L_1\\ 
0 & 0 & 0 & 0 & 0 & 1 & 1 & 1 & 1 & 1 & 0   & -1& \hspace{1.5ex} L_2 \\
0 & 0 & 0 & 0 & 0 & 0 & 0 & 0 & 0 & 0 & 1   & 1  & \hspace{1.5ex} H
\end{array}
\]
and $\Amp =\langle L_1, L_2 , H \rangle$; that is, $A \GIT T$ is the
projective bundle $\PP(\cO(-1,0)\oplus \cO(0,-1))$ over $\PP^4\times
\PP^4$. The non-trivial element of the Weyl group $W=\ZZ/2\ZZ$
exchanges the two factors of $\PP^4\times \PP^4$. The representation
$\cV$ induces the vector bundle $\cV_G = L \oplus M^{\oplus 3}$ over
$A \GIT G = \Fl$, whereas the representation $\cV$ induces the vector
bundle $\cV_T = H \oplus (L_1+L_2)^{\oplus 3}$ over $A \GIT T$.

\subsection*{The Abelian/non-Abelian correspondence}
Let $p_1$, $p_2$, and $p_3 \in H^2(A\GIT T;\QQ)$ denote the first
Chern classes of the line bundles $L_1$, $L_2$, and $H$
respectively. We fix a lift of $H^\bullet(A \GIT G;\QQ)$ to
$H^\bullet(A \GIT T,\QQ)^W$ in the sense of \cite[\S3]{CFKS}; there
are many possible choices for such a lift, and the precise choice made
will be unimportant in what follows.  The lift allows us to regard
$H^\bullet(A \GIT G;\QQ)$ as a subspace of $H^\bullet(A \GIT
T,\QQ)^W$, which maps isomorphically to the Weyl-anti-invariant part
$H^\bullet(A \GIT T,\QQ)^a$ of $H^\bullet(A \GIT T,\QQ)$ via:
\[
\xymatrix{
  H^\bullet(A \GIT T,\QQ)^W \ar[rr]^{\cup(p_2-p_1)} &&
  H^\bullet(A \GIT T,\QQ)^a}
\]
We compute the quantum period of $X$ by computing the $J$-function of
$\Fl = A \GIT G$ twisted \cite{Coates--Givental} by the Euler class and
the bundle $\cV_G$, using the Abelian/non-Abelian correspondence
\cite{CFKS}.

Our first step is to compute the $J$-function of $A \GIT T$ twisted by
the Euler class and the bundle $\cV_T$.  As in
\S\ref{sec:quantum_lefschetz}, and as in \cite{CFKS}, consider the
bundles $\cV_T$ and $\cV_G$ equipped with the canonical
$\Cstar$-action that rotates fibers and acts trivially on the base.
We will compute the twisted $J$-function $J_{\be,\cV_T}$ of $A \GIT T$
using the Quantum Lefschetz theorem; $J_{\be,\cV_T}$ was defined in
equation \eqref{eq:twisted_J} above, and is the restriction to the
locus $\tau \in H^0(A \GIT T) \oplus H^2(A \GIT T)$ of what was
denoted by $J^{S \times \Cstar}_{\cV_T}(\tau)$ in \cite{CFKS}.  The
toric variety $A \GIT T$ is Fano, so Theorem~\ref{thm:toric_mirror}
gives:
\[
J_{A \GIT T}(\tau) = 
e^{\tau/z}
\sum_{l, m, n \geq 0}
{
  Q_1^l Q_2^m Q_3^n 
  e^{l \tau_1} e^{m \tau_2} e^{m \tau_3}
  \over
  \prod_{k=1}^{k=l} (p_1 + k z)^5
  \prod_{k=1}^{k=m} (p_2 + k z)^5
}
{
\prod_{k =-\infty}^{k=0} p_3-p_1 + k z 
\over
\prod_{k=-\infty}^{k=n-l} p_3-p_1 + k z 
}
{
\prod_{k = -\infty}^{k=0} p_3-p_2 + k z 
\over
\prod_{k=-\infty}^{k= n-m} p_3-p_2 + k z 
}
\]
where $\tau = \tau_1 p_1 + \tau_2 p_2 + \tau_3 p_3$ and we have
identified the group ring $\QQ[H_2(A \GIT T;\ZZ)]$ with
$\QQ[Q_1,Q_2,Q_3]$ via the $\QQ$-linear map that sends $Q^\beta$ to
$Q_1^{\langle \beta, p_1 \rangle} Q_2^{\langle \beta, p_2\rangle} Q_3^{\langle \beta, p_3 \rangle}$.
The line bundles $L_1$, $L_2$, and $H$ are nef, and $c_1(A \GIT T) -
c_1(\cV_T)$ is ample, so Theorem~\ref{thm:toric_ci_ql} gives:
\begin{multline*}
  J_{\be,\cV_T}(\tau) = \\
  e^{-Q_3 e^{\tau_3}/z}  e^{\tau/z}
  \sum_{l, m, n \geq 0}
  Q_1^l Q_2^m Q_3^n e^{l \tau_1} e^{m \tau_2} e^{m \tau_3}
  {
    \prod_{k=1}^{k=n} (\lambda + p_3 + k z)
    \prod_{k=1}^{k=l+m} (\lambda + p_1 + p_2 + k z)^3
    \over
    \prod_{k=1}^{k=l} (p_1 + k z)^5
    \prod_{k=1}^{k=m} (p_2 + k z)^5
  } \times \\
  {
    \prod_{k =-\infty}^{k=0} p_3-p_1 + k z 
    \over
    \prod_{k=-\infty}^{k=n-l} p_3-p_1 + k z 
  }
  {
    \prod_{k = -\infty}^{k=0} p_3-p_2 + k z 
    \over
    \prod_{k=-\infty}^{k= n-m} p_3-p_2 + k z 
  }
\end{multline*}

Consider now $F = A \GIT G = \Fl$ and a point $t \in H^\bullet(F)$.
Recall that $\Fl=\PP(S)$ is the projectivization of the universal
bundle $S$ of subspaces on $\Gr$. Let $\epsilon_1 \in H^2(F;\QQ)$ be
the pullback to $F$ (under the projection map $q\colon\Fl\to \Gr$) of
the ample generator of $H^2(\Gr)$, and let $\epsilon_2 \in H^2(F;\QQ)$
be the first Chern class of $\cO_{\PP(S)}(1)$.  Identify the group
ring $\QQ[H_2(F;\ZZ)]$ with $\QQ[q_1,q_2]$ via the $\QQ$-linear map
which sends $Q^\beta$ to $q_1^{\langle \beta,\epsilon_1 \rangle}
q_2^{\langle \beta,\epsilon_2 \rangle}$.  In \cite{CFKS}*{\S6.1} the
authors consider the lift $\tilde{J}^{S \times \Cstar}_{\cV_G}(t)$
of their twisted $J$-function $J^{S \times \Cstar}_{\cV_G}(t)$
determined by a choice of lift $H^\bullet(A \GIT G;\QQ) \to
H^\bullet(A \GIT T,\QQ)^W$.  We restrict to the locus $t \in H^0(A
\GIT G;\QQ) \oplus H^2(A \GIT G;\QQ)$, considering the lift:
\begin{align*}
  \tilde{J}_{\be,\cV_G}(t) := \tilde{J}^{S \times \Cstar}_{\cV_G}(t)
  && t \in H^0(A \GIT G;\QQ) \oplus H^2(A \GIT G;\QQ)
\end{align*}
of our twisted $J$-function $J_{\be,\cV_G}$ determined by our choice
of lift $H^\bullet(A \GIT G;\QQ) \to H^\bullet(A \GIT T,\QQ)^W$.
Theorems~4.1.1 and~6.1.2 in \cite{CFKS} imply that:
\[
\tilde{J}_{\be,\cV_G}\big(\theta(t)\big) \cup (p_2 - p_1) = \Big[ 
\textstyle 
\big(z {\partial \over \partial \tau_2} 
- 
z {\partial \over \partial \tau_1} 
\big) J_{\be,\cV_T}(\tau) \Big]_{\tau=t, Q_1 = Q_2 = -q_1, Q_3=q_2}
\]
for some\footnote{In fact the mirror map $\theta$ takes values in
  $H^0(A \GIT G;\Lambda_G) \oplus H^2(A \GIT G;\Lambda_G)$.  This
  follows from homogeneity considerations, as in the proof of
  Proposition~\ref{pro:ABC}. We will see explicitly that $\theta(0)
  \in H^0 \oplus H^2$.} function $\theta:H^2(A \GIT G;\QQ) \to
H^\bullet(A \GIT G; \Lambda_G)$. Setting $t = 0$ gives:
\begin{multline*}
  \tilde{J}_{\be,\cV_G}\big(\theta(0)\big) \cup (p_2 - p_1) = \\
  e^{-q_2/z}  
  \sum_{l, m, n \geq 0}
  (-1)^{l+m} q_1^{l+m} q_2^n 
  {
    \prod_{k=1}^{k=n} (\lambda + p_3 + k z)
    \prod_{k=1}^{k=l+m} (\lambda + p_1 + p_2 + k z)^3
    \over
    \prod_{k=1}^{k=l} (p_1 + k z)^5
    \prod_{k=1}^{k=m} (p_2 + k z)^5
  } \times \\
  {
    \prod_{k =-\infty}^{k=0} p_3-p_1 + k z 
    \over
    \prod_{k=-\infty}^{k=n-l} p_3-p_1 + k z 
  }
  {
    \prod_{k = -\infty}^{k=0} p_3-p_2 + k z 
    \over
    \prod_{k=-\infty}^{k= n-m} p_3-p_2 + k z 
  }
  \big(p_2-p_1 + (m-l) z \big)
\end{multline*}
For symmetry reasons the right-hand side here is divisible by
$p_2-p_1$; it takes the form:
\[
(p_2-p_1)\Big(1 + q_1 z^{-1}+ O(z^{-2})\Big)
\]
whereas:
\[
\tilde{J}_{\be,\cV_G}\big(\theta(0)\big) \cup (p_2 - p_1) = 
(p_2-p_1) \Big(1 + \theta(0) z^{-1} + O(z^{-2}) \Big)
\]
We conclude that $\theta(0) = q_1$ and hence, via the String
Equation, that:
\[
J_{\be,\cV_G}\big(\theta(0)\big) = e^{q_1/z} J_{\be,\cV_G}(0)
\]
Thus:
\begin{multline}
  \label{eq:almost_there}
  \tilde{J}_{\be,\cV_G}(0) \cup (p_2 - p_1) = \\
  e^{-(q_1+q_2)/z}  
  \sum_{l, m, n \geq 0}
  (-1)^{l+m} q_1^{l+m} q_2^n 
  {
    \prod_{k=1}^{k=n} (\lambda + p_3 + k z)
    \prod_{k=1}^{k=l+m} (\lambda + p_1 + p_2 + k z)^3
    \over
    \prod_{k=1}^{k=l} (p_1 + k z)^5
    \prod_{k=1}^{k=m} (p_2 + k z)^5
  } \times \\
  {
    \prod_{k =-\infty}^{k=0} p_3-p_1 + k z 
    \over
    \prod_{k=-\infty}^{k=n-l} p_3-p_1 + k z 
  }
  {
    \prod_{k = -\infty}^{k=0} p_3-p_2 + k z 
    \over
    \prod_{k=-\infty}^{k= n-m} p_3-p_2 + k z 
  }
  \big(p_2-p_1 + (m-l) z \big)
\end{multline}

We saw in Example~\ref{ex:ql_with_mirror_map} how to extract the quantum period $G_X$ from
the twisted $J$-function $J_{\be,\cV_G}(0)$: we take the
non-equivariant limit, extract the component along the unit class $1
\in H^\bullet(A \GIT G;\QQ)$, set $z=1$, and set $Q^\beta = t^{\langle
  \beta, {-K_X} \rangle}$.  Thus we consider the right-hand side of
\eqref{eq:almost_there}, take the non-equivariant limit, extract the
coefficient of $p_2-p_1$, set $z=1$, and set $q_1 = q_2 =t$,
obtaining:
\begin{multline*}
  G_X(t) = e^{-2t} \sum_{l=0}^\infty \sum_{m=0}^\infty
  \sum_{n=\max(l,m)}^\infty 
  (-1)^{l+m}
  \frac
  {
    n! ((l+m)!)^3
  }
  {
    (l!)^5 (m!)^5 (n-l)! (n-m)!
  }
  t^{l+m+n} \\
  + 
  e^{-2t} \sum_{l=0}^\infty \sum_{m=l+1}^\infty
  \sum_{n=m}^\infty 
  (-1)^{l+m}
  \frac
  {
    n! ((l+m)!)^3
    (m-l)(5H_l-5H_m+H_{n-m}-H_{n-l})
  }
  {
    (l!)^5 (m!)^5 (n-l)! (n-m)!
  }  
  t^{l+m+n} \\
  + 
  e^{-2t} \sum_{l=0}^\infty \sum_{m=0}^\infty
  \sum_{n=l}^{m-1}
  (-1)^{l+n}
  \frac
  {
    n! ((l+m)!)^3
    (m-l)(m-n-1)!
  }
  {
    (l!)^5 (m!)^5 (n-l)! 
  }  
  t^{l+m+n} 
\end{multline*}
Regularizing yields:
\[
\hG_X(t) = 1+6 t^2+24 t^3+138 t^4+1080 t^5+6540 t^6+50400 t^7+362250 t^8+2713200 t^9+\cdots
\]

\subsection*{Minkowski period sequence:} \href{http://www.grdb.co.uk/search/period3?id=69&printlevel=2}{69}


\addtocounter{CustomSectionCounter}{1}

\section{The Fano Manifold $\MM{2}{23}$}
\label{anchor:2--23}

\subsection*{Mori--Mukai name:} 2--23

\subsection*{Mori--Mukai construction:} The blow-up of a quadric
\mbox{3-fold} $Q\subset \PP^4$ with centre an intersection of $ A \in
|\cO_Q(1)|$ and $ B \in |\cO_Q(2)|$ such that:
\begin{itemize}
\item[(a)] $A$ is nonsingular;
\item[(b)] $A$ is singular.
\end{itemize}

\subsection*{Our construction:} A codimension-2 complete intersection
$X$ of type $ (L+M)\cap (2L)$ in the toric variety $F$ with weight
data:
\[ 
\begin{array}{rrrrrrrl} 
  \multicolumn{1}{c}{s_0} & 
  \multicolumn{1}{c}{s_1} & 
  \multicolumn{1}{c}{s_2} & 
  \multicolumn{1}{c}{s_3} & 
  \multicolumn{1}{c}{s_4} & 
  \multicolumn{1}{c}{x} & 
  \multicolumn{1}{c}{x_5} & \\ 
  \cmidrule{1-7}
  1 & 1 & 1 & 1 & 1 & -1 & 0 & \hspace{1.5ex} L\\ 
  0 & 0 & 0 & 0 & 0 & 1 & 1 & \hspace{1.5ex} M \\
\end{array}
\]
and  $\Amp F = \langle L, M\rangle$.  We have:
\begin{itemize}
\item $-K_F=4L+2M$ is ample, that is $F$ is a Fano variety;
\item $X$ is the intersection of two nef divisors on $F$;
\item $-(K_F+\Lambda)\sim L+M$ is ample.
\end{itemize}

\subsection*{The two constructions coincide:} Apply
Lemma~\ref{lem:blowups} with $V = \cO_Q(-1) \oplus \cO_Q$, $W =
\cO_Q(1)$, and $f \colon V \to W$ given by the matrix $
\begin{pmatrix}
  B & A
\end{pmatrix}
$.  This exhibits $X$ as a member of $|\pi^\star W(1)|$ on $\PP(V)$,
or in other words as a complete intersection of type $(L+M) \cap (2L)$
on the toric variety $F$.

\subsection*{The quantum period:}  Corollary~\ref{cor:QL} yields:
\[
G_X(t) = e^{-t} \sum_{l=0}^\infty \sum_{m=l}^\infty
t^{l+m}
\frac{(l+m)!(2l)!}
{(l!)^5 (m-l)! m!}
\]
and regularizing gives:
\[
\hG_X(t) = 1+8 t^2+12 t^3+216 t^4+720 t^5+8540 t^6+42000 t^7+410200 t^8+2503200 t^9+ \cdots
\]

\subsection*{Minkowski period sequence:} \href{http://www.grdb.co.uk/search/period3?id=78&printlevel=2}{78}


\addtocounter{CustomSectionCounter}{1}

\section{The Fano Manifold $\MM{2}{24}$}
\label{anchor:2--24}

\subsection*{Mori--Mukai name:} 2--24

\subsection*{Mori--Mukai construction:} A divisor of bidegree $(1,2)$ on $\PP^2\times \PP^2$.

\subsection*{Our construction:} A member $X$ of $|L+2M|$ in the toric
variety $F=\PP^2\times \PP^2$.

\subsection*{The two constructions coincide:} Obvious.

\subsection*{The quantum period:}
The toric variety $F$ has weight data:
\[ 
\begin{array}{rrrrrrl} 
1 & 1 & 1 & 0 & 0 & 0 & \hspace{1.5ex} L\\ 
0 & 0 & 0 & 1 & 1 & 1 & \hspace{1.5ex} M \\
\end{array}
\]
and $\Amp F = \langle L,M\rangle$. We have:
\begin{itemize}
\item $F$ is a Fano variety;
\item $X\sim L+2M$ is ample;
\item $-(K_F+X)\sim 2L+M$ is ample.
\end{itemize}

Corollary~\ref{cor:QL} yields:
\[
G_X(t) = e^{-2t} \sum_{l=0}^\infty \sum_{m=0}^\infty 
t^{2l+m}
\frac{(l+2m)!}
{(l!)^3 (m!)^3}
\]
and regularizing gives:
\begin{multline*}
  \hG_X(t) = 1+4 t^2+24 t^3+132 t^4+780 t^5+5800 t^6+40320 t^7+283780 t^8+2105880 t^9+ \cdots
\end{multline*}

\subsection*{Minkowski period sequence:} \href{http://www.grdb.co.uk/search/period3?id=44&printlevel=2}{44}


\addtocounter{CustomSectionCounter}{1}

\section{The Fano Manifold $\MM{2}{25}$}
\label{anchor:2--25}

\subsection*{Mori--Mukai name:} 2--25

\subsection*{Mori--Mukai construction:} The blow up of $\PP^3$ with
centre an elliptic curve that is an intersection of two quadrics.

\subsection*{Our construction:} A member $X$ of $|L+2M|$ in the toric
variety $F=\PP^1\times \PP^3$.

\subsection*{The two constructions coincide:} Apply
Lemma~\ref{lem:blowups} with $V = \cO_{\PP^3} \oplus \cO_{\PP^3}$,
$W=\cO_{\PP^3}(2)$, and $f \colon V \to W$ the map given by the two quadrics
that define the elliptic curve.

\subsection*{The quantum period:}
The toric variety $F$ has weight data:
\[ 
\begin{array}{rrrrrrl} 
1 & 1 & 0 & 0 & 0 & 0 & \hspace{1.5ex} L\\ 
0 & 0 & 1 & 1 & 1 & 1 & \hspace{1.5ex} M \\
\end{array}
\]
and $\Amp F = \langle L,M\rangle$.  We have:
\begin{itemize}
\item $F$ is a Fano variety;
\item $X\sim L+2M$ is ample;
\item $-(K_F+X)\sim L+2M$ is ample.
\end{itemize}

Corollary~\ref{cor:QL} yields:
\[
G_X(t) = e^{-t} \sum_{l=0}^\infty \sum_{m=0}^\infty
t^{l+2m}
\frac{(l+2m)!}
{(l!)^2 (m!)^4}
\]
and regularizing gives:
\begin{multline*}
  \hG_X(t) = 1+4 t^2+24 t^3+60 t^4+720 t^5+3640 t^6+21840 t^7+175420 t^8+1024800 t^9+ \cdots
\end{multline*}

\subsection*{Minkowski period sequence:} \href{http://www.grdb.co.uk/search/period3?id=43&printlevel=2}{43}


\addtocounter{CustomSectionCounter}{1}
\section{The Fano Manifold $\MM{2}{26}$}
\label{sec:2-26}
\label{anchor:2--26}

\subsection*{Mori--Mukai name:} 2--26

\subsection*{Mori--Mukai construction:} The blow-up of $B_5\subset
\PP^6$ with centre a line on it.

\subsection*{Our construction:} 
Let $S$ be the universal bundle of subspaces over $\Gr = \Gr(2,4)$,
and let $E$ be the rank-$3$ vector bundle $E = \CC \oplus S^\star$ on
$\Gr$.  Let $q\colon \PP(E) \to \Gr$ denote the projection.  Then $X$
is the vanishing locus of a general section of:
\[
q^\star \det S^\star \oplus \Bigl( (q^\star \det S^\star)\otimes
\cO_{\PP(E)}(1) \Bigr)^{\oplus 2}
\]
on the key variety $F=\PP(E)$.

\subsection*{The two constructions coincide:} 

Write $V=\CC^5$ with basis $e_0,\dots,e_4$, and write $\CC^4=V/\CC
e_0$. Consider $\Gr$ as the Grassmannian of two-dimensional subspaces
of this $\CC^4$.  There is an exact sequence:
\[
0\to T \to q^\star E^\star \to \cO_{\PP(E)}(1)\to 0
\]
on $F = \PP(E)$, where $T$ is a rank-$2$ vector bundle.

First we construct a morphism $p\colon F\to \Gr(2, V)=\Gr (2,
\CC^5)$. Let $U$ denote the universal bundle of subspaces on
$\Gr(2,5)$. The morphism $p$ arises, by the universal property of
$\Gr(2,\CC^5)$, from the inclusion:
\[
T \subset q^\star E^\star = \CC\oplus q^\star S \subset \CC\oplus
q^\star \CC^4=\CC e_0\oplus \CC^4=\CC^5
\]
i.e. there is a unique $p\colon F\to \Gr(2,\CC^5)$ such that
$S=p^\star U$.

Next we claim that the morphism $p\colon F \to \Gr(2,5)$ that we just
constructed is the blow-up of $\Gr(2,5)$ along the locus 
\[
Z=\{W_2\subset \CC^5\mid e_0\in W_2\}
\]
of two-dimensional vector subspaces that contain $e_0$.  Denote by
$\pi\colon \CC^5\to \CC^5/\CC e_0$ the natural projection. Indeed for
$W_2\in \Gr(2, 5)$ \emph{either}:
\begin{itemize}
\item $e_0\not \in W_2$, in which case $\pi(W_2)=V_2\subset \CC^4$ is
  a $2$-dimensional subspace and $p$ is an isomorphism above $W_2$, \emph{or}
\item $e_0 \in W_2$, in which case $\pi(W_2)$ is a $1$-dimensional
  subspace and 
\[
q(p^{-1} W_2) =\{ V_2\in \Gr(2, 4) \mid \pi(W_2)\subset V_2\}
\]
\end{itemize}

The statement follows easily from the claim just shown. Indeed, on the
one hand $Z\cong \PP^3$ and the Pl\"ucker embedding of $\Gr(2,5)$
embeds $Z$ linearly in $\PP^9$. In other words, $p\colon F \to
\Gr(2,5)$ is the blow up of $\Gr(2, 5)\subset \PP^9$ along a
$\PP^3\subset \Gr(2,5)$. On the other hand, the rational map
\[
q p^{-1} \colon \Gr(2, 5)\dasharrow \Gr(2,4)\subset \PP^5
\]
where $\Gr(2,4)\subset \PP^5$ is the Pl\"ucker embedding of
$\Gr(2,4)$, is the map corresponding to the linear system of
hyperplane sections of $\Gr(2,5)\subset \PP^9$, in \emph{its}
Pl\"ucker embedding, that contain $Z$.

In other words, let now $Y\subset \Gr(2, 4)$ be a general hyperplane
section, and $H_1, H_2\subset \Gr(2, 5)$ be two general hyperplane sections
of $\Gr(2,5)$, then
\[
p\colon q^{-1}(Y)\cap p^{-1} (H_1\cap H_2)  \to p q^{-1}(Y) \cap H_1\cap H_2
\]
is the blow-up of $B_5 =p q^{-1}(Y) \cap H_1\cap H_2\subset \Gr(2,5)$
along the line $Z\cap B_5$.

\subsection*{Abelianization:} 

Consider the situation as in \S3.1 of \cite{CFKS} with:
\begin{itemize}
\item the space that is denoted by $X$ in \cite{CFKS} set equal to $A =
  \CC^{11}$, regarded as the space of pairs:
  \[
  \{(v,w) : \text{$v$ is a $2\times 4$ complex matrix, $w\in \CC^3$ is
    a column vector} \}
  \]
\item $G = \GL_2(\CC) \times \CC^\star$, acting on $A$ as:
  \[
  (g, \lambda) \colon (v, w) \mapsto (g v , \lambda \rho(g) w )
  \]
  where $GL_2(\CC)$ acts by left multiplication on $M(2,4)$ and $\rho
  = \rho_\text{std}\oplus 0$ is the direct sum of a copy of the
  standard representation of $GL_2(\CC)$ and a copy of the trivial
  representation.
\item $T = (\Cstar)^3$, the diagonal subtorus in $G$;
\item the group that is denoted by $S$ in \cite{CFKS} set equal to the
  trivial group;
\item $\cV$ equal to the representation of $G$ given by:
\[
\psi \oplus (\chi_3\otimes \psi)^{\oplus 2}
\]
where: $\psi\colon G\to \Cstar$ is $\det \rho_\text{std}$ on the first
factor and trivial on the second factor; whereas $\chi_3\colon G \to
\Cstar$ is trivial on the first factor and the identity on the second factor.
\end{itemize}
Then $A \GIT G$ is the key variety $F=\PP(E)$ introduced above (this
follows from Lemma~\ref{lem:P(E)}), whereas $A \GIT
T$ is the toric variety with weight data:
\[
\begin{array}{rrrrrrrrrrrl} 
1 & 1 & 1 & 1 & 0 & 0 & 0 & 0 & 1   & 0 & 0  & \hspace{1.5ex} L_1\\ 
0 & 0 & 0 & 0 & 1 & 1 & 1 & 1 & 0   & 1 & 0  & \hspace{1.5ex} L_2 \\
0 & 0 & 0 & 0 & 0 & 0 & 0 & 0 & 1   & 1 & 1  & \hspace{1.5ex} L_3
\end{array}
\]
and $\Amp =\langle L_1, L_2 , L_1+L_2+L_3 \rangle$; that is, $A \GIT
T$ is the projective bundle $\PP\bigl(\cO(-1,0)\oplus \cO(0,-1) \oplus
\cO(-1,-1)\bigr)$ over $\PP^3\times \PP^3$.  The Weyl group
$W=\ZZ/2\ZZ$ exchanges the first and second factors of $\PP^3\times
\PP^3$, that is, it exchanges the first set of four co-ordinates with
the second set of four coordinates in the table giving the weight
data. The representation $\cV$ induces the vector bundle $q^\star \det
S^\star \oplus \Bigl( (q^\star \det S^\star)(1)\Bigl)^{\oplus 2}$ over
$A \GIT G = F$, whereas the representation $\cV$ induces the vector
bundle
\[
(L_1+L_2)\oplus (L_1+L_2+L_3)^{\oplus 2}
\]
on $A\GIT T$.

\subsection*{The Abelian/non-Abelian correspondence:} 
Let $p_1$, $p_2$, and $p_3 \in H^2(A\GIT T;\QQ)$ denote the first
Chern classes of the line bundles $L_1$, $L_2$, and $L_1\otimes L_2
\otimes L_3$ respectively. We fix a lift of $H^\bullet(A \GIT G;\QQ)$
to $H^\bullet(A \GIT T,\QQ)^W$ in the sense of \cite[\S3]{CFKS}; as
before there are many possible choices for such a lift, and the
precise choice made will be unimportant in what follows.  The lift
allows us to regard $H^\bullet(A \GIT G;\QQ)$ as a subspace of
$H^\bullet(A \GIT T,\QQ)^W$, which maps isomorphically to the
Weyl-anti-invariant part $H^\bullet(A \GIT T,\QQ)^a$ of $H^\bullet(A
\GIT T,\QQ)$ via:
\[
\xymatrix{
  H^\bullet(A \GIT T,\QQ)^W \ar[rr]^{\cup(p_2-p_1)} &&
  H^\bullet(A \GIT T,\QQ)^a}
\]
We compute the quantum period of $X$ by computing the $J$-function of
$\Fl = A \GIT G$ twisted \cite{Coates--Givental} by the Euler class and
the bundle $\cV_G$, using the Abelian/non-Abelian correspondence
\cite{CFKS}.

We begin by computing the $J$-function of $A \GIT T$ twisted by the
Euler class and the bundle $\cV_T$.  Consider the bundles $\cV_T$ and
$\cV_G$ equipped with the canonical $\Cstar$-action that rotates
fibers and acts trivially on the base.  We will compute the twisted
$J$-function $J_{\be,\cV_T}$ of $A \GIT T$ using the Quantum Lefschetz
theorem; $J_{\be,\cV_T}$ was defined in equation \eqref{eq:twisted_J}
above, and is the restriction to the locus $\tau \in H^0(A \GIT T)
\oplus H^2(A \GIT T)$ of what was denoted by $J^{S \times
  \Cstar}_{\cV_T}(\tau)$ in \cite{CFKS}.  The toric variety $A \GIT T$
is Fano, so Theorem~\ref{thm:toric_mirror} gives:
\begin{multline*}
  J_{A \GIT T}(\tau) = 
  e^{\tau/z}
  \sum_{l, m, n \geq 0}
  {
    Q_1^l Q_2^m Q_3^n 
    e^{l \tau_1} e^{m \tau_2} e^{m \tau_3}
    \over
    \prod_{k=1}^{k=l} (p_1 + k z)^4
    \prod_{k=1}^{k=m} (p_2 + k z)^4
  }
  {
    \prod_{k =-\infty}^{k=0} p_3-p_2 + k z 
    \over
    \prod_{k=-\infty}^{k=n-m} p_3-p_2 + k z 
  }
  \\ \times
  {
    \prod_{k = -\infty}^{k=0} p_3-p_1 + k z 
    \over
    \prod_{k=-\infty}^{k= n-l} p_3-p_1 + k z 
  }
  {
    \prod_{k = -\infty}^{k=0} p_3-p_1-p_2 + k z 
    \over
    \prod_{k=-\infty}^{k= n-l-m} p_3-p_1-p_2 + k z 
  }
\end{multline*}
where $\tau = \tau_1 p_1 + \tau_2 p_2 + \tau_3 p_3$ and we have
identified the group ring $\QQ[H_2(A \GIT T;\ZZ)]$ with
$\QQ[Q_1,Q_2,Q_3]$ via the $\QQ$-linear map that sends $Q^\beta$ to
$Q_1^{\langle \beta, p_1 \rangle} Q_2^{\langle \beta, p_2\rangle}
Q_3^{\langle \beta, p_3 \rangle}$.  The line bundles $L_1 + L_2$, and
$L_1\otimes L_2 \otimes L_3$ are nef, and $c_1(A \GIT T) - c_1(\cV_T)$
is ample, so Theorem~\ref{thm:toric_ci_ql} gives:
\begin{multline*}
  J_{\be,\cV_T}(\tau) = \\
  e^{-Q_3 e^{\tau_3}/z}  e^{\tau/z}
  \sum_{l, m, n \geq 0}
  Q_1^l Q_2^m Q_3^n e^{l \tau_1} e^{m \tau_2} e^{m \tau_3}
  {
    \prod_{k=1}^{k=l+m} (\lambda + p_1 + p_2 + k z)
    \prod_{k=1}^{k=n} (\lambda + p_3 + k z)^2
    \over
    \prod_{k=1}^{k=l} (p_1 + k z)^4
    \prod_{k=1}^{k=m} (p_2 + k z)^4
  } \times \\
    {
    \prod_{k =-\infty}^{k=0} p_3-p_2 + k z 
    \over
    \prod_{k=-\infty}^{k=n-m} p_3-p_2 + k z 
  }
  {
    \prod_{k = -\infty}^{k=0} p_3-p_1 + k z 
    \over
    \prod_{k=-\infty}^{k= n-l} p_3-p_1 + k z 
  }
  {
    \prod_{k = -\infty}^{k=0} p_3-p_1-p_2 + k z 
    \over
    \prod_{k=-\infty}^{k= n-l-m} p_3-p_1-p_2 + k z 
  }
\end{multline*}

Consider now $F = A \GIT G = \PP(E)$ and a point $t \in H^\bullet(F)$.
Let $\epsilon_1 \in H^2(F;\QQ)$ be the pullback to $F$ (under the
projection map $q\colon\PP(E)\to \Gr(2,4)$) of the ample generator of
$H^2(\Gr(2,4))$, and let $\epsilon_2 \in H^2(F;\QQ)$ be the first
Chern class of $(q^\star \det S^\star)\otimes\cO_{\PP(E)}(1)$.
Identify the group ring $\QQ[H_2(F;\ZZ)]$ with $\QQ[q_1,q_2]$ via the
$\QQ$-linear map which sends $Q^\beta$ to $q_1^{\langle
  \beta,\epsilon_1 \rangle} q_2^{\langle \beta,\epsilon_2 \rangle}$.
In \cite{CFKS}*{\S6.1} the authors consider the lift $\tilde{J}^{S
  \times \Cstar}_{\cV_G}(t)$ of their twisted $J$-function $J^{S
  \times \Cstar}_{\cV_G}(t)$ determined by a choice of lift
$H^\bullet(A \GIT G;\QQ) \to H^\bullet(A \GIT T,\QQ)^W$.  We restrict
to the locus $t \in H^0(A \GIT G;\QQ) \oplus H^2(A \GIT G;\QQ)$,
considering the lift:
\begin{align*}
  \tilde{J}_{\be,\cV_G}(t) := \tilde{J}^{S \times
    \Cstar}_{\cV_G}(t) && t \in H^0(A \GIT G;\QQ) \oplus H^2(A
  \GIT G;\QQ)
\end{align*}
of our twisted $J$-function $J_{\be,\cV_G}$ determined by our choice
of lift $H^\bullet(A \GIT G;\QQ) \to H^\bullet(A \GIT T,\QQ)^W$.
Theorems~4.1.1 and~6.1.2 in \cite{CFKS} imply that:
\[
\tilde{J}_{\be,\cV_G}\big(\theta(t)\big) \cup (p_2-p_1) = \Big[ 
\big(z \textstyle {\partial \over \partial \tau_2} 
- 
z {\partial \over \partial \tau_1} 
\big) J_{\be,\cV_T}(\tau) \Big]_{\tau = t, Q_1  = Q_2 = -q_1, Q_3 = q_2}
\]
for some\footnote{As in Theorem~\ref{thm:rank_1_A_nA} and
  footnote~\ref{footnote:2-17}, the map $\theta$ is grading
  preserving and satisfies $\theta \equiv \id$ modulo $q_1, q_2$.}  function $\theta
\colon H^2(A \GIT G;\QQ) \to H^\bullet(A \GIT G; \Lambda_{A \GIT G})$.
Setting $t = 0$ gives:
\begin{multline*}
  \tilde{J}_{\be,\cV_G}\big(\theta(0)\big) \cup (p_2-p_1) = \\
  e^{-q_2/z}  
  \sum_{l, m, n \geq 0}
  (-1)^{l+m} q_1^{l+m} q_2^n 
  {
    \prod_{k=1}^{k=l+m} (\lambda + p_1 + p_2 + k z)
    \prod_{k=1}^{k=n} (\lambda + p_3 + k z)^2
    \over
    \prod_{k=1}^{k=l} (p_1 + k z)^4
    \prod_{k=1}^{k=m} (p_2 + k z)^4
  } \times \\
  \shoveright{  {
    \prod_{k =-\infty}^{k=0} p_3-p_2 + k z 
    \over
    \prod_{k=-\infty}^{k=n-m} p_3-p_2 + k z 
  }
  {
    \prod_{k = -\infty}^{k=0} p_3-p_1 + k z 
    \over
    \prod_{k=-\infty}^{k= n-l} p_3-p_1 + k z 
  } 
  {
    \prod_{k = -\infty}^{k=0} p_3-p_1-p_2 + k z 
    \over
    \prod_{k=-\infty}^{k= n-l-m} p_3-p_1-p_2 + k z 
  } \times} \\
  \big(p_2-p_1 + (m-l)z\big)
\end{multline*}
The left-hand side here takes the form:
\begin{align*}
  & (p_2-p_1) \Big( 1 + \theta(0) z^{-1} + O(z^{-2})\Big)  \\
  \intertext{whereas the right-hand side is:}
  & (p_2-p_1) \Big( 1 + O(z^{-2})\Big) 
\end{align*}
and therefore $\theta(0) = 0$.  Thus:
\begin{multline}
  \label{eq:2-26_almost_there}
  \tilde{J}_{\be,\cV_G}(0) \cup (p_2-p_1) = \\
  e^{-q_2/z}  
  \sum_{l, m, n \geq 0}
  (-1)^{l+m} q_1^{l+m} q_2^n 
  {
    \prod_{k=1}^{k=l+m} (\lambda + p_1 + p_2 + k z)
    \prod_{k=1}^{k=n} (\lambda + p_3 + k z)^2
    \over
    \prod_{k=1}^{k=l} (p_1 + k z)^4
    \prod_{k=1}^{k=m} (p_2 + k z)^4
  } \times \\
  \shoveright{  {
    \prod_{k =-\infty}^{k=0} p_3-p_2 + k z 
    \over
    \prod_{k=-\infty}^{k=n-m} p_3-p_2 + k z 
  }
  {
    \prod_{k = -\infty}^{k=0} p_3-p_1 + k z 
    \over
    \prod_{k=-\infty}^{k= n-l} p_3-p_1 + k z 
  } 
  {
    \prod_{k = -\infty}^{k=0} p_3-p_1-p_2 + k z 
    \over
    \prod_{k=-\infty}^{k= n-l-m} p_3-p_1-p_2 + k z 
  } \times} \\
  \big(p_2-p_1 + (m-l)z\big)
\end{multline}

We saw in Example~\ref{ex:ql_with_mirror_map} how to extract the
quantum period $G_X$ from the twisted $J$-function $J_{\be,\cV_G}(0)$:
we take the non-equivariant limit $\lambda \to 0$, extract the
component along the unit class $1 \in H^\bullet(A \GIT G;\QQ)$, set
$z=1$, and set $Q^\beta = t^{\langle \beta, {-K_X} \rangle}$.  Thus we
consider the right-hand side of \eqref{eq:2-26_almost_there}, take the
non-equivariant limit, extract the coefficient of $p_2-p_1$, set
$z=1$, set $q_1=t$, and set $q_2=t$.  This yields:
\begin{multline*}
  G_X(t) = e^{-t} 
  \sum_{l=0}^\infty \sum_{m=0}^\infty \sum_{n=l+m}^\infty
  (-1)^{l+m} t^{l+m+n}
  \frac
  {
    (l+m)!(n!)^2
  }
  {
    (l!)^4 (m!)^4 (n-m)! (n-l)! (n-l-m)!
  }
  \times \\
  \Big(1 + (m-l)(H_{n-m}-4H_m)\Big)
\end{multline*}
where $H_k$ is the $k$th harmonic number.  Regularizing gives:
\begin{multline*}
  \hG_X(t) = 1 + 6 t^2 + 12 t^3 + 114 t^4 + 540 t^5 + 3480 t^6 + 22680 t^7 + 137970 t^8 + 978600 t^9+\cdots
\end{multline*}

\subsection*{Minkowski period sequence:} \href{http://www.grdb.co.uk/search/period3?id=58&printlevel=2}{58}


\addtocounter{CustomSectionCounter}{1}

\section{The Fano Manifold $\MM{2}{27}$}
\label{anchor:2--27}

\subsection*{Mori--Mukai name:} 2--27

\subsection*{Mori--Mukai construction:} The blow up of $\PP^3$ with
centre a twisted cubic.

\subsection*{Our construction:} A codimension-2 complete intersection
$X$ of type $(L+M) \cap (L+M)$ in the toric variety $F=\PP^3\times
\PP^2$.

\subsection*{The two constructions coincide:} The twisted cubic in
$\PP^3$ with co-ordinates $x_0, x_1, x_2, x_3$ is given by the
condition:
\[ 
\rk
\begin{pmatrix}
  x_0 & x_1 & x_2 \\ 
  x_1 & x_2 & x_3
\end{pmatrix} < 2
\]
Applying Lemma~\ref{lem:blowups} with $V = \cO_{\PP^3}^{\oplus 3}$, $W
= \cO_{\PP^3}(1)^{\oplus 2}$, and the map $f\colon V \to W$ given by $\begin{pmatrix}
  x_0 & x_1 & x_2 \\ 
  x_1 & x_2 & x_3
\end{pmatrix}$, we see that $X$ is cut out of $\PP(V)$ by a section of
$\pi^\star W \otimes \cO_{\PP(E)}(1)$.  In other words, $X$ is a
complete intersection in $\PP^3 \times \PP^2$ of type $ (L+M) \cap
(L+M)$.

\subsection*{The quantum period:} The toric variety $F$ has weight
data:
\[ 
\begin{array}{rrrrrrrl} 
1 & 1 & 1 & 1 & 0 & 0 & 0 & \hspace{1.5ex} L\\ 
0 & 0 & 0 & 0 & 1 & 1 & 1 &\hspace{1.5ex} M \\
\end{array}
\]
and $\Amp F = \langle L,M\rangle$.  We have that:
\begin{itemize}
\item $F$ is a Fano variety;
\item $X$ is the intersection of two ample divisors on $F$;
\item $-(K_F+\Lambda)\sim 2L+M$ is ample.
\end{itemize}

Corollary~\ref{cor:QL} yields:
\[
G_X(t) = e^{-t} \sum_{l=0}^\infty \sum_{m=0}^\infty
t^{2l+m}
\frac{(l+m)!(l+m)!}
{(l!)^4 (m!)^3}
\]
and regularizing gives:
\[
\hG_X(t) = 1+2 t^2+18 t^3+30 t^4+240 t^5+1730 t^6+5880 t^7+41230 t^8+262080 t^9+ \cdots
\]

\subsection*{Minkowski period sequence:} \href{http://www.grdb.co.uk/search/period3?id=19&printlevel=2}{19}


\addtocounter{CustomSectionCounter}{1}

\section{The Fano Manifold $\MM{2}{28}$}
\label{anchor:2--28}

\subsection*{Mori--Mukai name:} 2--28

\subsection*{Mori--Mukai construction:} The blow-up of $\PP^3$ with
centre a plane cubic.

\subsection*{Our construction:} A member $X$ of $|L+M|$ in the toric
variety $F$ with weight data:
\[ 
\begin{array}{rrrrrrl} 
  \multicolumn{1}{c}{s_0} & 
  \multicolumn{1}{c}{s_1} & 
  \multicolumn{1}{c}{s_2} & 
  \multicolumn{1}{c}{s_3} & 
  \multicolumn{1}{c}{x} & 
  \multicolumn{1}{c}{y} & \\ 
  \cmidrule{1-6}
  1 & 1 & 1 & 1 & -2 & 0 & \hspace{1.5ex} L\\ 
  0 & 0 & 0 & 0  & 1 & 1 & \hspace{1.5ex} M \\
\end{array}
\]
and $\Amp F = \langle L, M\rangle$.  We have:
\begin{itemize}
\item $-K_F=2L+2M$ is ample, that is $F$ is a Fano variety;
\item $X\sim L+M$ is ample;
\item $-(K_F+X)\sim L+M$ is ample.
\end{itemize}

\subsection*{The two constructions coincide:}
Suppose that the centre of the blow-up is defined by the simultaneous
vanishing of $A$ and $B$, where $A$ is a member of $\cO_{\PP^3}(3)$
and $B$ is a member of $\cO_{\PP^3}(1)$.  Apply
Lemma~\ref{lem:blowups} with $V = \cO_{\PP^3}(-2) \oplus \cO_{\PP^3}$,
$W=\cO_{\PP^3}(1)$, and the map $f\colon V \to W$ given by $
\begin{pmatrix}
  A & B
\end{pmatrix}
$.

\subsection*{The quantum period:}
Corollary~\ref{cor:QL} yields:
\[
G_X(t) = e^{-t} \sum_{l =0}^\infty \sum_{m = 2l}^\infty 
t^{l+m}
\frac{(l+m)!}
{(l!)^4 (m-2l)! m!}
\]
and regularizing gives:
\[
\hG_X(t) = 1+18 t^3+24 t^4+1350 t^6+3780 t^7+2520 t^8+141120 t^9+ \cdots
\]

\subsection*{Minkowski period sequence:} \href{http://www.grdb.co.uk/search/period3?id=5&printlevel=2}{5}


\addtocounter{CustomSectionCounter}{1}
\section{The Fano Manifold $\MM{2}{29}$}
\label{sec:2-29}
\label{anchor:2--29}

\subsection*{Mori--Mukai name:} 2--29

\subsection*{Mori--Mukai construction:} The blow-up of a quadric
\mbox{3-fold} $Q\subset \PP^3$ with centre a conic on it.

\subsection*{Our construction:} A member $X$ of $|2M|$ in the toric
variety $F$ with weight data:
\[ 
\begin{array}{rrrrrrl} 
  \multicolumn{1}{c}{s_0} & 
  \multicolumn{1}{c}{s_1} & 
  \multicolumn{1}{c}{x} & 
  \multicolumn{1}{c}{x_2} & 
  \multicolumn{1}{c}{x_3} & 
  \multicolumn{1}{c}{x_4} & \\ 
  \cmidrule{1-6}
  1 & 1 & -1 & 0 & 0 & 0 & \hspace{1.5ex} L\\ 
  0 & 0 & 1 & 1  & 1 & 1 & \hspace{1.5ex} M \\
\end{array}
\]
and $\Amp F = \langle L, M\rangle$.  We have:
\begin{itemize}
\item $-K_F=L+4M$ is ample, that is $F$ is a Fano variety;
\item $X\sim 2M$ is nef and big;
\item $-(K_F+X)\sim L+2M$ is ample.
\end{itemize}
 
\subsection*{The two constructions coincide:} The morphism $F \to
\PP^4$ that sends (contravariantly) the homogeneous co-ordinate
functions $[x_0,\dots,x_4]$ to $[xs_0, xs_1, x_2, x_3, x_4]$
blows up the plane $(x_0=x_1=0)$ in $\PP^4$. Thus a generic member
of $|2M|$ on $F$ is the blow-up of a quadric \mbox{3-fold} with centre
a conic on it.

\subsection*{The quantum period:}
Corollary~\ref{cor:QL} yields:
\[
G_X(t) = \sum_{l=0}^\infty \sum_{m=l}^\infty
t^{l+2m}
\frac{(2m)!}
{(l!)^2 (m-l)! (m!)^3}
\]
and regularizing gives:
\[
\hG_X(t) = 1+4 t^2+12 t^3+36 t^4+360 t^5+940 t^6+8400 t^7+38500 t^8+210000 t^9+ \cdots
\]

\subsection*{Minkowski period sequence:} \href{http://www.grdb.co.uk/search/period3?id=35&printlevel=2}{35}


\addtocounter{CustomSectionCounter}{1}

\section{The Fano Manifold $\MM{2}{30}$}
\label{anchor:2--30}

\subsection*{Mori--Mukai name:} 2--30

\subsection*{Mori--Mukai construction:} The blow-up of $\PP^3$ with centre a conic.

\subsection*{Our construction:} A member $X$ of $|L+M|$ in the toric
variety $F$ with weight data:
\[ 
\begin{array}{rrrrrrl} 
  \multicolumn{1}{c}{s_0} & 
  \multicolumn{1}{c}{s_1} & 
  \multicolumn{1}{c}{s_2} & 
  \multicolumn{1}{c}{s_3} & 
  \multicolumn{1}{c}{x} & 
  \multicolumn{1}{c}{x_4} & \\ 
  \cmidrule{1-6}
  1 & 1 & 1 & 1 & -1 & 0 & \hspace{1.5ex} L\\ 
  0 & 0 & 0 & 0  & 1 & 1 & \hspace{1.5ex} M \\
\end{array}
\]
and $\Amp F = \langle L, M\rangle$.  We have:
\begin{itemize}
\item $-K_F=3L+2M$ is ample, that is $F$ is a Fano variety;
\item $X\sim L+M$ is ample;
\item $-(K_F+X)\sim 2L+M$ is ample.
\end{itemize}

\subsection*{The two constructions coincide:} Suppose that the centre
of the blow-up is defined by the simultaneous vanishing of $A$ and
$B$, where $A$ is a member of $\cO_{\PP^3}(2)$ and $B$ is a member of
$\cO_{\PP^3}(1)$.  Apply Lemma~\ref{lem:blowups} with $V =
\cO_{\PP^3}(-1) \oplus \cO_{\PP^3}$, $W=\cO_{\PP^3}(1)$, and the map
$f\colon V \to W$ given by $
\begin{pmatrix}
  A & B
\end{pmatrix}
$.

\subsection*{The quantum period:} Corollary~\ref{cor:QL} yields:
\[
G_X(t) = e^{-t} \sum_{l=0}^\infty \sum_{m=l}^\infty
t^{2l+m}
\frac{(l+m)!}
{(l!)^4 (m-l)! m!}
\]
and regularizing gives:
\[
\hG_X(t) = 1+12 t^3+24 t^4+540 t^6+2520 t^7+2520 t^8+33600 t^9+ \cdots
\]

\subsection*{Minkowski period sequence:} \href{http://www.grdb.co.uk/search/period3?id=4&printlevel=2}{4}


\addtocounter{CustomSectionCounter}{1}

\section{The Fano Manifold $\MM{2}{31}$}
\label{anchor:2--31}

\subsection*{Mori--Mukai name:} 2--31

\subsection*{Mori--Mukai construction:} The blow-up of a quadric
\mbox{3-fold} $Q\subset \PP^4$ with centre a line on it.

\subsection*{Our construction:}
A member $X$ of $|L+M|$ in the toric variety $F$ with weight data:
\[ 
\begin{array}{rrrrrrl} 
  \multicolumn{1}{c}{s_0} & 
  \multicolumn{1}{c}{s_1} & 
  \multicolumn{1}{c}{s_2} & 
  \multicolumn{1}{c}{x} & 
  \multicolumn{1}{c}{x_3} & 
  \multicolumn{1}{c}{x_4} & \\ 
  \cmidrule{1-6}
  1 & 1 & 1 & -1 & 0 & 0 & \hspace{1.5ex} L\\ 
  0 & 0 & 0 & 1  & 1 & 1 & \hspace{1.5ex} M \\
\end{array}
\]
and $\Amp F = \langle L, M\rangle$.  We have:
\begin{itemize}
\item $-K_F=2L+3M$ is ample, that is $F$ is a Fano variety;
\item $X\sim L+M$ is ample;
\item $-(K_F+X)\sim L+2M$ is ample.
\end{itemize}

\subsection*{The two constructions coincide:} 
The morphism $F \to \PP^4$ that sends (contravariantly) the
homogeneous co-ordinate functions $[x_0,\dots,x_4]$ to $[xs_0, xs_1, x
s_2, x_3, x_4]$ blows up the line $(x_0=x_1=x_2=0)$ in $\PP^4$, and
$X$ is the proper transform of a quadric containing this line.

\subsection*{The quantum period:} Corollary~\ref{cor:QL} yields:
\[
G_X(t) = \sum_{l=0}^\infty \sum_{m=l}^\infty 
t^{l+2m}
\frac{(l+m)!}
{(l!)^3 (m-l)! (m!)^2}
\]
and regularizing gives:
\[
\hG_X(t) = 1+2 t^2+12 t^3+6 t^4+180 t^5+560 t^6+1680 t^7+16870 t^8+46200 t^9+ \cdots
\]

\subsection*{Minkowski period sequence:} \href{http://www.grdb.co.uk/search/period3?id=15&printlevel=2}{15}


\addtocounter{CustomSectionCounter}{1}

\section{The Fano Manifold $\MM{2}{32}$ (also known as $W$)}
\label{sec:W}
\label{anchor:2--32}

\subsection*{Mori--Mukai name:} 2--32

\subsection*{Mori--Mukai construction:} The divisor $W$ of bidegree $
(1,1)$ on $ \PP^2 \times \PP^2$.

\subsection*{Our construction:} A member $X$ of $|L+M|$ on the toric
variety $F=\PP^2\times\PP^2$.

\subsection*{The two constructions coincide:} Obvious.

\subsection*{The quantum period:}
The toric variety $F$ has weight data:
\[ 
\begin{array}{rrrrrrrl} 
1 & 1 & 1 & 0 & 0 & 0 & \hspace{1.5ex} L\\ 
0 & 0 & 0 & 1 & 1 & 1 & \hspace{1.5ex} M \\
\end{array}
\]
and $\Amp F = \langle L,M\rangle$.  We have that:
\begin{itemize}
\item $F$ is a Fano variety;
\item $X\sim L+M$ is ample;
\item $-(K_F+X)\sim 2L+2M$ is ample.
\end{itemize}

Corollary~\ref{cor:QL} yields:
\[
G_X(t) = \sum_{l=0}^\infty \sum_{m=0}^\infty
t^{2l+2m}
\frac{(l+m)!}
{(l!)^3 (m!)^3}
\]
and regularizing gives:
\[
\hG_X(t) = 1+4 t^2+60 t^4+1120 t^6+24220 t^8+567504 t^{10}+ \cdots
\]

\subsection*{Minkowski period sequence:} \href{http://www.grdb.co.uk/search/period3?id=24&printlevel=2}{24}


\addtocounter{CustomSectionCounter}{1}
\section{The Fano Manifold $\MM{2}{33}$}
\label{sec:2-33}
\label{anchor:2--33}

\subsection*{Mori--Mukai name:} 2--33

\subsection*{Mori--Mukai construction:} The blow-up of $\PP^3$ with
centre a line.

\subsection*{Our construction:} The toric Fano variety $X$ with weight
data:
\[ 
\begin{array}{rrrrrl} 
  \multicolumn{1}{c}{s_0} & 
  \multicolumn{1}{c}{s_1} & 
  \multicolumn{1}{c}{x} & 
  \multicolumn{1}{c}{x_2} & 
  \multicolumn{1}{c}{x_3} & \\ 
  \cmidrule{1-5}
  1 & 1 & -1 & 0 & 0 & \hspace{1.5ex} L\\ 
  0 & 0 &  1 & 1 & 1 & \hspace{1.5ex} M \\
\end{array}
\]
and $\Amp X = \langle L, M\rangle$.  

\subsection*{The two constructions coincide:} The blow-up $X
\to \PP^3$ sends (contravariantly) the homogeneous co-ordinate
functions $[x_0,x_1,x_2,x_3]$ to $[xs_0, xs_1, x_2, x_3]$.

\subsection*{The quantum period:}
Corollary~\ref{cor:toric_mirror} yields:
\[
G_X(t) = \sum_{l=0}^\infty \sum_{m=l}^\infty
\frac{t^{l+3m}}
{(l!)^2 (m-l)! (m!)^2}
\]
and regularizing gives:
\[
\hG_X(t) = 1+6 t^3+24 t^4+90 t^6+1260 t^7+2520 t^8+1680 t^9+ \cdots
\]

\subsection*{Minkowski period sequence:} \href{http://www.grdb.co.uk/search/period3?id=2&printlevel=2}{2}


\addtocounter{CustomSectionCounter}{1}

\section{The Fano Manifold $\MM{2}{34}$}
\label{anchor:2--34}

\subsection*{Mori--Mukai name:} 2--34

\subsection*{Mori--Mukai construction:} $\PP^1 \times \PP^2$

\subsection*{Our construction:} $\PP^1\times \PP^2$

\subsection*{The two constructions coincide:} Obvious.

\subsection*{The quantum period:} $X = \PP^1 \times \PP^2$ is the
toric Fano variety with weight data:
\[ 
\begin{array}{rrrrrl} 
  1 & 1 & 0 & 0 & 0 & \hspace{1.5ex} L\\ 
  0 & 0 & 1 & 1 & 1 & \hspace{1.5ex} M \\
\end{array}
\]
and $\Amp X = \langle L, M\rangle$.  Corollary~\ref{cor:toric_mirror}
yields:
\[
G_X(t) = \sum_{l=0}^\infty \sum_{m=0}^\infty 
\frac{t^{2l+3m}}
{(l!)^2 (m!)^3}
\]
and regularizing gives:
\[
\hG_X(t) = 1+2 t^2+6 t^3+6 t^4+120 t^5+110 t^6+1260 t^7+5110 t^8+11760 t^9+ \cdots
\]

\subsection*{Minkowski period sequence:} \href{http://www.grdb.co.uk/search/period3?id=10&printlevel=2}{10}


\addtocounter{CustomSectionCounter}{1}
\section{The Fano Manifold $\MM{2}{35}$ (also known as $B_7$)}
\label{sec:B7}
\label{anchor:2--35}

\subsection*{Mori--Mukai name:} 2--35

\subsection*{Mori--Mukai construction:} $ B_7$, the blow-up of $
\PP^3$ at a point; equivalently, the $\PP^1$-bundle $\PP(\cO+\cO(1))$
over $\PP^2$.

\subsection*{Our construction:} The toric Fano variety $X$ with weight
data:
\[ 
\begin{array}{rrrrrl} 
  \multicolumn{1}{c}{s_0} & 
  \multicolumn{1}{c}{s_1} & 
  \multicolumn{1}{c}{s_2} & 
  \multicolumn{1}{c}{x} & 
  \multicolumn{1}{c}{x_3} & \\ 
  \cmidrule{1-5}
  1 & 1 & 1 & -1 & 0 & \hspace{1.5ex} L\\ 
  0 & 0 &  0 & 1 & 1 & \hspace{1.5ex} M \\
\end{array}
\]
and $\Amp X = \langle L, M\rangle$. 

\subsection*{The two constructions coincide:} The blow-up $X
\to \PP^3$ sends (contravariantly) the homogeneous co-ordinate
functions $[x_0,x_1,x_2,x_3]$ to $[xs_0, xs_1, xs_2, x_3]$.

\subsection*{The quantum period:}
Corollary~\ref{cor:toric_mirror} yields:
\[
G_X(t) = \sum_{l=0}^\infty \sum_{m=l}^\infty
\frac{t^{2l+2m}}
{(l!)^3 (m-l)! m!}
\]
and regularizing gives:
\[
\hG_X(t) = 1+2 t^2+30 t^4+380 t^6+5950 t^8+101052 t^{10}+ \cdots
\]

\subsection*{Minkowski period sequence:} \href{http://www.grdb.co.uk/search/period3?id=7&printlevel=2}{7}


\addtocounter{CustomSectionCounter}{1}

\section{The Fano Manifold $\MM{2}{36}$}
\label{anchor:2--36}

\subsection*{Mori--Mukai name:} 2--36

\subsection*{Mori--Mukai construction:} The blow-up of the Veronese
cone $W_4 \subset \PP^6$ with centre the vertex; equivalently, the
$\PP^1$-bundle $ \PP(\cO \oplus \cO(2))$ over $ \PP^2$.

\subsection*{Our construction:} 
The toric Fano variety $X$ with weight data:
\[ 
\begin{array}{rrrrrl} 
  \multicolumn{1}{c}{s_0} & 
  \multicolumn{1}{c}{s_1} & 
  \multicolumn{1}{c}{s_2} & 
  \multicolumn{1}{c}{x} & 
  \multicolumn{1}{c}{y} & \\ 
  \cmidrule{1-5}
  1 & 1 & 1 & -2 & 0 & \hspace{1.5ex} L\\ 
  0 & 0 &  0 & 1 & 1 & \hspace{1.5ex} M \\
\end{array}
\]
and $\Amp F = \langle L, M\rangle$.

\subsection*{The two constructions coincide:} Obvious.

\subsection*{The quantum period:}
Corollary~\ref{cor:toric_mirror} yields:
\[
G_X(t) = \sum_{l=0}^\infty \sum_{m=2l}^\infty
\frac{t^{l+2m}}
{(l!)^3 (m-2l)! m!}
\]
and regularizing gives:
\[
\hG_X(t) = 1+2 t^2+6 t^4+60 t^5+20 t^6+840 t^7+70 t^8+7560 t^9+ \cdots
\]

\subsection*{Minkowski period sequence:} \href{http://www.grdb.co.uk/search/period3?id=6&printlevel=2}{6}


\addtocounter{CustomSectionCounter}{1}

\section{The Fano Manifold $\MM{3}{1}$}
\label{anchor:3--1}

\subsection*{Mori--Mukai name:} 3--1
\label{sec:restriction_of_nef_not_restriction_of_ample}

\subsection*{Mori--Mukai construction:} A double cover of $\PP^1\times
\PP^1 \times \PP^1$ branched along a divisor of tridegree $(2,2,2)$.

\subsection*{Our construction:} A member $X$ of $|2L+2M+2N|$ in the
toric variety $F$ with weight data:
\[
\begin{array}{rrrrrrrl} 
  \multicolumn{1}{c}{x_0} & 
  \multicolumn{1}{c}{x_1} & 
  \multicolumn{1}{c}{y_0} & 
  \multicolumn{1}{c}{y_1} & 
  \multicolumn{1}{c}{z_0} & 
  \multicolumn{1}{c}{z_1} &  
  \multicolumn{1}{c}{w} &  \\
  \cmidrule{1-7}
  1 & 1 & 0 & 0 & 0 & 0 & 1 & \hspace{1.5ex} L\\ 
  0 & 0 & 1 & 1 & 0 & 0 & 1 & \hspace{1.5ex} M \\
  0 & 0 & 0 & 0 & 1 & 1 & 1 & \hspace{1.5ex} N \\
\end{array}
\]
and $\Amp F =\langle L,M,L+M+N\rangle$.  The secondary fan for $F$ has
three maximal cones; the corresponding three toric varieties are
isomorphic. It is easy to see that $\Amp X = \langle L,M,N \rangle $.
We have:
\begin{itemize}
\item $-K_F=3(L+M+N)$ is nef and big but not ample;
\item $X\sim 2(L+M+N)$ is nef and big but not ample;
\item $-(K_F+X)\sim L+M+N$ is nef and big but not ample. 
\end{itemize}

\subsection*{The two constructions coincide:} Consider the equation
$w^2 = f(x_0,x_1,y_0,y_1,z_0,z_1)$ where $f$ is a generic polynomial
of degree $2$ in $x_0$ and $x_1$, degree $2$ in $y_0$ and $y_1$, and
degree $2$ in $z_0$, $z_1$.

\subsection*{The quantum period:}  
Let $p_1$, $p_2, p_3 \in H^\bullet(F;\ZZ)$ denote the first Chern
classes of $L$, $M$, and $L \otimes M \otimes N$ respectively; these
classes form a basis for $H^2(F;\ZZ)$.  Write $\tau \in H^2(F;\QQ)$ as
$\tau = \tau_1 p_1 + \tau_2 p_2+ \tau_3 p_3$ and identify the group
ring $\QQ[H_2(F;\ZZ)]$ with the polynomial ring $\QQ[Q_1,Q_2,Q_3]$ via
the $\QQ$-linear map that sends the element $Q^\beta \in
\QQ[H_2(F;\ZZ)]$ to $Q_1^{\langle \beta,p_1\rangle} Q_2^{\langle
  \beta,p_2\rangle} Q_3^{\langle \beta,p_3\rangle}$.  We have:
\begin{align*}
  I_F(\tau)  & = e^{\tau/z} 
  \sum_{l, m, n\geq 0} 
  \frac{
    Q_1^l Q_2^m Q_3^n e^{l \tau_1} e^{m \tau_2} e^{n \tau_3}
  }
  {
    \prod_{k=1}^l (p_1 + k z)^2
    \prod_{k=1}^m (p_2 + k z)^2
    \prod_{k=1}^n (p_3 + k z)
  }
  \frac{\prod_{k = -\infty}^0 (p_3-p_1-p_2 + k z)^2}
  {\prod_{k = -\infty}^{n-l-m} (p_3-p_1-p_2 + k z)^2}
  \\
  &=
  1 + \tau z^{-1} + O(z^{-2})
\end{align*}
Theorem~\ref{thm:toric_mirror} gives:
\[
J_F(\tau) = I_F(\tau)
\]
and hence:
\begin{multline*}
  I_{\be,E}(\tau) = 
  e^{\tau/z}  
  \sum_{l, m, n\geq 0} 
  \frac{
    Q_1^l Q_2^m Q_3^n e^{l \tau_1} e^{m \tau_2} e^{n \tau_3}
    \prod_{k=1}^{2n} (\lambda + 2p_3 + k z)
  }
  {
    \prod_{k=1}^l (p_1 + k z)^2
    \prod_{k=1}^m (p_2 + k z)^2
    \prod_{k=1}^n (p_3 + k z)
  } \\
  \times
  \frac{\prod_{k = -\infty}^0 (p_3-p_1-p_2 + k z)^2}
  {\prod_{k = -\infty}^{n-l-m} (p_3-p_1-p_2 + k z)^2}
\end{multline*}
Since:
\[
  I_{\be,E}(\tau) =  1 + \big(\tau + 2Q_3 + 2 Q_1 Q_3 + 2Q_2 Q_3\big) z^{-1} + O(z^{-2}) 
\]
applying Theorem~\ref{thm:toric_ci_ql} yields:
\[
J_{\be,E}\big(\tau + 2Q_3 + 2 Q_1 Q_3 + 2Q_2 Q_3\big) = I_{\be,E}(\tau)
\]
The String Equation now implies that:
\[
J_{\be,E}\big(\tau\big) = 
e^{-(2Q_3 + 2 Q_1 Q_3 + 2Q_2 Q_3)/z} I_{\be,E}(\tau)
\]
and taking the non-equivariant limit $\lambda \to 0$ gives:
\begin{multline*}
  J_{F,X}(\tau) = 
  e^{-(2Q_3 + 2 Q_1 Q_3 + 2Q_2 Q_3)/z}
  e^{\tau/z}  
  \sum_{l, m, n\geq 0} 
  \frac{
    Q_1^l Q_2^m Q_3^n e^{l \tau_1} e^{m \tau_2} e^{n \tau_3}
    \prod_{k=1}^{2n} (2p_3 + k z)
  }
  {
    \prod_{k=1}^l (p_1 + k z)^2
    \prod_{k=1}^m (p_2 + k z)^2
    \prod_{k=1}^n (p_3 + k z)
  } \\
  \times
  \frac{\prod_{k = -\infty}^0 (p_3-p_1-p_2 + k z)^2}
  {\prod_{k = -\infty}^{n-l-m} (p_3-p_1-p_2 + k z)^2}
\end{multline*}
We now proceed exactly as in the proof of Corollary~\ref{cor:QL},
obtaining:
\[
G_X(t) = 
e^{-6 t} 
\sum_{l=0}^\infty \sum_{m=0}^\infty \sum_{n=l+m}^\infty
t^n
\frac
{(2n)!}
{(l!)^2(m!)^2n!\big((n-l-m)!\big)^2}
\]
Regularizing gives:
\begin{multline*}
  \hG_X(t) = 1+54 t^2+672 t^3+15642 t^4+336960 t^5+7919460
  t^6+191177280 t^7\\
  +4751272890 t^8+120527514240 t^9 + \cdots
\end{multline*}

\subsection*{Minkowski period sequence:} \href{http://www.grdb.co.uk/search/period3?id=154&printlevel=2}{154}


\addtocounter{CustomSectionCounter}{1}

\section{The Fano Manifold $\MM{3}{2}$}
\label{sec:3-2}
\label{anchor:3--2}

\subsection*{Mori--Mukai name:} 3--2

\subsection*{Mori--Mukai construction:} A member of $|\L^{\otimes 2}
\otimes_{\cO_{\PP^1\times \PP^1}}\cO_{\PP^1 \times \PP^1}(2,3)|$ on
the $\PP^2$-bundle
\[
\PP\bigl(\cO_{\PP^1 \times \PP^1}\oplus\cO_{\PP^1 \times
  \PP^1}(-1,-1)^{\oplus 2}\bigr)
\] 
over $\PP^1\times \PP^1$ such that $X\cap Y$ is irreducible, where $\L$
is the tautological line bundle (that is, the fiberwise $\cO(1)$ on the $\PP^2$-bundle) and $Y$ is a member of $|\L|$.

\subsection*{Our construction:} A member $X$ of $|M+2N|$ in the toric
variety $F$ with weight data:
\[
\begin{array}{rrrrrrrl} 
  \multicolumn{1}{c}{x_0} & 
  \multicolumn{1}{c}{x_1} & 
  \multicolumn{1}{c}{y_0} & 
  \multicolumn{1}{c}{y_1} & 
  \multicolumn{1}{c}{t} & 
  \multicolumn{1}{c}{t_0} &  
  \multicolumn{1}{c}{t_1} &  \\
  \cmidrule{1-7}
  1 & 1 & 0 & 0 & -1 & 0 & 0 & \hspace{1.5ex} L\\ 
  0 & 0 & 1 & 1 & -1 & 0 & 0 & \hspace{1.5ex} M \\
  0 & 0 & 0 & 0 & 1 & 1 & 1 & \hspace{1.5ex} N \\
\end{array}
\]
and $\Amp F=\langle L,M,N\rangle$.

We have:
\begin{itemize}
\item $-K_F=L+M+3N$ is ample, that is $F$ is a Fano variety;
\item $X\sim M+2N$ is nef and big;
\item ${-(K_F+X)}\sim L+N$ is nef and big but not ample on $F$ (it is
  ample when restricted to $X$). 
\end{itemize}

\subsection*{The two constructions coincide:} Mori--Mukai use
different weight conventions to ours, so their construction exhibits
$X$ as a member of $|2L'+3M'+2N'|$ in the toric variety with weight
data:
\[
\begin{array}{rrrrrrrl} 
  1 & 1 & 0 & 0 & 0 & 1 & 1 & \hspace{1.5ex} L'\\ 
  0 & 0 & 1 & 1 & 0 & 1 & 1 & \hspace{1.5ex} M' \\
  0 & 0 & 0 & 0 & 1 & 1 & 1 & \hspace{1.5ex} N' \\
\end{array}
\]
and $\Amp F=\langle L',M',L'+M'+N'\rangle$.  Changing basis yields our construction.

\subsection*{Remarks on our construction:}
Note that the secondary fan for $F$ has three maximal cones as in
Fig.~\ref{fig:3-2}.

\begin{figure}[h!]
  \centering
  \resizebox{6cm}{!}{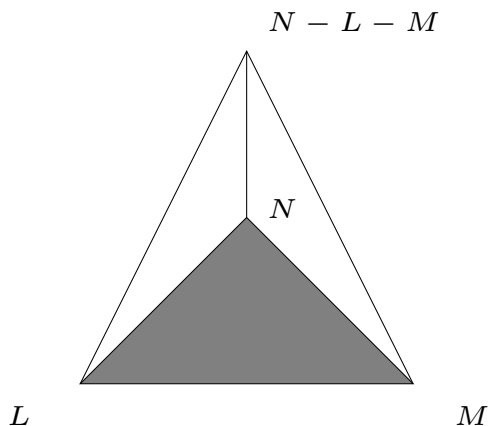}
  \caption{The secondary fan for $F$ in 3--2}
  \label{fig:3-2}
\end{figure}

The following table gives more detail about the irrelevant ideal,
unstable locus, and quotient variety corresponding to each of the
maximal cones of the secondary fan.

\begin{table}[h!]
  \centering
  \begin{tabular}{lllc} \toprule
    \multicolumn{1}{c}{chamber}& 
    \multicolumn{1}{c}{irrelevant ideal}&
    \multicolumn{1}{c}{unstable locus}&
    \multicolumn{1}{c}{$\CC^7 \GIT_\omega T $}\\
    \midrule
    $\langle L,M,N\rangle$&$(x_iy_jt_k,x_iy_jt)$&$(x_0=x_1=0)\cup(y_0=y_1=0)\cup(t=t_0=t_1=0)$&$F$\\ 
    $\langle L,N,N-L-M\rangle$&$(x_it_kt,x_iy_jt)$&$(t=0)\cup(x_0=x_1=0)\cup(y_0=y_1=t_0=t_1=0)$&$G$\\
    $\langle
    M,N,N-L-M\rangle$&$(y_jt_kt,x_iy_jt)$&$(t=0)\cup(y_0=y_1=0)\cup(x_0=x_1=t_0=t_1=0)$&$G^\prime$ \\
    \bottomrule
  \end{tabular}
 \label{tab:3-2}
\end{table}

The shape of the unstable locus shows that the second and third
maximal cones are ``hollow'', that is, taking the GIT quotient with
respect to these stability conditions leads to toric
varieties of Picard rank~2. We discuss briefly the variety $G$, which is the most
relevant for understanding the geometry of $X$. Since $t\not=0$, we
can use the $M$-torus to reduce to $t=1$ and eliminate $t$. We are
left with the toric variety $G$ with weight data:
\[
\begin{array}{rrrrrrl} 
  \multicolumn{1}{c}{x_0} & 
  \multicolumn{1}{c}{x_1} & 
  \multicolumn{1}{c}{u_0} & 
  \multicolumn{1}{c}{u_1} & 
  \multicolumn{1}{c}{t_0} &  
  \multicolumn{1}{c}{t_1} &  \\
  \cmidrule{1-6}
  1 & 1 & -1 & -1 & 0 & 0 & \hspace{1.5ex} L'\\ 
  0 & 0 & 1 & 1 & 1 & 1 & \hspace{1.5ex} N' \\
\end{array}
\] 
and $\Amp G = \langle L' , N'\rangle$. The morphism $f\colon F \to G$
is given (contravariantly) by:
\[ 
[x_0,x_1,u_0,u_1,t_0,t_1] \mapsto [x_0,x_1,t y_0, t y_1, t_0, t_1]
\] 
and we have $L=f^\star L'$, $N=f^\star N'$.

The divisor that Mori--Mukai denote by $Y$ is, in our notation,
$(t=0)\cong \PP^1_{x_0,x_1}\times \PP^1_{y_0,y_1}\times
\PP^1_{t_0,t_1}$.  The complete linear system $|{-(K_F+X)}|$ defines
the morphism $f\colon F \to G$, which (a) contracts the divisor $Y$ to
$\PP^1_{x_0,x_1}\times \PP^1_{t_0,t_1}$ and (b) is an isomorphism of
$X$ to its image.  Under $f\colon F \to G$, $X$ maps isomorphically to
a member $X'$ of $|{-L'}+3N'|$ on $G$.  This makes it clear that $X$
is Fano, because ${-(K_G+X')} = L'+N'$ is ample on $G$; however
because $X'$ is not nef on $G$ this construction, economical though it
is, is useless for calculating the quantum cohomology of $X$, as the
convexity assumption on the bundle in Quantum Lefschetz is not
satisfied.

\subsection*{The quantum period:}  This is
Example~\ref{ex:ql_with_mirror_map}.  We have:
\[
\hG_X(t) = 1 + 58t^2 + 600t^3 + 13182t^4 + 247440t^5 +
5212300t^6 + 
111835920t^7 + 2480747710t^8 + 56184565920t^9 + \cdots
\]

\subsection*{Minkowski period sequence:} \href{http://www.grdb.co.uk/search/period3?id=157&printlevel=2}{157}


\addtocounter{CustomSectionCounter}{1}
\section{The Fano Manifold $\MM{3}{3}$}
\label{sec:3-3}
\label{anchor:3--3}

\subsection*{Mori--Mukai name:} 3--3

\subsection*{Mori--Mukai construction:} A divisor of tridegree
$(1,1,2)$ on $\PP^1\times \PP^1\times \PP^2$.

\subsection*{Our construction:} A member $X$ of $|L+M+2N|$ on the
toric variety $F$ with weight data:
\[ 
\begin{array}{rrrrrrrrl} 
1 & 1 & 0 & 0 & 0 & 0 & 0 & \hspace{1.5ex} L\\ 
0 & 0 & 1 & 1 & 0 & 0 & 0 & \hspace{1.5ex} M \\
0 & 0 & 0 & 0 & 1 & 1 & 1 & \hspace{1.5ex} N \\
\end{array}
\]
and $\Amp F = \langle L,M,N\rangle$.

\subsection*{The two constructions coincide:} Obvious.
\subsection*{The quantum period:}
We have that:
\begin{itemize}
\item $F$ is a Fano variety;
\item $X\sim L+M+2N$ is ample;
\item ${-(K_F+X)}\sim L+M+N$ is ample. 
\end{itemize}
Corollary~\ref{cor:QL} yields:
\[
G_X(t) = e^{-4t} \sum_{l=0}^\infty \sum_{m=0}^\infty \sum_{n=0}^\infty
t^{l+m+n}
\frac{(l+m+2n)!}
{(l!)^2 (m!)^2 (n!)^3}
\]
and regularizing gives:
\begin{multline*}
  \hG_X(t) = 1+20 t^2+132 t^3+1812 t^4+21720 t^5+289100 t^6+3927840
  t^7\\
  +54999700 t^8+785606640 t^9+ \cdots
\end{multline*}

\subsection*{Minkowski period sequence:} \href{http://www.grdb.co.uk/search/period3?id=135&printlevel=2}{135}


\addtocounter{CustomSectionCounter}{1}
\section{The Fano Manifold $\MM{3}{4}$}
\label{sec:3-4}
\label{anchor:3--4}

\subsection*{Mori--Mukai name:} 3--4

\subsection*{Mori--Mukai construction:} The blow-up of the variety $Y$
constructed in \S\ref{sec:2-18} (i.e.~number 18 on the Mori--Mukai
list of smooth Fano 3-folds of rank~2) with centre a smooth fibre of
the composition:
\[
\xymatrix{
  Y \ar[rrr]^-{\text{double cover}} &&& \PP^2 \times \PP^1 \ar[rr]^-{\text{projection}} && \PP^2
}
\]

\subsection*{Our construction:} A member $X$ of $|2N|$ on the toric
variety $F$ with weight data:
\[
\begin{array}{rrrrrrrrl} 
  \multicolumn{1}{c}{t_0} & 
  \multicolumn{1}{c}{t_1} & 
  \multicolumn{1}{c}{x} & 
  \multicolumn{1}{c}{x_2} & 
  \multicolumn{1}{c}{y_0} &  
  \multicolumn{1}{c}{y_1} &  
  \multicolumn{1}{c}{z} &  \\
  \cmidrule{1-7}
  1 & 1 & -1 & 0 & 0 & 0 & 0 & \hspace{1.5ex} L\\ 
  0 & 0 & 1 & 1 & -1 & -1 & 0 & \hspace{1.5ex} M \\
  0 & 0 & 0 & 0 & 1 & 1 & 1 & \hspace{1.5ex} N \\
\end{array}
\]
and $\Amp F=\langle L, M, N\rangle$. The secondary fan has four
maximal cones as in Fig.~\ref{fig:3-4}. We have:
\begin{itemize}
\item $-K_F=L+3N$ is nef and big but not ample;
\item $X\sim 2 N$ is nef and big but not ample;
\item $-(K_F+X)\sim L+N$ is nef and big but not ample.
\end{itemize}

\begin{figure}[h!]
  \centering
  \resizebox{6cm}{!}{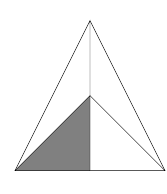}
  \caption{The secondary fan for $F$ in 3--4}
  \label{fig:3-4}
\end{figure}

\subsection*{The two constructions coincide:} Recall\footnote{The
  description here differs from the weight data in \S\ref{sec:2-18} by
  a change of lattice basis and by relabelling of co-ordinates.} from
\S\ref{sec:2-18} that $Y$ is a member of $|N|$ in the toric variety
$G$ with weight data:
\[
\begin{array}{rrrrrrrl} 
  \multicolumn{1}{c}{x_0} & 
  \multicolumn{1}{c}{x_1} & 
  \multicolumn{1}{c}{x_2} & 
  \multicolumn{1}{c}{y_0} & 
  \multicolumn{1}{c}{y_1} &  
  \multicolumn{1}{c}{z} &  \\
  \cmidrule{1-6}
  1 & 1 & 1 & -1 & -1 & 0 & \hspace{1.5ex} M\\ 
  0 & 0 & 0 & 1 & 1 & 1 & \hspace{1.5ex} N \\
\end{array}
\]
and $\Amp G= \langle M, N\rangle$.  The unstable locus is
$(x_0=x_1=x_2=0)\cup (y_0=y_1=z=0)$. The linear system
$|M|=|x_0,x_1,x_2|$ manifestly defines a morphism $G\to
\PP^2_{x_0,x_1,x_2}$ with fibre $\PP^2$.  If $F$ is the blow-up of $G$
along $(x_0=x_1=0)$ then $X$ is the proper transform of $Y$.  It is
clear that $F$ is a toric variety with the weight data given above,
and that the morphism $F\to G$ is given by $x_0=xt_0$, $x_1=xt_1$.

\subsection*{The quantum period:}
Let $p_1$, $p_2, p_3 \in H^\bullet(F;\ZZ)$ denote the first Chern
classes of $L$, $M$, and $N$ respectively; these classes form a basis
for $H^2(F;\ZZ)$.  Write $\tau \in H^2(F;\QQ)$ as $\tau = \tau_1 p_1 +
\tau_2 p_2+ \tau_3 p_3$ and identify the group ring $\QQ[H_2(F;\ZZ)]$
with the polynomial ring $\QQ[Q_1,Q_2,Q_3]$ via the $\QQ$-linear map
that sends the element $Q^\beta \in \QQ[H_2(F;\ZZ)]$ to $Q_1^{\langle
  \beta,p_1\rangle} Q_2^{\langle \beta,p_2\rangle} Q_3^{\langle
  \beta,p_3\rangle}$.  We have:
\begin{multline*}
  I_F(\tau)  = e^{\tau/z} 
  \sum_{l, m, n\geq 0} 
  \frac{
    Q_1^l Q_2^m Q_3^n e^{l \tau_1} e^{m \tau_2} e^{n \tau_3}
  }
  {
    \prod_{k=1}^l (p_1 + k z)^2
    \prod_{k=1}^m (p_2 + k z)
    \prod_{k=1}^n (p_3 + k z)
  }
  \frac{\prod_{k = -\infty}^0 (p_2-p_1 + k z)}
  {\prod_{k = -\infty}^{m-l} (p_2-p_1 + k z)}
  \\
  \frac{\prod_{k = -\infty}^0 (p_3-p_2 + k z)^2}
  {\prod_{k = -\infty}^{n-m} (p_3-p_2 + k z)^2}
\end{multline*}
Since:
\[
  I_F(\tau)  = 1 + \tau z^{-1} + O(z^{-2})
\]
Theorem~\ref{thm:toric_mirror} gives:
\[
J_F(\tau) = I_F(\tau)
\]
We now proceed exactly as in the case of 3--1
(\S\ref{sec:restriction_of_nef_not_restriction_of_ample}), obtaining:
\[
G_X(t) = 
e^{-4 t} \sum_{l=0}^\infty \sum_{n=0}^\infty \sum_{m=l}^n
t^{l+n}
\frac
{(2n)!}
{(l!)^2m!n!(m-l)!\big((n-m)!\big)^2}
\]
Regularizing gives:
\begin{multline*}
  \hG_X(t) = 1+24 t^2+156 t^3+2280 t^4+27960 t^5+387060 t^6+5450760
  t^7\\ +79246440 t^8+1175608560 t^9 + \cdots
\end{multline*}

\subsection*{Minkowski period sequence:} \href{http://www.grdb.co.uk/search/period3?id=142&printlevel=2}{142}


\addtocounter{CustomSectionCounter}{1}
\section{The Fano Manifold $\MM{3}{5}$}
\label{sec:3-5}
\label{anchor:3--5}

\subsection*{Mori--Mukai name:} 3--5

\subsection*{Mori--Mukai construction:} The blow-up of $\PP^1\times
\PP^2$ with centre a curve $C$ of bidegree $(5,2)$ such that the
composition $C\hookrightarrow \PP^1\times\PP^2 \to \PP^2$ with
projection to the second factor is an embedding.

\subsection*{Our construction:} A codimension-2 complete intersection
$X$ of type $(M+N)\cap (M+N)$ in the toric variety $F$ with weight
data:
\[ 
\begin{array}{rrrrrrrrl} 
  \multicolumn{1}{c}{t_0} & 
  \multicolumn{1}{c}{t_1} & 
  \multicolumn{1}{c}{y_0} & 
  \multicolumn{1}{c}{y_1} & 
  \multicolumn{1}{c}{y_2} &  
  \multicolumn{1}{c}{x} &  
  \multicolumn{1}{c}{x_0} &  
  \multicolumn{1}{c}{x_1} &  \\
  \cmidrule{1-8}
  1 & 1 & 0 & 0 & 0 & -1 & 0 & 0 & \hspace{1.5ex} L\\ 
  0 & 0 & 1 & 1 & 1 & -1 & 0 & 0 & \hspace{1.5ex} M\\ 
  0 & 0 & 0 & 0 & 0 & 1 & 1 & 1 & \hspace{1.5ex} N\\ 
\end{array}
\]
and $\Amp F=\langle L,M,N\rangle$.  The secondary fan for $F$ is the
same as that for the toric variety in 3--2 (\S\ref{sec:3-2}) and is
shown in Fig.~\ref{fig:3-2}.  We have:
\begin{itemize}
\item $-K_F=L+2M+3N$ is ample, that is $F$ is a Fano variety;
\item $X$ is complete intersection of two nef divisors on $F$;
\item $-(K_F+\Lambda)=L+N$ is nef  and big but not ample on $F$.
\end{itemize}

\subsection*{The two constructions coincide:}
Apply Lemma~\ref{lem:blowups} with $G = \PP^1 \times \PP^2$ and:
\begin{align*}
  & V = \cO_{\PP^1 \times \PP^2}(-1,-1) \oplus \cO_{\PP^1 \times
    \PP^2} \oplus \cO_{\PP^1 \times \PP^2}\\
  & W = \cO_{\PP^1 \times \PP^2}(0,1) \oplus \cO_{\PP^1 \times
    \PP^2}(0,1) 
\end{align*}
with $f\colon V \to W$ given by the matrix:
\[
\begin{pmatrix} 
  t_0 A_2(y) & y_0 & y_1 \\
  t_1 B_2(y) & y_1 & y_2 
\end{pmatrix}
\]
where $[t_0:t_1]$ are homogeneous co-ordinates on $\PP^1$ and
$[y_0:y_1:y_2]$ are homogeneous co-ordinates on $\PP^2$.  This
exhibits $X$ as the blow-up of $\PP^1 \times \PP^2$ in the locus $Z$
defined by the condition
\[
\rk
\begin{pmatrix} 
  t_0 A_2(y) & y_0 & y_1 \\
  t_1 B_2(y) & y_1 & y_2 
\end{pmatrix}
<2
\]
and it is easy to see that $C$ is described in this way. For instance,
it is immediate that $Z$ projects isomorphically to a conic in
$\PP^2$, and that the projection to $\PP^1$ has degree $5$.

\subsection*{The quantum period:} We proceed as in
Example~\ref{ex:ql_with_mirror_map}.  Let $p_1$, $p_2, p_3 \in
H^\bullet(F;\ZZ)$ denote the first Chern classes of $L$, $M$, and $N$
respectively; these classes form a basis for $H^2(F;\ZZ)$.  Write
$\tau \in H^2(F;\QQ)$ as $\tau = \tau_1 p_1 + \tau_2 p_2+ \tau_3 p_3$
and identify the group ring $\QQ[H_2(F;\ZZ)]$ with the polynomial ring
$\QQ[Q_1,Q_2,Q_3]$ via the $\QQ$-linear map that sends the element
$Q^\beta \in \QQ[H_2(F;\ZZ)]$ to $Q_1^{\langle \beta,p_1\rangle}
Q_2^{\langle \beta,p_2\rangle} Q_3^{\langle \beta,p_3\rangle}$.
Theorem~\ref{thm:toric_mirror} gives:
\[
J_F(\tau) = e^{\tau/z} 
\sum_{l, m, n\geq 0} 
\frac{
  Q_1^l Q_2^m Q_3^n e^{l \tau_1} e^{m \tau_2} e^{n \tau_3}
}
{
  \prod_{k=1}^l (p_1 + k z)^2
  \prod_{k=1}^m (p_2 + k z)^3
  \prod_{k=1}^n (p_3 + k z)^2
}
\frac{\prod_{k = -\infty}^0 (p_3-p_1-p_2 + k z)}
{\prod_{k = -\infty}^{n-l-m} (p_3-p_1-p_2 + k z)}
\]
and hence:
\[
I_{\be,E}(\tau) = e^{\tau/z} 
\sum_{l, m, n \geq 0} 
\frac
{
  Q_1^l Q_2^m Q_3^n e^{l \tau_1} e^{m \tau_2} e^{n \tau_3}
  \prod_{k=1}^{m+n} (\lambda + p_2+p_3 + k z)^2
}
{
  \prod_{k=1}^l (p_1 + k z)^2 
  \prod_{k=1}^m (p_2 + k z)^3
  \prod_{k=1}^n (p_3 + k z)^2 
}
\frac
{
  \prod_{k = -\infty}^0 (p_3-p_2-p_1 + k z)
}
{
  \prod_{k = -\infty}^{n-l-m} (p_3-p_2-p_1 + k z)
}
\]
Note that:
\[
I_{\be,E}(0) = A+ B  z^{-1} + O(z^{-2})
\]
where:
\begin{align*}
  A & = 1\\
  B & = (Q_3 + 4 Q_2 Q_3) 1 + (p_3-p_2-p_1)\sum_{m > 0}
  {(-1)^{m-1} Q_2^m \over m} \\
  &= Q_3(1 +4 Q_2) 1 + (p_3-p_2-p_1)\log(1+Q_2)
\end{align*}
Arguing exactly as in Example~\ref{ex:ql_with_mirror_map}, we find
that:
\[
J_{\be,E}\big((p_3-p_2-p_1)\log(1+Q_2)\big) =
e^{-Q_3(1+4 Q_2)/z} I_{\be,E}(0)
\]
and:
\[
J_{\be,E}\big((p_3-p_2-p_1)\log(1+Q_2)\big)
= 
e^{(p_3-p_2-p_1)\log(1+Q_2)/z}
\Big[ J_{\be,E}(0)\Big]_{Q_1 = {Q_1 \over 1+Q_2}, Q_2 = {Q_2 \over
    1+Q_2}, Q_3 = Q_3(1+Q_2)}
\]

Hence, using the inverse mirror map \eqref{eq:inverse_mirror_map}, we have:
\begin{align*}
  J_{\be,E}(0) &= 
  \Big[e^{-(p_3-p_2-p_1)\log(1+Q_2)/z}
  J_{\be,E}\big((p_3-p_2-p_1)\log(1+Q_2)\big)\Big]_{Q_1 = {Q_1 \over 1-Q_2}, Q_2 = {Q_2 \over
      1-Q_2}, Q_3 = Q_3(1-Q_2)}  \\
  &= 
  e^{(p_3-p_2-p_1)\log(1-Q_2)/z}
  \Big[e^{-Q_3(1+4 Q_2) /z} I_{\be,E}(0)\Big]_{Q_1 = {Q_1 \over 1-Q_2}, Q_2 = {Q_2 \over
      1-Q_2}, Q_3 = Q_3(1-Q_2)} 
\end{align*}
Taking the non-equivariant limit yields:
\begin{multline*}
  J_{Y,X}(0) = e^{(p_3-p_2-p_1)\log(1-Q_2)/z} e^{-Q_3(1+3 Q_2)} \times \\
  \sum_{l, m, n \geq 0} 
  \frac
  {
    Q_1^l Q_2^m Q_3^n (1-Q_2)^{n-l-m}
    \prod_{k=1}^{m+n} (p_2+p_3 + k z)^2
  }
  {
    \prod_{k=1}^l (p_1 + k z)^2 
    \prod_{k=1}^m (p_2 + k z)^3 
    \prod_{k=1}^n (p_3 + k z)^2 
  }
  \frac
  {
    \prod_{k = -\infty}^0 (p_3-p_2-p_1 + k z)
  }
  {
    \prod_{k = -\infty}^{n-l-m} (p_3-p_2-p_1 + k z)
  }
\end{multline*}

Recall that the quantum period $G_X$ is obtained from the component of
$J_X(0)$ along the unit class $1 \in H^\bullet(X;\QQ)$ by setting $z =
1$ and $Q^\beta = t^{\langle \beta, {-K_X} \rangle}$.  In view of
equation \eqref{eq:KKP}, therefore, to obtain $G_X$ we extract the
component of $J_{Y,X}(0)$ along the unit class $1 \in
H^\bullet(Y;\QQ)$, set $z=1$, set $Q_1 = t$, set $Q_2 = 1$, and set
$Q_3 =t$.  This gives:
\[
G_X(t) = e^{-4t} \sum_{l=0}^\infty \sum_{m = 0}^\infty 
t^{2l+m} \frac{(l+2m)!(l+2m)!}{(l!)^2 (m!)^3 ((l+m)!)^2}
\]
Regularizing gives:
\begin{align*}
  \hG_X(t) = 1 + 22t^2 + 126t^3 + 1722t^4 + 18780t^5 + 236470t^6 +
  2998380t^7 + 39440170t^8 \\
  + 528743880t^9 + \cdots
\end{align*}

\subsection*{Minkowski period sequence:} \href{http://www.grdb.co.uk/search/period3?id=138&printlevel=2}{138}


\addtocounter{CustomSectionCounter}{1}
\section{The Fano Manifold $\MM{3}{6}$}
\label{sec:3-6}
\label{anchor:3--6}

\subsection*{Mori--Mukai name:} 3--6

\subsection*{Mori--Mukai construction:} The blow-up of $\PP^3$ with
centre a disjoint union of a line and an elliptic curve of degree $4$.

\subsection*{Our construction:} A member $X$ of $|2M+N|$ in the toric
variety with weight data:
\[ 
\begin{array}{rrrrrrrl} 
  \multicolumn{1}{c}{s_0} & 
  \multicolumn{1}{c}{s_1} & 
  \multicolumn{1}{c}{x} & 
  \multicolumn{1}{c}{x_2} & 
  \multicolumn{1}{c}{x_3} &  
  \multicolumn{1}{c}{y_0} &  
  \multicolumn{1}{c}{y_1} &  \\
  \cmidrule{1-7}
  1 & 1 & -1 & 0 & 0  & 0 & 0 & \hspace{1.5ex} L\\ 
  0 & 0 & 1 & 1 & 1  & 0 & 0 & \hspace{1.5ex} M\\ 
  0 & 0 & 0 & 0 & 0  & 1 & 1 & \hspace{1.5ex} N\\ 
\end{array}
\]
and $\Amp F = \langle L,M,N\rangle$. The secondary fan for $F$ has two
maximal cones as in Fig.~\ref{fig:3-6}.

\begin{figure}[h!]
  \centering
  \resizebox{6cm}{!}{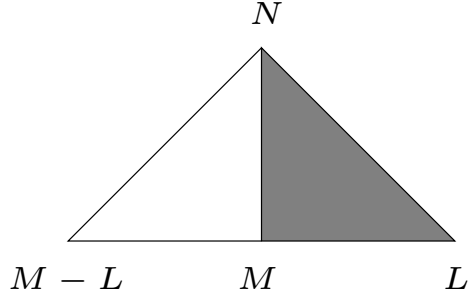}
  \caption{The secondary fan for $F$ in 3--6}
  \label{fig:3-6}
\end{figure}

We have:
\begin{itemize}
\item $-K_F=L+3M+2N$ is ample, that is $F$ is a Fano variety;
\item $X\sim 2M+N$ is nef;
\item $-(K_F+X)\sim L+M+N$ is ample.
\end{itemize}

\subsection*{The two constructions coincide:} An elliptic curve
$\Gamma \subset \PP^3$ is a $(2,2)$-complete intersection in $ \PP^3$
so $X$ is constructed by applying Lemma~\ref{lem:blowups} twice.  In
more detail, the equation of $X$ has the form:
\[
y_0 A(s_0x, s_1x, x_2,x_3)+ y_1B(s_0x, s_1x, x_2,x_3) =0
\]
where $A$, $B$ are homogeneous quadratic polynomials in the variables $x_0=s_0x,
x_1=s_1x, x_2, x_3$. The obvious morphism $X\to \PP^3_{x_0,x_1,x_2,x_3}$
blows up the line $x_0=x_1=0$ and the elliptic curve $A=B=0$.

\subsection*{The quantum period:}
Corollary~\ref{cor:QL} yields:
\[
G_X(t) = e^{-3t} \sum_{l=0}^\infty \sum_{n=0}^\infty \sum_{m=l}^\infty 
t^{l+m+n}
\frac{(2m+n)!}
{(l!)^2 (m-l)! (m!)^2(n!)^2}
\]
and regularizing gives:
\[
\hG_X(t) = 1+14 t^2+66 t^3+762 t^4+6960 t^5+73490 t^6+780360 t^7+8578570 t^8+96096000 t^9+ \cdots
\]

\subsection*{Minkowski period sequence:} \href{http://www.grdb.co.uk/search/period3?id=117&printlevel=2}{117}


\addtocounter{CustomSectionCounter}{1}

\section{The Fano Manifold $\MM{3}{7}$}
\label{anchor:3--7}

\subsection*{Mori--Mukai name:} 3--7

\subsection*{Mori--Mukai construction:} The blow-up of $W\subset
\PP^2\times \PP^2$ with centre an elliptic curve which is an
intersection of two members of $|{-\frac{1}{2}}K_W|$.  Here $W$ is a
divisor of bidegree $(1,1)$ in $\PP^2 \times \PP^2$.

\subsection*{Our construction:} A complete intersection $X$ of type
$(M+N)\cap (L+M+N)$ in the toric variety $F= \PP^1 \times \PP^2
\times\PP^2$.

\subsection*{The two constructions coincide:}  Apply
Lemma~\ref{lem:blowups}.

\subsection*{The quantum period:}
The toric variety $F$ has weight data:
\[ 
\begin{array}{rrrrrrrrl} 
  1 & 1 & 0 & 0 & 0  & 0 & 0 & 0 & \hspace{1.5ex} L\\ 
  0 & 0 & 1 & 1 & 1  & 0 & 0 & 0 & \hspace{1.5ex} M\\ 
  0 & 0 & 0 & 0 & 0  & 1 & 1 & 1 & \hspace{1.5ex} N\\ 
\end{array}
\]
and $\Amp F=\langle L,M,N\rangle$.  We have that:
\begin{itemize}
\item $F$ is a Fano variety;
\item $X$ is the complete intersection of two nef divisors on $F$;
\item ${-(K_F+\Lambda)}=L+M+N$ is ample on $F$.
\end{itemize}

Corollary~\ref{cor:QL} yields:
\[
G_X(t) = e^{-3t} \sum_{l=0}^\infty \sum_{m=0}^\infty \sum_{n=0}^\infty
t^{l+m+n}
\frac{(l+m+n)!(m+n)!}
{(l!)^2 (m!)^3 (n!)^3}
\]
and regularizing gives:
\[
  \hG_X(t) = 1+10 t^2+48 t^3+438 t^4+3720 t^5+33940 t^6+320040
  t^7+3096310 t^8+30581040 t^9 + \cdots
\]

\subsection*{Minkowski period sequence:} \href{http://www.grdb.co.uk/search/period3?id=103&printlevel=2}{103}


\addtocounter{CustomSectionCounter}{1}

\section{The Fano Manifold $\MM{3}{8}$}
\label{anchor:3--8}

\subsection*{Mori--Mukai name:} 3--8

\subsection*{Mori--Mukai construction:} A member of the linear system
$|p_1^\star g^\star \cO(1) \otimes p_2^\star \cO(2) |$ on $\FF_1\times
\PP^2$ where $p_i\,(i=1,2)$ is the projection to the $i$th factor and
$g\colon \FF_1\to \PP^2$ is the blowing-up.

\subsection*{Our construction:} A member $X$ of $|M+2N|$ in the toric
variety $F$ with weight data:
\[
\begin{array}{rrrrrrrl} 
  \multicolumn{1}{c}{s_0} & 
  \multicolumn{1}{c}{s_1} & 
  \multicolumn{1}{c}{x} & 
  \multicolumn{1}{c}{x_2} & 
  \multicolumn{1}{c}{y_0} &  
  \multicolumn{1}{c}{y_1} &  
  \multicolumn{1}{c}{y_2} &  \\
  \cmidrule{1-7}
  1 & 1 & -1 & 0 & 0  & 0 & 0 & \hspace{1.5ex} L\\ 
  0 & 0 & 1 & 1 & 0  & 0 & 0 & \hspace{1.5ex} M\\ 
  0 & 0 & 0 & 0 & 1  & 1 & 1 & \hspace{1.5ex} N\\ 
\end{array}
\]
and $\Amp F = \langle L, M, N\rangle$.  The secondary fan for $F$ is
the same as that for the toric variety in 3--6 (\S\ref{sec:3-6}) and is
shown in Fig.~\ref{fig:3-6}.  We have:
\begin{itemize}
\item $-K_F=L+2M+3N$ is ample, that is $F$ is a Fano variety;
\item $X\sim M+2N$ is nef;
\item $-(K_F+X)\sim L+M+N$ is ample.
\end{itemize}

\subsection*{The two constructions coincide:} Obvious. 

\subsection*{The quantum period:}
Corollary~\ref{cor:QL} yields:
\[
G_X(t) = e^{-3t} \sum_{l=0}^\infty \sum_{m=l}^\infty \sum_{n=0}^\infty 
t^{l+m+n}
\frac{(m+2n)!}
{(l!)^2 (m-l)! (m!) (n!)^3}
\]
and regularizing gives:
\[
\hG_X(t) = 1+12 t^2+54 t^3+540 t^4+4620 t^5+43770 t^6+425880 t^7
+4256700 t^8+43462440 t^9+ \cdots
\]

\subsection*{Minkowski period sequence:} \href{http://www.grdb.co.uk/search/period3?id=112&printlevel=2}{112}


\addtocounter{CustomSectionCounter}{1}

\section{The Fano Manifold $\MM{3}{9}$}
\label{sec:3-9}
\label{anchor:3--9}

\subsection*{Mori--Mukai name:} 3--9

\subsection*{Mori--Mukai construction:} The blow-up of the cone
$W_4\subset \PP^6$ over the Veronese surface $R_4\subset \PP^5$ with
centre a disjoint union of the vertex and a quartic in $R_4\cong
\PP^2$.

\subsection*{Our construction:} A member $X$ of $|2M|$ in the toric
variety $F$ with weight data: 
\[ 
\begin{array}{rrrrrrl}  
  \multicolumn{1}{c}{s_0} & 
  \multicolumn{1}{c}{s_1} & 
  \multicolumn{1}{c}{s_2} & 
  \multicolumn{1}{c}{x} & 
  \multicolumn{1}{c}{y_0} &  
  \multicolumn{1}{c}{y_1} &  \\
  \cmidrule{1-6}
  1 & 1 & 1 & -2 & 0  & 0 &  \hspace{1.5ex} L\\ 
  0 & 0 & 0 & 1 & 1  & 1 & \hspace{1.5ex} M\\ 
\end{array}
\]
and $\Amp F =\langle L,M \rangle$.

We have that:
\begin{itemize}
\item $-K_F=L+3M$ is ample, so $F$ is a Fano variety;
\item $X\sim 2M$ is nef;
\item $-(K_F+X)\sim L+M$ is ample. 
\end{itemize}

\subsection*{The two constructions coincide:} The variety $X$ is cut
out by:
\[
y_0 y_1 +x^2 A_4 (s_0,s_1,s_2)=0
\]
where $A_4$ is a generic homogeneous polynomial of degree~4 in
$s_0$,~$s_1$,~$s_2$.  Note the morphisms $\pi \colon F\to \PP^2$ given
by the linear system $|L|$, and $f\colon F\to \PP(1,1,1,2,2)$ given
(contravariantly) by $[x_0,x_1,x_2,y_0,y_1] \mapsto
[s_0\sqrt{x},s_1\sqrt{x} , s_2\sqrt{x} , y_0, y_1]$. The exceptional
set of $f$ is the divisor $E=(x=0)= \PP^2_{s_0,s_1,s_2}\times
\PP^1_{y_0,y_1}$ that maps to $\PP^1_{y_0,y_1}\subset
\PP(1,1,1,2,2)$. Note that $E\cap X$ is \emph{two} copies of $\PP^2$,
one above $[y_0:y_1]=[1:0]$ and one above $[y_0:y_1]=[0:1]$. This
explains how $X$ has rank~$3$ when $F$ has rank~$2$.

To see that our construction coincides with the construction of
Mori--Mukai, set $W = f(X)$, note that:
\[
W=\bigl(y_0 y_1 +A_4 (x_0,x_1,x_2)=0\bigr)\subset \PP(1,1,1,2,2)
\]
and note that the morphism $f\colon X \to W$ contracts one copy of
$\PP^2$, with normal bundle $\cO(-2)$, to each of the two singular
points $W\cap \PP^1_{y_0,y_1}$. Consider the rational projection:
\[
g\colon \PP(1,1,1,2,2)\dasharrow \PP(1,1,1,2)_{x_0,x_1,x_2,y_0}
\]
which omits the homogeneous co-ordinate $y_1$. It is clear that
$g|_{W}\colon W \dasharrow \PP(1,1,1,2)$ extends to a morphism after
blowing up the singular point $[0:0:0:0:1]\in W$, and that this
morphism contracts the surface
$\bigl(y_0=A_4(x_0,x_1,x_2)=0\bigr)\subset W$ to the curve
$\bigl(y_0=A_4(x_0,x_1,x_2)=0\bigr)\subset \PP(1,1,1,2)$.

\subsection*{The quantum period:}
Corollary~\ref{cor:QL} yields:
\[
G_X(t) = e^{-2t} \sum_{l=0}^\infty \sum_{m=2l}^\infty
t^{l+m}
\frac{(2m)!}
{(l!)^3 (m-2l)! (m!)^2}
\]
and regularizing gives:
\[
\hG_X(t) = 1+2 t^2+36 t^3+198 t^4+840 t^5+9200 t^6+79800 t^7+520870 t^8+4289040 t^9+ \cdots
\]

\subsection*{Minkowski period sequence:} \href{http://www.grdb.co.uk/search/period3?id=22&printlevel=2}{22}


\addtocounter{CustomSectionCounter}{1}
\section{The Fano Manifold $\MM{3}{10}$}
\label{sec:3-10}
\label{anchor:3--10}

\subsection*{Mori--Mukai name:} 3--10

\subsection*{Mori--Mukai construction:} The blow-up of a quadric
\mbox{3-fold} $Q\subset \PP^4$ with centre a disjoint union of two
conics on it.

\subsection*{Our construction:} A member $X$ of $|2N|$ in the toric
variety $F$ with weight data:
\[
\begin{array}{rrrrrrrl} 
  \multicolumn{1}{c}{s_0} & 
  \multicolumn{1}{c}{s_1} & 
  \multicolumn{1}{c}{t_2} & 
  \multicolumn{1}{c}{t_3} & 
  \multicolumn{1}{c}{x} & 
  \multicolumn{1}{c}{y} &  
  \multicolumn{1}{c}{x_4} &  \\
  \cmidrule{1-7}
  1 & 1 & 0 & 0 & -1  & 0 & 0 & \hspace{1.5ex} L\\ 
  0 & 0 & 1 & 1 & 0 & -1 & 0 & \hspace{1.5ex} M\\ 
  0 & 0 & 0 & 0 & 1  & 1 & 1 & \hspace{1.5ex} N\\ 
\end{array}
\]
and $\Amp F=\langle L,M,N\rangle$. The secondary fan for $F$ has $4$
maximal cones as in Fig.~\ref{fig:3-10}.

\begin{figure}[h!]
  \centering
  \resizebox{4cm}{!}{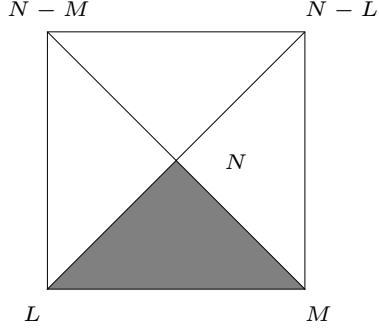}
  \caption{The secondary fan for $F$ in 3--10}
  \label{fig:3-10}
\end{figure}

We have:
\begin{itemize}
\item $-K_F=L+M+3N$ is ample, so that $F$ is a Fano variety;
\item $X\sim 2N$ is nef;
\item $-(K_F+X)\sim L+M+N$ is ample.
\end{itemize}

\subsection*{The two constructions coincide:} We take $Q$ to be the
locus $ x_0 x_1 + x_2 x_3 + x_4^2 = 0$ in $
\PP^4_{x_0,x_1,x_2,x_3,x_4}$, and take the conics to be cut out of $Q$ 
by the two complete intersections $(x_0 = x_1 = 0)$ and $(x_2 = x_3 =
0)$; note that the intersection of these two planes misses $Q$.  The
morphism $F\to \PP^4$ given (contravariantly) by:
\[
[x_0:x_1:x_2:x_3:x_4] \mapsto [s_0x:s_1x:t_2y:t_3y:x_4]
\]
blows up the planes $(x_0 = x_1 = 0)$ and $(x_2 = x_3 = 0)$.  Taking
the proper transform of $Q$ yields $X$.  

\subsection*{The quantum period:}
Corollary~\ref{cor:QL} yields:
\[
G_X(t) = e^{-2t} \sum_{l=0}^\infty \sum_{m=0}^\infty \sum_{n=\max(l,m)}^\infty
t^{l+m+n}
\frac{(2n)!}
{(l!)^2 (m!)^2 n! (n-l)! (n-m)!}
\]
and regularizing gives:
\[
\hG_X(t) = 1+10 t^2+36 t^3+366 t^4+2640 t^5+23320 t^6+200760 t^7+1815310 t^8+16611840 t^9+ \cdots
\]

\subsection*{Minkowski period sequence:} \href{http://www.grdb.co.uk/search/period3?id=99&printlevel=2}{99}


\addtocounter{CustomSectionCounter}{1}

\section{The Fano Manifold $\MM{3}{11}$}
\label{anchor:3--11}

\subsection*{Mori--Mukai name:} 3--11

\subsection*{Mori--Mukai construction:} The blow--up of $B_7$ (see
\S\ref{sec:B7}) with centre an elliptic curve that is the intersection
of two members of $|{-\frac{1}{2}}K_{B_7}|$.

\subsection*{Our construction:} A member $X$ of $|L+M+N|$ in the toric
variety $F$ with weight data:
\[
\begin{array}{rrrrrrrl} 
  \multicolumn{1}{c}{s_0} & 
  \multicolumn{1}{c}{s_1} & 
  \multicolumn{1}{c}{s_2} & 
  \multicolumn{1}{c}{x} & 
  \multicolumn{1}{c}{x_3} & 
  \multicolumn{1}{c}{y_0} &  
  \multicolumn{1}{c}{y_1} &  \\
  \cmidrule{1-7}
  1 & 1 & 1 & -1 & 0  & 0 & 0 & \hspace{1.5ex} L\\ 
  0 & 0 & 0 & 1 & 1 & 0 & 0 & \hspace{1.5ex} M\\ 
  0 & 0 & 0 & 0 & 0  & 1 & 1 & \hspace{1.5ex} N\\ 
\end{array}
\]
and $\Amp F= \langle L,M,N \rangle$. In other words, $F\cong B_7\times
\PP^1$.  The secondary fan of $F$ is the same as that of the toric
variety in 3--6 (\S\ref{sec:3-6}) and is shown in Fig.~\ref{fig:3-6} .

We have:
\begin{itemize}
\item $-K_F=2L+2M+2N$ is ample, so $F$ is a Fano variety;
\item $X\sim L+M+N$ is ample;
\item $-(K_F+X)\sim L+M+N$ is ample.
\end{itemize}

\subsection*{The two constructions coincide:} Recall from
\S\ref{sec:B7} that $B_7$ is the toric variety with weight data:
\[
\begin{array}{rrrrrl} 
  \multicolumn{1}{c}{s_0} & 
  \multicolumn{1}{c}{s_1} & 
  \multicolumn{1}{c}{s_2} & 
  \multicolumn{1}{c}{x} & 
  \multicolumn{1}{c}{x_3} \\
  \cmidrule{1-5}
  1 & 1 & 1 & -1 & 0  &  \hspace{1.5ex} L\\ 
  0 & 0 & 0 & 1 & 1 &  \hspace{1.5ex} M\\ 
\end{array}
\]
and $\Amp B_7 = \langle L, M \rangle$.  Now apply
Lemma~\ref{lem:blowups} with $V = \cO_{B_7} \oplus \cO_{B_7}$, $W =
{-\frac{1}{2}}K_{B_7}$, and the map $f\colon V \to W$ given by $
\begin{pmatrix}
  A & B
\end{pmatrix}
$ where $A$, $B$ are the sections of ${-\frac{1}{2}}K_{B_7}$ that
define the centre of the blow-up.

\subsection*{The quantum period:}
Corollary~\ref{cor:QL} yields:
\[
G_X(t) = e^{-2t} \sum_{l=0}^\infty \sum_{n=0}^\infty \sum_{m=l}^\infty
t^{l+m+n}
\frac{(l+m+n)!}
{(l!)^3 (m-l)! m! (n!)^2}
\]
and regularizing gives:
\[
\hG_X(t) = 1+6 t^2+30 t^3+186 t^4+1380 t^5+10230 t^6+78540 t^7+620970 t^8+5020680 t^9+ \cdots
\]

\subsection*{Minkowski period sequence:} \href{http://www.grdb.co.uk/search/period3?id=72&printlevel=2}{72}


\addtocounter{CustomSectionCounter}{1}

\section{The Fano Manifold $\MM{3}{12}$}
\label{anchor:3--12}

\subsection*{Mori--Mukai name:} 3--12

\subsection*{Mori--Mukai construction:} The blow-up of $\PP^3$ with
centre a disjoint union of a line and a twisted cubic.

\subsection*{Our construction:} A codimension-$2$ complete
intersection $X$ of type $(M+N)\cap (M+N)$ in the toric variety $F$
with weight data:
\[
\begin{array}{rrrrrrrrl} 
  \multicolumn{1}{c}{s_0} & 
  \multicolumn{1}{c}{s_3} & 
  \multicolumn{1}{c}{x} & 
  \multicolumn{1}{c}{x_1} & 
  \multicolumn{1}{c}{x_2} & 
  \multicolumn{1}{c}{y_0} &  
  \multicolumn{1}{c}{y_1} &  
  \multicolumn{1}{c}{y_2} &  \\
  \cmidrule{1-8}
  1 & 1 & -1 & 0 & 0  & 0 & 0 & 0 & \hspace{1.5ex} L\\ 
  0 & 0 & 1 & 1 & 1 & 0 & 0 & 0 & \hspace{1.5ex} M\\ 
  0 & 0 & 0 & 0 & 0  & 1 & 1 & 1 & \hspace{1.5ex} N\\ 
\end{array}
\]
and $\Amp F =\langle L,M,N\rangle$.  The secondary fan of $F$ is the
same as that of the toric variety in 3--6 (\S\ref{sec:3-6}) and is
shown in Fig.~\ref{fig:3-6}.  We have:
\begin{itemize}
\item $-K_F=L + 3M + 3N$ is ample, so $F$ is a Fano variety;
\item $X$ is the complete intersection of two nef divisors on $F$;
\item $-(K_F+\Lambda)\sim L+M+N$ is ample.
\end{itemize}

\subsection*{The two constructions coincide:}
The twisted cubic $\Gamma$ is cut out of $ \PP^3_{x_0,\dots, x_3}$ by
the equations:
\[ 
\rk \begin{pmatrix} x_0 & x_1 &x_ 2\\ x_1
  & x_2 & x_3 \end{pmatrix} <2
\]
By Lemma~\ref{lem:blowups} the blow up of $ \PP^3$ along $ \Gamma$ is
cut out of $\PP^3_{x_0,\dots, x_3}\times \PP^2_{y_0,y_1,y_2}$ by the
equations:
\[
 \begin{pmatrix} 
   x_0 & x_1 & x_2 \\ 
   x_1 & x_2 & x_3 
 \end{pmatrix} \cdot 
 \begin{pmatrix} 
   y_0 \\ y_1\\y_2 
 \end{pmatrix} = 0
\] 
Observe that $ \Gamma$ is disjoint from the line $(x_0=x_3=0)$. We
therefore blow up $ \PP^3_{x_0,\dots, x_3}\times \PP^2_{y_0,y_1,y_2}$
along the locus $ x_0=x_3=0$, obtaining the toric variety $F$.  The
equations defining $X$ inside $F$ are:
\[
\begin{pmatrix} 
  s_0 x & x_1 & x_2 \\ 
  x_1 & x_2 & s_3 x 
\end{pmatrix}
\cdot 
\begin{pmatrix} 
  y_0 \\ y_1\\y_2 
\end{pmatrix} = 0
\]
and so $ X$ is a complete intersection of type $(M+N)\cap (M+N)$.

\subsection*{The quantum period:}
Corollary~\ref{cor:QL} yields:
\[
G_X(t) = e^{-2t} \sum_{l=0}^\infty \sum_{n=0}^\infty \sum_{m=l} ^\infty
t^{l+m+n}
\frac{(m+n)!(m+n)!}
{(l!)^2 (m-l)! (m!)^2 (n!)^3}
\]
and regularizing gives:
\[
\hG_X(t) = 1+8 t^2+30 t^3+240 t^4+1740 t^5+13130 t^6+106680 t^7+862960
t^8+7248360 t^9+ \cdots
\]

\subsection*{Minkowski period sequence:} \href{http://www.grdb.co.uk/search/period3?id=85&printlevel=2}{85}


\addtocounter{CustomSectionCounter}{1}

\section{The Fano Manifold $\MM{3}{13}$}
\label{anchor:3--13}

\subsection*{Mori--Mukai name:} 3--13

\subsection*{Mori--Mukai construction:} The blow-up of $W\subset
\PP^2\times \PP^2$ with centre a curve $C$ of bidegree $(2,2)$ on it
such that $C\hookrightarrow W \to
\PP^2\times\PP^2\overset{p_i}{\longrightarrow}\PP^2$ is an embedding
for both $i=1,2$.  Here $W$ is a divisor of bidegree $(1,1)$ in $\PP^2
\times \PP^2$ and $p_i:\PP^2 \times \PP^2 \to \PP^2$ is projection to
the $i$th factor.

\subsection*{Our construction:} A codimension-$3$ complete
intersection $X$ of type $(L+M)\cap (L+N)\cap (M+N)$ in $\PP^2\times
\PP^2 \times \PP^2$.

\subsection*{The two constructions coincide:} First choose
co-ordinates $x_0, x_1, x_2$, $y_0,y_1,y_2$ on $\PP^2\times \PP^2$
such that the curve $C$ is contained in the surface $\Sigma$ given by
the condition:
\[
\rk
\begin{pmatrix}
  x_0 & x_1 & x_2 \\
  y_0 & y_1 & y_2
\end{pmatrix}
<2
\]
Note that $\Sigma$ is just $\PP^2$ embedded diagonally in $\PP^2\times
\PP^2$. In these coordinates, $W_{1,1}=\{ f_{1,1 }(x, y)=0 \}$ where
$f_{1,1}\in \Gamma \bigl(\PP^2\times \PP^2; \cO(1,1) \bigr)$ is a
general section, and $C=\Sigma \cdot W_{1,1}$. By
Lemma~\ref{lem:blowups}, $X$ is given by the equations:
\[
\begin{cases}
  x_0z_0+x_1z_1+x_2z_2 & = 0\\
  y_0z_0+y_1z_1+y_2z_2 & = 0\\
  f_{1,1}(x,y) & = 0
\end{cases}
\]
in $\PP^2_{x_0,x_1,x_2}\times\PP^2_{y_0,y_1,y_2}\times \PP^2_{z_0,z_1,z_2}$.

\subsection*{The quantum period:} 
$F = \PP^2 \times \PP^2 \times \PP^2$ is the toric variety with weight data:
\[
\begin{array}{rrrrrrrrrl} 
  1 & 1 & 1 & 0 & 0 & 0 & 0 & 0 & 0 & \hspace{1.5ex} L\\ 
  0 & 0 & 0 & 1 & 1 & 1 & 0 & 0 & 0 & \hspace{1.5ex} M\\ 
  0 & 0 & 0 & 0 & 0 & 0 & 1 & 1 & 1 & \hspace{1.5ex} N\\ 
\end{array}
\]
and $\Amp F =\langle L,M,N\rangle$.  We have that:
\begin{itemize}
\item $F$ is a Fano variety;
\item $X$ is the complete intersection of three nef divisors on $F$;
\item $-(K_F+\Lambda)\sim L+M+N$ is ample.
\end{itemize}
Corollary~\ref{cor:QL} yields:
\[
G_X(t) = e^{-3t} \sum_{l=0}^\infty \sum_{m=0}^\infty \sum_{n=0} ^\infty
t^{l+m+n}
\frac{(l+m)!(l+n)!(m+n)!}
{(l!)^3 (m!)^3 (n!)^3}
\]
and regularizing gives:
\[
\hG_X(t) = 1 + 6 t^2 + 24 t^3 + 162 t^4 + 1080 t^5 + 7620 t^6 +
55440 t^7 + 415170 t^8 + 3166800 t^9 + \cdots
\]

\subsection*{Minkowski period sequence:} \href{http://www.grdb.co.uk/search/period3?id=70&printlevel=2}{70}


\addtocounter{CustomSectionCounter}{1}

\section{The Fano Manifold $\MM{3}{14}$}
\label{sec:3-14}
\label{anchor:3--14}

\subsection*{Mori--Mukai name:} 3--14

\subsection*{Mori--Mukai construction:} The blow-up of $\PP^3$ with
centre a union of a cubic in a plane $S$ and a point not in $S$.
 
\subsection*{Our construction:} A member $X$ of $|M+N|$ in the toric
variety $F$ with weight data:
\[ 
\begin{array}{rrrrrrrl}  
  \multicolumn{1}{c}{s_0} & 
  \multicolumn{1}{c}{s_1} & 
  \multicolumn{1}{c}{s_2} & 
  \multicolumn{1}{c}{x} & 
  \multicolumn{1}{c}{x_3} &  
  \multicolumn{1}{c}{u} &  
  \multicolumn{1}{c}{v} &  \\
  \cmidrule{1-7}
  1 & 1 & 1 & -1 & 0  & -2 & 0 & \hspace{1.5ex} L\\ 
  0 & 0 & 0 &  1 & 1  &  0 &  0 & \hspace{1.5ex} M\\ 
  0 & 0 & 0 &  0 & 0  &  1 & 1 & \hspace{1.5ex} N\\ 
\end{array}
\]
and $\Amp F =\langle L,M, N \rangle$. The secondary fan of $F$ is shown
in~\ref{fig:3-14}.

\begin{figure}[h!]
  \centering
  \resizebox{6cm}{!}{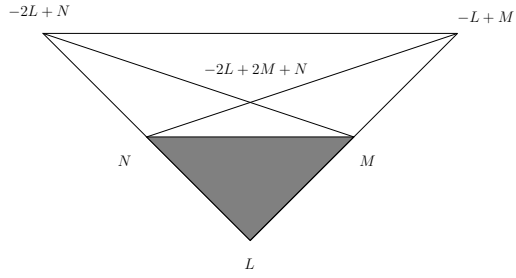}
  \caption{The secondary fan for $F$ in 3--14}
  \label{fig:3-14}
\end{figure}

We have:
\begin{itemize}
\item $-K_F=2M+2N$ is nef and big but not ample;
\item $X\sim M+N$ is nef and big but not ample;
\item $-(K_F+X)\sim M+N$ is nef and big but not ample. 
\end{itemize}

\subsection*{The two constructions coincide:} The variety $X$ is cut
out by:
\[
vx_3 + u x A_3 (s_0,s_1,s_2)=0
\]
Note the obvious morphism $\pi \colon F \to B_7$ with fibre
$\PP^1_{u,v}$, where $B_7$ is the toric variety with weight data:
\[ 
\begin{array}{rrrrrl}
  \multicolumn{1}{c}{s_0} & 
  \multicolumn{1}{c}{s_1} & 
  \multicolumn{1}{c}{s_2} & 
  \multicolumn{1}{c}{x} & 
  \multicolumn{1}{c}{x_3} \\
 \cmidrule{1-5}
  1 & 1 & 1 & -1 & 0 & \hspace{1.5ex} L\\ 
  0 & 0 & 0 &  1 & 1 & \hspace{1.5ex} M
 \end{array}
\]
and $\Amp B_7 = \langle L, M \rangle$.  (The weight data and co-ordinates for $B_7 $ here are exactly as in \S\ref{sec:B7}.)  The birational morphism $B_7\to
\PP^3$ given (contravariantly) by $[x_0,\dots ,x_3]\mapsto
[s_0x,s_1x,s_2x,x_3]$ identifies $B_7$ with the blow-up of $\PP^3$ at
the point $[0:0:0:1]$. The equation defining $X$ is of degree~$1$ in
$\PP^1_{u,v}$: it follows that the morphism $\pi|_{X} \colon X \to
B_7$ is birational and blows up the locus\footnote{Note that, with our
  choice of stability condition for $F$, $(x_3=x=0)\subset \CC^7$ is
  part of the unstable locus.}  $(x_3=A_3(s_0,s_1,s_2)=0)\subset B_7$.

\subsection*{The quantum period:}
Let $p_1$, $p_2, p_3 \in H^\bullet(F;\ZZ)$ denote the first Chern
classes of $L$, $M$, and $N$ respectively; these classes form a basis
for $H^2(F;\ZZ)$.  Write $\tau \in H^2(F;\QQ)$ as $\tau = \tau_1 p_1 +
\tau_2 p_2+ \tau_3 p_3$ and identify the group ring $\QQ[H_2(F;\ZZ)]$
with the polynomial ring $\QQ[Q_1,Q_2,Q_3]$ via the $\QQ$-linear map
that sends the element $Q^\beta \in \QQ[H_2(F;\ZZ)]$ to $Q_1^{\langle
  \beta,p_1\rangle} Q_2^{\langle \beta,p_2\rangle} Q_3^{\langle
  \beta,p_3\rangle}$.  We have:
\begin{multline*}
  I_F(\tau)  = e^{\tau/z} 
  \sum_{l, m, n\geq 0} 
  \frac{
    Q_1^l Q_2^m Q_3^n e^{l \tau_1} e^{m \tau_2} e^{n \tau_3}
  }
  {
    \prod_{k=1}^l (p_1 + k z)^3
    \prod_{k=1}^m (p_2 + k z)
    \prod_{k=1}^n (p_3 + k z)
  }
  \frac{\prod_{k = -\infty}^0 (p_2-p_1 + k z)}
  {\prod_{k = -\infty}^{m-l} (p_2-p_1 + k z)}
  \\
  \frac{\prod_{k = -\infty}^0 (p_3-2p_1 + k z)^2}
  {\prod_{k = -\infty}^{n-2l} (p_3-2p_1 + k z)^2}
\end{multline*}
Since:
\[
  I_F(\tau)  = 1 + \tau z^{-1} + O(z^{-2})
\]
Theorem~\ref{thm:toric_mirror} gives:
\[
J_F(\tau) = I_F(\tau)
\]
We now proceed exactly as in the case of 3--1
(\S\ref{sec:restriction_of_nef_not_restriction_of_ample}), obtaining:
\[
G_X(t) = 
e^{-2 t} \sum_{l=0}^\infty \sum_{m=l}^\infty \sum_{n=2l}^\infty
t^{m+n}
\frac
{(m+n)!}
{(l!)^3m!n!(m-l)!(n-2l)!}
\]
Regularizing gives:
\[
\hG_X(t) = 1+2 t^2+18 t^3+102 t^4+420 t^5+2810 t^6+21000 t^7+129430
t^8+813960 t^9 + \cdots
\]

\subsection*{Minkowski period sequence:} \href{http://www.grdb.co.uk/search/period3?id=21&printlevel=2}{21}


\addtocounter{CustomSectionCounter}{1}

\section{The Fano Manifold $\MM{3}{15}$}
\label{anchor:3--15}

\subsection*{Mori--Mukai name:} 3--15

\subsection*{Mori--Mukai construction:} The blow-up of a quadric
\mbox{3-fold} $Q\subset \PP^4$ with centre a disjoint union of a line
and a conic on it.

\subsection*{Our construction:} A member $X$ of $|L+N|$ in a toric
variety $F$ with weight data:
\[
\begin{array}{rrrrrrrl} 
  \multicolumn{1}{c}{s_0} & 
  \multicolumn{1}{c}{s_1} & 
  \multicolumn{1}{c}{s_2} & 
  \multicolumn{1}{c}{t_3} & 
  \multicolumn{1}{c}{t_4} &  
  \multicolumn{1}{c}{y} &  
  \multicolumn{1}{c}{z} &  \\
  \cmidrule{1-7}
  1 & 1 & 1 & 0 & 0  & -1 & 0 & \hspace{1.5ex} L\\ 
  0 & 0 & 0 & 1 & 1 & 0 & -1 & \hspace{1.5ex} M\\ 
  0 & 0 & 0 & 0 & 0  & 1 & 1 & \hspace{1.5ex} N\\ 
\end{array}
\]
and $\Amp F =\langle L,M,N\rangle$. The secondary fan for $F$ is the
same as that for the toric variety in 3--10 (\S\ref{sec:3-10}) and is
shown in Fig.~\ref{fig:3-10}.  We have:
\begin{itemize}
\item $-K_F=2L+M+2N$ is ample, that is $F$ is a Fano variety;
\item $X\sim N+L$ is nef;
\item $-(K_F+X)\sim L+M+N$ is ample on $F$.
\end{itemize}

\subsection*{The two constructions coincide:} The morphism $F\to
\PP^4$ given (contravariantly) by:
\[
[x_0,x_1,x_2,x_3,x_4] \mapsto [s_0y, s_1y, s_2y, t_3z, t_4z]
\]
is the blow-up of $\PP^2$ along the disjoint union of the line
$(x_0=x_1=x_2=0)$ and the plane $(x_3=x_4=0)$. $X$ is the proper
transform of the (nonsingular) quadric defined by the equation:
\[
x_0^2+x_1x_3+x_2x_4=0
\]
Note that this quadric contains the line $x_0=x_1=x_2=0$.

\subsection*{The quantum period:}
Corollary~\ref{cor:QL} yields:
\[
G_X(t) = e^{-t} \sum_{l=0}^\infty \sum_{m=0}^\infty \sum_{n=\max(l,m)}^\infty
t^{l+m+n}
\frac{(l+n)!}
{(l!)^3 (m!)^2 (n-l)! (n-m)!}
\]
and regularizing gives:
\[
\hG_X(t) = 1+6 t^2+18 t^3+138 t^4+780 t^5+5370 t^6+36120 t^7+253050 t^8+1811880 t^9+ \cdots
\]

\subsection*{Minkowski period sequence:} \href{http://www.grdb.co.uk/search/period3?id=67&printlevel=2}{67}


\addtocounter{CustomSectionCounter}{1}

\section{The Fano Manifold $\MM{3}{16}$}
\label{anchor:3--16}

\subsection*{Mori--Mukai name:} 3--16

\subsection*{Mori--Mukai construction:} The blow-up of $B_7$ (see
\S\ref{sec:B7}) with centre the strict transform of a twisted cubic
passing through the centre of the blow-up $B_7\to \PP^3$.

\subsection*{Our construction:} A complete intersection $X$ of type $N
\cap N$ in the toric variety $F$ with weight data:
\[
\begin{array}{rrrrrrrrl} 
  \multicolumn{1}{c}{s_1} & 
  \multicolumn{1}{c}{s_2} & 
  \multicolumn{1}{c}{s_3} & 
  \multicolumn{1}{c}{x} & 
  \multicolumn{1}{c}{x_0} & 
  \multicolumn{1}{c}{y_0} &  
  \multicolumn{1}{c}{y_1} &  
  \multicolumn{1}{c}{y_2} &  \\
  \cmidrule{1-8}
  1 & 1 & 1 & -1 & 0  & 0 & -1 & -1 & \hspace{1.5ex} L\\ 
  0 & 0 & 0 & 1 & 1 & -1 & 0 & 0 & \hspace{1.5ex} M\\ 
  0 & 0 & 0 & 0 & 0  & 1 & 1 & 1 & \hspace{1.5ex} N\\ 
\end{array}
\]
and $\Amp F = \langle L, M, N \rangle$. The secondary fan for $F$ is
shown schematically in Fig.~\ref{fig:3-16}. We have:
\begin{itemize}
\item $-K_F=M+3N$ is nef and big but not ample;
\item $X$ is the complete intersection of two nef divisors on $F$;
\item $-(K_F+X)\sim M+N$ is nef and big but not ample on $F$.
\end{itemize}

\begin{figure}[h!]
  \centering
  \resizebox{6cm}{!}{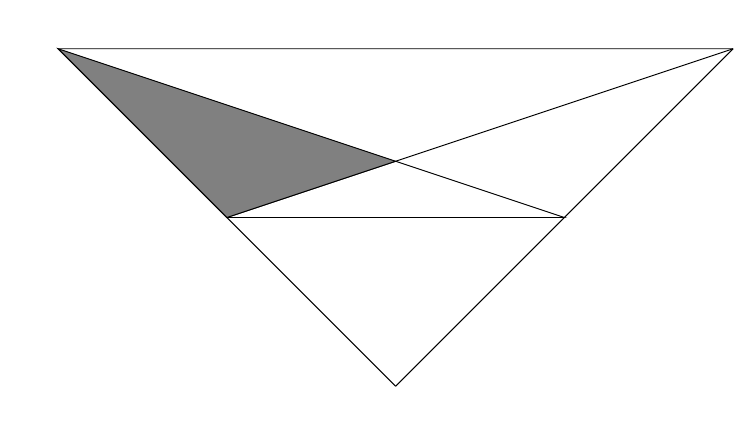}
  \caption{The secondary fan for $F$ in 3--16}
  \label{fig:3-16}
\end{figure}

\subsection*{The two constructions coincide:} Consider the rational
normal curve:
\[
\Gamma =
\Bigg\{\rk
\begin{pmatrix}
  x_0 & x_1 & x_2\\
  x_1 & x_2 & x_3
\end{pmatrix}
<2\Bigg\}
\]
in $\PP^3_{x_0,x_1,x_2,x_3}$ and note that $P = [1:0:0:0]$ lies on
$\Gamma$. Recall from \S\ref{sec:B7} that $B_7$ is the toric variety
with weight data:
\[
\begin{array}{rrrrrl} 
  \multicolumn{1}{c}{s_1} & 
  \multicolumn{1}{c}{s_2} & 
  \multicolumn{1}{c}{s_3} & 
  \multicolumn{1}{c}{x} & 
  \multicolumn{1}{c}{x_0} \\
  \cmidrule{1-5}
  1  & 1   & 1   &-1 & 0 & \hspace{1.5ex} L\\ 
  0  & 0   & 0   & 1 & 1 & \hspace{1.5ex} M
\end{array} 
\]
and $\Amp B_7 = \langle L, M \rangle$, and that the blow-up morphism
$B_7 \to \PP^3$ is given (contravariantly) by $[x_0, x_1,x_2,x_3]
\mapsto [x_0, s_1 x, s_2 x, s_3 x]$. The proper transform of the curve
$\Gamma$ is the curve $\Gamma^\prime$ defined by the condition:
\[
\rk
\begin{pmatrix}
  x_0   & s_1 & s_2\\
  x s_1 & s_2 & s_3
\end{pmatrix}
<2 
\]
Now apply Lemma~\ref{lem:blowups} with $G = B_7$, $V = M^{-1} \oplus
L^{-1} \oplus L^{-1}$, $W = \cO_G \oplus \cO_G$, and the map $f\colon V \to
W$ given by the matrix:
\[
\begin{pmatrix}
  x_0   & s_1 & s_2\\
  x s_1 & s_2 & s_3
\end{pmatrix}
\]

\subsection*{The quantum period:}
Let $p_1$, $p_2, p_3 \in H^\bullet(F;\ZZ)$ denote the first Chern
classes of $L$, $M$, and $N$ respectively; these classes form a basis
for $H^2(F;\ZZ)$.  Write $\tau \in H^2(F;\QQ)$ as $\tau = \tau_1 p_1 +
\tau_2 p_2+ \tau_3 p_3$ and identify the group ring $\QQ[H_2(F;\ZZ)]$
with the polynomial ring $\QQ[Q_1,Q_2,Q_3]$ via the $\QQ$-linear map
that sends the element $Q^\beta \in \QQ[H_2(F;\ZZ)]$ to $Q_1^{\langle
  \beta,p_1\rangle} Q_2^{\langle \beta,p_2\rangle} Q_3^{\langle
  \beta,p_3\rangle}$.  We have:
\begin{multline*}
  I_F(\tau)  = e^{\tau/z} 
  \sum_{l, m, n\geq 0} 
  \frac{
    Q_1^l Q_2^m Q_3^n e^{l \tau_1} e^{m \tau_2} e^{n \tau_3}
  }
  {
    \prod_{k=1}^l (p_1 + k z)^3
    \prod_{k=1}^m (p_2 + k z)
  }
  \frac{\prod_{k = -\infty}^0 (p_2-p_1 + k z)}
  {\prod_{k = -\infty}^{m-l} (p_2-p_1 + k z)}
  \\ \times
  \frac{\prod_{k = -\infty}^0 (p_3-p_2 + k z)}
  {\prod_{k = -\infty}^{n-m} (p_3-p_2 + k z)}
  \frac{\prod_{k = -\infty}^0 (p_3-p_1 + k z)^2}
  {\prod_{k = -\infty}^{n-l} (p_3-p_1 + k z)^2}
\end{multline*}
and, since $I_F(\tau) = 1 + \tau z^{-1} + O(z^{-2})$,
Theorem~\ref{thm:toric_mirror} gives:
\[
J_F(\tau) = I_F(\tau)
\]
We now proceed exactly as in the case of 3--1
(\S\ref{sec:restriction_of_nef_not_restriction_of_ample}), obtaining:
\[
G_X(t) = 
e^{-t} \sum_{l=0}^\infty \sum_{n=l}^\infty \sum_{m=l}^n
t^{m+n}
\frac
{n! n!}
{(l!)^3 m! (m-l)! (n-m)! \big((n-l)!\big)^2}
\]
Regularizing gives:
\[
\hG_X(t) = 1+4 t^2+18 t^3+84 t^4+540 t^5+3190 t^6+20160 t^7+130900 t^8+859320 t^9 + \cdots
\]

\subsection*{Minkowski period sequence:} \href{http://www.grdb.co.uk/search/period3?id=42&printlevel=2}{42}


\addtocounter{CustomSectionCounter}{1}

\section{The Fano Manifold $\MM{3}{17}$}
\label{anchor:3--17}

\subsection*{Mori--Mukai name:} 3--17

\subsection*{Mori--Mukai construction:} A nonsingular divisor of tridegree
$(1,1,1)$ on $\PP^1\times \PP^1 \times \PP^2$.

\subsection*{Our construction:} A member $X$ of $|L+M+N|$ on the toric
variety $F$ with weight data:
\[ 
\begin{array}{rrrrrrrrl} 
  1 & 1 & 0 & 0 & 0 & 0 & 0 & \hspace{1.5ex} L\\ 
  0 & 0 & 1 & 1 & 0 & 0 & 0 & \hspace{1.5ex} M\\ 
  0 & 0 & 0 & 0 & 1 & 1 & 1 & \hspace{1.5ex} N 
\end{array}
\]
and $\Amp F = \langle L, M, N \rangle$.

\subsection*{The two constructions coincide:} Obvious.

\subsection*{The quantum period:} 
Corollary~\ref{cor:QL} yields:
\[
G_X(t) = e^{-2t} \sum_{l=0}^\infty \sum_{m=0}^\infty \sum_{n=0}^\infty 
t^{l+m+2n}
\frac{(l+m+n)!}
{(l!)^2 (m!)^2 (n!)^3}
\]
and regularizing gives:
\[
\hG_X(t) = 1+4 t^2+12 t^3+84 t^4+360 t^5+2380 t^6+13440 t^7+83860
t^8+512400 t^9 + \cdots
\]

\subsection*{Minkowski period sequence:} \href{http://www.grdb.co.uk/search/period3?id=39&printlevel=2}{39}


\addtocounter{CustomSectionCounter}{1}

\section{The Fano Manifold $\MM{3}{18}$}
\label{anchor:3--18}

\subsection*{Mori--Mukai name:} 3--18

\subsection*{Mori--Mukai construction:} The blow-up of $\PP^3$ with
centre the disjoint union of a line and a conic.

\subsection*{Our construction:} A member $X$ of $|M+N|$ on the toric
variety $F$ with weight data:
\[
\begin{array}{rrrrrrrl} 
  \multicolumn{1}{c}{s_0} & 
  \multicolumn{1}{c}{s_1} & 
  \multicolumn{1}{c}{x} & 
  \multicolumn{1}{c}{x_2} & 
  \multicolumn{1}{c}{x_3} & 
  \multicolumn{1}{c}{y_0} &  
  \multicolumn{1}{c}{y_1} &  \\
  \cmidrule{1-7}
  1 & 1 & -1 & 0 & 0 &   0 & 0 & \hspace{1.5ex} L\\ 
  0 & 0 &   1 & 1 & 1 & -1 & 0 & \hspace{1.5ex} M\\ 
  0 & 0 &   0 & 0 & 0 &   1 & 1 & \hspace{1.5ex} N\\ 
\end{array}
\]
and $\Amp F = \langle L, M, N\rangle$. The secondary fan of $F$ is the
same as that of the toric variety in 3--4 (\S\ref{sec:3-4}) and it is
shown schematically in Fig.~\ref{fig:3-4}.  We have:
\begin{itemize}
\item $-K_F=L+2M+2N$ ample, that is $F$ is a Fano variety;
\item $X\sim M+N$ is nef;
\item $-(K_F+X)\sim L+ M+N$ is ample.
\end{itemize}

\subsection*{The two constructions coincide:} We construct $X$, for
example, as the blow-up of $\PP^3_{x_0,x_1,x_2,x_3}$ along the
(disjoint) union of the line $(x_0=x_1=0)$ and the conic
$(x_0x_1+x_2^2=x_3=0)$. Thus $X$ is given in $F$ by the equation:
\[
y_0(s_0s_1x^2+x_2^2)+y_1x_3=0 
\]
where the morphism $F\to \PP^3$ is given (contravariantly) by:
\[
[x_0,x_1,x_2,x_3]\mapsto [s_0x,s_1x,x_2,x_3] 
\]

\subsection*{The quantum period:}
Corollary~\ref{cor:QL} yields:
\[
G_X(t) = e^{-t} \sum_{l=0}^\infty \sum_{m=l}^\infty \sum_{n=m}^\infty
t^{l+m+n}
\frac{(m+n)!}
{(l!)^2 (m-l)! (m!)^2 (n-m)! n!}
\]
and regularizing gives:
\[
\hG_X(t) = 1+4 t^2+18 t^3+60 t^4+480 t^5+2470 t^6+14280 t^7+94780 t^8+564480 t^9 + \cdots
\]

\subsection*{Minkowski period sequence:} \href{http://www.grdb.co.uk/search/period3?id=41&printlevel=2}{41}


\addtocounter{CustomSectionCounter}{1}

\section{The Fano Manifold $\MM{3}{19}$}
\label{sec:3-19}
\label{anchor:3--19}

\subsection*{Mori--Mukai name:} 3--19

\subsection*{Mori--Mukai construction:} The blow-up of a quadric
\mbox{3-fold} $Q\subset \PP^4$ with centre two points $P_1$ and $P_2$
on it which are not collinear.

\subsection*{Our construction:} A member $X$ of $|2M|$ in the rank 2
toric variety $F$ with weight data:
\[ 
\begin{array}{rrrrrrl} 
  \multicolumn{1}{c}{s_0} & 
  \multicolumn{1}{c}{s_1} & 
  \multicolumn{1}{c}{s_2} & 
  \multicolumn{1}{c}{x} & 
  \multicolumn{1}{c}{x_3} & 
  \multicolumn{1}{c}{x_4} & \\ 
  \cmidrule{1-6}
  1 & 1 & 1 & -1 & 0 & 0  & \hspace{1.5ex} L\\ 
  0 & 0 & 0 &   1 & 1 & 1 & \hspace{1.5ex} M \\
\end{array}
\]
and $\Amp(F) = \langle L, M \rangle$.  We have: 
\begin{itemize}
\item $-K_F=2L+3M$ is ample, that is $F$ is a Fano variety;
\item $X\sim 2M$ is nef;
\item $-(K_F+X)\sim 2L+M$ is ample.
\end{itemize}

\subsection*{The two constructions coincide:} The variety $F$ is
manifestly the blow-up of $\PP^4_{x_0,x_1,x_2,x_3,x_4}$ along the line
$(x_0=x_1=x_2=0)$, and $X$ is the strict transform of a general
quadric.

\subsection*{The quantum period:}
Corollary~\ref{cor:QL} yields:
\[
G_X(t) = e^{-2t} \sum_{l=0}^\infty \sum_{m=l}^\infty
t^{2l+m}
\frac{(2m)!}
{(l!)^3 (m-l)! (m!)^2 }
\]
and regularizing gives:  
\[
\hG_X(t) = 1 + 2 t^2 + 12 t^3 + 54 t^4 + 240 t^5 + 1280 t^6 + 7560 t^7 + 
 42070 t^8 + 235200 t^9 + \cdots
\]

\subsection*{Minkowski period sequence:} \href{http://www.grdb.co.uk/search/period3?id=18&printlevel=2}{18}


\addtocounter{CustomSectionCounter}{1}

\section{The Fano Manifold $\MM{3}{20}$}
\label{sec:3-20}
\label{anchor:3--20}

\subsection*{Mori--Mukai name:} 3--20

\subsection*{Mori--Mukai construction:} The blow-up of a quadric
\mbox{3-fold} $Q\subset \PP^4$ with centre two disjoint lines on it.

\subsection*{Our construction:} A member $X$ of $|L+M|$ in the toric
variety $F$ with weight data:
\[
\begin{array}{rrrrrrrl} 
  \multicolumn{1}{c}{s_0} & 
  \multicolumn{1}{c}{s_1} & 
  \multicolumn{1}{c}{t_2} & 
  \multicolumn{1}{c}{t_3} & 
  \multicolumn{1}{c}{u_4} & 
  \multicolumn{1}{c}{x} &  
  \multicolumn{1}{c}{y} &  \\
  \cmidrule{1-7}
  1 & 1 & 0 & 0 & 1 &   -1 & 0 & \hspace{1.5ex} L\\ 
  0 & 0 & 1 & 1 & 1 & 0 & -1 & \hspace{1.5ex} M\\ 
  0 & 0 & 0 & 0 & -1 &   1 & 1 & \hspace{1.5ex} N\\ 
\end{array}
\]
and $\Amp F = \langle L, M, N\rangle$. The secondary fan of $F$ is the
same as that for $F$ in 3--16; it is shown schematically in Fig.~\ref{fig:3-20}.

\begin{figure}[h!]
  \centering
  \resizebox{6cm}{!}{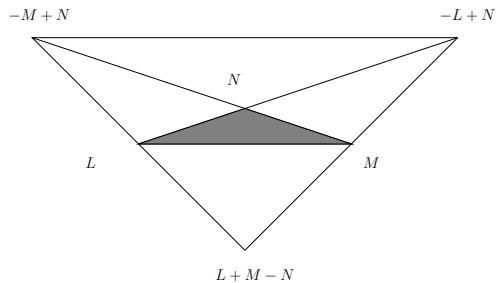}
  \caption{The secondary fan for $F$ in 3--20}
  \label{fig:3-20}
\end{figure}

We have:
\begin{itemize}
\item $-K_F=2L+2M+N$ is ample, that is $F$ is a Fano variety;
\item $X\sim L+M$ is nef;
\item $-(K_F+X)\sim L+ M+N$ is ample.
\end{itemize}

\subsection*{The two constructions coincide:} We blow up the disjoint
union of the line $(x_2=x_3=x_4=0)$ and the line $(x_0=x_1=x_4=0)$ in
$\PP^4_{x_0,x_1,x_2,x_3,x_4}$ and take $X$ to be the proper transform
of the quadric $x_0x_3+x_1x_2+x_4^2=0$ constructed to contain the two
lines. The morphism $F\to \PP^4$ is given (contravariantly) by:
\[
[x_0,x_1,x_2,x_3,x_4]\mapsto [s_0x,s_1x,t_2y,t_3y,u_4xy] 
\]

\subsection*{The quantum period:}
Corollary~\ref{cor:QL} yields:
\[
G_X(t) = \sum_{l=0}^\infty \sum_{m=0}^\infty \sum_{n=\max(l,m)}^{l+m}
t^{l+m+n}
\frac{(l+m)!}
{(l!)^2 (m!)^2 (l+m-n)! (n-l)! (n-m)! }
\]
and regularizing gives:  
\[
\hG_X(t) = 1+4 t^2+12 t^3+60 t^4+360 t^5+1660 t^6+10920 t^7+57820 t^8+361200 t^9 + \cdots
\]

\subsection*{Minkowski period sequence:} \href{http://www.grdb.co.uk/search/period3?id=38&printlevel=2}{38}


\addtocounter{CustomSectionCounter}{1}

\section{The Fano Manifold $\MM{3}{21}$}
\label{anchor:3--21}

\subsection*{Mori--Mukai name:} 3--21

\subsection*{Mori--Mukai construction:} The blow-up of $\PP^1\times
\PP^2$ with centre a curve of bidegree $(2,1)$.

\subsection*{Our construction:} A member $X$ of $|M+N|$ on the toric
variety $F$ with weight data:
\[
\begin{array}{rrrrrrrl} 
  \multicolumn{1}{c}{x_0} & 
  \multicolumn{1}{c}{x_1} & 
  \multicolumn{1}{c}{y_0} & 
  \multicolumn{1}{c}{y_1} & 
  \multicolumn{1}{c}{y_2} & 
  \multicolumn{1}{c}{s} &  
  \multicolumn{1}{c}{t} &  \\
  \cmidrule{1-7}
  1 & 1 & 0 & 0 & 0 & 0 & -1 & \hspace{1.5ex} L\\ 
  0 & 0 & 1 & 1 & 1 & 0 & -1 & \hspace{1.5ex} M\\ 
  0 & 0 & 0 & 0 & 0 & 1 & 1 & \hspace{1.5ex} N\\ 
\end{array}
\]
and $\Amp F = \langle L, M, N\rangle$. The secondary fan of $F$ is the
same as that of the toric variety in 3--2 (\S\ref{sec:3-2}) and it is
shown schematically in Fig.~\ref{fig:3-2}.  We have:
\begin{itemize}
\item $-K_F=L+2M+2N$ is ample, that is $F$ is a Fano variety;
\item $X\sim M+N$ is nef;
\item $-(K_F+X)\sim L+ M+N$ is ample.
\end{itemize}

\subsection*{The two constructions coincide:} A complete intersection
of type $(0,1) \cap (1,2)$ on $\PP^1 \times \PP^2$ is a curve of
bidegree (2,1).  Apply Lemma~\ref{lem:blowups} with $G =
\PP^1_{x_0,x_1} \times \PP^2_{y_0,y_1,y_2}$, $V = \cO_{\PP^1 \times
  \PP^2}\oplus \cO_{\PP^1 \times \PP^2}(-1,-1)$, $W = \cO_{\PP^1 \times
  \PP^2}(0,1)$, and $f\colon V \to W$ given by the matrix $
\begin{pmatrix}
  y_0 & x_0q_0+x_1q_1
\end{pmatrix}$ where $q_0$, $q_1$ are homogeneous quadratic
polynomials in $y_0,y_1,y_2$.

\subsection*{The quantum period:}
Corollary~\ref{cor:QL} yields:
\[
G_X(t) = e^{-t} \sum_{l=0}^\infty \sum_{m=0}^\infty \sum_{n=l+m}^\infty
t^{l+m+n}
\frac{(m+n)!}
{(l!)^2 (m!)^3 n! (n-l-m)! }
\]
and regularizing gives:  
\[
\hG_X(t) = 1+6 t^2+6 t^3+114 t^4+240 t^5+3030 t^6+9660 t^7+95970 t^8+394800 t^9 + \cdots
\]

\subsection*{Minkowski period sequence:} \href{http://www.grdb.co.uk/search/period3?id=49&printlevel=2}{49}


\addtocounter{CustomSectionCounter}{1}

\section{The Fano Manifold $\MM{3}{22}$}
\label{anchor:3--22}

\subsection*{Mori--Mukai name:} 3--22

\subsection*{Mori--Mukai construction:} The blow-up of $\PP^1\times
\PP^2$ with centre a conic in $t\times \PP^2$ ($t\in \PP^1$). 

\subsection*{Our construction:} A member $X$ of $|N|$ on the toric
variety $F$ with weight data:
\[
\begin{array}{rrrrrrrl} 
  \multicolumn{1}{c}{x_0} & 
  \multicolumn{1}{c}{x_1} & 
  \multicolumn{1}{c}{y_0} & 
  \multicolumn{1}{c}{y_1} & 
  \multicolumn{1}{c}{y_2} & 
  \multicolumn{1}{c}{s} &  
  \multicolumn{1}{c}{t} &  \\
  \cmidrule{1-7}
  1 & 1 & 0 & 0 & 0 & -1 & 0 & \hspace{1.5ex} L\\ 
  0 & 0 & 1 & 1 & 1 & 0 & -2 & \hspace{1.5ex} M\\ 
  0 & 0 & 0 & 0 & 0 & 1 & 1 & \hspace{1.5ex} N\\ 
\end{array}
\]
and $\Amp F = \langle L, M, N\rangle$. The secondary fan of $F$ is
similar to that of the toric variety in 3--10 (\S\ref{sec:3-10}).  We
have:
\begin{itemize}
\item $-K_F=L+M+2N$ is ample, that is $F$ is a Fano variety;
\item $X\sim N$ is nef;
\item $-(K_F+X)\sim L+ M+N$ is ample.
\end{itemize}
\subsection*{The two constructions coincide:} Apply
Lemma~\ref{lem:blowups} with $G = \PP^1_{x_0,x_1} \times
\PP^2_{y_0,y_1,y_2}$, $V = \cO_{\PP^1 \times \PP^2}(-1,0)\oplus
\cO_{\PP^1 \times \PP^2}(0,-2)$, $W = \cO_{\PP^1 \times \PP^2}$, and
$f\colon V \to W$ given by the matrix $
\begin{pmatrix}
  x_0-t x_1 & y_0y_2-y_1^2
\end{pmatrix}$.

\subsection*{The quantum period:}
Corollary~\ref{cor:QL} yields:
\[
G_X(t) = e^{-t} \sum_{l=0}^\infty \sum_{m=0}^\infty \sum_{n=\max(l,2m)}^\infty
t^{l+m+n}
\frac{n!}
{(l!)^2 (m!)^3 (n-l)! (n-2m)! }
\]
and regularizing gives:  
\[
\hG_X(t) = 1+2 t^2+6 t^3+54 t^4+180 t^5+830 t^6+4620 t^7+26950 t^8+140280 t^9+ \cdots
\]

\subsection*{Minkowski period sequence:} \href{http://www.grdb.co.uk/search/period3?id=13&printlevel=2}{13}


\addtocounter{CustomSectionCounter}{1}

\section{The Fano Manifold $\MM{3}{23}$}
\label{anchor:3--23}

\subsection*{Mori--Mukai name:} 3--23

\subsection*{Mori--Mukai construction:} The blow-up of $B_7$ (see
\S\ref{sec:B7}) with centre a conic passing through the centre of the
blow-up $B_7\to \PP^3$.

\subsection*{Our construction:} A member $X$ of $|L+N|$ in the toric
variety $F$ with weight data:
\[
\begin{array}{rrrrrrrl} 
  \multicolumn{1}{c}{s_1} & 
  \multicolumn{1}{c}{s_2} & 
  \multicolumn{1}{c}{s_3} & 
  \multicolumn{1}{c}{x} & 
  \multicolumn{1}{c}{x_0} & 
  \multicolumn{1}{c}{u} &  
  \multicolumn{1}{c}{v} &  \\
  \cmidrule{1-7}
  1 & 1 & 1 & -1 & 0 & 0 & 0 & \hspace{1.5ex} L\\ 
  0 & 0 & 0 & 1 & 1 & -1 & 0 & \hspace{1.5ex} M\\ 
  0 & 0 & 0 & 0 & 0 & 1 & 1 & \hspace{1.5ex} N\\ 
\end{array}
\]
and $\Amp F = \langle L, M, N\rangle$. The secondary fan for $F$ is
the same as that of the toric variety in 3--4 (\S\ref{sec:3-4}) and is
shown in Fig.~\ref{fig:3-4}.

We have:
\begin{itemize}
\item $-K_F=2L+M+2N$ is ample, that is $F$ is a Fano variety;
\item $X\sim L+N$ is nef;
\item $-(K_F+X)\sim L+M+N$ is ample.
\end{itemize}

\subsection*{The two constructions coincide:} Consider the conic
$\Gamma$ given by $(x_3=x_0x_1+x_2^2=0)$ in $\PP^3_{x_0, \dots, x_3}$,
and note that $P = [1:0:0:0]$ lies on $\Gamma$. Recall from
\S\ref{sec:B7} that $B_7$ is the toric variety with weight data:
\[
\begin{array}{rrrrrl} 
  \multicolumn{1}{c}{s_1} & 
  \multicolumn{1}{c}{s_2} & 
  \multicolumn{1}{c}{s_3} & 
  \multicolumn{1}{c}{x} & 
  \multicolumn{1}{c}{x_0} \\
  \cmidrule{1-5}
  1  & 1   & 1   &-1 & 0 & \hspace{1.5ex} L\\ 
  0  & 0   & 0   & 1 & 1 & \hspace{1.5ex} M
\end{array} 
\]
and $\Amp B_7 = \langle L, M \rangle$, and that the blow-up morphism
$B_7 \to \PP^3$ is given (contravariantly) by $[x_0, x_1,x_2,x_3]
\mapsto [x_0, s_1 x, s_2 x, s_3 x]$.  The proper transform of the
curve $\Gamma$ is the curve $\Gamma^\prime$ defined by the equations:
\[
s_3=x_0s_1+xs_2^2=0
\]
Now apply Lemma~\ref{lem:blowups} with $G = B_7$, $V = M^{-1} \oplus
\cO_G$, $W = L$, and the map $f\colon V \to W$ given by the matrix $
\begin{pmatrix}
  x_0s_1+xs_2^2 & s_3
\end{pmatrix}
$.

\subsection*{The quantum period:}
Corollary~\ref{cor:QL} yields:
\[
G_X(t) = e^{-t} \sum_{l=0}^\infty \sum_{m=l}^\infty \sum_{n=m}^\infty
t^{l+m+n}
\frac{(l+n)!}
{(l!)^3 (m-l)!m! (n-m)! n! }
\]
and regularizing gives:  
\[
\hG_X(t) = 1+2 t^2+12 t^3+30 t^4+180 t^5+920 t^6+4200 t^7+22750 t^8+121800 t^9+ \cdots
\]

\subsection*{Minkowski period sequence:} \href{http://www.grdb.co.uk/search/period3?id=17&printlevel=2}{17}


\addtocounter{CustomSectionCounter}{1}

\section{The Fano Manifold $\MM{3}{24}$}
\label{anchor:3--24}

\subsection*{Mori--Mukai name:} 3--24

\subsection*{Mori--Mukai construction:} The fibre product
$W\times_{\PP^2} \FF_1$, where $W\to \PP^2$ is a $\PP^1$-bundle and
$p\colon \FF_1\to \PP^2$ is the blow-up.  Here $W$ (see \S\ref{sec:W})
is a divisor of bidegree $(1,1)$ on $\PP^2 \times \PP^2$.

\subsection*{Our construction:} A member $X$ of $|M+N|$ on the toric
variety $\FF_1\times \PP^2$, where $M$ is the line bundle $p^\star
\cO(1)$ on $\FF_1$, and $N=\cO(1)$.  In other words, $X$ is a member
of $|M+N|$ on the toric variety $F$ with weight data:
\[
\begin{array}{rrrrrrrl} 
  \multicolumn{1}{c}{s_0} & 
  \multicolumn{1}{c}{s_1} & 
  \multicolumn{1}{c}{x} & 
  \multicolumn{1}{c}{x_2} & 
  \multicolumn{1}{c}{y_0} & 
  \multicolumn{1}{c}{y_1} &  
  \multicolumn{1}{c}{y_2} &  \\
  \cmidrule{1-7}
  1 & 1 & -1 & 0 & 0 & 0 & 0 & \hspace{1.5ex} L\\ 
  0 & 0 & 1 & 1 & 0 & 0 & 0 & \hspace{1.5ex} M\\ 
  0 & 0 & 0 & 0 & 1 & 1 & 1 & \hspace{1.5ex} N\\ 
\end{array}
\]
and $\Amp F = \langle L, M, N\rangle$.  We have:
\begin{itemize}
\item $-K_F=L+2M+3N$ is ample, that is $F$ is a Fano variety;
\item $X\sim M+N$ is nef;
\item $-(K_F+X)\sim L+M+2N$ is ample.
\end{itemize}
 
\subsection*{The two constructions coincide:} First we show that $X$
is the blow up of $\PP^1\times \PP^2$ along a curve of bidegree
$(1,1)$.  To see this, note first that $X$ is cut out of
$\PP^2_{x_0,x_1,x_2}\times \PP^2_{y_0,y_1,y_2}\times \PP^1_{s_0,s_1}$
by the equations:
\[
\begin{cases}
  y_0x_0+y_1x_1+y_2x_2=0\\
 s_0x_0+s_1x_1=0
\end{cases}
\]
The first equation here cuts $W$ out of $\PP^2_{x_0,x_1,x_2}\times
\PP^2_{y_0,y_1,y_2}$; the second equation cuts $\FF_1$ out of
$\PP^2_{x_0,x_1,x_2}\times \PP^1_{s_0,s_1} $, as it is the equation
defining the blow-up of $\PP^2$ at the point $[0:0:1]$.  We now
exhibit $X$ as the blow-up of a curve in $\PP^2_{y_0,y_1,y_2}\times
\PP^1_{s_0,s_1}$. The projection to $\PP^2_{y_0,y_1,y_2}\times
\PP^1_{s_0,s_1}$ is an isomorphism away from the locus where the
matrix
\[
\begin{pmatrix}
  y_0 & y_1 & y_2 \\
  s_0 & s_1 & 0
\end{pmatrix}
\]
drops rank.  This locus is:
\[
\begin{cases}
  y_2=0\\
  y_0s_1-y_1s_0=0
\end{cases}
\]
i.e. a curve in of bidegree $(1,1)$ as claimed.  We can further
simplify things by writing $X$ as a hypersurface in $\FF_1\times
\PP^2$: the two equations defining $X$ (given above) reduce to the
single equation:
\[
s_0xy_0+s_1xy_1+x_2y_2=0 
\]
in $\FF^1 \times\PP^2$.

\subsection*{The quantum period:}
Corollary~\ref{cor:QL} yields:
\[
G_X(t) = e^{-t} \sum_{l=0}^\infty \sum_{m=l}^\infty \sum_{n=0}^\infty
t^{l+m+2n}
\frac{(m+n)!}
{(l!)^2 (m-l)!m! (n!)^3 }
\]
and regularizing gives:  
\[
\hG_X(t) = 1+4 t^2+6 t^3+60 t^4+180 t^5+1210 t^6+5460 t^7+30940 t^8+165480 t^9+ \cdots
\]

\subsection*{Minkowski period sequence:} \href{http://www.grdb.co.uk/search/period3?id=31&printlevel=2}{31}


\addtocounter{CustomSectionCounter}{1}
\section{The Fano Manifold $\MM{3}{25}$}
\label{sec:3-25}
\label{anchor:3--25}

\subsection*{Mori--Mukai name:} 3--25

\subsection*{Mori--Mukai construction:} The blow-up of $\PP^3$ with
centre two disjoint lines; equivalently\footnote{Note that Mori--Mukai
  use different weight conventions for projective bundles than we
  do.}, $\PP\bigl(\cO(1,0)\oplus \cO(0,1) \bigr)$ over $\PP^1\times
\PP^1$.

\subsection*{Our construction:} The toric variety $X$
with weight data:
\[
\begin{array}{rrrrrrl} 
  \multicolumn{1}{c}{s_0} & 
  \multicolumn{1}{c}{s_1} & 
  \multicolumn{1}{c}{t_2} & 
  \multicolumn{1}{c}{t_3} & 
  \multicolumn{1}{c}{x} & 
  \multicolumn{1}{c}{y} &  \\
  \cmidrule{1-6}
  1 & 1 & 0 & 0 & -1 & 0 & \hspace{1.5ex} L\\ 
  0 & 0 & 1 & 1 & 0 & -1 & \hspace{1.5ex} M\\ 
  0 & 0 & 0 & 0 & 1 & 1 & \hspace{1.5ex} N\\ 
\end{array}
\]
and $\Amp X = \langle L, M, N\rangle$.  

\subsection*{The two constructions coincide:} 
The morphism $X \to \PP^3$ that sends (contravariantly) the
homogeneous co-ordinate functions $[x_0,x_1,x_2,x_3]$ to
$[s_0x,s_1x,t_2y,t_3y]$ manifestly blows up the union of the line
$(x_0=x_1=0)$ and the line $(x_2=x_3=0)$.  These lines are disjoint.

\subsection*{The quantum period:}
Corollary~\ref{cor:toric_mirror} yields:
\[
G_X(t) = \sum_{l=0}^\infty \sum_{m=0}^\infty \sum_{n=\max(l,m)}^\infty
\frac{t^{l+m+2n}}
{(l!)^2 (m!)^2 (n-l)! (n-m)!}
\]
and regularizing gives:
\[
\hG_X(t) = 1+2 t^2+12 t^3+30 t^4+120 t^5+920 t^6+3360 t^7+16030 t^8+99120 t^9+ \cdots
\]

\subsection*{Minkowski period sequence:} \href{http://www.grdb.co.uk/search/period3?id=16&printlevel=2}{16}


\addtocounter{CustomSectionCounter}{1}

\section{The Fano Manifold $\MM{3}{26}$}
\label{anchor:3--26}

\subsection*{Mori--Mukai name:} 3--26

\subsection*{Mori--Mukai construction:} The blow-up of $\PP^3$ with
centre a disjoint union of a point and a line.

\subsection*{Our construction:} The toric variety $X$
with weight data:
\[
\begin{array}{rrrrrrl} 
  \multicolumn{1}{c}{s_0} & 
  \multicolumn{1}{c}{s_1} & 
  \multicolumn{1}{c}{t_2} & 
  \multicolumn{1}{c}{u_3} & 
  \multicolumn{1}{c}{x} & 
  \multicolumn{1}{c}{y} &  \\
  \cmidrule{1-6}
  1 & 1 & 0 & 1 & -1 & 0 & \hspace{1.5ex} L\\ 
  0 & 0 & 1 & 1 & 0 & -1 & \hspace{1.5ex} M\\ 
  0 & 0 & 0 & -1 & 1 & 1 & \hspace{1.5ex} N\\ 
\end{array}
\]
and $\Amp X = \langle L, M, N\rangle$. The secondary fan of $X$ is the
same as that of the toric variety in 3--20 (\S\ref{sec:3-20}) and it
is shown in Fig.~\ref{fig:3-20}.

\subsection*{The two constructions coincide:} The morphism to $\PP^3$
is given by the complete linear system $|N|$ on $X$; it sends
(contravariantly) the homogeneous co-ordinates $[x_0,x_1,x_2,x_3]$ to
$[s_0x,s_1x,t_2y,u_3xy]$.  The divisor $(x=0)\subset X$ contracts to
the point $[0:0:1:0]\in \PP^3$, and the divisor $(y=0)\subset X$ contracts
to the line $(x_2=x_3=0)\subset \PP^3$.
\subsection*{The quantum period:}
Corollary~\ref{cor:toric_mirror} yields:
\[
G_X(t) = \sum_{l=0}^\infty \sum_{m=0}^\infty \sum_{n = \max(l,m)}^{l+m}
\frac{t^{2l+m+n}}
{(l!)^2 m! (l+m-n)! (n-l)! (n-m)!}
\]
and regularizing gives:
\[
\hG_X(t) = 1+2 t^2+6 t^3+30 t^4+120 t^5+470 t^6+2520 t^7+10990
t^8+57120 t^9 + \cdots
\]

\subsection*{Minkowski period sequence:} \href{http://www.grdb.co.uk/search/period3?id=12&printlevel=2}{12}


\addtocounter{CustomSectionCounter}{1}

\section{The Fano Manifold $\MM{3}{27}$}
\label{anchor:3--27}

\subsection*{Mori--Mukai name:} 3--27

\subsection*{Mori--Mukai construction:} $\PP^1\times \PP^1 \times \PP^1$.

\subsection*{Our construction:} 
The toric variety $X$ with weight data:
\[
\begin{array}{rrrrrrl}
  1 & 1 & 0 & 0 & 0 & 0 & \hspace{1.5ex} L\\ 
  0 & 0 & 1 & 1 & 0 & 0 & \hspace{1.5ex} M\\ 
  0 & 0 & 0 & 0 & 1 & 1 & \hspace{1.5ex} N\\ 
\end{array}
\]
and $\Amp(X) = \langle L, M, N \rangle$.

\subsection*{The two constructions coincide:}
Obvious.

\subsection*{The quantum period:} 
Corollary~\ref{cor:toric_mirror} yields:
\[
G_X(t) = \sum_{l=0}^\infty \sum_{m=0}^\infty \sum_{n=0}^\infty 
\frac{t^{2l+2m+2n}}
{(l!)^2 (m!)^2 (n!)^2}
\]
and regularizing gives:
\[
\hG_X(t) = 1+6 t^2+90 t^4+1860 t^6+44730 t^8+1172556 t^{10} + \cdots
\]

\subsection*{Minkowski period sequence:} \href{http://www.grdb.co.uk/search/period3?id=45&printlevel=2}{45}


\addtocounter{CustomSectionCounter}{1}

\section{The Fano Manifold $\MM{3}{28}$}
\label{anchor:3--28}

\subsection*{Mori--Mukai name:} 3--28

\subsection*{Mori--Mukai construction:} $\PP^1\times \FF_1$.

\subsection*{Our construction:} The toric variety $X$ with weight
data:
\[
\begin{array}{rrrrrrl}
  1 & 1 & 0 & 0 & 0 & 0 & \hspace{1.5ex} L\\ 
  0 & 0 & 1 & 1 & -1 & 0 & \hspace{1.5ex} M\\ 
  0 & 0 & 0 & 0 & 1 & 1 & \hspace{1.5ex} N\\ 
\end{array}
\]
and $\Amp(X) = \langle L, M, N \rangle$.

\subsection*{The two constructions coincide:} Obvious.

\subsection*{The quantum period:}
Corollary~\ref{cor:toric_mirror} yields:
\[
G_X(t) = \sum_{l=0}^\infty \sum_{m=0}^\infty \sum_{n=m}^\infty
\frac{t^{2l+m+2n}}
{(l!)^2 (m!)^2 (n-m)!n!}
\]
and regularizing gives:
\[
\hG_X(t) = 1+4 t^2+6 t^3+36 t^4+180 t^5+490 t^6+4200 t^7+11620 t^8+89880 t^9 + \cdots
\]

\subsection*{Minkowski period sequence:} \href{http://www.grdb.co.uk/search/period3?id=28&printlevel=2}{28}


\addtocounter{CustomSectionCounter}{1}

\section{The Fano Manifold $\MM{3}{29}$}
\label{anchor:3--29}

\subsection*{Mori--Mukai name:} 3--29

\subsection*{Mori--Mukai construction:} The blow-up of $B_7$ (see
\S\ref{sec:B7}) with centre a line on the exceptional divisor $D\cong
\PP^2$ of the blow-up $B_7\to \PP^3$.

\subsection*{Our construction:} The toric variety $X$ with weight
data:
\[
\begin{array}{rrrrrrl} 
  \multicolumn{1}{c}{x_0} & 
  \multicolumn{1}{c}{s_1} & 
  \multicolumn{1}{c}{s_2} & 
  \multicolumn{1}{c}{t_3} & 
  \multicolumn{1}{c}{x} & 
  \multicolumn{1}{c}{y} &  \\
  \cmidrule{1-6}
  1 & 0 & 0 & -1 & 0 & 1 & \hspace{1.5ex} L\\ 
  0 & 1 & 1 & 0 & -2 & 1 & \hspace{1.5ex} M\\ 
  0 & 0 & 0 & 1 & 1 & -1 & \hspace{1.5ex} N\\ 
\end{array}
\]
and $\Amp X = \langle L, M, N\rangle$. The secondary fan of $X$ is
shown schematically in Fig.~\ref{fig:3-29}.

\begin{figure}[h!]
  \centering
  \resizebox{6cm}{!}{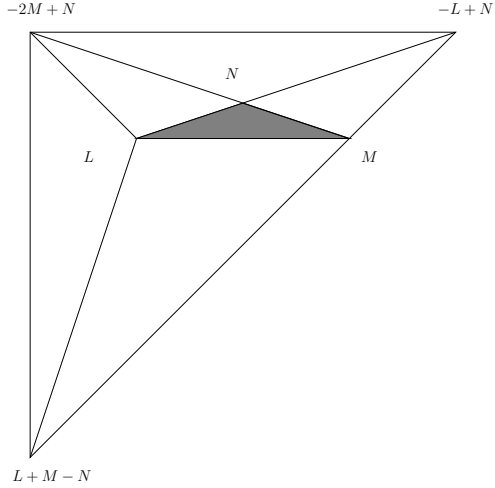}
  \caption{The secondary fan for $X$ in 3--29}
  \label{fig:3-29}
\end{figure}

\subsection*{The two constructions coincide:} The morphism $X \to \PP^3$
sends (contravariantly) the homogeneous co-ordinate functions
$[x_0,x_1,x_2,x_3]$ to $[x_0,s_1xy,s_2xy,t_3xy^2]$.

\subsection*{The quantum period:}
Corollary~\ref{cor:toric_mirror} yields:
\[
G_X(t) = \sum_{l=0}^\infty \sum_{m=0}^\infty \sum_{n =\max(l,2m)}^{l+m} 
\frac{t^{l+m+n}}
{l! (m!)^2 (n-l)! (n-2m)!(l+m-n)!}
\]
and regularizing gives:
\[
\hG_X(t) = 1+2 t^2+30 t^4+60 t^5+380 t^6+840 t^7+5950 t^8+22680 t^9+ \cdots
\]

\subsection*{Minkowski period sequence:} \href{http://www.grdb.co.uk/search/period3?id=8&printlevel=2}{8}


\addtocounter{CustomSectionCounter}{1}

\section{The Fano Manifold $\MM{3}{30}$}
\label{anchor:3--30}

\subsection*{Mori--Mukai name:} 3--30

\subsection*{Mori--Mukai construction:} The blow-up of $B_7$ (see
\S\ref{sec:B7}) with centre the strict transform of a line passing
through the centre of the blow-up $B_7\to \PP^3$.

\subsection*{Our construction:} The toric variety $X$ with weight data:
\[
\begin{array}{rrrrrrl} 
  \multicolumn{1}{c}{t_0} & 
  \multicolumn{1}{c}{t_1} & 
  \multicolumn{1}{c}{x} & 
  \multicolumn{1}{c}{s_2} & 
  \multicolumn{1}{c}{y} & 
  \multicolumn{1}{c}{x_3} &  \\
  \cmidrule{1-6}
  1 & 1 & -1 & 0 & 0 & 0 & \hspace{1.5ex} L\\ 
  0 & 0 & 1 & 1 & -1 & 0 & \hspace{1.5ex} M\\ 
  0 & 0 & 0 & 0 & 1 & 1 & \hspace{1.5ex} N\\ 
\end{array}
\]

\subsection*{The two constructions coincide:} The morphism $X \to
\PP^3$ sends (contravariantly) the homogeneous co-ordinate functions
$[x_0,x_1,x_2,x_3]$ to $[t_0xy,t_1xy,s_2y,x_3]$.

\subsection*{The quantum period:}
Corollary~\ref{cor:toric_mirror} yields:
\[
G_X(t) = \sum_{l=0}^\infty \sum_{m=l}^\infty \sum_{n=m}^\infty
\frac{t^{l+m+2n}}
{(l!)^2 (m-l)! m! (n-m)!n!}
\]
and regularizing gives:
\[
\hG_X(t) = 1+2 t^2+6 t^3+30 t^4+60 t^5+470 t^6+1680 t^7+7630 t^8+34440
t^9 + \cdots
\]

\subsection*{Minkowski period sequence:} \href{http://www.grdb.co.uk/search/period3?id=11&printlevel=2}{11}


\addtocounter{CustomSectionCounter}{1}

\section{The Fano Manifold $\MM{3}{31}$}
\label{anchor:3--31}

\subsection*{Mori--Mukai name:} 3--31

\subsection*{Mori--Mukai construction:} The blow-up of the cone over a
nonsingular quadric surface in $\PP^3$ with centre the vertex;
equivalently, the $\PP^1$-bundle $\PP\bigl(\cO\oplus \cO(1,1)\bigr)$
over $\PP^1\times \PP^1$.

\subsection*{Our construction:} The toric variety $X$ with weight
data:
\[
\begin{array}{rrrrrrl} 
  \multicolumn{1}{c}{s_0} & 
  \multicolumn{1}{c}{s_1} & 
  \multicolumn{1}{c}{t_0} & 
  \multicolumn{1}{c}{t_1} & 
  \multicolumn{1}{c}{x} & 
  \multicolumn{1}{c}{y} &  \\
  \cmidrule{1-6}
  1 & 1 & 0 & 0 & -1 & 0 & \hspace{1.5ex} L\\ 
  0 & 0 & 1 & 1 & -1 & 0 & \hspace{1.5ex} M\\ 
  0 & 0 & 0 & 0 & 1 & 1 & \hspace{1.5ex} N\\ 
\end{array}
\]
and $\Amp(X) = \langle L, M, N \rangle$.

\subsection*{The two constructions coincide:} Obvious.

\subsection*{The quantum period:}
Corollary~\ref{cor:toric_mirror} yields:
\[
G_X(t) = \sum_{l=0}^\infty \sum_{m=0}^\infty \sum_{n=l+m}^\infty
\frac{t^{l+m+2n}}
{(l!)^2 (m!)^2 (n-l-m)!n!}
\]
and regularizing gives:
\[
\hG_X(t) = 1+2 t^2+12 t^3+6 t^4+120 t^5+560 t^6+840 t^7+10150 t^8+38640 t^9+ \cdots
\]

\subsection*{Minkowski period sequence:} \href{http://www.grdb.co.uk/search/period3?id=14&printlevel=2}{14}


\addtocounter{CustomSectionCounter}{1}

\section{The Fano Manifold $\MM{4}{1}$}
\label{anchor:4--1}

\subsection*{Mori--Mukai name:} 4--1

\subsection*{Mori--Mukai construction:} A divisor of multidegree
$(1,1,1,1)$ in $\PP^1 \times \PP^1 \times \PP^1 \times \PP^1$.

\subsection*{Our construction:} A member $X$ of $|A+B+C+D|$ in the toric
variety $F$ with weight data:
\[ 
\begin{array}{rrrrrrrrl} 
  1 & 1 & 0 & 0 & 0 & 0 & 0 & 0 & \hspace{1.5ex} A\\ 
  0 & 0 & 1 & 1 & 0 & 0 & 0 & 0 & \hspace{1.5ex} B\\ 
  0 & 0 & 0 & 0 & 1 & 1 & 0 & 0 & \hspace{1.5ex} C\\ 
  0 & 0 & 0 & 0 & 0 & 0 & 1 & 1 & \hspace{1.5ex} D\\ 
\end{array}
\]
and $\Amp(X) = \langle A, B, C, D \rangle$.

\subsection*{The two constructions coincide:} Obvious.

\subsection*{The quantum period:}
Corollary~\ref{cor:QL} yields:
\[
G_X(t) = e^{-4t} \sum_{a=0}^\infty \sum_{b=0}^\infty \sum_{c=0}^\infty \sum_{d=0}^\infty 
t^{a+b+c+d}
\frac{(a+b+c+d)!}
{(a!)^2 (b!)^2 (c!)^2 (d!)^2}
\]
and regularizing gives:
\[
\hG_X(t) = 1+12 t^2+48 t^3+540 t^4+4320 t^5+42240 t^6+403200
t^7+4038300 t^8+40958400 t^9 + \cdots
\]

\subsection*{Minkowski period sequence:} \href{http://www.grdb.co.uk/search/period3?id=111&printlevel=2}{111}


\addtocounter{CustomSectionCounter}{1}
\section{The Fano Manifold $\MM{4}{2}$}
\label{sec:4-2}
\label{anchor:4--2}

\subsection*{Mori--Mukai name:} 4--2 \footnote{Mori and Mukai
  initially missed this variety
  \citelist{\cite{MM:Manuscripta}\cite{MM:fanoconf}}. We put it where
  it belongs in their scheme.}

\subsection*{Mori--Mukai construction:} The blow-up of $\PP^1\times
\PP^1 \times \PP^1$ with centre a curve  of tridegree $(1,1,3)$.  

\subsection*{Our construction:} A member $X$ of $|B+C+D|$ in the toric
variety $F$ with weight data:
\[ 
\begin{array}{rrrrrrrrl} 
  \multicolumn{1}{c}{x_0} & 
  \multicolumn{1}{c}{x_1} & 
  \multicolumn{1}{c}{y_0} & 
  \multicolumn{1}{c}{y_1} & 
  \multicolumn{1}{c}{z_0} &  
  \multicolumn{1}{c}{z_1} &  
  \multicolumn{1}{c}{u} &  
  \multicolumn{1}{c}{v} &  \\
  \cmidrule{1-8}
1 & 1 &  0  &  0  & 0 &  0 & -1 &  0 &  \hspace{1.5ex} A \\ 
0 & 0 &  1  &  1  & 0 &  0 & -1 &  0 &  \hspace{1.5ex} B \\ 
0 & 0 &  0  &  0  & 1 &  1 &  0  &  0 & \hspace{1.5ex} C \\
0 & 0 &  0  &  0  & 0 &  0 &  1  &  1 & \hspace{1.5ex} D 
\end{array}
\]
and $\Amp F = \langle A, B, C, D \rangle$.  We have:
\begin{itemize}
\item $-K_F=A+B+2C+2D$ is ample, that is $F$ is a Fano variety;
\item $X\sim B+C+D$ is nef;
\item $-(K_F+X)\sim A+C+D$ is nef and big but not ample.
\end{itemize}

\subsection*{The two constructions coincide:} The curve is a complete
intersection of type $(1,2,1)\cap (0,1,1)$ in $\PP^1 \times \PP^1
\times \PP^1$, so $X$ is constructed by
applying Lemma~\ref{lem:blowups} with 
\begin{align*}
  & G = \PP^1 \times \PP^1 \times \PP^1 \\
  & V = \cO_{\PP^1 \times \PP^1 \times \PP^1}(-1,-1,0) \oplus
  \cO_{\PP^1 \times \PP^1 \times \PP^1} \\
  & W = \cO_{\PP^1 \times \PP^1 \times \PP^1}(0,1,1)
\end{align*}
and $f\colon V \to W$ given by the matrix $
\begin{pmatrix}
  A & B
\end{pmatrix}
$ where $A \in \Gamma\big(\PP^1 \times \PP^1 \times \PP^1; \cO_{\PP^1
  \times \PP^1 \times \PP^1}(1,2,1)\big)$ and $B \in \Gamma\big(\PP^1
\times \PP^1 \times \PP^1; \cO_{\PP^1 \times \PP^1 \times
  \PP^1}(0,1,1)\big)$ are the sections that define the centre of the
blow-up.

\subsection*{The quantum period:}
Let $p_1$, $p_2$, $p_3$, $p_4 \in H^\bullet(F;\ZZ)$ denote the first
Chern classes of $A$, $B$, $C$, and $D$ respectively; these classes
form a basis for $H^2(F;\ZZ)$.  Write $\tau \in H^2(F;\QQ)$ as $\tau =
\tau_1 p_1 + \tau_2 p_2+ \tau_3 p_3+ \tau_4 p_4$ and identify the
group ring $\QQ[H_2(F;\ZZ)]$ with the polynomial ring
$\QQ[Q_1,Q_2,Q_3,Q_4]$ via the $\QQ$-linear map that sends the element
$Q^\beta \in \QQ[H_2(F;\ZZ)]$ to $Q_1^{\langle \beta,p_1\rangle}
Q_2^{\langle \beta,p_2\rangle} Q_3^{\langle \beta,p_3\rangle}
Q_4^{\langle \beta,p_4\rangle}$.  Theorem~\ref{thm:toric_mirror}
gives:
\begin{multline*}
  J_F(\tau) = e^{\tau/z} 
  \sum_{a,b,c,d \geq 0}
  \frac{
    Q_1^a Q_2^b Q_3^c Q_4^d 
    e^{a \tau_1} e^{b \tau_2} e^{c \tau_3} e^{d \tau_4}
  }
  {
    \prod_{k=1}^a (p_1 + k z)^2
    \prod_{k=1}^b (p_2 + k z)^2
    \prod_{k=1}^c (p_3 + k z)^2
    \prod_{k=1}^d (p_4 + k z)
  } \\
  \times
  \frac{\prod_{k = -\infty}^0 (p_4-p_1-p_2 + k z)}
  {\prod_{k = -\infty}^{d-a-b} (p_4-p_1-p_2 + k z)}
\end{multline*}
and hence:
\begin{multline*}
    I_{\be,E}(\tau) = e^{\tau/z} 
  \sum_{a,b,c,d \geq 0}
  \frac{
    Q_1^a Q_2^b Q_3^c Q_4^d 
    e^{a \tau_1} e^{b \tau_2} e^{c \tau_3} e^{d \tau_4}
    \prod_{k=1}^{b+c+d} (\lambda + p_2+p_3+p_4 + k z)
  }
  {
    \prod_{k=1}^a (p_1 + k z)^2
    \prod_{k=1}^b (p_2 + k z)^2
    \prod_{k=1}^c (p_3 + k z)^2
    \prod_{k=1}^d (p_4 + k z)
  } \\
  \times
  \frac{\prod_{k = -\infty}^0 (p_4-p_1-p_2 + k z)}
  {\prod_{k = -\infty}^{d-a-b} (p_4-p_1-p_2 + k z)}
\end{multline*}
Note that, much as in Example~\ref{ex:ql_with_mirror_map}, we have:
\[
I_{\be,E}(0) = 1 + \Big((Q_3 + Q_4 + 2 Q_3Q_4)1 + (p_4 - p_1 - p_2)
\log(1+Q_2)\Big) z^{-1} + O(z^{-2})
\]
Arguing exactly as in Example~\ref{ex:ql_with_mirror_map}, we find
that:
\[
J_{\be,E}\big((p_4-p_2-p_1)\log(1+Q_2)\big) =
e^{-(Q_3+Q_4+2 Q_2 Q_4)/z} I_{\be,E}(0)
\]
and:
\begin{multline*}
  J_{\be,E}\big((p_3-p_2-p_1)\log(1+Q_2)\big)
  =  \\
  e^{(p_4-p_2-p_1)\log(1+Q_2)/z}
  \Big[ J_{\be,E}(0)\Big]_{Q_1 = {Q_1 \over 1+Q_2}, Q_2 = {Q_2 \over
      1+Q_2}, Q_3=Q_3, Q_4 = Q_4(1+Q_2)}
\end{multline*}

Hence, using the inverse mirror map:
\begin{align*}
  Q_1 = {Q_1 \over 1-Q_2} && 
  Q_2 = {Q_2 \over 1-Q_2} && 
  Q_3 = Q_3 &&
  Q_4 = Q_4(1-Q_2)
\end{align*}
we have that $J_{\be,E}(0)$ is equal to:
\begin{multline*}
  \Big[e^{-(p_4-p_2-p_1)\log(1+Q_2)/z}
  J_{\be,E}\big((p_4-p_2-p_1)\log(1+Q_2)\big)\Big]_{Q_1 = {Q_1 \over 1-Q_2}, Q_2 = {Q_2 \over
      1-Q_2}, Q_3 = Q_3, Q_4=Q_4(1-Q_2)}  \\
  = 
  e^{(p_4-p_2-p_1)\log(1-Q_2)/z}
  \Big[e^{-(Q_3+Q_4+2 Q_2 Q_4)/z} I_{\be,E}(0)\Big]_{Q_1 = {Q_1 \over 1-Q_2}, Q_2 = {Q_2 \over
      1-Q_2}, Q_3 = Q_3, Q_4=Q_4(1-Q_2)} 
\end{multline*}
Taking the non-equivariant limit yields:
\begin{multline*}
  J_{Y,X}(0) = e^{(p_4-p_2-p_1)\log(1-Q_2)/z} e^{-(Q_3+Q_4+Q_2 Q_4)} \times \\
  \sum_{a,b,c,d \geq 0}
  \frac{
    Q_1^a Q_2^b Q_3^c Q_4^d (1-Q_2)^{d-a-b}
    \prod_{k=1}^{b+c+d} (p_2+p_3+p_4 + k z)
  }
  {
    \prod_{k=1}^a (p_1 + k z)^2
    \prod_{k=1}^b (p_2 + k z)^2
    \prod_{k=1}^c (p_3 + k z)^2
    \prod_{k=1}^d (p_4 + k z)
  } \\
  \times
  \frac{\prod_{k = -\infty}^0 (p_4-p_1-p_2 + k z)}
  {\prod_{k = -\infty}^{d-a-b} (p_4-p_1-p_2 + k z)}
\end{multline*}

We saw in Example~\ref{ex:ql_with_mirror_map} how to obtain the
quantum period $G_X$ from $J_{Y,X}(0)$: we extract the component along
the unit class $1 \in H^\bullet(Y;\QQ)$, set $z=1$, and set $Q^\beta =
t^{\langle \beta, {-K_X}\rangle}$ (i.e.~set $Q_1 = t$, $Q_2 = 1$, $Q_3
= t$, and $Q_4 =t$).  This yields:
\[
G_X(t) = e^{-3t} \sum_{a=0}^\infty \sum_{b = 0}^\infty \sum_{c = 0}^\infty 
t^{2a+b+c} \frac{(a+2b+c)!}{(a!)^2 (b!)^2 (c!)^2 (a+b)!}
\]
Regularizing gives:
\[
\hG_X(t) = 1+12 t^2+42 t^3+468 t^4+3360 t^5+31350 t^6+275940 t^7+2599380 t^8+24566640 t^9 + \cdots
\]

\subsection*{Minkowski period sequence:} \href{http://www.grdb.co.uk/search/period3?id=110&printlevel=2}{110}


\addtocounter{CustomSectionCounter}{1}

\section{The Fano Manifold $\MM{4}{3}$}
\label{sec:4-3}
\label{anchor:4--3}

\subsection*{Mori--Mukai name:} 4--3

\subsection*{Mori--Mukai construction:} The blow-up of the cone $Y$ over a
smooth quadric surface $S$ in $\PP^3$ with centre the disjoint union of the vertex and an elliptic curve on $S$.

\subsection*{Our construction:} A member $X$ of $|2N|$ in the toric
variety with weight data: 
\[ 
\begin{array}{rrrrrrrl}  
  \multicolumn{1}{c}{s_0} & 
  \multicolumn{1}{c}{s_1} & 
  \multicolumn{1}{c}{t_0} & 
  \multicolumn{1}{c}{t_1} & 
  \multicolumn{1}{c}{x} & 
  \multicolumn{1}{c}{y_0} &  
  \multicolumn{1}{c}{y_1} &  \\
  \cmidrule{1-7}
  1 & 1 & 0 & 0 & -1 & 0  & 0 &  \hspace{1.5ex} L\\ 
  0 & 0 & 1 & 1 & -1 & 0  & 0 &  \hspace{1.5ex} M\\ 
  0 & 0 & 0 & 0 & 1   & 1  & 1 & \hspace{1.5ex} N\\ 
\end{array}
\]
and $\Amp F =\langle L,M,N \rangle$. The toric variety $F$ is the same
as for 3--3 (\S\ref{sec:3-3}) and the secondary fan for $F$ is shown
in Fig.~\ref{fig:3-2}.

We have that:
\begin{itemize}
\item $-K_F=L+M+3N$ is ample, so $F$ is a Fano variety;
\item $X\sim 2N$ is nef;
\item $-(K_F+X)\sim L+M+N$ is ample. 
\end{itemize}

\subsection*{The two constructions coincide:} The variety $X$ is cut
out by:
\[
y_0 y_1 +x^2 A_{2,2} (s_0,s_1;t_0,t_1)=0
\]
where $A_{2,2}$ is a generic bihomogeneous polynomial of degrees~$2$
in $s_0$,~$s_1$ and $2$ in $t_0$,~$t_1$.  Note the obvious morphism
$\pi \colon F\to \PP^1_{s_0,s_1}\times \PP^1_{t_0,t_1}$, and the
morphism $f\colon F\to G$ to the double cone $G\subset \PP^5$ over
$\PP^1\times \PP^1$ given (contravariantly) by
$[y_0,y_1,y_2,y_3,y_4,y_5]\mapsto [y_0,y_1,s_0t_0x,s_0t_1x, s_1t_0
x,s_1t_1x]$. The exceptional set of $f$ is the divisor $E=(x=0)=
\PP^1_{s_0,s_1}\times \PP^1_{t_0,t_1} \times \PP^1_{y_0,y_1}$ that
maps to $\PP^1_{y_0,y_1}\subset G$.  Note that $E\cap X$ is \emph{two}
copies of $\PP^1_{s_0,s_1}\times \PP^1_{t_0,t_1}$, one above
$[y_0:y_1]=[1:0]$ and one above $[y_0:y_1]=[0:1]$. This explains how $X$ has
rank~$4$ when $F$ has rank~$3$.

To see that our construction coincides with the construction of
Mori--Mukai, set $W=f(X)$, note that:
\[
W=\bigl(y_0 y_1 +\tilde{A}_2 (y_2,y_3,y_4,y_5)=0\bigr)\subset G
\]
for some degree~$2$ homogeneous polynomial $\tilde{A}_2$, and note
that the morphism $f\colon X \to W$ contracts one copy of
$\PP^1_{s_0,s_1}\times \PP^1_{t_0,t_1}$, with normal bundle
$\cO(-1,-1)$, to each of the two singular points $W\cap
\PP^1_{y_0,y_1}$. Consider next the rational projection $g\colon G
\dasharrow \PP^4_{y_1,\dots ,y_5}$ which omits the homogeneous
co-ordinate $y_0$. It is clear that $g|_{W}\colon W \dasharrow \PP^4$
is birational onto its image $Y$ (the cone over $\PP^1\times \PP^1$),
that it extends to a morphism after blowing up the singular point
$[1:0:0:0:0:0]\in W$, and that this morphism contracts the surface
$\bigl(y_1=\tilde{A}_2(y_2,y_3,y_4,y_5)=0\bigr)\subset W$ to the
elliptic curve $\bigl(y_1=\tilde{A}_2(y_2,y_3,y_4,y_5)=0\bigr)\subset
Y$.

\subsection*{The quantum period:} 
Corollary~\ref{cor:QL} yields:
\[
G_X(t) = e^{-2t} \sum_{l=0}^\infty \sum_{m=0}^\infty \sum_{n=l+m}^\infty
t^{l+m+n}
\frac{(2n)!}
{(l!)^2 (m!)^2 (n-l-m)! (n!)^2}
\]
and regularizing gives:
\[
\hG_X(t) = 1 + 10 t^2 + 24 t^3 + 318 t^4 + 1680 t^5 + 16300 t^6 + 115920 t^7 + 1040830 t^8 + 8403360 t^9 + \cdots
\]

\subsection*{Minkowski period sequence:} \href{http://www.grdb.co.uk/search/period3?id=88&printlevel=2}{88}


\addtocounter{CustomSectionCounter}{1}

\section{The Fano Manifold $\MM{4}{4}$}
\label{anchor:4--4}

\subsection*{Mori--Mukai name:} 4--4

\subsection*{Mori--Mukai construction:} The blow-up of $\PP^1\times
\PP^1\times \PP^1$ with centre a curve $\Gamma$ of tridegree
$(1,1,2)$.

\subsection*{Our construction:} A member $X$ of $|A+B+D|$ in the toric
variety $F$ with weight data:
\[ 
\begin{array}{rrrrrrrrl} 
  \multicolumn{1}{c}{x_0} & 
  \multicolumn{1}{c}{x_1} & 
  \multicolumn{1}{c}{y_0} & 
  \multicolumn{1}{c}{y_1} & 
  \multicolumn{1}{c}{z_0} &  
  \multicolumn{1}{c}{z_1} &  
  \multicolumn{1}{c}{u} &  
  \multicolumn{1}{c}{v} &  \\
  \cmidrule{1-8}
1 & 1 &  0  &  0  & 0 & 0 & 0 &  0 &  \hspace{1.5ex} A \\ 
0 & 0 &  1  &  1  & 0 & 0 & 0 &  0 &  \hspace{1.5ex} B \\ 
0 & 0 &  0  &  0  & 1 & 1 &  -1  &  0 & \hspace{1.5ex} C \\
0 & 0 &  0  &  0  & 0 & 0 &  1  &  1 & \hspace{1.5ex} D 
\end{array}
\]
and $\Amp F = \langle A,B,C,D\rangle$.  We have:
\begin{itemize}
\item $-K_F=2A+2B+C+2D$ is ample, that is $F$ is a Fano variety;
\item $X\sim A+B+D$ is nef;
\item $-(K_F+X)\sim A+B+C+D$ is ample.
\end{itemize}

\subsection*{The two constructions coincide:} We can take $\Gamma
\subset \PP^1_{x_0,x_1}\times \PP^1_{y_0,y_1}\times \PP^1_{z_0,z_1}$
to be parameterised as
\[
[x_0: x_1:y_0: y_1:z_0: z_1]\mapsto [s_0:s_1:s_0:s_1:s_0^2:s_1^2] 
\]
thus $\Gamma$ is the complete intersection in $\PP^1\times \PP^1\times
\PP^1$ given by the equations
$x_0y_1-x_1y_0=z_1x_0y_0-z_0x_1y_1=0$.  Now apply
Lemma~\ref{lem:blowups} with:
\begin{align*}
  & G = \PP^1 \times \PP^1 \times \PP^1 \\
  & V = \cO_{\PP^1 \times \PP^1 \times \PP^1}(0,0,-1) \oplus
  \cO_{\PP^1 \times \PP^1 \times \PP^1} \\
  & W = \cO_{\PP^1 \times \PP^1 \times \PP^1}(1,1,0)
\end{align*}
and $f\colon V \to W$ given by the matrix $
\begin{pmatrix}
  z_1x_0y_0-z_0x_1y_1 & x_0y_1-x_1y_0
\end{pmatrix}
$.

\subsection*{The quantum period:}
Corollary~\ref{cor:QL} yields:
\[
G_X(t) = e^{-3t} \sum_{a=0}^\infty \sum_{b=0}^\infty \sum_{c=0}^\infty \sum_{d=c}^\infty 
t^{a+b+c+d}
\frac{(a+b+d)!}
{(a!)^2 (b!)^2 (c!)^2 (d-c)! d!}
\]
and regularizing gives:
\[
\hG_X(t) = 1+8 t^2+24 t^3+216 t^4+1320 t^5+10160 t^6+74760 t^7+584920 t^8+4598160 t^9 + \cdots
\]

\subsection*{Minkowski period sequence:} \href{http://www.grdb.co.uk/search/period3?id=83&printlevel=2}{83}


\addtocounter{CustomSectionCounter}{1}

\section{The Fano Manifold $\MM{4}{5}$}
\label{anchor:4--5}

\subsection*{Mori--Mukai name:} 4--5

\subsection*{Mori--Mukai construction:} The blow-up of $Y_{3-19}$,
which is the blow-up of a quadric \mbox{3-fold} $Q\subset \PP^4$ with
centre two points $P_1$ and $P_2$ on it which are not collinear (see
\S\ref{sec:3-19}), with centre the strict transform of a conic
containing $P_1$ and $P_2$.

\subsection*{Our construction:} A member $X$ of 
$|2N|$ in the toric variety $F$ with weight data:
\[ 
\begin{array}{rrrrrrrrl} 
  \multicolumn{1}{c}{s_0} & 
  \multicolumn{1}{c}{s_1} & 
  \multicolumn{1}{c}{x} & 
  \multicolumn{1}{c}{x_2} & 
  \multicolumn{1}{c}{y} &  
  \multicolumn{1}{c}{x_3} &  
  \multicolumn{1}{c}{y_4} &  \\
  \cmidrule{1-7}
1 & 1 & -1  &  0  & 0 &  0  &  0 &  \hspace{1.5ex} L \\ 
0 & 0 &  1  &  1  &-1 & 0  &  0 &  \hspace{1.5ex} M \\ 
0 & 0 &  0  &  0  & 1 &  1 &  1 &   \hspace{1.5ex} N
\end{array}
\]
and $\Amp F = \langle L,M,N\rangle$. We have:
\begin{itemize}
\item $-K_F=L+M+3N$ is ample, that is $F$ is a Fano variety;
\item $X\sim 2N$ is nef;
\item $-(K_F+X)\sim L+M+N$ is ample.
\end{itemize}
\subsection*{The two constructions coincide:} The complete linear
system $|N|$ defines a morphism $F\to \PP^4$ which sends
(contravariantly) the homogeneous co-ordinate functions
$[x_0,x_1,x_2,x_3,x_4]$ to:
\[
[s_0xy,s_1xy, x_2y,x_3,x_4]
\]
This morphism identifies $F$ with the blow-up of the line
$(x_2=x_3=x_4=0)\subset \PP^4$ followed by the blow up of the proper
transform of the plane $(x_3=x_4=0)$. The variety $X$ is the strict
transform of a general quadric in $\PP^4$: in other words, $X$ is a
general member of the linear system $|2N|$ on $F$.

\subsection*{Remark:} Note that $X$ has rank~$4$ even though the
ambient space $F$ has rank~$3$; there is no contradiction here because
$2N$ is not ample on $F$.

\subsection*{The quantum period:}
Corollary~\ref{cor:QL} yields:
\[
G_X(t) = e^{-2t} \sum_{l=0}^\infty \sum_{m=l}^\infty \sum_{n=m}^\infty
t^{l+m+n}
\frac{(2n)!}
{(l!)^2 (m-l)!m!(n-m)!(n!)^2}
\]
and regularizing gives:
\[
\hG_X(t) = 1+6 t^2+24 t^3+138 t^4+960 t^5+6180 t^6+43680 t^7+311850 t^8+2274720 t^9 + \cdots
\]

\subsection*{Minkowski period sequence:} \href{http://www.grdb.co.uk/search/period3?id=68&printlevel=2}{68}


\addtocounter{CustomSectionCounter}{1}

\section{The Fano Manifold $\MM{4}{6}$}
\label{anchor:4--6}

\subsection*{Mori--Mukai name:} 4--6 

\subsection*{Mori--Mukai construction:} The blow-up of $\PP^2\times
\PP^1$ with centre two disjoint curves, one of bidegree $(1,2)$ and the other of bidegree $(0,1)$.
 
\subsection*{Our construction:} A member $X$ of $|C+D|$ in the toric
variety with weight data: 
\[ 
\begin{array}{rrrrrrrrl}  
  \multicolumn{1}{c}{s_0} & 
  \multicolumn{1}{c}{s_1} & 
  \multicolumn{1}{c}{t_0} &
   \multicolumn{1}{c}{t_1} &
  \multicolumn{1}{c}{x} & 
  \multicolumn{1}{c}{x_2} &  
  \multicolumn{1}{c}{u} &  
  \multicolumn{1}{c}{v} &  \\
  \cmidrule{1-8}
  1 & 1 & 0 & 0 & -1 & 0  & -1 & 0 & \hspace{1.5ex} A\\
  0 & 0 & 1 & 1 &  0 & 0  & -1 & 0 & \hspace{1.5ex} B\\ 
  0 & 0 & 0 & 0 & 1  & 1 &  0 &  0  & \hspace{1.5ex} C\\ 
  0 & 0 & 0 &  0 & 0 & 0   &  1 & 1 & \hspace{1.5ex} D\\ 
\end{array}
\]
and $\Amp F =\langle A,B, C,D \rangle$. 

We have
\begin{itemize}
\item $-K_F=B+2C+2D$ is nef and big but not ample;
\item $X\sim C+D$ is nef and big but not ample;
\item $-(K_F+X)\sim B+C+D$ is nef and big but not ample. 
\end{itemize}

\subsection*{The two constructions coincide:} The variety $X$ is cut
out by:
\[
vx_2 + u x A_{2,1} (s_0,s_1;t_0,t_1)=0
\]
Note the obvious morphism $\pi \colon F \to G$ with fibre $\PP^1_{u,v}$, where $G$ is the
toric variety with weight data:
\[ 
\begin{array}{rrrrrrl}  
  \multicolumn{1}{c}{s_0} & 
  \multicolumn{1}{c}{s_1} & 
  \multicolumn{1}{c}{t_0} &
   \multicolumn{1}{c}{t_1} &
  \multicolumn{1}{c}{x} & 
  \multicolumn{1}{c}{x_2} \\
  \cmidrule{1-6}
  1 & 1 & 0 & 0 & -1 & 0  &  \hspace{1.5ex} A\\
  0 & 0 & 1 & 1 &  0 & 0  &  \hspace{1.5ex} B\\ 
  0 & 0 & 0 & 0 & 1  & 1 & \hspace{1.5ex} C\\ 
\end{array}
\]
and $\Amp G =\langle A,B, C \rangle$.  The birational morphism $G\to
\PP^2_{x_0,x_1,x_2}\times\PP^1_{t_0,t_1}$ given (contravariantly) by
$[x_0,x_1,x_2,t_0,t_1]\mapsto [s_0x,s_1x,x_2,t_0,t_1]$ identifies $G$
with the blow-up of the curve $\{[0:0:1]\}\times \PP^1\subset
\PP^2\times \PP^1$; this curve has bidegree $(0,1)$. The equation
defining $X$ has degree~$1$ in $\PP^1_{u,v}$: it follows that the
morphism $\pi|_{X} \colon X \to G$ is birational and blows up the
locus\footnote{Note that, with our choice of stability condition for
  $F$, $(x_2=x=0)\subset \CC^8$ is part of the unstable locus.}
$(x_2=A_{2,1}(s_0,s_1;t_0,t_1)=0)\subset G$.

\subsection*{The quantum period:}
Let $p_1$, $p_2, p_3, p_4 \in H^\bullet(F;\ZZ)$ denote the first Chern
classes of $A$, $B$, $C$, and $D$ respectively; these classes form a
basis for $H^2(F;\ZZ)$.  Write $\tau \in H^2(F;\QQ)$ as $\tau = \tau_1
p_1 + \tau_2 p_2+ \tau_3 p_3 + \tau_4 p_4$ and identify the group ring
$\QQ[H_2(F;\ZZ)]$ with the polynomial ring $\QQ[Q_1,Q_2,Q_3,Q_4]$ via the
$\QQ$-linear map that sends the element $Q^\beta \in \QQ[H_2(F;\ZZ)]$
to $Q_1^{\langle \beta,p_1\rangle} Q_2^{\langle \beta,p_2\rangle}
Q_3^{\langle \beta,p_3\rangle} Q_4^{\langle \beta,p_4\rangle}$.  We have:
\begin{multline*}
  I_F(\tau)  = e^{\tau/z} 
  \sum_{a,b,c,d \geq 0}
  \frac{
    Q_1^a Q_2^b Q_3^c Q_4^d
    e^{a \tau_1} e^{b \tau_2} e^{c \tau_3} e^{d \tau_4}
  }
  {
    \prod_{k=1}^a (p_1 + k z)^2
    \prod_{k=1}^b (p_2 + k z)^2
    \prod_{k=1}^c (p_3 + k z)
    \prod_{k=1}^d (p_4 + k z)
  }
  \\
  \frac
  {
    \prod_{k = -\infty}^0 (p_3-p_1 + k z)
  }
  {
    \prod_{k = -\infty}^{c-a} (p_3-p_1 + k z)
  }
  \frac
  {
    \prod_{k = -\infty}^0 (p_4-p_1-p_2 + k z)
  }
  {
    \prod_{k = -\infty}^{d-a-b} (p_4-p_1-p_2 + k z)
  }
\end{multline*}
Since:
\[
  I_F(\tau)  = 1 + \tau z^{-1} + O(z^{-2})
\]
Theorem~\ref{thm:toric_mirror} gives:
\[
J_F(\tau) = I_F(\tau)
\]
We now proceed exactly as in the case of 3--1
(\S\ref{sec:restriction_of_nef_not_restriction_of_ample}), obtaining:
\[
G_X(t) = 
e^{-2 t} 
\sum_{a=0}^\infty \sum_{b=0}^\infty \sum_{c=a}^\infty \sum_{d=a+b}^\infty 
t^{b+c+d}
\frac
{(c+d)!}
{(a!)^2 (b!)^2 c! d! (c-a)! (d-a-b)!}
\]
Regularizing gives:
\[
\hG_X(t) = 1+8 t^2+18 t^3+192 t^4+960 t^5+7550 t^6+49980 t^7+374080 t^8+2741760 t^9 + \cdots
\]

\subsection*{Minkowski period sequence:} \href{http://www.grdb.co.uk/search/period3?id=81&printlevel=2}{81}


\addtocounter{CustomSectionCounter}{1}

\section{The Fano Manifold $\MM{4}{7}$}
\label{anchor:4--7}

\subsection*{Mori--Mukai name:} 4--7

\subsection*{Mori--Mukai construction:} the blow-up of $\PP^1\times
\PP^1 \times \PP^1$ with centre the curve of tridegree $(1,1,1)$.

\subsection*{Our construction:} A codimension-2 complete intersection
$X$ of type $D \cap D$ in the toric variety $F$ with weight data:
\[ 
\begin{array}{rrrrrrrrrl} 
  \multicolumn{1}{c}{x_0} & 
  \multicolumn{1}{c}{x_1} & 
  \multicolumn{1}{c}{y_0} & 
  \multicolumn{1}{c}{y_1} & 
  \multicolumn{1}{c}{z_0} &  
  \multicolumn{1}{c}{z_1} &  
  \multicolumn{1}{c}{u_0} &  
  \multicolumn{1}{c}{u_1} &  
  \multicolumn{1}{c}{u_2} &  \\
  \cmidrule{1-9}
1 & 1 & 0  &  0  & 0 & 0 & -1 &  0 & 0 &  \hspace{1.5ex} A \\ 
0 & 0 &  1  &  1  & 0 & 0 & 0 &  -1 &  0 &\hspace{1.5ex} B \\ 
0 & 0 &  0  &  0  & 1 & 1 &  0  &  0 & -1 & \hspace{1.5ex} C \\
0 & 0 &  0  &  0  & 0 & 0 &  1  &  1 & 1 & \hspace{1.5ex} D 
\end{array}
\]
and $\Amp F =\langle A, B, C, D \rangle$.  We have:
\begin{itemize}
\item $-K_F=A+B+C+3D$ is ample, that is $F$ is a Fano variety;
\item $X$ is complete intersection of two nef divisors on $F$.
\item $-(K_F+\Lambda)\sim A+B+C+D$ is ample.
\end{itemize}

\subsection*{The two constructions coincide:} Without loss of
generality, the curve to be blown up is defined in
$\PP^1_{x_0,x_1}\times \PP^1_{y_0,y_1}\times \PP^1_{z_0,z_1}$ by the
condition:
\[
\rk
\begin{pmatrix}
  x_0 & y_0 & z_0 \\  x_1 & y_1 & z_1 \\
\end{pmatrix}
<2 
\]
Now apply Lemma~\ref{lem:blowups} with:
\begin{align*}
  & G = \PP^1 \times \PP^1 \times \PP^1 \\
  & V = \cO_{\PP^1 \times \PP^1 \times \PP^1}(-1,0,0) \oplus 
  \cO_{\PP^1 \times \PP^1 \times \PP^1}(0,-1,0) \oplus 
  \cO_{\PP^1 \times \PP^1 \times \PP^1}(0,0,-1) \\
  & W = \cO_{\PP^1 \times \PP^1 \times \PP^1}\oplus
  \cO_{\PP^1 \times \PP^1 \times \PP^1}
\end{align*}
and the map $f\colon V \to W$ given by the matrix:
\[
\begin{pmatrix}
  x_0 & y_0 & z_0 \\  x_1 & y_1 & z_1 \\
\end{pmatrix}
\]

\subsection*{The quantum period:}
Corollary~\ref{cor:QL} yields:
\[
G_X(t) = e^{-t} \sum_{a=0}^\infty \sum_{b=0}^\infty \sum_{c=0}^\infty \sum_{d=\max(a,b,c)}^\infty 
t^{a+b+c+d}
\frac{(d!)^2}
{(a!)^2 (b!)^2 (c!)^2 (d-a)! (d-b)! (d-c)!}
\]
and regularizing gives:
\[
\hG_X(t) = 1+6 t^2+18 t^3+114 t^4+720 t^5+4290 t^6+28980 t^7+193410 t^8+1320480 t^9+ \cdots
\]

\subsection*{Minkowski period sequence:} \href{http://www.grdb.co.uk/search/period3?id=65&printlevel=2}{65}


\addtocounter{CustomSectionCounter}{1}

\section{The Fano Manifold $\MM{4}{8}$}
\label{anchor:4--8}

\subsection*{Mori--Mukai name:} 4--8

\subsection*{Mori--Mukai construction:} The blow-up of $W$ (a divisor
of bidegree $(1,1)$ in $\PP^2\times \PP^2$; see \S\ref{sec:W}) with
centre two disjoint curves on it, of bi-degree $(0,1)$ and $(1,0)$.

\subsection*{Our construction:} A member $X$ of $|B+D|$ in the toric
variety $F$ with weight data:
\[ 
\begin{array}{rrrrrrrrrl} 
  \multicolumn{1}{c}{s_0} & 
  \multicolumn{1}{c}{s_1} & 
  \multicolumn{1}{c}{x} & 
  \multicolumn{1}{c}{x_2} & 
  \multicolumn{1}{c}{t_0} &   
  \multicolumn{1}{c}{t_1} &  
  \multicolumn{1}{c}{y} &  
 \multicolumn{1}{c}{y_2} &  \\
  \cmidrule{1-8}
1 & 1 & -1  &  0  & 0 & 0  &  0&    0 &  \hspace{1.5ex} A \\ 
0 & 0 &  1  &  1  & 0 &  0 &   0 &   0 & \hspace{1.5ex} B \\ 
0 & 0 &  0  &  0  & 1 &  1 & -1 &   0 & \hspace{1.5ex} C \\
0 & 0 &  0  &  0  & 0 &  0 &  1  &  1 & \hspace{1.5ex} D 
\end{array}
\]
and $\Amp F =\langle A, B, C, D \rangle$.  We have:
\begin{itemize}
\item $-K_F=A+2B+C+2D$ is ample, that is $F$ is a Fano variety;
\item $X\sim B+D$ is nef.
\item $-(K_F+X)\sim A+B+C+D$ is ample.
\end{itemize}

\subsection*{The two constructions coincide:} We take $W$ to be the
divisor:
\[
W= (x_0y_0+x_1y_1+x_2y_2=0)\subset \PP^2_{x_0,x_1,x_2}\times
\PP^2_{y_0,y_1,y_2} 
\] 
It is clear that the morphism $f\colon F \to \PP^2\times \PP^2$ which
sends (contravariantly) $[x_0,x_1,x_2,y_0,y_1,y_2]$ to
$[s_0x,s_1x,x_2,t_0y,t_1y,y_2]$ blows up the disjoint union of
$(x_0=x_1=0)$ and $(y_0=y_1=0)$ in $\PP^2\times \PP^2$.  This morphism induces the
required blow-up of $W$.

\subsection*{The quantum period:}  
Corollary~\ref{cor:QL} yields:
\[
G_X(t) = e^{-2t} \sum_{a=0}^\infty \sum_{b=a}^\infty \sum_{c=0}^\infty \sum_{d=c}^\infty
t^{a+b+c+d}
\frac{(b+d)!}
{(a!)^2 (b-a)!b! (c!)^2 (d-c)!d!}
\]
and regularizing gives:
\[
\hG_X(t) = 1 + 6 t^2 + 12 t^3 + 114 t^4 + 480 t^5 + 3480 t^6 + 19320 t^7 + 131250 t^8 + 
    819840 t^9+ \cdots
\]

\subsection*{Minkowski period sequence:} \href{http://www.grdb.co.uk/search/period3?id=57&printlevel=2}{57}


\addtocounter{CustomSectionCounter}{1}

\section{The Fano Manifold $\MM{4}{9}$}
\label{anchor:4--9}

\subsection*{Mori--Mukai name:} 4--9

\subsection*{Mori--Mukai construction:} the blow-up of $\PP^1\times
\PP^1\times \PP^1$ with centre a curve of tridegree $(0,1,1$).

\subsection*{Our construction:} A member $X$ of $|D|$ in the toric
variety $F$ with weight data:
\[
\begin{array}{rrrrrrrrrl} 
  \multicolumn{1}{c}{x_0} & 
  \multicolumn{1}{c}{x_1} & 
  \multicolumn{1}{c}{y_0} & 
  \multicolumn{1}{c}{y_1} & 
  \multicolumn{1}{c}{z_0} &   
  \multicolumn{1}{c}{z_1} &  
  \multicolumn{1}{c}{u} &  
 \multicolumn{1}{c}{v} &  \\
  \cmidrule{1-8}
  1 & 1 &  0  &  0  & 0 & 0  &  0&    -1 &  \hspace{1.5ex} A \\ 
  0 & 0 &  1  &  1  & 0 &  0 &   0 &  -1 & \hspace{1.5ex} B \\ 
  0 & 0 &  0  &  0  & 1 &  1 & -1 &   0 & \hspace{1.5ex} C \\
  0 & 0 &  0  &  0  & 0 &  0 &  1  &  1 & \hspace{1.5ex} D 
\end{array} 
\]
and $\Amp F =\langle A, B, C, D \rangle$.  We have:
\begin{itemize}
\item $-K_F=A+B+C+2D$ is ample, that is $F$ is a Fano variety;
\item $X\sim D$ is nef.
\item $-(K_F+X)\sim A+B+C+D$ is ample.
\end{itemize}

\subsection*{The two constructions coincide:}  The curve to be blown
up is the complete intersection:
\[
(z_0=x_0y_0+x_1y_1=0) \subset \PP^1_{x_0,x_1}\times
\PP^1_{y_0,y_1}\times \PP^1_{z_0,z_1} \, .
\]
We apply Lemma~\ref{lem:blowups} with:
\begin{align*}
  & G = \PP^1 \times \PP^1 \times \PP^1 \\
  & V = \cO_{\PP^1 \times \PP^1 \times \PP^1}(0,0,-1) \oplus 
  \cO_{\PP^1 \times \PP^1 \times \PP^1}(-1,-1,0) \\
  & W = \cO_{\PP^1 \times \PP^1 \times \PP^1}
\end{align*}
and the map $f\colon V \to W$ given by the matrix $
\begin{pmatrix}
  z_0 & x_0y_0+x_1y_1
\end{pmatrix}
$.

\subsection*{The quantum period:}
Corollary~\ref{cor:QL} yields:
\[
G_X(t) = e^{-t}  \sum_{a=0}^\infty \sum_{b=0}^\infty \sum_{c=0}^\infty \sum_{d=\max(a+b,c)}^\infty
t^{a+b+c+d}
\frac{d!}
{(a!)^2 (b!)^2 (c!)^2 (d-c)! (d-a-b)!}
\]
and regularizing gives:
\[
\hG_X(t) = 1+6 t^2+12 t^3+90 t^4+480 t^5+2400 t^6+16800 t^7+88410 t^8+608160 t^9+ \cdots
\]

\subsection*{Minkowski period sequence:} \href{http://www.grdb.co.uk/search/period3?id=54&printlevel=2}{54}


\addtocounter{CustomSectionCounter}{1}

\section{The Fano Manifold $\MM{4}{10}$}
\label{anchor:4--10}

\subsection*{Mori--Mukai name:} 4--10

\subsection*{Mori--Mukai construction:} The blow-up of $Y_{3-25}$,
which is the blow-up of $\PP^3$ with centre two disjoint lines (see
\S\ref{sec:3-25}), with centre an exceptional line of the blow-up
$Y\to \PP^3$.

\subsection*{Our construction:} The toric variety $X$ with weight data:
\[
\begin{array}{rrrrrrrrl} 
  \multicolumn{1}{c}{s_0} & 
  \multicolumn{1}{c}{s_1} & 
  \multicolumn{1}{c}{t_2} & 
  \multicolumn{1}{c}{t_3} & 
  \multicolumn{1}{c}{x} &   
  \multicolumn{1}{c}{y} &  
  \multicolumn{1}{c}{z} &  \\
  \cmidrule{1-7}
  1 & 1 &  0  & 0  & -1 & 0  &  0&     \hspace{1.5ex} A \\ 
  0 & 0 &  1  &  1  & 0 & -1 &  0&     \hspace{1.5ex} B \\ 
  0 & 0 &  0  & -1  & 0 &  1 &   1&    \hspace{1.5ex} C \\
  0 & 0 &  0  &  1  & 1 &  0  &  -1 &   \hspace{1.5ex} D 
\end{array} 
\]
and $\Amp X =\langle A, B, C, D \rangle$.

\subsection*{The two constructions coincide:} The morphism $X\mapsto
\PP^3$ is given by the complete linear system $|C|$.  It sends
(contravariantly) the homogeneous co-ordinate functions
$[x_0,x_1,x_2,x_3]$ to $[s_0xz, s_1xz, t_2 y, t_3yz]$. The morphism
blows up first the lines $(x_0=x_1=0)$ (the image of the divisor $x=0$
in $X$) and $(x_2=x_3=0)$ (the image of the divisor $y=0$ in $X$), and
then the fibre over the point $[0:0:1:0]$ (the image of the divisor
$z=0$ in $X$).

\subsection*{The quantum period:}
Corollary~\ref{cor:toric_mirror} yields:
\[
G_X(t) = \sum_{a=0}^\infty \sum_{b= 0}^\infty \sum_{d=a}^\infty \sum_{c=\max(b,d)}^{b+d} 
\frac{t^{a+b+c+d}}
{(a!)^2 b! (b-c+d)! (d-a)! (c-b)! (c-d)!}
\]
and regularizing gives:
\[
\hG_X(t) = 1+4 t^2+12 t^3+60 t^4+300 t^5+1660 t^6+8820 t^7+51100 t^8+293160 t^9+ \cdots
\]

\subsection*{Minkowski period sequence:} \href{http://www.grdb.co.uk/search/period3?id=37&printlevel=2}{37}


\addtocounter{CustomSectionCounter}{1}

\section{The Fano Manifold $\MM{4}{11}$}
\label{anchor:4--11}

\subsection*{Mori--Mukai name:} 4--11

\subsection*{Mori--Mukai construction:} $S_7\times \PP^1$.

\subsection*{Our construction:} $S_7\times \PP^1$.

\subsection*{The two constructions coincide:}  Obvious.

\subsection*{The quantum period:}
Combining Corollary~\ref{cor:products} with Example~\ref{ex:P1} and
Example~\ref{ex:S7} yields:
\[
G_X(t) = \sum_{a \geq 0} \sum_{b \geq 0} \sum_{c=\max(a,b)}^{a+b}
\sum_{d \geq 0} 
\frac{t^{a+b+c+2d}}
{a! b! (a+b-c)! (c-a)! (c-b)! (d!)^2}
\]
Regularizing gives:
\[
\hG_X(t) = 1+6 t^2+6 t^3+90 t^4+240 t^5+1950 t^6+8400 t^7+53130 t^8+288960 t^9+ \cdots
\]

\subsection*{Minkowski period sequence:} \href{http://www.grdb.co.uk/search/period3?id=48&printlevel=2}{48}


\addtocounter{CustomSectionCounter}{1}

\section{The Fano Manifold $\MM{4}{12}$}
\label{anchor:4--12}

\subsection*{Mori--Mukai name:} 4--12

\subsection*{Mori--Mukai construction:} The blow-up of $\PP^1\times
\FF_1$ with centre $t\times e$, where $t\in \PP^1$ and $e$ is the
exceptional curve on $\FF_1$.

\subsection*{Our construction:} The toric variety $X$ with weight data:
\[
\begin{array}{rrrrrrrrl} 
  \multicolumn{1}{c}{y_0} & 
  \multicolumn{1}{c}{y_1^\prime} & 
  \multicolumn{1}{c}{s_0} & 
  \multicolumn{1}{c}{s_1} & 
  \multicolumn{1}{c}{x^\prime} &   
  \multicolumn{1}{c}{x_2} &  
  \multicolumn{1}{c}{w} &  \\
  \cmidrule{1-7}
  1 & 0 &  0  &  0  & -1 & 0 &   1& \hspace{1.5ex}    A \\ 
  0 & 0 &  1  &  1   & -1 & 0 & 0 & \hspace{1.5ex}   B \\ 
  0 &-1&  0  &  0   & 0 &  1 &  1 & \hspace{1.5ex}   C \\
  0 & 1 &  0  &  0  & 1&  0 &  -1  & \hspace{1.5ex}   D 
\end{array} 
\]
and $\Amp X =\langle A,B,C,D  \rangle$.

\subsection*{The two constructions coincide:} Let $[y_0:y_1]$ be
homogeneous co-ordinates on $\PP^1$, and recall that $\FF_1$ is the
toric variety with weight data:
\[
\begin{array}{rrrrl}
  \multicolumn{1}{c}{s_0} & 
  \multicolumn{1}{c}{s_1} & 
  \multicolumn{1}{c}{x} & 
  \multicolumn{1}{c}{x_2} \\
  \cmidrule{1-4}
  1 & 1 & -1 & 0 & \hspace{1.5ex} L \\
  0 & 0 & 1 & 1 & \hspace{1.5ex} M \\
\end{array}
\]
The morphism $X \to \PP^1\times \FF_1$ is given (contravariantly) by:
\[
[y_0,y_1,s_0,s_1,x,x_2]\mapsto [y_0, y_1^\prime w, s_0,s_1,x^\prime w,
x_2] 
\]

\subsection*{The quantum period:}
Corollary~\ref{cor:toric_mirror} yields:
\[
G_X(t) = \sum_{a=0}^\infty \sum_{b=0}^\infty \sum_{c=0}^\infty 
\sum_{d=\max(a+b,c)}^{a+c}
\frac{t^{a+b+c+d}}
{a! (d-c)! (b!)^2 (d-a-b)! c! (a+c-d)! }
\]
and regularizing gives:
\[
\hG_X(t) = 1+4 t^2+12 t^3+36 t^4+300 t^5+940 t^6+6300 t^7+31780 t^8+157080 t^9+ \cdots
\]

\subsection*{Minkowski period sequence:} \href{http://www.grdb.co.uk/search/period3?id=34&printlevel=2}{34}


\addtocounter{CustomSectionCounter}{1}

\section{The Fano Manifold $\MM{4}{13}$}
\label{anchor:4--13}

\subsection*{Mori--Mukai name:} 4--13

\subsection*{Mori--Mukai construction:} The blow-up of $Y_{2-33}$,
which is the blow-up of $\PP^3$ with centre a line (see
\S\ref{sec:2-33}), with centre two exceptional lines of the blow-up
$Y\to \PP^3$.

\subsection*{Our construction:} The toric variety $X$ with weight data:
\[
\begin{array}{rrrrrrrrl} 
  \multicolumn{1}{c}{s_0} & 
  \multicolumn{1}{c}{s_1} & 
  \multicolumn{1}{c}{x} & 
  \multicolumn{1}{c}{y_2} & 
  \multicolumn{1}{c}{y_3} &   
  \multicolumn{1}{c}{u} &  
  \multicolumn{1}{c}{v} &  \\
  \cmidrule{1-7}
  1 & 1 & -1 &  0  & 0 & 0 & 0 &   \hspace{1.5ex}  A \\ 
  0 & 0 & -1 &  0  & 0 & 1 & 1 &  \hspace{1.5ex}  B \\ 
  0 & 0 &  1  &  1  & 0 &-1& 0 &  \hspace{1.5ex}  C \\
  0 & 0 &  1  &  0  & 1 & 0 &-1&  \hspace{1.5ex}  D 
\end{array} 
\]
and $\Amp X =\langle A,B,C,D  \rangle$.

\subsection*{The two constructions coincide:} Recall from
\S\ref{sec:2-33} that $Y_{2-33}$ is the toric variety with weight
data:
\[
\begin{array}{rrrrr}
  \multicolumn{1}{c}{s_0} & 
  \multicolumn{1}{c}{s_1} & 
  \multicolumn{1}{c}{x} & 
  \multicolumn{1}{c}{x_2} & 
  \multicolumn{1}{c}{x_3} \\
  \midrule
 1 & 1 & -1 & 0 & 0 \\
 0 & 0 & 1   & 1 & 1  
\end{array}
\]
and that the morphism $Y_{2-33}\to \PP^3$ sends (contravariantly) the
homogeneous co-ordinate functions $[x_0,x_1,x_2,x_3]$ on $\PP^3$ to
$[s_0x, s_1x, x_2, x_3]$. The blow-up $X\to Y_{2-33}$ is given (again
contravariantly) by $[s_0,s_1,x,x_2,x_3] \mapsto [s_0,s_1,uvx,ux_2,vx_3]$.

\subsection*{The quantum period:}
Corollary~\ref{cor:toric_mirror} yields:
\[
G_X(t) = \sum_{a=0}^\infty \sum_{b=0}^\infty \sum_{c=0}^b \sum_{d=\max(0,a+b-c)}^{b}
\frac{t^{a+b+c+d}}
{(a!)^2 (c+d-a-b)!c!d! (b-c)! (b-d)! }
\]
and regularizing gives:
\[
\hG_X(t) = 1+4 t^2+6 t^3+60 t^4+120 t^5+1210 t^6+3360 t^7+27580 t^8+97440 t^9+ \cdots
\]

\subsection*{Minkowski period sequence:} \href{http://www.grdb.co.uk/search/period3?id=29&printlevel=2}{29}


\addtocounter{CustomSectionCounter}{1}

\section{The Fano Manifold $\MM{5}{1}$}
\label{anchor:5--1}

\subsection*{Mori--Mukai name:} 5--1

\subsection*{Mori--Mukai construction:} The blow-up of $Y_{2-29}$,
which is the blow-up of a quadric \mbox{3-fold} $Q\subset \PP^3$ with
centre a conic on it (see \S\ref{sec:2-29}), with centre three
exceptional lines of the blow-up $Y\to Q$.

\subsection*{Our construction:} A member $X$ of $|2A+2B+C+D+E|$ in the toric
variety $F$ with weight data:
\[
\begin{array}{rrrrrrrrrl} 
  \multicolumn{1}{c}{z_0} & 
  \multicolumn{1}{c}{z_1} & 
  \multicolumn{1}{c}{z_2} & 
  \multicolumn{1}{c}{s_3} & 
  \multicolumn{1}{c}{s_4} &  
  \multicolumn{1}{c}{x} &  
  \multicolumn{1}{c}{t_{12}} &  
  \multicolumn{1}{c}{t_{02}} &  
  \multicolumn{1}{c}{t_{01}} &  \\
  \cmidrule{1-9}
  1 & 1 &  1  &  1  & 1 & 0  &  0&   0 &  0 & \hspace{1.5ex} A  \\ 
  1 & 1 &  1  &  0  & 0 & 1  &  0 &  0 &  0 & \hspace{1.5ex} B  \\ 
  1 & 0 &  0  &  0  & 0 & 0  & 1 &   0 &  0 & \hspace{1.5ex} C \\
  0 & 1 &  0  &  0  & 0 &  0 & 0  &  1 &  0 & \hspace{1.5ex} D \\
  0  & 0 & 1  &  0  & 0 &  0 & 0  &  0 &  1 & \hspace{1.5ex} E 
\end{array} 
\]
and $\Amp F =\langle A, A+B+D+E, A+B+C+E, A+B+C+D, A+B+C+D+E,
2A+2B+C+D+E\rangle$.  We have:
\begin{itemize}
\item $-K_F=5A+4B+2C+2D+2E=2(2A+2B+C+D+E)+(A)$ is nef and big but
  not ample;
\item $X\sim 2A+2B+C+D+E$ is nef.
\item $-(K_F+X)\sim 3A+2B+C+D+E$ is nef and big but not ample.
\end{itemize}

\subsection*{The two constructions coincide:} There is a
morphism\footnote{The class $-K_F$ belongs to $7$ simplicial cones and
  a non-simplicial cone (the one that we chose to be $\Amp F$). It
  turns out that the class $2A+2B+C+D+E$ also belongs to all of these
  cones. However, only one of these cones contains $A+B+C+D+E$: this
  is the cone that we chose to be $\Amp F$.} $F\to \PP^4$ given by the
complete linear system $|A+B+C+D+E|$; it sends (contravariantly) the
homogeneous co-ordinate functions $[x_0,x_1,x_2,x_3,x_4]$ on $\PP^4$
to
$[z_0t_{02}t_{01},z_1t_{12}t_{01},z_2t_{12}t_{02},s_3xt_{12}t_{02}t_{01},s_3xt_{12}t_{02}t_{01}]$.
This morphism can be factorized by first blowing up the plane
$\Pi=(x_3=x_4=0)\subset \PP^4$, and subsequently blowing up the three
fibres over the co-ordinate points $P_0=[1:0:0:0:0]$,
$P_1=[0:1:0:0:0]$ and $P_2=[0:0:1:0:0]$ in $\Pi$. Thus we can take $X$
to be the proper transform of any quadric $Q\subset \PP^4$ containing
the three points $P_0$, $P_1$, $P_2$ but not containing the plane
$\Pi$, for instance the quadric given by the equation:
\[
x_0x_1+x_1x_2+x_2x_0+x_3^2+x_4^2=0
\]

\subsection*{The quantum period:}
Corollary~\ref{cor:QL} yields:
\[
G_X(t) = e^{-3t} \sum_{a=0}^\infty \sum_{b=0}^\infty \sum_{c=0}^\infty \sum_{d=0}^\infty \sum_{e=0}^\infty
t^{3a+2b+c+d+e}
\frac{(2a+2b+c+d+e)!}
{(a+b+c)! (a+b+d)! (a+b+e)! (a!)^2 b! c! d! e!}
\]
and regularizing gives:
\[
\hG_X(t) = 1 + 10 t^2 + 42 t^3 + 342 t^4 + 2640 t^5 + 21250 t^6 + 180600 t^7 + 1562470 t^8 
    + 13851600 t^9 + \cdots
\]

\subsection*{Minkowski period sequence:} \href{http://www.grdb.co.uk/search/period3?id=100&printlevel=2}{100}


\addtocounter{CustomSectionCounter}{1}

\section{The Fano Manifold $\MM{5}{2}$}
\label{anchor:5--2}

\subsection*{Mori--Mukai name:} 5--2 

\subsection*{Mori--Mukai construction:} The blow-up of $Y_{3-25}$,
which is the blow-up of $\PP^3$ with centre two disjoint lines (see
\S\ref{sec:3-25}), with centre two exceptional lines $\ell$,
$\ell^\prime$ of the blow-up $f \colon Y\to \PP^3$ such that $\ell$
and $\ell^\prime$ lie on the same irreducible component of the
exceptional set of $f$.

\subsection*{Our construction:} The toric variety $X$ with weight data:
\[
\begin{array}{rrrrrrrrl} 
  \multicolumn{1}{c}{s_0} & 
  \multicolumn{1}{c}{s_1} & 
  \multicolumn{1}{c}{t_2} & 
  \multicolumn{1}{c}{t_3} & 
  \multicolumn{1}{c}{x} &  
 \multicolumn{1}{c}{y} &  
  \multicolumn{1}{c}{u} &  
  \multicolumn{1}{c}{v} &  \\
  \cmidrule{1-8}
  1 & 1 &  0  &  0  &-1 & 0  &  0&   0 & \hspace{1.5ex} A  \\ 
  0 & 0 &  1  &  1  & 0 & -1 &  0 &  0 &   \hspace{1.5ex} B  \\ 
  0 & 0 &  0  &  1  & 1 & 0   & -1 &   0& \hspace{1.5ex} C \\
  0 & 0 &  1  &  0  & 1 & 0   &  0  &  -1 & \hspace{1.5ex} D \\
  0  & 0 &-1  & -1&-1 & 1 & 1   &  1&  \hspace{1.5ex} E 
\end{array} 
\]
and $\Amp X =\langle A, B, C, D, E, B+C+D-E \rangle$.

\subsection*{The two constructions coincide:} Consider the morphism
$f\colon X\to \PP^3$ given by the complete linear system $E$.  The
morphism $f$ sends (contravariantly) the homogeneous co-ordinate
functions $[x_0,x_1,x_2,x_3]$ on $\PP^3$ to
$[s_0xuv,s_1xuv,t_2yv,t_3yu]$; it contracts:
\begin{itemize}
\item the divisors $(x=0)$ and $(y=0)$ to the lines $x_0=x_1=0$ and
  $x_2=x_3=0$, and
\item the divisors $(u=0)$ and $(v=0)$ to the points $P_0=[0:0:0:1]$ and $P_1=[0:0:1:0]$.
\end{itemize}

\subsection*{The quantum period:}
Corollary~\ref{cor:QL} yields:
\begin{multline*}
  G_X(t) = \sum_{a=0}^\infty \sum_{b=0}^\infty \sum_{c=0}^\infty
  \sum_{d=0}^\infty 
  \\ 
  \sum_{e=\max(b,c,d)}^{\min(b+c,b+d,c+d-a)}
  \frac{t^{a+b+c+d}}
  {(a!)^2(b+d-e)! (b+c-e)! (c+d-a-e)! (e-b)! (e-c)! (e-d)!}
\end{multline*}
and regularizing gives:
\[
\hG_X(t) = 1 + 6 t^2 + 18 t^3 + 114 t^4 + 660 t^5 + 3930 t^6 + 25620
t^7 + 163170 t^8 + 1101240 t^9 + \cdots
\]

\subsection*{Minkowski period sequence:} \href{http://www.grdb.co.uk/search/period3?id=64&printlevel=2}{64}


\addtocounter{CustomSectionCounter}{1}

\section{The Fano Manifold $\MM{5}{3}$}
\label{anchor:5--3}

\subsection*{Mori--Mukai name:} 5--3

\subsection*{Mori--Mukai construction:} $S_6 \times \PP^1$.

\subsection*{Our construction:} $S_6 \times \PP^1$.

\subsection*{The two constructions coincide:}  Obvious.

\subsection*{The quantum period:}
Combining Corollary~\ref{cor:products} with Example~\ref{ex:P1} and
Example~\ref{ex:S6} yields:
\[
G_X(t) = \sum_{a=0}^\infty \sum_{b=0}^\infty \sum_{c=0}^\infty \sum_{d=\max(a-c,0)}^{a+b} \sum_{e=0}^\infty
  \frac{t^{a+2b+2c+d+2e}}{a!b!c!d!(a+b-d)!(c+d-a)!(e!)^2}
\]
Regularizing gives:
\[
\hG_X(t) = 1 + 8 t^2 + 12 t^3 + 168 t^4 + 600 t^5 + 5300 t^6 + 27720
t^7 + 210280 t^8 + 1308720 t^9 + \cdots
\]

\subsection*{Minkowski period sequence:} \href{http://www.grdb.co.uk/search/period3?id=76&printlevel=2}{76}


\addtocounter{CustomSectionCounter}{1}

\section{The Fano Manifold $\MM{6}{1}$}
\label{anchor:6--1}

\subsection*{Mori--Mukai name:} 6--1 

\subsection*{Mori--Mukai construction:} $S_5 \times \PP^1$.

\subsection*{Our construction:} $S_5 \times \PP^1$.

\subsection*{The two constructions coincide:}  Obvious.

\subsection*{The quantum period:}
Combining Corollary~\ref{cor:products} with Example~\ref{ex:P1} and
Example~\ref{ex:S5} yields:
\[
G_X(t) =  e^{-3t}
  \sum_{l=0}^\infty \sum_{m=0}^\infty \sum_{n=0}^\infty
  t^{l+m+2n} \frac{(l+2m)!}{(l!)^2(m!)^3(n!)^2}
\]
Regularizing gives:
\[
\hG_X(t) = 1 + 12 t^2 + 30 t^3 + 396 t^4 + 2160 t^5 + 20370 t^6 +
149520 t^7 + 1315020 t^8 + 10864560 t^9 + \cdots
\]

\subsection*{Minkowski period sequence:} \href{http://www.grdb.co.uk/search/period3?id=107&printlevel=2}{107}


\addtocounter{CustomSectionCounter}{1}

\section{The Fano Manifold $\MM{7}{1}$}
\label{anchor:7--1}

\subsection*{Mori--Mukai name:} 7--1 

\subsection*{Mori--Mukai construction:} $S_4 \times \PP^1$.

\subsection*{Our construction:} $S_4 \times \PP^1$.

\subsection*{The two constructions coincide:}  Obvious.

\subsection*{The quantum period:}
Combining Corollary~\ref{cor:products} with Example~\ref{ex:P1} and
Example~\ref{ex:S4} yields:
\[
G_X(t) =  e^{-4t}
  \sum_{l=0}^\infty \sum_{m=0}^\infty
  t^{l+2m} \frac{(2l)!(2l)!}{(l!)^5(m!)^2}
\]
Regularizing gives:
\begin{multline*}
  \hG_X(t) = 1 + 22 t^2 + 96 t^3 + 1434 t^4 + 12480 t^5 + 148900 t^6 +
  1606080 t^7 \\ 
  + 18905530 t^8 + 220617600 t^9 + \cdots
\end{multline*}

\subsection*{Minkowski period sequence:} \href{http://www.grdb.co.uk/search/period3?id=136&printlevel=2}{136}


\addtocounter{CustomSectionCounter}{1}

\section{The Fano Manifold $\MM{8}{1}$}
\label{anchor:8--1}

\subsection*{Mori--Mukai name:} 8--1

\subsection*{Mori--Mukai construction:} $S_3 \times \PP^1$.

\subsection*{Our construction:} $S_3 \times \PP^1$.

\subsection*{The two constructions coincide:}  Obvious.

\subsection*{The quantum period:}
Combining Corollary~\ref{cor:products} with Example~\ref{ex:P1} and
Example~\ref{ex:S3} yields:
\[
G_X(t) =  e^{-6t}
  \sum_{l=0}^\infty \sum_{m=0}^\infty
  t^{l+2m} \frac{(3l)!}{(l!)^4(m!)^2}
\]
Regularizing gives:
\begin{multline*}
  \hG_X(t) = 1 + 56 t^2 + 492 t^3 + 10536 t^4 + 168600 t^5 + 3180980 t^6
  + 58753800 t^7 \\
  + 1129788520 t^8 + 21955158960 t^9 + \cdots
\end{multline*}

\subsection*{Minkowski period sequence:} \href{http://www.grdb.co.uk/search/period3?id=155&printlevel=2}{155}


\addtocounter{CustomSectionCounter}{1}

\section{The Fano Manifold $\MM{9}{1}$}
\label{anchor:9--1}

\subsection*{Mori--Mukai name:} 9--1  

\subsection*{Mori--Mukai construction:} $S_2 \times \PP^1$.

\subsection*{Our construction:} $S_2 \times \PP^1$.

\subsection*{The two constructions coincide:}  Obvious.

\subsection*{The quantum period:}
Combining Corollary~\ref{cor:products} with Example~\ref{ex:P1} and
Example~\ref{ex:S2} yields:
\[
G_X(t) =  e^{-12t}
  \sum_{l=0}^\infty \sum_{m=0}^\infty
  t^{l+2m} \frac{(4l)!}{(l!)^3(2l)!(m!)^2}
\]
Regularizing gives:
\begin{multline*}
  \hG_X(t) = 1 + 278 t^2 + 6816 t^3 + 317850 t^4 + 12989760 t^5 +
  578870180 t^6 + 26074520640 t^7 \\
  + 1202038745530 t^8 + 56188933046400 t^9 + \cdots
\end{multline*}

\subsection*{Minkowski period sequence:} None.  Note that the
anticanonical line bundle of $S_2 \times \PP^1$ is not very ample.


\addtocounter{CustomSectionCounter}{1}

\section{The Fano Manifold $\MM{10}{1}$}
\label{anchor:10--1}

\subsection*{Mori--Mukai name:} 10--1

\subsection*{Mori--Mukai construction:} $S_1 \times \PP^1$.

\subsection*{Our construction:} $S_1 \times \PP^1$.

\subsection*{The two constructions coincide:}  Obvious.

\subsection*{The quantum period:}
Combining Corollary~\ref{cor:products} with Example~\ref{ex:P1} and
Example~\ref{ex:S1} yields:
\[
G_X(t) =   e^{-60t}
  \sum_{l=0}^\infty \sum_{m=0}^\infty
  t^{l+2m} \frac{(6l)!}{(l!)^2(2l)!(3l)!(m!)^2}
\]
Regularizing gives:
\begin{multline*}
  \hG_X(t) = 1 + 10262 t^2 + 2021280 t^3 + 618997146 t^4 +
  184490852160 t^5 + 57894898611620 t^6 \\ + 18577980262739520 t^7 +
  6078628630941923770 t^8 + 2017980469547810194560 t^9 + \cdots
\end{multline*}

\subsection*{Minkowski period sequence:} None.  Note that the
anticanonical line bundle of $S_1 \times \PP^1$ is not very ample.


\section*{Conclusion}

This completes the calculation of the quantum periods for all
3-dimensional Fano manifolds, and the proof of
Theorem~\ref{thm:models}.  It also completes the proof of our
conjecture with Golyshev \cite{CCGGK}: that there is a one-to-one
correspondence between deformation families of smooth 3-dimensional
Fano manifolds $X$ with very ample anticanonical bundle and
equivalence classes of Minkowski polynomials $f$ of manifold type,
such that the regularized quantum period $\hG_X$ of $X$ coincides with
the period $\pi_f$ of $f$.

\addtocounter{CustomSectionCounter}{1}
\section{A Fano Manifold With Non-Unirational Moduli Space}
\label{sec:moduli_not_unirational}
We conclude by giving an example of a Fano manifold $X$ such that the
moduli space of $X$ is not unirational.  The manifold $X$ has complex
dimension~66 and, since unirationality of moduli spaces is a
straightforward consequence of Theorem~\ref{thm:models}, this example
shows that the analog of Theorem~\ref{thm:models} fails in
dimension~66.  The same technique allows one to construct Fano
manifolds $X_{3k}$ of dimension $3k$ for every $k \geq 22$ such that
the moduli space of $X_{3k}$ is not unirational.  Let $C$ be a smooth
curve of genus~23, let $L$ be a line bundle of degree~1 on $C$, and
let $X$ be the moduli space of stable vector bundles over $C$ of
rank~2 with fixed determinant~$L$.  It is known that $X$ is a
non-singular projective variety \cite{Newstead} which is Fano
\cite{Ramanan}.  The moduli space of $X$ is isomorphic to the moduli
space of curves of genus~23 \cite{Tjurin}*{\S2}, which has
non-negative Kodaira dimension \cite{Harris--Mumford} and thus is
not unirational.

\appendix

\section{Laurent Polynomial Mirrors for 3-Dimensional Fano Manifolds}
\label{appendix}

The table below exhibits Laurent polynomial mirrors for each of
the~105 deformation families of 3-dimensional Fano manifolds.  The
`Method' column summarizes the method by which we computed the quantum
period in each case: ``Quantum Lefschetz'' means ``Quantum Lefschetz
with Fano ambient space and no mirror map''; ``Quantum Lefschetz with
weak Fano ambient'' means ``Quantum Lefschetz with non-Fano but weak
Fano ambient space''; ``Quantum Lefschetz with mirror map'' means
``Quantum Lefschetz with non-trivial mirror map''; and other entries
should be self-explanatory.  The `Minkowski~ID' column records the ID
in the Graded Ring Database \cite{CGK} of the corresponding Minkowski
period sequence of manifold type; there are only~98 non-trivial
entries in this column as only the~98 deformation families of
3-dimensional Fano manifolds with very ample anticanonical bundle give
rise to Minkowski polynomial mirrors.  There are in general many
Minkowski polynomials (and infinitely many other Laurent polynomials)
mirror to a given 3-dimensional Fano manifold, but we have listed only
one such Laurent polynomial in each case.

\begin{landscape}
  \begin{center}
\footnotesize
\begin{longtable}{>{\hspace{0.8ex}}lr<{\hspace{2ex}}p{9.5cm}p{6.3cm}r<{\hspace{5ex}}}
\caption{Mirror Laurent polynomials for 3-dimensional Fano manifolds.}\\

\toprule 
\multicolumn{1}{c}{{Name}}&
\multicolumn{1}{c}{{Degree}} & 
\multicolumn{1}{c}{{Laurent polynomial}} & 
\multicolumn{1}{c}{{Method}} & 
\multicolumn{1}{c}{{Minkowski ID}} 
\\ \midrule \endfirsthead

\multicolumn{5}{c}{{\tablename\ \thetable{}: Mirror Laurent polynomials for 3-dimensional Fano manifolds -- continued from previous page}} \\ \addlinespace[1.7ex] \midrule
\multicolumn{1}{c}{{Name}}&
\multicolumn{1}{c}{{Degree}} & 
\multicolumn{1}{c}{{Laurent polynomial}} & 
\multicolumn{1}{c}{{Method}} & 
\multicolumn{1}{c}{{Minkowski ID}} 
\\ \midrule\endhead

\midrule \multicolumn{5}{c}{{Continued on next page}} \endfoot

\bottomrule \endlastfoot

\hyperref[anchor:V2]{$V_2$} & 2 & $xy^6+6xy^5z+6xy^5+15xy^4z^2+30xy^4z+15xy^4+20xy^3z^3+60xy^3z^2+60xy^3z+20xy^3+15xy^2z^4+60xy^2z^3+90xy^2z^2+60xy^2z+15xy^2+6xyz^5+30xyz^4+60xyz^3+60xyz^2+30xyz+6xy+xz^6+6xz^5+15xz^4+20xz^3+15xz^2+6xz+x+\frac{6y^2}{z}+30y+\frac{30y}{z}+60z+\frac{60}{z}+\frac{60z^2}{y}+\frac{180z}{y}+\frac{180}{y}+\frac{60}{yz}+\frac{30z^3}{y^2}+\frac{120z^2}{y^2}+\frac{180z}{y^2}+\frac{120}{y^2}+\frac{30}{y^2z}+\frac{6z^4}{y^3}+\frac{30z^3}{y^3}+\frac{60z^2}{y^3}+\frac{60z}{y^3}+\frac{30}{y^3}+\frac{6}{y^3z}+\frac{15}{xy^2z^2}+\frac{60}{xy^3z}+\frac{60}{xy^3z^2}+\frac{90}{xy^4}+\frac{180}{xy^4z}+\frac{90}{xy^4z^2}+\frac{60z}{xy^5}+\frac{180}{xy^5}+\frac{180}{xy^5z}+\frac{60}{xy^5z^2}+\frac{15z^2}{xy^6}+\frac{60z}{xy^6}+\frac{90}{xy^6}+\frac{60}{xy^6z}+\frac{15}{xy^6z^2}+\frac{20}{x^2y^6z^3}+\frac{60}{x^2y^7z^2}+\frac{60}{x^2y^7z^3}+\frac{60}{x^2y^8z}+\frac{120}{x^2y^8z^2}+\frac{60}{x^2y^8z^3}+\frac{20}{x^2y^9}+\frac{60}{x^2y^9z}+\frac{60}{x^2y^9z^2}+\frac{20}{x^2y^9z^3}+\frac{15}{x^3y^{10}z^4}+\frac{30}{x^3y^{11}z^3}+\frac{30}{x^3y^{11}z^4}+\frac{15}{x^3y^{12}z^2}+\frac{30}{x^3y^{12}z^3}+\frac{15}{x^3y^{12}z^4}+\frac{6}{x^4y^{14}z^5}+\frac{6}{x^4y^{15}z^4}+\frac{6}{x^4y^{15}z^5}+\frac{1}{x^5y^{18}z^6}$ & Weighted projective complete intersection & n/a\\ \addlinespace[1.3ex] 
\rowcolor[gray]{0.95}
\hyperref[anchor:V4]{$V_4$} & 4 & $xy^4+4xy^3z+4xy^3+6xy^2z^2+12xy^2z+6xy^2+4xyz^3+12xyz^2+12xyz+4xy+xz^4+4xz^3+6xz^2+4xz+x+\frac{4y^2}{z}+12y+\frac{12y}{z}+12z+\frac{12}{z}+\frac{4z^2}{y}+\frac{12z}{y}+\frac{12}{y}+\frac{4}{yz}+\frac{6}{xz^2}+\frac{12}{xyz}+\frac{12}{xyz^2}+\frac{6}{xy^2}+\frac{12}{xy^2z}+\frac{6}{xy^2z^2}+\frac{4}{x^2y^2z^3}+\frac{4}{x^2y^3z^2}+\frac{4}{x^2y^3z^3}+\frac{1}{x^3y^4z^4}$ & Quantum Lefschetz & \href{http://www.grdb.co.uk/search/period3?id=165&printlevel=2}{165}\\ \addlinespace[1.3ex] 
\hyperref[anchor:V6]{$V_6$} & 6 & $xy^2z^3+3xy^2z^2+3xy^2z+xy^2+2xyz^3+6xyz^2+6xyz+2xy+xz^3+3xz^2+3xz+x+3yz+6y+\frac{3y}{z}+6z+\frac{6}{z}+\frac{3z}{y}+\frac{6}{y}+\frac{3}{yz}+\frac{3}{xz}+\frac{3}{xz^2}+\frac{6}{xyz}+\frac{6}{xyz^2}+\frac{3}{xy^2z}+\frac{3}{xy^2z^2}+\frac{1}{x^2yz^3}+\frac{2}{x^2y^2z^3}+\frac{1}{x^2y^3z^3}$ & Quantum Lefschetz & \href{http://www.grdb.co.uk/search/period3?id=164&printlevel=2}{164}\\ \addlinespace[1.3ex] 
\rowcolor[gray]{0.95}
\hyperref[anchor:V8]{$V_8$} & 8 & $xy^2+2xyz^2+4xyz+2xy+xz^4+4xz^3+6xz^2+4xz+x+\frac{4y}{z}+4z+\frac{4}{z}+\frac{6}{xz^2}+\frac{2}{xy}+\frac{4}{xyz}+\frac{2}{xyz^2}+\frac{4}{x^2yz^3}+\frac{1}{x^3y^2z^4}$ & Quantum Lefschetz & \href{http://www.grdb.co.uk/search/period3?id=163&printlevel=2}{163}\\ \addlinespace[1.3ex] 
\hyperref[anchor:B1]{$B_1$} & 8 & $xz^4+4xz^3+6xz^2+4xz+x+yz^4+4yz^3+6yz^2+4yz+y+\frac{2}{yz^2}+\frac{4}{yz^3}+\frac{2}{yz^4}+\frac{2}{xz^2}+\frac{4}{xz^3}+\frac{2}{xz^4}+\frac{1}{xy^2z^8}+\frac{1}{x^2yz^8}$ & Weighted projective complete intersection & n/a\\ \addlinespace[1.3ex] 
\rowcolor[gray]{0.95}
\hyperref[anchor:V10]{$V_{10}$} & 10 & $xyz^3+3xyz^2+3xyz+xy+xz^2+2xz+x+yz^2+2yz+y+3z+\frac{3}{z}+\frac{2}{y}+\frac{2}{yz}+\frac{2}{x}+\frac{2}{xz}+\frac{3}{xyz}+\frac{3}{xyz^2}+\frac{1}{xy^2z^2}+\frac{1}{x^2yz^2}+\frac{1}{x^2y^2z^3}$ & Abelian/non-Abelian correspondence & \href{http://www.grdb.co.uk/search/period3?id=160&printlevel=2}{160}\\ \addlinespace[1.3ex] 
\hyperref[anchor:V12]{$V_{12}$} & 12 & $x^2y^3z+x^2y^2z+2xy^2z+xy^2+2xyz+2xy+x+yz+3y+z+\frac{2}{y}+\frac{1}{x}+\frac{1}{xz}+\frac{2}{xy}+\frac{3}{xyz}+\frac{1}{xy^2}+\frac{3}{xy^2z}+\frac{1}{xy^3z}$ & Abelian/non-Abelian correspondence & \href{http://www.grdb.co.uk/search/period3?id=150&printlevel=2}{150}\\ \addlinespace[1.3ex] 
\rowcolor[gray]{0.95}
\hyperref[anchor:V14]{$V_{14}$} & 14 & $xz+x+\frac{x}{yz}+yz^3+3yz^2+3yz+y+z+\frac{3}{z}+\frac{1}{yz}+\frac{3}{yz^2}+\frac{1}{y^2z^3}+\frac{z}{x}+\frac{1}{x}+\frac{1}{xyz}$ & Abelian/non-Abelian correspondence & \href{http://www.grdb.co.uk/search/period3?id=147&printlevel=2}{147}\\ \addlinespace[1.3ex] 
\hyperref[anchor:V16]{$V_{16}$} & 16 & $x+\frac{2x}{yz}+\frac{x}{y^2z^2}+yz^2+2yz+y+2z+\frac{2}{z}+\frac{1}{y}+\frac{2}{yz}+\frac{1}{yz^2}+\frac{z}{x}+\frac{2}{x}+\frac{1}{xz}$ & Abelian/non-Abelian correspondence & \href{http://www.grdb.co.uk/search/period3?id=143&printlevel=2}{143}\\ \addlinespace[1.3ex] 
\rowcolor[gray]{0.95}
\hyperref[anchor:B2]{$B_2$} & 16 & $xy^2+2xyz+2xy+xz^2+2xz+x+\frac{2}{xz}+\frac{2}{xy}+\frac{2}{xyz}+\frac{1}{x^3y^2z^2}$ & Weighted projective complete intersection & \href{http://www.grdb.co.uk/search/period3?id=140&printlevel=2}{140}\\ \addlinespace[1.3ex] 
\hyperref[anchor:V18]{$V_{18}$} & 18 & $xy^2+2xy+x+2y+z+\frac{1}{z}+\frac{2}{y}+\frac{1}{yz}+\frac{1}{x}+\frac{2z}{xy}+\frac{1}{xy}+\frac{1}{xy^2}+\frac{z}{x^2y^2}$ & Abelian/non-Abelian correspondence & \href{http://www.grdb.co.uk/search/period3?id=124&printlevel=2}{124}\\ \addlinespace[1.3ex] 
\rowcolor[gray]{0.95}
\hyperref[anchor:V22]{$V_{22}$} & 22 & $xy+\frac{xy}{z}+x+y+\frac{2y}{z}+z+\frac{2z}{y}+\frac{1}{y}+\frac{z}{y^2}+\frac{y}{xz}+\frac{1}{x}$ & Abelian/non-Abelian correspondence & \href{http://www.grdb.co.uk/search/period3?id=113&printlevel=2}{113}\\ \addlinespace[1.3ex] 
\hyperref[anchor:B3]{$B_3$} & 24 & $x+\frac{x}{yz}+y+z+\frac{2}{z}+\frac{2}{y}+\frac{y}{xz}+\frac{2}{x}+\frac{z}{xy}$ & Quantum Lefschetz & \href{http://www.grdb.co.uk/search/period3?id=106&printlevel=2}{106}\\ \addlinespace[1.3ex] 
\rowcolor[gray]{0.95}
\hyperref[anchor:B4]{$B_4$} & 32 & $x+yz^2+2yz+y+\frac{2}{yz}+\frac{1}{xy^2z^2}$ & Quantum Lefschetz & \href{http://www.grdb.co.uk/search/period3?id=75&printlevel=2}{75}\\ \addlinespace[1.3ex] 
\hyperref[anchor:B5]{$B_5$} & 40 & $x+y+z+\frac{1}{z}+\frac{1}{y}+\frac{1}{x}+\frac{1}{xyz}$ & Abelian/non-Abelian correspondence & \href{http://www.grdb.co.uk/search/period3?id=46&printlevel=2}{46}\\ \addlinespace[1.3ex] 
\rowcolor[gray]{0.95}
\hyperref[anchor:Q3]{$Q^3$} & 54 & $x+y+z+\frac{1}{xz}+\frac{1}{xy}$ & Quantum Lefschetz & \href{http://www.grdb.co.uk/search/period3?id=3&printlevel=2}{3}\\ \addlinespace[1.3ex] 
\hyperref[anchor:P3]{$\PP^3$} & 64 & $x+y+z+\frac{1}{xyz}$ & Toric variety & \href{http://www.grdb.co.uk/search/period3?id=1&printlevel=2}{1}\\ \addlinespace[1.3ex] 
\midrule
\rowcolor[gray]{0.95}
\hyperref[anchor:2--1]{$\MM{2}{1}$} & 4 & $x^7y^7z^{18}+6x^6y^6z^{15}+6x^5y^5z^{13}+15x^5y^5z^{12}+30x^4y^4z^{10}+20x^4y^4z^9+x^4y^3z^9+x^3y^4z^9+15x^3y^3z^8+60x^3y^3z^7+15x^3y^3z^6+3x^3y^2z^6+3x^2y^3z^6+60x^2y^2z^5+60x^2y^2z^4+6x^2y^2z^3+3x^2yz^4+3x^2yz^3+3xy^2z^4+3xy^2z^3+20xyz^3+90xyz^2+30xyz+xy+6xz+x+6yz+y+\frac{60}{z}+\frac{6}{z^2}+\frac{3}{yz}+\frac{3}{yz^2}+\frac{3}{xz}+\frac{3}{xz^2}+\frac{15}{xyz^2}+\frac{60}{xyz^3}+\frac{15}{xyz^4}+\frac{3}{xy^2z^4}+\frac{3}{x^2yz^4}+\frac{30}{x^2y^2z^5}+\frac{20}{x^2y^2z^6}+\frac{1}{x^2y^3z^6}+\frac{1}{x^3y^2z^6}+\frac{6}{x^3y^3z^7}+\frac{15}{x^3y^3z^8}+\frac{6}{x^4y^4z^{10}}+\frac{1}{x^5y^5z^{12}}$ & Hypersurface in product & n/a\\ \addlinespace[1.3ex] 
\hyperref[anchor:2--2]{$\MM{2}{2}$} & 6 & $xy^2+2xyz+2xy+xz^2+2xz+x+\frac{y^2}{z}+4y+\frac{4y}{z}+6z+\frac{6}{z}+\frac{4z^2}{y}+\frac{14z}{y}+\frac{14}{y}+\frac{4}{yz}+\frac{z^3}{y^2}+\frac{4z^2}{y^2}+\frac{6z}{y^2}+\frac{4}{y^2}+\frac{1}{y^2z}+\frac{4}{xz}+\frac{12}{xy}+\frac{12}{xyz}+\frac{12z}{xy^2}+\frac{25}{xy^2}+\frac{12}{xy^2z}+\frac{4z^2}{xy^3}+\frac{12z}{xy^3}+\frac{12}{xy^3}+\frac{4}{xy^3z}+\frac{6}{x^2y^2z}+\frac{12}{x^2y^3}+\frac{12}{x^2y^3z}+\frac{6z}{x^2y^4}+\frac{12}{x^2y^4}+\frac{6}{x^2y^4z}+\frac{4}{x^3y^4z}+\frac{4}{x^3y^5}+\frac{4}{x^3y^5z}+\frac{1}{x^4y^6z}$ & Quantum Lefschetz with mirror map & n/a\\ \addlinespace[1.3ex] 
\rowcolor[gray]{0.95}
\hyperref[anchor:2--3]{$\MM{2}{3}$} & 8 & $x^2y^5z^2+4x^2y^4z^2+6x^2y^3z^2+4x^2y^2z^2+x^2yz^2+xy^3z^2+4xy^3z+2xy^2z^2+12xy^2z+xy^2+xyz^2+12xyz+2xy+4xz+x+2yz+6y+2z+\frac{2}{z}+\frac{6}{y}+\frac{2}{yz}+\frac{1}{xy}+\frac{4}{xyz}+\frac{4}{xy^2z}+\frac{1}{xy^2z^2}+\frac{1}{x^2y^3z^2}$ & Hypersurface in product & n/a\\ \addlinespace[1.3ex] 
\hyperref[anchor:2--4]{$\MM{2}{4}$} & 10 & $xyz^3+3xyz^2+3xyz+xy+xz^2+2xz+x+yz^2+2yz+y+4z+\frac{3}{z}+\frac{2}{y}+\frac{2}{yz}+\frac{2}{x}+\frac{2}{xz}+\frac{4}{xyz}+\frac{3}{xyz^2}+\frac{1}{xy^2z^2}+\frac{1}{x^2yz^2}+\frac{1}{x^2y^2z^3}$ & Quantum Lefschetz & \href{http://www.grdb.co.uk/search/period3?id=161&printlevel=2}{161}\\ \addlinespace[1.3ex] 
\rowcolor[gray]{0.95}
\hyperref[anchor:2--5]{$\MM{2}{5}$} & 12 & $\frac{x^2}{yz}+x+\frac{3x}{z}+\frac{3x}{y}+\frac{x}{yz}+y+\frac{3y}{z}+z+\frac{2}{z}+\frac{3z}{y}+\frac{2}{y}+\frac{y^2}{xz}+\frac{3y}{x}+\frac{y}{xz}+\frac{3z}{x}+\frac{2}{x}+\frac{z^2}{xy}+\frac{z}{xy}$ & Quantum Lefschetz & \href{http://www.grdb.co.uk/search/period3?id=158&printlevel=2}{158}\\ \addlinespace[1.3ex] 
\hyperref[anchor:2--6]{$\MM{2}{6}$} & 12 & $x^2yz^2+2xyz^2+2xyz+2xz+x+yz^2+2yz+y+2z+\frac{2}{z}+\frac{1}{y}+\frac{2}{yz}+\frac{1}{x}+\frac{2}{xz}+\frac{1}{xz^2}+\frac{2}{xyz}+\frac{2}{xyz^2}+\frac{1}{xy^2z^2}$ & Quantum Lefschetz & \href{http://www.grdb.co.uk/search/period3?id=149&printlevel=2}{149}\\ \addlinespace[1.3ex] 
\rowcolor[gray]{0.95}
\hyperref[anchor:2--7]{$\MM{2}{7}$} & 14 & $xy^3z^3+xy^2z^3+3xy^2z^2+xyz^2+3xyz+x+y^2z+yz+y+z+\frac{3}{yz}+\frac{1}{xz}+\frac{2}{xyz}+\frac{3}{xy^2z^2}+\frac{1}{x^2y^3z^3}$ & Quantum Lefschetz & \href{http://www.grdb.co.uk/search/period3?id=148&printlevel=2}{148}\\ \addlinespace[1.3ex] 
\hyperref[anchor:2--8]{$\MM{2}{8}$} & 14 & $\frac{x^2}{y^2z}+x+\frac{x}{y}+\frac{2x}{yz}+\frac{x}{y^2}+yz+y+z+\frac{1}{z}+\frac{3}{y}+\frac{y^2z}{x}+\frac{2yz}{x}+\frac{y}{x}+\frac{3}{x}+\frac{y^2z}{x^2}+\frac{y}{x^2}$ & Quantum Lefschetz with weak Fano ambient & \href{http://www.grdb.co.uk/search/period3?id=144&printlevel=2}{144}\\ \addlinespace[1.3ex] 
\rowcolor[gray]{0.95}
\hyperref[anchor:2--9]{$\MM{2}{9}$} & 16 & $x+\frac{x}{z}+\frac{x}{yz}+\frac{x}{yz^2}+yz^2+2yz+y+2z+\frac{2}{z}+\frac{1}{y}+\frac{1}{yz}+\frac{1}{yz^2}+\frac{yz}{x}+\frac{2}{x}+\frac{1}{xyz}$ & Quantum Lefschetz & \href{http://www.grdb.co.uk/search/period3?id=139&printlevel=2}{139}\\ \addlinespace[1.3ex] 
\hyperref[anchor:2--10]{$\MM{2}{10}$} & 16 & $xy^2+2xy+x+\frac{x}{z}+y^2z+2yz+2y+z+\frac{2}{y}+\frac{2}{yz}+\frac{1}{x}+\frac{2}{xy}+\frac{1}{xy^2}+\frac{1}{xy^2z}$ & Quantum Lefschetz & \href{http://www.grdb.co.uk/search/period3?id=145&printlevel=2}{145}\\ \addlinespace[1.3ex] 
\rowcolor[gray]{0.95}
\hyperref[anchor:2--11]{$\MM{2}{11}$} & 18 & $x+\frac{x}{z}+\frac{x}{y}+yz+y+z+\frac{2}{z}+\frac{2}{y}+\frac{yz}{x}+\frac{y}{x}+\frac{z}{x}+\frac{1}{x}+\frac{1}{xz}+\frac{1}{xy}$ & Quantum Lefschetz & \href{http://www.grdb.co.uk/search/period3?id=120&printlevel=2}{120}\\ \addlinespace[1.3ex] 
\hyperref[anchor:2--12]{$\MM{2}{12}$} & 20 & $\frac{x^2}{yz}+x+\frac{x}{y}+\frac{2x}{yz}+y+z+\frac{1}{y}+\frac{1}{yz}+\frac{2yz}{x}+\frac{y}{x}+\frac{1}{x}+\frac{y^2z}{x^2}$ & Quantum Lefschetz & \href{http://www.grdb.co.uk/search/period3?id=118&printlevel=2}{118}\\ \addlinespace[1.3ex] 
\rowcolor[gray]{0.95}
\hyperref[anchor:2--13]{$\MM{2}{13}$} & 20 & $xy+x+\frac{x}{z}+y+z+\frac{2}{z}+\frac{z}{y}+\frac{2}{y}+\frac{1}{yz}+\frac{z}{x}+\frac{2}{x}+\frac{1}{xz}$ & Quantum Lefschetz & \href{http://www.grdb.co.uk/search/period3?id=119&printlevel=2}{119}\\ \addlinespace[1.3ex] 
\hyperref[anchor:2--14]{$\MM{2}{14}$} & 20 & $xy^2+2xy+x+2y+z+\frac{2}{y}+\frac{1}{x}+\frac{z}{xy}+\frac{1}{xy}+\frac{1}{xyz}+\frac{1}{xy^2}+\frac{1}{xy^2z}$ & Hypersurface in product & \href{http://www.grdb.co.uk/search/period3?id=122&printlevel=2}{122}\\ \addlinespace[1.3ex] 
\rowcolor[gray]{0.95}
\hyperref[anchor:2--15]{$\MM{2}{15}$} & 22 & $x+\frac{x}{z}+\frac{x}{yz}+y+\frac{y}{z}+z+\frac{2}{z}+\frac{2}{y}+\frac{y}{xz}+\frac{2}{x}+\frac{z}{xy}$ & Quantum Lefschetz & \href{http://www.grdb.co.uk/search/period3?id=109&printlevel=2}{109}\\ \addlinespace[1.3ex] 
\hyperref[anchor:2--16]{$\MM{2}{16}$} & 22 & $xy+x+y+z+\frac{1}{z}+\frac{z}{y}+\frac{2}{y}+\frac{1}{yz}+\frac{z}{x}+\frac{2}{x}+\frac{1}{xz}$ & Quantum Lefschetz & \href{http://www.grdb.co.uk/search/period3?id=104&printlevel=2}{104}\\ \addlinespace[1.3ex] 
\rowcolor[gray]{0.95}
\hyperref[anchor:2--17]{$\MM{2}{17}$} & 24 & $\frac{x^2}{yz}+\frac{x^2}{yz^2}+x+\frac{2x}{z}+\frac{x}{yz}+y+z+\frac{2z}{x}+\frac{1}{x}+\frac{z}{x^2}$ & Abelian/non-Abelian correspondence & \href{http://www.grdb.co.uk/search/period3?id=101&printlevel=2}{101}\\ \addlinespace[1.3ex] 
\hyperref[anchor:2--18]{$\MM{2}{18}$} & 24 & $x+\frac{x}{z}+\frac{x}{yz}+yz+y+z+\frac{1}{y}+\frac{y}{x}+\frac{2}{x}+\frac{1}{xy}$ & Quantum Lefschetz & \href{http://www.grdb.co.uk/search/period3?id=74&printlevel=2}{74}\\ \addlinespace[1.3ex] 
\rowcolor[gray]{0.95}
\hyperref[anchor:2--19]{$\MM{2}{19}$} & 26 & $\frac{x^2}{yz}+x+\frac{2x}{yz}+y+z+\frac{1}{yz}+\frac{2yz}{x}+\frac{y}{x}+\frac{y^2z}{x^2}$ & Quantum Lefschetz & \href{http://www.grdb.co.uk/search/period3?id=86&printlevel=2}{86}\\ \addlinespace[1.3ex] 
\hyperref[anchor:2--20]{$\MM{2}{20}$} & 26 & $x+\frac{x}{y}+y+\frac{y}{z}+z+\frac{1}{z}+\frac{1}{y}+\frac{z}{x}+\frac{2}{x}+\frac{1}{xz}$ & Abelian/non-Abelian correspondence & \href{http://www.grdb.co.uk/search/period3?id=87&printlevel=2}{87}\\ \addlinespace[1.3ex] 
\rowcolor[gray]{0.95}
\hyperref[anchor:2--21]{$\MM{2}{21}$} & 28 & $x+\frac{x}{yz}+y^2z+2yz+y+z+\frac{2}{yz}+\frac{1}{xyz}$ & Abelian/non-Abelian correspondence & \href{http://www.grdb.co.uk/search/period3?id=84&printlevel=2}{84}\\ \addlinespace[1.3ex] 
\hyperref[anchor:2--22]{$\MM{2}{22}$} & 30 & $xy+x+\frac{x}{z}+y+z+\frac{1}{z}+\frac{1}{y}+\frac{1}{x}+\frac{z}{xy}$ & Abelian/non-Abelian correspondence & \href{http://www.grdb.co.uk/search/period3?id=69&printlevel=2}{69}\\ \addlinespace[1.3ex] 
\rowcolor[gray]{0.95}
\hyperref[anchor:2--23]{$\MM{2}{23}$} & 30 & $x^2y+2xy+x+y+z+\frac{2}{xy}+\frac{1}{x^2y^2z}$ & Quantum Lefschetz & \href{http://www.grdb.co.uk/search/period3?id=78&printlevel=2}{78}\\ \addlinespace[1.3ex] 
\hyperref[anchor:2--24]{$\MM{2}{24}$} & 30 & $\frac{xy}{z}+x+\frac{x}{z}+y+z+\frac{z}{y}+\frac{1}{y}+\frac{y}{x}+\frac{1}{x}$ & Quantum Lefschetz & \href{http://www.grdb.co.uk/search/period3?id=44&printlevel=2}{44}\\ \addlinespace[1.3ex] 
\rowcolor[gray]{0.95}
\hyperref[anchor:2--25]{$\MM{2}{25}$} & 32 & $x+\frac{x}{z}+y+z+\frac{1}{y}+\frac{1}{yz}+\frac{yz}{x}+\frac{1}{x}$ & Quantum Lefschetz & \href{http://www.grdb.co.uk/search/period3?id=43&printlevel=2}{43}\\ \addlinespace[1.3ex] 
\hyperref[anchor:2--26]{$\MM{2}{26}$} & 34 & $xy+x+y+z+\frac{1}{z}+\frac{1}{y}+\frac{1}{x}+\frac{1}{xyz}$ & Abelian/non-Abelian correspondence & \href{http://www.grdb.co.uk/search/period3?id=58&printlevel=2}{58}\\ \addlinespace[1.3ex] 
\rowcolor[gray]{0.95}
\hyperref[anchor:2--27]{$\MM{2}{27}$} & 38 & $x+\frac{x}{z}+y+z+\frac{1}{yz}+\frac{1}{x}+\frac{1}{xy}$ & Quantum Lefschetz & \href{http://www.grdb.co.uk/search/period3?id=19&printlevel=2}{19}\\ \addlinespace[1.3ex] 
\hyperref[anchor:2--28]{$\MM{2}{28}$} & 40 & $xyz^2+xyz+x+y+z+\frac{1}{yz}+\frac{1}{xz}$ & Quantum Lefschetz & \href{http://www.grdb.co.uk/search/period3?id=5&printlevel=2}{5}\\ \addlinespace[1.3ex] 
\rowcolor[gray]{0.95}
\hyperref[anchor:2--29]{$\MM{2}{29}$} & 40 & $x+\frac{x}{y}+y+z+\frac{2}{x}+\frac{1}{x^2z}$ & Quantum Lefschetz & \href{http://www.grdb.co.uk/search/period3?id=35&printlevel=2}{35}\\ \addlinespace[1.3ex] 
\hyperref[anchor:2--30]{$\MM{2}{30}$} & 46 & $xyz+x+y+z+\frac{1}{xz}+\frac{1}{xy}$ & Quantum Lefschetz & \href{http://www.grdb.co.uk/search/period3?id=4&printlevel=2}{4}\\ \addlinespace[1.3ex] 
\rowcolor[gray]{0.95}
\hyperref[anchor:2--31]{$\MM{2}{31}$} & 46 & $x+\frac{x}{y}+y+z+\frac{1}{yz}+\frac{1}{x}$ & Quantum Lefschetz & \href{http://www.grdb.co.uk/search/period3?id=15&printlevel=2}{15}\\ \addlinespace[1.3ex] 
\hyperref[anchor:2--32]{$\MM{2}{32}$} & 48 & $x+y+z+\frac{1}{y}+\frac{1}{x}+\frac{1}{xyz}$ & Quantum Lefschetz & \href{http://www.grdb.co.uk/search/period3?id=24&printlevel=2}{24}\\ \addlinespace[1.3ex] 
\rowcolor[gray]{0.95}
\hyperref[anchor:2--33]{$\MM{2}{33}$} & 54 & $x+\frac{x}{z}+y+z+\frac{1}{xy}$ & Toric variety & \href{http://www.grdb.co.uk/search/period3?id=2&printlevel=2}{2}\\ \addlinespace[1.3ex] 
\hyperref[anchor:2--34]{$\MM{2}{34}$} & 54 & $x+y+z+\frac{1}{yz}+\frac{1}{x}$ & Toric variety & \href{http://www.grdb.co.uk/search/period3?id=10&printlevel=2}{10}\\ \addlinespace[1.3ex] 
\rowcolor[gray]{0.95}
\hyperref[anchor:2--35]{$\MM{2}{35}$} & 56 & $x+\frac{x}{yz}+y+z+\frac{1}{x}$ & Toric variety & \href{http://www.grdb.co.uk/search/period3?id=7&printlevel=2}{7}\\ \addlinespace[1.3ex] 
\hyperref[anchor:2--36]{$\MM{2}{36}$} & 62 & $\frac{x^2}{yz}+x+y+z+\frac{1}{x}$ & Toric variety & \href{http://www.grdb.co.uk/search/period3?id=6&printlevel=2}{6}\\ \addlinespace[1.3ex] 
\midrule
\rowcolor[gray]{0.95}
\hyperref[anchor:3--1]{$\MM{3}{1}$} & 12 & $xy^2+2xyz+2xy+xz^2+2xz+x+2y+\frac{2y}{z}+2z+\frac{2}{z}+\frac{2z}{y}+\frac{2}{y}+\frac{1}{x}+\frac{2}{xz}+\frac{1}{xz^2}+\frac{2}{xy}+\frac{2}{xyz}+\frac{1}{xy^2}$ & Quantum Lefschetz with weak Fano ambient & \href{http://www.grdb.co.uk/search/period3?id=154&printlevel=2}{154}\\ \addlinespace[1.3ex] 
\hyperref[anchor:3--2]{$\MM{3}{2}$} & 14 & $xyz^2+xyz+3xz+x+\frac{3x}{y}+\frac{x}{y^2z}+3yz+y+z+\frac{1}{y}+\frac{3}{yz}+\frac{3y}{x}+\frac{1}{x}+\frac{3}{xz}+\frac{y}{x^2z}$ & Quantum Lefschetz with mirror map & \href{http://www.grdb.co.uk/search/period3?id=157&printlevel=2}{157}\\ \addlinespace[1.3ex] 
\rowcolor[gray]{0.95}
\hyperref[anchor:3--3]{$\MM{3}{3}$} & 18 & $x+\frac{2x}{y}+\frac{x}{yz}+\frac{x}{y^2}+yz+y+z+\frac{2}{z}+\frac{2}{y}+\frac{yz}{x}+\frac{2y}{x}+\frac{y}{xz}+\frac{1}{x}$ & Quantum Lefschetz & \href{http://www.grdb.co.uk/search/period3?id=135&printlevel=2}{135}\\ \addlinespace[1.3ex] 
\hyperref[anchor:3--4]{$\MM{3}{4}$} & 18 & $xyz+x+yz^2+2yz+y+2z+\frac{2}{z}+\frac{1}{y}+\frac{2}{yz}+\frac{1}{yz^2}+\frac{z}{x}+\frac{2}{x}+\frac{1}{xz}$ & Quantum Lefschetz with weak Fano ambient & \href{http://www.grdb.co.uk/search/period3?id=142&printlevel=2}{142}\\ \addlinespace[1.3ex] 
\rowcolor[gray]{0.95}
\hyperref[anchor:3--5]{$\MM{3}{5}$} & 20 & $xyz+xz^2+2xz+x+y+2z+\frac{2}{z}+\frac{1}{y}+\frac{1}{yz}+\frac{1}{x}+\frac{2}{xz}+\frac{1}{xz^2}$ & Quantum Lefschetz with mirror map & \href{http://www.grdb.co.uk/search/period3?id=138&printlevel=2}{138}\\ \addlinespace[1.3ex] 
\hyperref[anchor:3--6]{$\MM{3}{6}$} & 22 & $\frac{x^2}{yz}+\frac{x^2}{y^2z}+x+\frac{2x}{y}+\frac{x}{yz}+y+z+\frac{1}{y}+\frac{2y}{x}+\frac{2}{x}+\frac{y}{x^2}$ & Quantum Lefschetz & \href{http://www.grdb.co.uk/search/period3?id=117&printlevel=2}{117}\\ \addlinespace[1.3ex] 
\rowcolor[gray]{0.95}
\hyperref[anchor:3--7]{$\MM{3}{7}$} & 24 & $\frac{xy}{z}+x+\frac{x}{z}+\frac{x}{y}+y+z+\frac{2}{y}+\frac{y}{x}+\frac{2}{x}+\frac{1}{xy}$ & Quantum Lefschetz & \href{http://www.grdb.co.uk/search/period3?id=103&printlevel=2}{103}\\ \addlinespace[1.3ex] 
\hyperref[anchor:3--8]{$\MM{3}{8}$} & 24 & $x+\frac{x}{z}+\frac{x}{y}+y+\frac{y}{z}+z+\frac{1}{z}+\frac{2}{y}+\frac{y}{x}+\frac{2}{x}+\frac{1}{xy}$ & Quantum Lefschetz & \href{http://www.grdb.co.uk/search/period3?id=112&printlevel=2}{112}\\ \addlinespace[1.3ex] 
\rowcolor[gray]{0.95}
\hyperref[anchor:3--9]{$\MM{3}{9}$} & 26 & $\frac{x^2}{yz}+x+\frac{2x}{yz}+y+z+\frac{1}{yz}+\frac{y}{x}+\frac{z}{x}+\frac{1}{x}$ & Quantum Lefschetz & \href{http://www.grdb.co.uk/search/period3?id=22&printlevel=2}{22}\\ \addlinespace[1.3ex] 
\hyperref[anchor:3--10]{$\MM{3}{10}$} & 26 & $\frac{xy}{z}+x+\frac{x}{y}+y+z+\frac{2}{y}+\frac{y}{x}+\frac{2}{x}+\frac{1}{xy}$ & Quantum Lefschetz & \href{http://www.grdb.co.uk/search/period3?id=99&printlevel=2}{99}\\ \addlinespace[1.3ex] 
\rowcolor[gray]{0.95}
\hyperref[anchor:3--11]{$\MM{3}{11}$} & 28 & $x+\frac{x}{z}+\frac{x}{yz}+y+z+\frac{1}{y}+\frac{y}{x}+\frac{2}{x}+\frac{1}{xy}$ & Quantum Lefschetz & \href{http://www.grdb.co.uk/search/period3?id=72&printlevel=2}{72}\\ \addlinespace[1.3ex] 
\hyperref[anchor:3--12]{$\MM{3}{12}$} & 28 & $xz+x+y+\frac{y}{z}+z+\frac{1}{z}+\frac{z}{y}+\frac{1}{y}+\frac{y}{xz}+\frac{1}{x}$ & Quantum Lefschetz & \href{http://www.grdb.co.uk/search/period3?id=85&printlevel=2}{85}\\ \addlinespace[1.3ex] 
\rowcolor[gray]{0.95}
\hyperref[anchor:3--13]{$\MM{3}{13}$} & 30 & $xy+x+y+z+\frac{1}{z}+\frac{1}{y}+\frac{1}{yz}+\frac{z}{x}+\frac{1}{x}$ & Quantum Lefschetz & \href{http://www.grdb.co.uk/search/period3?id=70&printlevel=2}{70}\\ \addlinespace[1.3ex] 
\hyperref[anchor:3--14]{$\MM{3}{14}$} & 32 & $\frac{x^2}{yz}+x+\frac{x}{yz}+y+z+\frac{y}{x}+\frac{z}{x}+\frac{1}{x}$ & Quantum Lefschetz with weak Fano ambient & \href{http://www.grdb.co.uk/search/period3?id=21&printlevel=2}{21}\\ \addlinespace[1.3ex] 
\rowcolor[gray]{0.95}
\hyperref[anchor:3--15]{$\MM{3}{15}$} & 32 & $x+\frac{x}{yz}+y+z+\frac{1}{y}+\frac{y}{x}+\frac{2}{x}+\frac{1}{xy}$ & Quantum Lefschetz & \href{http://www.grdb.co.uk/search/period3?id=67&printlevel=2}{67}\\ \addlinespace[1.3ex] 
\hyperref[anchor:3--16]{$\MM{3}{16}$} & 34 & $x+\frac{x}{y}+y+\frac{y}{z}+z+\frac{1}{y}+\frac{y}{xz}+\frac{1}{x}$ & Quantum Lefschetz with weak Fano ambient & \href{http://www.grdb.co.uk/search/period3?id=42&printlevel=2}{42}\\ \addlinespace[1.3ex] 
\rowcolor[gray]{0.95}
\hyperref[anchor:3--17]{$\MM{3}{17}$} & 36 & $x+y+\frac{y}{z}+z+\frac{1}{y}+\frac{y}{xz}+\frac{1}{x}+\frac{1}{xy}$ & Quantum Lefschetz & \href{http://www.grdb.co.uk/search/period3?id=39&printlevel=2}{39}\\ \addlinespace[1.3ex] 
\hyperref[anchor:3--18]{$\MM{3}{18}$} & 36 & $x+\frac{x}{y}+y+z+\frac{z}{x}+\frac{2}{x}+\frac{1}{xz}$ & Quantum Lefschetz & \href{http://www.grdb.co.uk/search/period3?id=41&printlevel=2}{41}\\ \addlinespace[1.3ex] 
\rowcolor[gray]{0.95}
\hyperref[anchor:3--19]{$\MM{3}{19}$} & 38 & $xz+x+y+z+\frac{1}{yz}+\frac{1}{x}+\frac{1}{xyz}$ & Quantum Lefschetz & \href{http://www.grdb.co.uk/search/period3?id=18&printlevel=2}{18}\\ \addlinespace[1.3ex] 
\hyperref[anchor:3--20]{$\MM{3}{20}$} & 38 & $xy+x+y+z+\frac{1}{y}+\frac{1}{x}+\frac{1}{xyz}$ & Quantum Lefschetz & \href{http://www.grdb.co.uk/search/period3?id=38&printlevel=2}{38}\\ \addlinespace[1.3ex] 
\rowcolor[gray]{0.95}
\hyperref[anchor:3--21]{$\MM{3}{21}$} & 38 & $x+yz+y+z+\frac{1}{z}+\frac{1}{y}+\frac{yz}{x}+\frac{1}{x}$ & Quantum Lefschetz & \href{http://www.grdb.co.uk/search/period3?id=49&printlevel=2}{49}\\ \addlinespace[1.3ex] 
\hyperref[anchor:3--22]{$\MM{3}{22}$} & 40 & $xz+x+\frac{x}{yz}+y+z+\frac{1}{yz}+\frac{1}{x}$ & Quantum Lefschetz & \href{http://www.grdb.co.uk/search/period3?id=13&printlevel=2}{13}\\ \addlinespace[1.3ex] 
\rowcolor[gray]{0.95}
\hyperref[anchor:3--23]{$\MM{3}{23}$} & 42 & $xz+x+\frac{x}{y}+y+z+\frac{1}{yz}+\frac{1}{x}$ & Quantum Lefschetz & \href{http://www.grdb.co.uk/search/period3?id=17&printlevel=2}{17}\\ \addlinespace[1.3ex] 
\hyperref[anchor:3--24]{$\MM{3}{24}$} & 42 & $x+y+z+\frac{1}{y}+\frac{y}{x}+\frac{1}{x}+\frac{1}{xyz}$ & Quantum Lefschetz & \href{http://www.grdb.co.uk/search/period3?id=31&printlevel=2}{31}\\ \addlinespace[1.3ex] 
\rowcolor[gray]{0.95}
\hyperref[anchor:3--25]{$\MM{3}{25}$} & 44 & $x+\frac{x}{z}+y+z+\frac{1}{x}+\frac{1}{xy}$ & Toric variety & \href{http://www.grdb.co.uk/search/period3?id=16&printlevel=2}{16}\\ \addlinespace[1.3ex] 
\hyperref[anchor:3--26]{$\MM{3}{26}$} & 46 & $xy+x+y+z+\frac{1}{yz}+\frac{1}{x}$ & Toric variety & \href{http://www.grdb.co.uk/search/period3?id=12&printlevel=2}{12}\\ \addlinespace[1.3ex] 
\rowcolor[gray]{0.95}
\hyperref[anchor:3--27]{$\MM{3}{27}$} & 48 & $x+y+z+\frac{1}{z}+\frac{1}{y}+\frac{1}{x}$ & Toric variety & \href{http://www.grdb.co.uk/search/period3?id=45&printlevel=2}{45}\\ \addlinespace[1.3ex] 
\hyperref[anchor:3--28]{$\MM{3}{28}$} & 48 & $x+\frac{x}{z}+y+z+\frac{1}{y}+\frac{1}{x}$ & Toric variety & \href{http://www.grdb.co.uk/search/period3?id=28&printlevel=2}{28}\\ \addlinespace[1.3ex] 
\rowcolor[gray]{0.95}
\hyperref[anchor:3--29]{$\MM{3}{29}$} & 50 & $xy+x+\frac{x}{yz}+y+z+\frac{1}{x}$ & Toric variety & \href{http://www.grdb.co.uk/search/period3?id=8&printlevel=2}{8}\\ \addlinespace[1.3ex] 
\hyperref[anchor:3--30]{$\MM{3}{30}$} & 50 & $x+\frac{x}{y}+y+\frac{y}{z}+z+\frac{1}{x}$ & Toric variety & \href{http://www.grdb.co.uk/search/period3?id=11&printlevel=2}{11}\\ \addlinespace[1.3ex] 
\rowcolor[gray]{0.95}
\hyperref[anchor:3--31]{$\MM{3}{31}$} & 52 & $x+\frac{x}{z}+\frac{x}{y}+y+z+\frac{1}{x}$ & Toric variety & \href{http://www.grdb.co.uk/search/period3?id=14&printlevel=2}{14}\\ \addlinespace[1.3ex] 
\midrule
\hyperref[anchor:4--1]{$\MM{4}{1}$} & 24 & $x^2z+2xz+x+y+z+\frac{1}{y}+\frac{y}{xz}+\frac{1}{x}+\frac{2}{xz}+\frac{1}{xyz}$ & Quantum Lefschetz & \href{http://www.grdb.co.uk/search/period3?id=111&printlevel=2}{111}\\ \addlinespace[1.3ex] 
\rowcolor[gray]{0.95}
\hyperref[anchor:4--2]{$\MM{4}{2}$} & 26 & $x+\frac{x}{z}+\frac{x}{y}+y+z+\frac{1}{z}+\frac{2}{y}+\frac{y}{x}+\frac{2}{x}+\frac{1}{xy}$ & Quantum Lefschetz with mirror map & \href{http://www.grdb.co.uk/search/period3?id=110&printlevel=2}{110}\\ \addlinespace[1.3ex] 
\hyperref[anchor:4--3]{$\MM{4}{3}$} & 28 & $\frac{x^2}{y^2z}+x+\frac{2x}{y}+y+z+\frac{2y}{x}+\frac{1}{x}+\frac{y}{x^2}$ & Quantum Lefschetz & \href{http://www.grdb.co.uk/search/period3?id=88&printlevel=2}{88}\\ \addlinespace[1.3ex] 
\rowcolor[gray]{0.95}
\hyperref[anchor:4--4]{$\MM{4}{4}$} & 30 & $x+y+z+\frac{1}{z}+\frac{z}{y}+\frac{2}{y}+\frac{1}{yz}+\frac{y}{x}+\frac{1}{x}$ & Quantum Lefschetz & \href{http://www.grdb.co.uk/search/period3?id=83&printlevel=2}{83}\\ \addlinespace[1.3ex] 
\hyperref[anchor:4--5]{$\MM{4}{5}$} & 32 & $x+\frac{x}{z}+y+z+\frac{1}{y}+\frac{y}{x}+\frac{2}{x}+\frac{1}{xy}$ & Quantum Lefschetz & \href{http://www.grdb.co.uk/search/period3?id=68&printlevel=2}{68}\\ \addlinespace[1.3ex] 
\rowcolor[gray]{0.95}
\hyperref[anchor:4--6]{$\MM{4}{6}$} & 32 & $x+y+\frac{y}{z}+z+\frac{1}{z}+\frac{z}{y}+\frac{1}{y}+\frac{y}{x}+\frac{1}{x}$ & Quantum Lefschetz with weak Fano ambient & \href{http://www.grdb.co.uk/search/period3?id=81&printlevel=2}{81}\\ \addlinespace[1.3ex] 
\hyperref[anchor:4--7]{$\MM{4}{7}$} & 34 & $x+\frac{x}{y}+y+z+\frac{1}{y}+\frac{z}{x}+\frac{2}{x}+\frac{1}{xz}$ & Quantum Lefschetz & \href{http://www.grdb.co.uk/search/period3?id=65&printlevel=2}{65}\\ \addlinespace[1.3ex] 
\rowcolor[gray]{0.95}
\hyperref[anchor:4--8]{$\MM{4}{8}$} & 36 & $x+y+z+\frac{1}{z}+\frac{z}{y}+\frac{1}{y}+\frac{1}{x}+\frac{1}{xz}$ & Quantum Lefschetz & \href{http://www.grdb.co.uk/search/period3?id=57&printlevel=2}{57}\\ \addlinespace[1.3ex] 
\hyperref[anchor:4--9]{$\MM{4}{9}$} & 38 & $xy+x+y+z+\frac{1}{y}+\frac{2}{x}+\frac{1}{x^2z}$ & Quantum Lefschetz & \href{http://www.grdb.co.uk/search/period3?id=54&printlevel=2}{54}\\ \addlinespace[1.3ex] 
\rowcolor[gray]{0.95}
\hyperref[anchor:4--10]{$\MM{4}{10}$} & 40 & $xy+x+y+z+\frac{1}{y}+\frac{1}{yz}+\frac{1}{x}$ & Toric variety & \href{http://www.grdb.co.uk/search/period3?id=37&printlevel=2}{37}\\ \addlinespace[1.3ex] 
\hyperref[anchor:4--11]{$\MM{4}{11}$} & 42 & $xy+x+y+z+\frac{1}{z}+\frac{1}{y}+\frac{1}{x}$ & Product & \href{http://www.grdb.co.uk/search/period3?id=48&printlevel=2}{48}\\ \addlinespace[1.3ex] 
\rowcolor[gray]{0.95}
\hyperref[anchor:4--12]{$\MM{4}{12}$} & 44 & $xy+x+\frac{x}{z}+y+z+\frac{1}{y}+\frac{1}{x}$ & Toric variety & \href{http://www.grdb.co.uk/search/period3?id=34&printlevel=2}{34}\\ \addlinespace[1.3ex] 
\hyperref[anchor:4--13]{$\MM{4}{13}$} & 46 & $xy+\frac{xy}{z}+x+y+z+\frac{1}{y}+\frac{1}{x}$ & Toric variety & \href{http://www.grdb.co.uk/search/period3?id=29&printlevel=2}{29}\\ \addlinespace[1.3ex] 
\midrule
\rowcolor[gray]{0.95}
\hyperref[anchor:5--1]{$\MM{5}{1}$} & 28 & $x+\frac{x}{z}+\frac{x}{y}+y+z+\frac{2}{y}+\frac{y}{x}+\frac{2}{x}+\frac{1}{xy}$ & Quantum Lefschetz with weak Fano ambient & \href{http://www.grdb.co.uk/search/period3?id=100&printlevel=2}{100}\\ \addlinespace[1.3ex] 
\hyperref[anchor:5--2]{$\MM{5}{2}$} & 36 & $x+\frac{x}{z}+\frac{x}{y}+y+z+\frac{1}{y}+\frac{y}{x}+\frac{1}{x}$ & Toric variety & \href{http://www.grdb.co.uk/search/period3?id=64&printlevel=2}{64}\\ \addlinespace[1.3ex] 
\rowcolor[gray]{0.95}
\hyperref[anchor:5--3]{$\MM{5}{3}$} & 36 & $x+y+\frac{y}{z}+z+\frac{1}{z}+\frac{z}{y}+\frac{1}{y}+\frac{1}{x}$ & Product & \href{http://www.grdb.co.uk/search/period3?id=76&printlevel=2}{76}\\ \addlinespace[1.3ex] 
\midrule
\hyperref[anchor:6--1]{$\MM{6}{1}$} & 30 & $x+\frac{x}{y}+y+z+\frac{1}{z}+\frac{2}{y}+\frac{y}{x}+\frac{2}{x}+\frac{1}{xy}$ & Product & \href{http://www.grdb.co.uk/search/period3?id=107&printlevel=2}{107}\\ \addlinespace[1.3ex] 
\midrule
\rowcolor[gray]{0.95}
\hyperref[anchor:7--1]{$\MM{7}{1}$} & 24 & $x+yz^2+2yz+y+2z+\frac{2}{z}+\frac{1}{y}+\frac{2}{yz}+\frac{1}{yz^2}+\frac{1}{x}$ & Product & \href{http://www.grdb.co.uk/search/period3?id=136&printlevel=2}{136}\\ \addlinespace[1.3ex] 
\midrule
\hyperref[anchor:8--1]{$\MM{8}{1}$} & 18 & $x+yz^3+3yz^2+3yz+y+3z+\frac{3}{z}+\frac{3}{yz}+\frac{3}{yz^2}+\frac{1}{y^2z^3}+\frac{1}{x}$ & Product & \href{http://www.grdb.co.uk/search/period3?id=155&printlevel=2}{155}\\ \addlinespace[1.3ex] 
\midrule
\rowcolor[gray]{0.95}
\hyperref[anchor:9--1]{$\MM{9}{1}$} & 12 & $xz^4+4xz^3+6xz^2+4xz+x+y+4z^2+12z+\frac{4}{z}+\frac{1}{y}+\frac{6}{x}+\frac{12}{xz}+\frac{6}{xz^2}+\frac{4}{x^2z^2}+\frac{4}{x^2z^3}+\frac{1}{x^3z^4}$ & Product & n/a\\ \addlinespace[1.3ex] 
\midrule
\hyperref[anchor:10--1]{$\MM{10}{1}$} & 6 & $xz^6+6xz^5+15xz^4+20xz^3+15xz^2+6xz+x+y+6z^3+30z^2+60z+\frac{30}{z}+\frac{6}{z^2}+\frac{1}{y}+\frac{15}{x}+\frac{60}{xz}+\frac{90}{xz^2}+\frac{60}{xz^3}+\frac{15}{xz^4}+\frac{20}{x^2z^3}+\frac{60}{x^2z^4}+\frac{60}{x^2z^5}+\frac{20}{x^2z^6}+\frac{15}{x^3z^6}+\frac{30}{x^3z^7}+\frac{15}{x^3z^8}+\frac{6}{x^4z^9}+\frac{6}{x^4z^{10}}+\frac{1}{x^5z^{12}}$ & Product & n/a\\ \addlinespace[1.3ex] 
\end{longtable}
\end{center}

\end{landscape}


\end{document}